\documentclass[11pt,a4paper]{article}
\usepackage{amsmath,amsfonts,amssymb,amsthm}
\usepackage{array,enumitem,fullpage,graphicx,mathrsfs,color,pstricks,bm}
\usepackage[symbols,nogroupskip, nomain, numberedsection=autolabel]{glossaries-extra}

\newtheorem{theorem}{Theorem}[section]
\newtheorem{thm}[theorem]{Theorem}
\newtheorem{lemma}[theorem]{Lemma}
\newtheorem{lem}[theorem]{Lemma}

\newtheorem{prop}[theorem]{Proposition}

\newtheorem{cor}[theorem]{Corollary}

\newtheorem{dfn}[theorem]{Definition}
\newtheorem{rem}[theorem]{Remark}

\newtheorem{assu}{Assumption}

\def\U{\mathcal{U}}
\def\sU{\mathcal{U}}
\def\dU{d_{\U}}
\def\mU{\mu_{\U}}
\def\pU{\phi_{\U}}
\def\rU{\rho_{\U}}
\def\BU{B_{\U}}
\newcommand\bP{\mathbf{P}}

\def\be{\begin{equation}}
\def\ee{\end{equation}}
\def\eps{\varepsilon}
\renewcommand{\P}{\mathbb{P}}
\newcommand{\E}{\mathbb{E}}
\renewcommand{\S}{\mathbb{S}}
\newcommand\N{\mathbb{N}}
\newcommand\Z{\mathbb{Z}}
\newcommand\R{\mathbb{R}}
\newcommand\T{\mathcal{T}}
\newcommand\cH{\mathcal{H}}
\def\uT{\underline{\mathcal{T}}}
\def\dT{d_{\T}}
\def\mT{\mu_{\T}}
\def\pT{\phi_{\T}}
\def\rT{\rho_{\T}}

\newcommand{\interior}[1]{%
  {\kern0pt#1}^{\mathrm{o}}%
}

\date{}

\DeclareMathOperator{\LE}{LE}
\DeclareMathOperator{\diam}{diam}
\DeclareMathOperator{\dist}{dist}
\DeclareMathOperator{\tr}{tr}
\DeclareMathOperator{\len}{len}
\DeclareMathOperator{\QL}{QL}

\makenoidxglossaries
\newglossaryentry{beta}
{
    name=\ensuremath{\beta},
    description={Growth exponent for the three-dimensional loop-erased random walk.},
    type=symbols,
    sort={0}
}

\newglossaryentry{ust}
{
    name=\ensuremath{\U},
    description={Uniform spanning tree on $\mathbb{Z}^3$.},
    type=symbols,
    sort={0}
}

\newglossaryentry{boldP}
{
    name=\ensuremath{\bP},
    description={Law of the uniform spanning tree $\U$.},
    type=symbols,
    sort={0}
}

\newglossaryentry{distanceU}
{
    name=\ensuremath{\dU},
    description={Intrinsic metric on the graph $\U$.},
    type=symbols,
    sort={0}
}
\newglossaryentry{ballU}
{
    name=\ensuremath{B_{\U} (x, r)},
    description={Ball in the metric space $\U$ of radius $r$ around $x$.},
    type=symbols,
    sort={0}
}

\newglossaryentry{muU}
{
    name=\ensuremath{\mu_\U},
    description={Counting measure on $\U$.},
    type=symbols,
    sort={0}
}

\newglossaryentry{phiU}
{
    name=\ensuremath{\phi_\U},
    description={Continuous embedding of $\U$ into $\mathbb{R}^3$.},
    type=symbols,
    sort={0}
}

\newglossaryentry{rhoU}
{
    name=\ensuremath{\rU},
    description={Root of $\U$. It is equal to $0$.},
    type=symbols,
    sort={0}
}

\newglossaryentry{boldPd}
{
    name=\ensuremath{\bP_\delta},
    description={Law of $ \left(\U, \delta^\beta \dU, \delta^3 \mU, \delta\pU, \rU \right) $.},
    type=symbols,
    sort={0}
}

\newglossaryentry{hausdorffDimT}
{
    name=\ensuremath{d_f},
    description={Fractal dimension of $\U$.},
    type=symbols,
    sort={0}
}

\newglossaryentry{Tlimit}
{
    name=\ensuremath{\T},
    description={Limit metric space of the scaled uniform spanning tree $\U$.},
    type=symbols,
    sort={0}
}

\newglossaryentry{Plimit}
{
    name=\ensuremath{\hat{\bP}},
    description={Law of the limit space $(\T,d_\T,\mu_\T,\phi_\T,\rho_\T)$.},
    type=symbols,
    sort={0}
}

\newglossaryentry{B-T}
{
    name=\ensuremath{B_\T(x,r)},
    description={Ball in the metric space $\T$ of radius $r$ around $x$.},
    type=symbols,
    sort={0}
}

\newglossaryentry{srwU}
{
    name=\ensuremath{X^\sU},
    description={Simple random walk on $\U$.},
    type=symbols,
    sort={0}
}

\newglossaryentry{quenchedP}
{
    name=\ensuremath{{P}_x^\sU},
    description={Quenched law of the simple random walk on  $\U$ started at $x$.},
    type=symbols,
    sort={0}
}

\newglossaryentry{annealedP}
{
    name=\ensuremath{\P^\U},
    description={Annealed law of the simple random walk on $\U$.},
    type=symbols,
    sort={0}
}

\newglossaryentry{resistanceU}
{
    name=\ensuremath{R_\U},
    description={Effective resistance on $\U$.},
    type=symbols,
    sort={0}
}

\newglossaryentry{hausdorffD}
{
    name=\ensuremath{d_H},
    description={Hausdorff metric.},
    type=symbols,
    sort={0}
}

\newglossaryentry{discreteBall}
{
    name=\ensuremath{B (x, r)},
    description={Discrete Euclidean ball of radius $R$ around $x$.},
    type=symbols,
    sort={0}
}

\newglossaryentry{scaledBall}
{
    name=\ensuremath{B_{\delta} (x, r)},
    description={Discrete Euclidean ball on $\delta \mathbb{Z}^3$ of radius $r$ around $x$.},
    type=symbols,
    sort={0}
}

\newglossaryentry{euclideanBall}
{
    name=\ensuremath{B_E (x, r)},
    description={Euclidean ball of radius $r$ around $x$.},
    type=symbols,
    sort={0}
}

\newglossaryentry{discreteCube}
{
    name=\ensuremath{D(x,r)},
    description={Discrete cube of side-length $2r$ centred at $x$.},
    type=symbols,
    sort={0}
}

\newglossaryentry{scaledCube}
{
    name=\ensuremath{D_{\delta} (x,r)},
    description={Discrete cube on $\delta \mathbb{Z}^3$ of side-length $2r$ centred at $x$.},
    type=symbols,
    sort={0}
}

\newglossaryentry{euclideanCube}
{
    name=\ensuremath{D_E (x,r)},
    description={Euclidean cube of side-length $2r$ centred at $x$.},
    type=symbols,
    sort={0}
}

\newglossaryentry{length}
{
    name=\ensuremath{\len(\gamma)},
    description={Length of a path $\gamma$.},
    type=symbols,
    sort={0}
}

\newglossaryentry{duration}
{
    name=\ensuremath{T (\gamma) },
    description={Duration of a curve $\gamma$.},
    type=symbols,
    sort={0}
}

\newglossaryentry{spaceCf}
{
    name=\ensuremath{ \mathcal{C}_f },
    description={Space of parameterized curves of finite duration.},
    type=symbols,
    sort={0}
}

\newglossaryentry{psi}
{
    name=\ensuremath{ \psi },
    description={Metric on $\mathcal{C}_f$.},
    type=symbols,
    sort={0}
}

\newglossaryentry{spaceC}
{
    name=\ensuremath{ \mathcal{C} },
    description={Space of transient parameterized curves.},
    type=symbols,
    sort={0}
}

\newglossaryentry{chi}
{
    name=\ensuremath{ \chi },
    description={Metric on $\mathcal{C}$.},
    type=symbols,
    sort={0}
}

\newglossaryentry{schrammMetric}
{
    name=\ensuremath{ d_{\gamma}^S },
    description={Schramm metric on a parameterized curve $\gamma$.},
    type=symbols,
    sort={0}
}

\newglossaryentry{intrinsicMetric}
{
    name=\ensuremath{ d_{\gamma} },
    description={Intrinsic metric on a parameterized curve $\gamma$.},
    type=symbols,
    sort={0}
}

\newglossaryentry{betaParameterization}
{
    name=\ensuremath{\bar{\gamma} },
    description={Loop-erased random walk endowed with its $\beta$-parameterization.},
    type=symbols,
    sort={0}
}

\newglossaryentry{ILERW}
{
    name=\ensuremath{\gamma^x_{\infty}},
    description={Infinite loop-erased random walk on $\mathbb{Z}^3$ starting at $x$.},
    type=symbols,
    sort ={0}
}

\newglossaryentry{dyadic}
{
    name=\ensuremath{\mathcal{P}},
    description={Dyadic polyhedron.},
    type=symbols,
    sort ={0}
}

\newglossaryentry{parameterizedTree}
{
    name=\ensuremath{\mathscr{T}},
    description={Parameterized tree.},
    type=symbols,
    sort ={0}
}

\newglossaryentry{parameterizedForest}
{
    name=\ensuremath{\mathscr{F}^K},
    description={Space of parameterized trees with $K$ leaves.},
    type=symbols,
    sort ={0}
}

\newglossaryentry{essentialBranches}
{
    name=\ensuremath{\Gamma^e (\mathscr{T})},
    description={Space of parameterized trees with $K$ leaves.},
    type=symbols,
    sort ={0}
}

\newglossaryentry{innerB}
{
    name=\ensuremath{\partial_{i} A},
    description={Inner boundary of $A \subset \mathbb{Z}^3$.},
    type=symbols,
    sort ={0}
}

\begin{document}

\title{\vspace{-10pt}Scaling limits of the three-dimensional uniform spanning tree\\and associated random walk}
\author{O.\ Angel\footnote{Department of Mathematics, University of British Columbia, Vancouver, BC, V6T 1Z2, Canada. Email: angel@math.ubc.ca.}, D.~A.~Croydon\footnote{Research Institute for Mathematical Sciences, Kyoto University, Kyoto 606-8502, Japan. Email: croydon@kurims.kyoto-u.ac.jp.}, S. Hernandez-Torres\footnote{Faculty of Mathematics and Faculty of Industrial Engineering and Management, Technion -- Israel Institute of Technology, Haifa 32000, Israel. Email: sarai.h@campus.technion.ac.il.} and D.~Shiraishi\footnote{Department of Advanced Mathematical Sciences, Graduate School of Informatics,
Kyoto University, Kyoto 606-8501, Japan. Email: shiraishi@acs.i.kyoto-u.ac.jp.}}

\maketitle

\vspace{-30pt}

\begin{abstract}
  We show that the law of the three-dimensional uniform spanning tree (UST) is tight under rescaling in a space whose elements are measured, rooted real trees, continuously embedded into Euclidean space. We also establish that the relevant laws actually converge along a particular scaling sequence. The techniques that we use to establish these results are further applied to obtain various properties of the intrinsic metric and measure of any limiting space, including showing that the Hausdorff dimension of such is given by $3/\beta$, where $\beta\approx 1.624\dots$ is the growth exponent of three-dimensional loop-erased random walk. Additionally, we study the random walk on the three-dimensional uniform spanning tree, deriving its walk dimension (with respect to both the intrinsic and Euclidean metric) and its spectral dimension, demonstrating the tightness of its annealed law under rescaling, and deducing heat kernel estimates for any diffusion that arises as a scaling limit.
\end{abstract}

\vspace{-10pt}

\renewcommand{\baselinestretch}{0.5}\normalsize
\setcounter{tocdepth}{1}
\tableofcontents
\renewcommand{\baselinestretch}{1.0}\normalsize

\section{Introduction}

Remarkable progress has been made in understanding the scaling limits of two-dimensional statistical mechanics models in recent years, much of which has depended in a fundamental way on the asymptotic conformal invariance of the models in question that has allowed many powerful tools from complex analysis to be harnessed. See \cite{LSW,Schramm,Smirnov} for some of the seminal works in this area, and \cite{Lawler} for more details. By contrast, no similar foothold for studying analogous problems in the (physically most relevant) case of three dimensions has yet been established.
It seems that there is currently little prospect of progress for the corresponding models in this dimension.

Nonetheless, in \cite{Kozma}, Kozma made the significant step of establishing the existence of a (subsequential) scaling limit for the trace of a three-dimensional loop-erased random walk (LERW). Moreover, in work that builds substantially on this, the time parametrisation of the LERW has been incorporated into the picture, with it being demonstrated that (again subsequentially) the three-dimensional LERW converges as a stochastic process, see \cite{LS} and the related articles \cite{Escape, S}.
The aim of this work is to apply the latter results in conjunction with the fundamental connection between uniform spanning trees (USTs) and LERWs -- specifically that paths between points in USTs are precisely LERWs \cite{Pemantle, Wilson} -- to determine the scaling behaviour of the three-dimensional UST (see Figure \ref{3dustfig}) and the associated random walk.

\begin{figure}[!b]
\begin{center}
\begin{tabular}{m{.47\textwidth} m{.47\textwidth}}
  \includegraphics[width=0.47\textwidth]{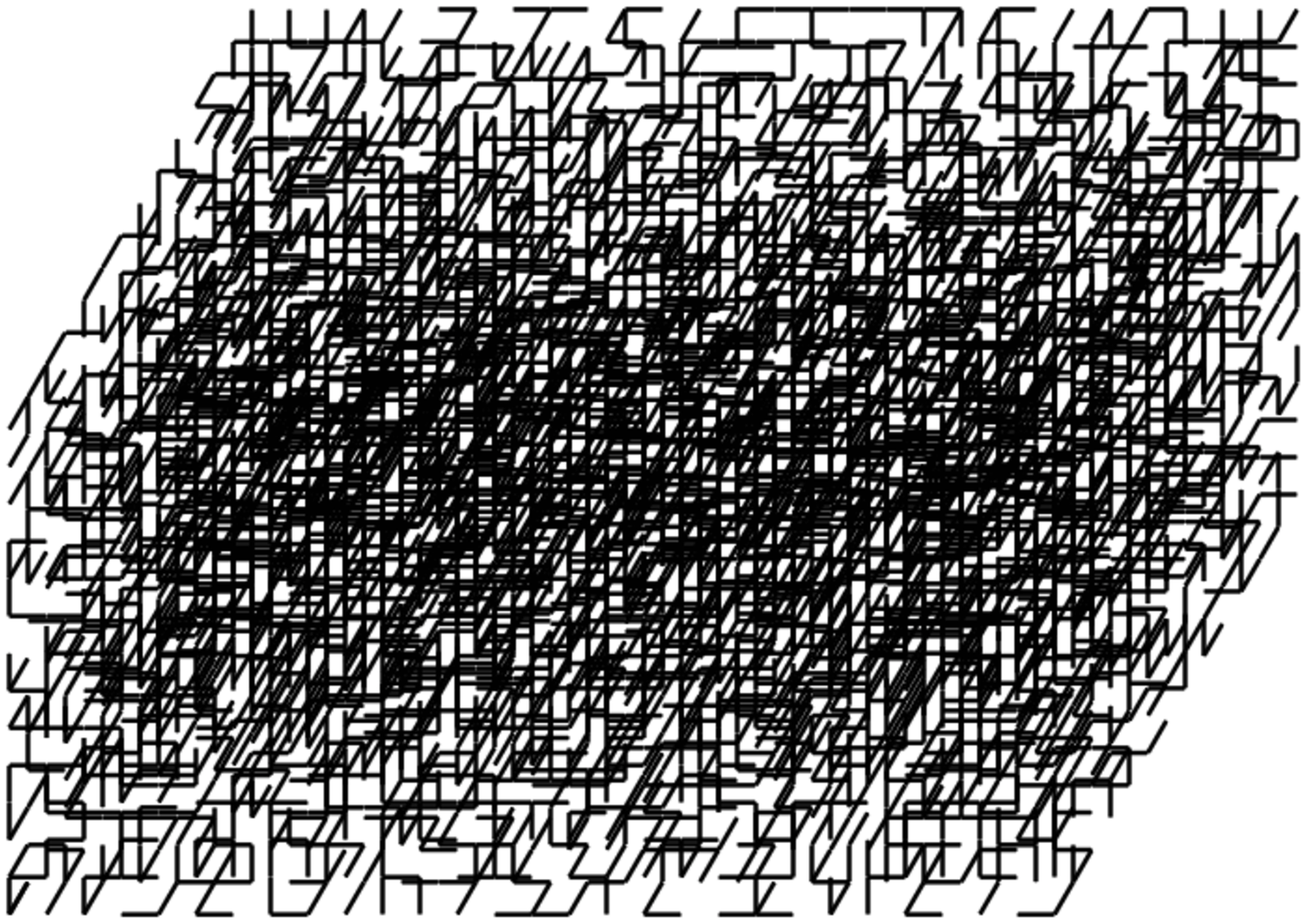} & \includegraphics[width=0.47\textwidth]{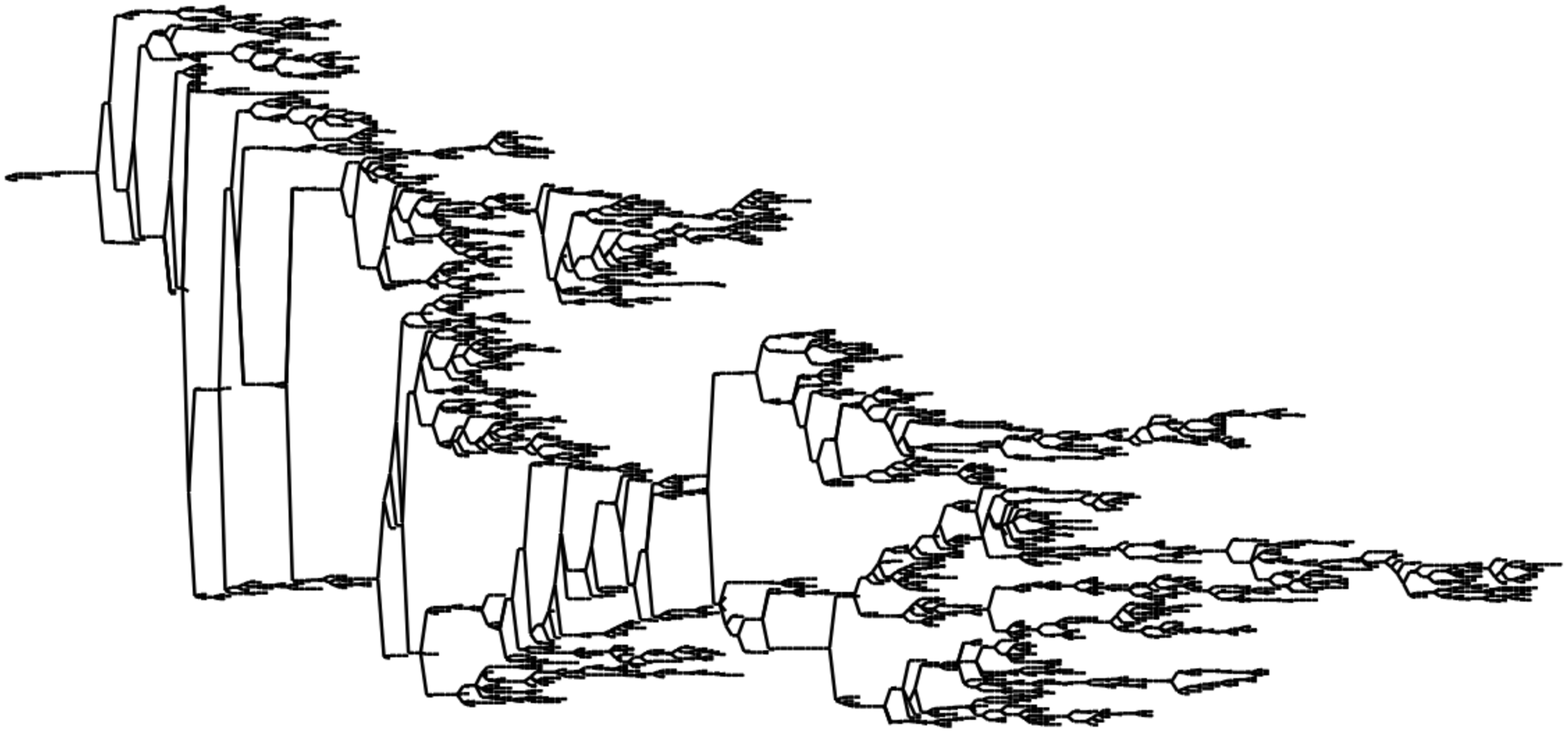}
  \end{tabular}
  \vspace{-15pt}
\end{center}
\caption{A realisation of the UST in a three-dimensional box, as embedded into $\R^3$ (left), and drawn as a planar graph tree (right). Source code adapted from two-dimensional version of Mike Bostock.}\label{3dustfig}
\end{figure}

Before stating our results, let us introduce some of our notation. For the reader's convenience, we include a list of notation in Appendix~\ref{symbols}. We follow closely the presentation of \cite{BCroyK}, where similar results were obtained in the two-dimensional case. Henceforth, we will write $\gls{ust}$ for the UST on $\Z^3$, and $\gls{boldP}$ the probability measure on the probability space on which this is built (the corresponding expectation will be denoted $\mathbf{E}$). We refer the reader to \cite{Pemantle} for Pemantle's construction of $\U$ in terms of a local limit of the USTs on the finite boxes $[-n,n]^3\cap\Z^3$ (equipped with nearest-neighbour bonds) as $n\rightarrow\infty$, and proof of the fact that the resulting graph is indeed a spanning tree of $\Z^3$. We will denote by $\gls{distanceU}$ the intrinsic (shortest path) metric on the graph $\U$, and $\gls{muU}$ the counting measure on $\U$ (i.e., the measure which places a unit mass at each vertex). Similarly to \cite{BCroyK}, in describing a scaling limit for $\U$, we will view $\U$ as a measured, rooted spatial tree. In particular, in addition to the metric measure space $(\U, \dU,\mu_\U)$, we will also consider the embedding $\gls{phiU} :\U\rightarrow\R^3$, which we take to be simply the identity on vertices; this will allow us to retain information about $\U$ in the Euclidean topology. Moreover, it will be convenient to suppose the space $(\U, \dU)$ is rooted at the origin of $\R^3$, which we will write as $\gls{rhoU}$. To fit the framework of \cite{BCroyK}, we extend $(\U,\dU)$ by adding unit line segments along edges, and linearly interpolate $\phi_\U$ between vertices.

\subsection{Scaling limits of the three-dimensional UST}

We have defined a random quintuplet $(\U, \dU,\mu_\U,\phi_\U, \rU)$. Our main result (Theorem \ref{mainthm1} below) is the existence of a certain subsequential scaling limit for this object in an appropriate Gromov-Hausdorff-type topology, the precise definition of which we postpone to Section \ref{topsec}. Moreover, the result incorporates the statement that the laws of the rescaled objects are tight even without taking the subsequence. One further quantity needed to state the result precisely is the growth exponent of the three-dimensional LERW. Let $M_n$ be the number of steps of the LERW on $\Z^3$ until its first exit from a ball of radius $n$. The \textbf{growth exponent} is defined by the limit:
\[ \gls{beta}:=\lim_{n\rightarrow\infty}\frac{\log\mathbf{E}M_n}{\log n},\]
(equivalently, $\mathbf{E}M_n = n^{\beta+o(1)}$). The existence of this limit was proved in \cite{S}. Whilst the exact value of $\beta$ is not known, rigourously proved bounds are $\beta\in(1,\frac{5}{3}]$, see \cite{LLERW}. Numerical estimates suggest that $\beta=1.624\dots$, see \cite{Wilest}. We remark that in two dimensions the corresponding exponent is $5/4$, first proved by Kenyon \cite{Kenyon}, and in dimension 4 or more its value is $2$. In three dimensions there is no conjecture for an exact value of $\beta$.

The exponent $\beta$ determines the scaling of $\dU$. Specifically, let $\gls{boldPd}$ be the law of the measured, rooted spatial tree
\begin{equation}\label{uddef}
  \left(\U, \delta^\beta \dU, \delta^3 \mU, \delta\pU, \rU \right),
\end{equation}
when $\U$ has law $\bP$.
For the rooted measured metric space $(\U,\dU,\mU,\rho)$ we consider the local Gromov-Hausdorff-Prohorov topology. This is extended with the locally uniform topology for the embedding $\pU$. As a straightforward consequence of our tightness and scaling results with respect to this Gromov-Hausdorff-type topology, we also obtain the corresponding conclusions with respect to Schramm's path ensemble topology. The latter topology was introduced in \cite{Schramm} as an approach to taking scaling limits of two-dimensional spanning trees. Roughly speaking this topology observes the set of all macroscopic paths in an object, in the Hausdorff topology. See Section \ref{topsec} for detailed definitions of these topologies.

\begin{thm}\label{mainthm1}
  The collection $(\bP_\delta)_{\delta\in(0,1]}$ is tight with respect to the local Gromov-Hausdorff-Prohorov topology with locally uniform topology for the embedding, and with respect to the path ensemble topology.  Moreover the limit of $\bP_\delta$ exists as $\delta=2^{-n}\to0$ exists in both topologies.
\end{thm}

\begin{rem}\label{extension}
  The reason that we only state convergence along the subsequence $(2^{-n})_{n\geq 0}$ stems from the fact that our argument fundamentally depends on the one-point function estimates for three-dimensional LERW from \cite{Escape}. Indeed, although the same subsequential restriction was also present in Kozma's original work on the scaling of three-dimensional LERW \cite{Kozma}, (on which we also rely heavily,) as was helpfully explained to us by a referee, this restriction can be removed by a slight rearrangement of the argument of \cite{Kozma}. That the latter is the case should allow the extension of the scaling limit of Theorem \ref{mainthm1} to an arbitrary sequence of $\delta$s with respect to the path ensemble topology. However, we choose to not include this here since the deficiency with respect to the Gromov-Hausdorff-type topology remains. We highlight that there is no reason to believe that taking a certain subsequence of $\delta$s is an essential requirement, and if one could extend Li and Shiraishi's work from \cite{Escape} to arbitrary sequence of $\delta$s, then the corresponding extension of our result would also follow.
\end{rem}

\begin{rem}
  An important open problem, for both the LERW and UST in three dimensions, is to describe the limiting object directly in the continuum.
  In two dimensions, there are connections between the LERW and SLE$_2$, as well as between the UST and SLE$_8$, see \cite{HS,LSW,Schramm}, which give a direct construction of the continuous objects.
  In the three-dimensional case, there is as yet no parallel theory.
  The development of such a representation would be a significant advance in three-dimensional statistical mechanics.
\end{rem}

Before continuing, we briefly outline the strategy of proof for the convergence part of the above result, for which there are two main elements. The first of these is a finite-dimensional convergence statement: Theorem \ref{thm:convergecetrees} states that the part of $\U$ spanning a finite collection of points converges under rescaling. Appealing to Wilson's algorithm \cite{Wilson}, which gives the means to construct $\U$ from LERW paths, this finite-dimensional result extends the scaling result for the three-dimensional LERW of \cite{LS}. Here we encounter a central hurdle: after the first walk, Wilson's algorithm requires us to take a LERW in an rough subdomain of $\Z^3$, namely the complement of the previous LERWs. Existing results in \cite{Kozma,LS} on scaling limits of LERWs require subdomains with smooth boundary, and some care is needed to extend the existence of the scaling limit. We resolve this difficulty by proving that we can approximate the rough subdomain with a simpler one, and showing the corresponding LERWs are close to each other as parametrized curves.

Secondly, to prove tightness, we need to check that the trees spanning a finite collection of points give a sufficiently good approximation of the entire UST $\U$, once the number of points is large. For this, we need to know that LERWs started from the remaining lattice points hit the trees spanning a finite collection of points quickly. In two dimensions, such a property was established using Beurling's estimate, which says that a simple random walk hits any given path quickly if it starts close to it in Euclidean terms, see \cite{Kesten1987}. In three dimensions, Beurling's estimate does not hold. In its place, we have a result from \cite{SS}, which yields that a simple random walk hits a typical LERW path quickly if it starts close to it. Thus, although the intuition in the three-dimensional case is similar, it requires us to remember much more about the structure of the part of the UST we have already constructed as Wilson's algorithm proceeds.

\subsection{Properties of the scaling limit}

While uniqueness of the scaling limit is as yet unproved, the techniques we use to establish Theorem \ref{mainthm1} allow us to deduce some properties of any possible scaling limit. These are collected below. NB.\ For the result, the scaling limits we consider are with respect to the Gromov-Hausdorff-type topology on the space of measured, rooted spatial trees, see Section \ref{topsec} below. The one-endedness of the limiting space matches the corresponding result in the discrete case, \cite[Theorem 4.3]{Pemantle}. We use $\gls{B-T}$ to denote the ball in the limiting metric space $\gls{Tlimit} =(\T,d_\T)$ of radius $r$ around $x$.
 It is natural to expect that the scaling limit will have dimension
\[\gls{hausdorffDimT} := \frac{3}{\beta}.\]
Moreover, one would expect that a ball of radius $r$ in the limiting object has measure of order $r^{3/\beta}$. The following theorem establishes uniform bounds of this magnitude for all small balls in the limiting tree, with a logarithmic correction for arbitrary centres and with iterated logarithmic corrections for a fixed centre, which may be fixed to be $\rho$. We use $f \preceq g$ to denote that $f\leq Cg$ for some absolute (i.e.\ deterministic, and not depending on the particular subsequence) constant $C$. We denote by $\gamma_\T(x,y)$ the path in the topological tree $\T$ between points $x$ and $y$. We write $\mathcal{L}$ to represent Lebesgue measure on $\R^3$. The definition of the `Schramm distance' below is inspired by \cite[Remark 10.15]{Schramm}.

\begin{thm}\label{mainthm2}
Let $\gls{Plimit}$ be a subsequential limit of $\bP_\delta$ as $\delta\to0$, and the random  measured, rooted spatial tree $(\T,d_\T,\mu_\T,\phi_\T,\rho_\T)$ have law $\hat\bP$. Then the following statements hold $\hat\bP$-a.s.
  \begin{enumerate}[nosep,label=(\alph*)]
  \item The tree $\T$ is one-ended (with respect to the topology induced by the metric $d_\T$).
  \item Every open ball in $(\T,d_\T)$ has Hausdorff dimension $d_f$.
  \item There exists an absolute constant $C<\infty$ so that: for any $R>0$, there exists a random $r_0(\T)>0$ such that
    \[r^{d_f}(\log r^{-1})^{-C}
      \preceq \inf_{x\in B_\T(\rho,R)}\mu_{\T}\left(B_\T(x,r)\right)\leq\sup_{x\in B_\T(\rho,R)}\mu_{\T}\left(B_\T(x,r)\right)
      \preceq r^{d_f}(\log r^{-1})^{C},\]
      for all $r<r_0$.
  \item For some absolute $C<\infty$, there exists a random $r_0(\T)>0$ such that
    \[r^{d_f}(\log\log r^{-1})^{-C}
      \preceq \mu_\T\left(B_\T(\rho,r)\right)
      \preceq r^{d_f}(\log\log r^{-1})^{C},\qquad\forall r<r_0.\]
  \item The metric $d$ is equivalent to the `Schramm metric' $d_\T^S$ on $\T$,
    defined by
    \begin{equation}\label{dtschramm}
      d_\T^S(x,y) := \diam\left(\phi_\T(\gamma_\T(x,y))\right),
    \end{equation}
    where $\diam$ is the diameter in the Euclidean metric.
  \item $\mu_\mathcal{T}=\mathcal{L}\circ\phi_\mathcal{T}$.
  \end{enumerate}
\end{thm}

\begin{rem}\label{volrem}
To establish parts (c) and (d) of Theorem \ref{mainthm2}, we need to extend some of the estimates applied in the proof of Theorem \ref{mainthm1}. In particular, in Section \ref{sec:assump}, we derive certain probabilistic volume bounds of a polynomial form, which are what we require for our proof of tightness of $\mathcal{U}$ with respect to the Gromov-Hausdorff-type topology of interest. However, to obtain the above logarithmic/loglogarithmic error terms, we need to sharpen such volume bounds to exponential ones. By suitably extending the argument of Section \ref{sec:assump}, this is done in Sections \ref{sec: exp lower volume} and \ref{sec:expupper}. (See Theorems \ref{2nd-goal} and \ref{darui-1} for precise results in this direction.)
\end{rem}

\subsubsection{Differences from the two-dimensional case}

Analogues for the properties described in Theorem \ref{mainthm2} (and others) were proved in the two-dimensional case in \cite{BCroyK}, see also the related earlier work \cite{Schramm}. There are, however, several notable differences in three dimensions. Following Schramm \cite{Schramm}, consider the \textbf{trunk} of the tree $\T$, denoted $\T^\circ$, which is the set of all points of $\T$ of degree greater than $1$, where the degree of $x$ is the number of connected components of $\T\setminus\{x\}$. In the two-dimensional case, it is known that the restriction of the continuous map $\phi_\T$ to the trunk is a homeomorphism between $\T^\circ$ (equipped with the induced topology from $\T$) and its image $\phi_\T(\T^\circ)$ (equipped with the induced Euclidean topology). Thus the image of the trunk, which is dense in $\R^2$, determines its topology. We do not expect the same to be true in three-dimensions. Indeed, due to the greater probability that three LERWs started from adjacent points on the integer lattice escape to a macroscopic distance before colliding, we expect that the image of the trunk $\phi_\T(\T^\circ)$ is no longer a topological tree in $\R^3$, see Figure \ref{fig2}. We aim to establish this as a result in a forthcoming work.

\begin{figure}[!ht]
\begin{center}
  \includegraphics[width=0.5\textwidth]{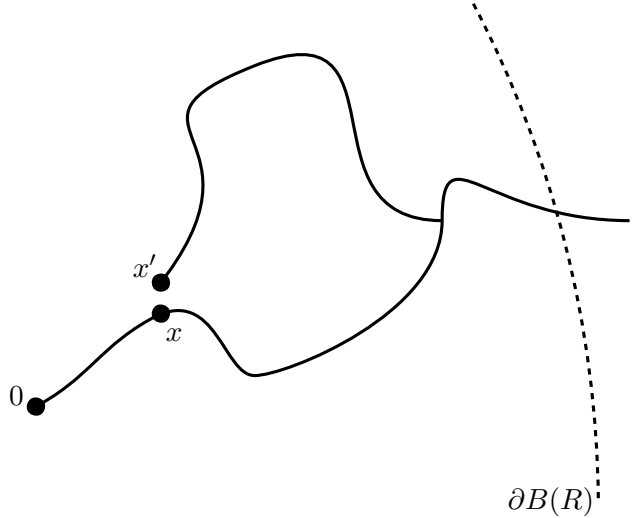}
  \rput(-234pt,40pt){$0$}
  \rput(-186pt,89pt){$x'$}
  \rput(-175pt,63pt){$x$}
  \rput(-34pt,0pt){$\partial B(R)$}
\end{center}
\caption{In the above sketch, $|x-x'|=1$, but the path in the UST between these points has Euclidean diameter greater than $R/3$. We expect that such pairs of points occur with positive probability, uniformly in $R$.}\label{fig2}
\end{figure}

Secondly, for the two-dimensional UST, it was shown in \cite{BCroyK} that the maximal degree in $\T$ is $3$, and that $\mu_\T$ is supported on the leaves of $\T$, i.e.\ the set of points of degree 1. We can show that the same is true in three dimensions, though we also postpone these results to a separate paper, since they are significantly harder than in two dimensional case.
Indeed, as well as appealing to the homeomorphism between the trunk and its embedding, the two-dimensional arguments in the literature depend on a duality argument that does not extend to three dimensions. We replace this with a more technical direct argument. The aforementioned homeomorphism and duality also allow it to be shown that in two dimensions $\max_{x\in\R^3}|\phi^{-1}(x)|=3$ (where we write $|A|$ to represent the cardinality of a set $A$), and, although not mentioned explicitly in \cite{BCroyK,Schramm}, it is also easy to deduce the Hausdorff dimension of the set of points with given pre-image size. Our forthcoming work will explore the corresponding results in the three dimensional case.

\subsection{Scaling the random walk on $\U$}

The metric-measure scaling of $\U$ yields various consequences for the associated simple random walk (SRW), which we next introduce. For a given realisation of the graph $\sU$, the SRW on $\sU$ is the discrete time Markov process $\gls{srwU}=((X^\U_n)_{n\geq0},({P}_x^\sU)_{x\in \Z^3})$ which at each time step jumps from its current location to a uniformly chosen neighbour in $\sU$. For $x\in \Z^3$, the law $\gls{quenchedP}$ is called the \textbf{quenched} law of the simple random walk on $\sU$ started at $x$. We then define the \textbf{annealed} or \textbf{averaged} law for the process started from $\rU$ as the semi-direct product of the environment law $\bP$ and the quenched law ${P}_0^\sU$ by setting
\[\gls{annealedP} \left(\cdot\right):=\int P_0^\sU(\cdot) d\bP.\]
We use $\E^\U$ for the corresponding quenched expectation.

The behaviour of the random walk on a graph is fundamentally linked to the associated electrical resistance. We refer the reader to \cite{Barlowbook,DS,LPW,LP} for introductions to this connection, including the definition of effective resistance in particular. For the three-dimensional UST, we will write $\gls{resistanceU}$ for the effective resistance on $\U$, considered as an electrical network with unit resistors placed along each edge.

As noted above, the typical measure of $B_\U(\rho,R)$ is of order $R^{d_f}$.
We show below that the effective resistance to the complement of the ball is typically of order $R$ (it is trivially at most $R$).
In light of these, and following \cite{KM}, we define the set of well-behaved scales with parameter $\lambda$ by
\[J(\lambda) := \left\{R\in[1,\infty):\:
    R^{-d_f} \mU\left(\BU(\rho,R)\right)\in [\lambda^{-1}, \lambda]
    \text{ and } R_\U\left(\rho,\BU(\rho,R)^c\right)\geq \lambda^{-1}R\right\}.\]
In particular, for $R$ to be in $J(\lambda)$, we require good control over the volume of the intrinsic ball centred at the root of $\U$ of radius $R$, and control over the resistance from the root to the boundary of this ball.
As our next result, we show that the these events hold with high probability, uniformly in $R$. (In particular, this adds the resistance estimate to the volume bounds discussed in Remark \ref{volrem}.)

\begin{thm}\label{mainthm3}
There exist constants $c, c_1,c_2\in(0,\infty)$ such that: for all $R,\lambda\geq 1$,
\[\bP\left(R\in J(\lambda)\right) \geq 1-c e^{-c_1\lambda^{c_2}}.\]
\end{thm}

The motivation for Theorem \ref{mainthm3} is provided by the general random walk estimates presented by Kumagai and Misumi in \cite{KM}. (Which builds on the work \cite{BJKS}.)
More specifically, Theorem \ref{mainthm3} establishes the conditions for the main results of \cite{KM}, which yield several important exponents governing aspects of the behaviour of the random walk.
Indeed, as is made precise in the following corollary, we obtain that the walk dimension with respect to intrinsic distance is given by
\[d_w:=1+d_f=\frac{3+\beta}{\beta},\]
the walk dimension with respect to extrinsic (Euclidean) distance $d_E$ is given by $\beta d_w=3+\beta$ (this requires a small amount of additional work to the tools of \cite{KM}), and the spectral dimension is given by
\begin{equation}\label{dsdef}
d_s:=\frac{2d_f}{d_w}=\frac{6}{3+\beta}.
\end{equation}
Various further consequences for the random walk on $\U$ also follow from the results of \cite{KM}, but rather than simply list these here, we refer the interested reader to that article for details. Table \ref{exponenttable} summarises the numerical estimates for the three-dimensional random walk exponents that follow from the above formulae, together with the numerical estimate for $\beta$ from \cite{Wilest}, and compares these with the known exponents in the two-dimensional model.

\begin{table}[t]
\begin{center}
\begin{tabular}{r|c|c|c}
  & General form & $d=2$ & $d=3$\\
  \hline
  LERW growth exponent & $\beta$ & 5/4 = 1.25 & 1.62\\
  Fractal dimension of $\sU$ & $d_f=d/\beta$ & 8/5 = 1.60 & 1.85 \\
  Intrinsic walk dimension of $\sU$ & $d_w=1+d_f$ & 13/5 = 2.60 & 2.85\\
  Extrinsic walk dimension of $\sU$ & $\beta d_w$ & 13/4 = 3.25 & 4.62\\
  Spectral dimension of $\sU$ & $2d_f/d_w$ & 16/13 = 1.23 & 1.30\\
\end{tabular}
\end{center}
\caption{Exponents associated with the LERW and UST in two and three dimensions. The two-dimensional exponents are known rigourously from \cite{BCroyK2, BCroyK, BM, Kenyon}. The three-dimensional values are based on the results of this study, together with the numerical estimate for the growth exponent of the three-dimensional LERW from \cite{Wilest}.}
\label{exponenttable}
\end{table}

{\cor\label{corsrxexp} (a) For ${\bP}$-a.e.~realisation of $\U$ and all $x\in\sU$,
\begin{equation}\label{star1}
\lim_{R\rightarrow \infty}\frac{\log E_x^\sU \tau^\sU_{x,R}}{\log R}=d_w,
\end{equation}
where $\tau^\sU_{x,R}:=\inf\{n\geq 0:\:d_\sU(x,X^\sU_n)>R\}$,
\begin{equation}\label{star2}
\lim_{R\rightarrow \infty}\frac{\log E_x^\sU \tau^E_{x,R}}{\log R}=\beta d_w,
\end{equation}
where $\tau^E_{x,R}:=\inf\{n\geq 0:\:d_E(x,X^\sU_n)>R\}$, and
\begin{equation}\label{specdim}
-\lim_{n\rightarrow \infty}\frac{2\log p^\mathcal{U}_{2n}(x,x)}{\log n}=d_s.
\end{equation}
(b) For ${\mathbb{P}}^{\sU}$-a.e.~realisation of $X^\U$,
\begin{equation}\label{star3}
\lim_{R\rightarrow \infty}\frac{\log \tau^\sU_{0,R}}{\log R}=d_w,\qquad
\lim_{n\rightarrow \infty}\frac{\log \max_{0\leq m\leq n}d_\sU(0,X^\sU_m)}{\log n}=\frac{1}{d_w},
\end{equation}
\begin{equation}\label{star4}
\lim_{R\rightarrow \infty}\frac{\log\tau^E_{0,R}}{\log R}=\beta d_w,\qquad
\lim_{n\rightarrow \infty}\frac{\log \max_{0\leq m\leq n}d_E(0,X^\sU_m)}{\log n}=\frac{1}{\beta d_w}.
\end{equation}
(c) It holds that
\begin{equation}\label{star5}
\lim_{R\rightarrow \infty}\frac{\log \mathbb{E}^{\sU}\left( \tau^\sU_{0,R}\right)}{\log R}=d_w,
\end{equation}
\begin{equation}\label{star6}
\lim_{R\rightarrow \infty}\frac{\log \mathbb{E}^{\sU}\left( \tau^E_{0,R}\right)}{\log R}=\beta d_w,
\end{equation}
where $\mathbb{E}^{\sU}$ is the expectation under $\mathbb{P}^{\sU}$, and
\begin{equation}\label{star7}
-\lim_{n\rightarrow \infty}\frac{2\log \mathbf{E}\left( p^\mathcal{U}_{2n}(0,0)\right)}{\log n}=d_s.
\end{equation}}

\begin{rem} In part (c) of the previous result, we do not provide averaged results for the distance travelled by the process up to time $n$ with respect to either the intrinsic or extrinsic metrics. In the two-dimensional case, the corresponding results were established in \cite{BCroyK2}, with the additional input being full off-diagonal annealed heat kernel estimates. Since the latter require a substantial amount of additional work, we leave deriving such as an open problem.
\end{rem}

Finally, it is by now well-understood how scaling limits of discrete trees transfer to scaling limits for the associated random walks on the trees, see \cite{ALW, BCroyK, CAIHP, Crange, Csl, Cres}.
We apply these techniques in our setting to deduce a (subsequential) scaling limit for $X^\U$.
As we will explain in Section \ref{srwsec}, the limiting process can be written as $(\phi_\T(X^\T_t))_{t\geq 0}$, where $((X^\T_t)_{t\geq 0},(P^\T_x)_{x\in\T})$ is the canonical Brownian motion on the limit space $(\T, d_\T ,\mu_\T )$.
This Brownian motion is constructed in \cite{Kden,AEW}.
Moreover, the volume estimates of Theorem \ref{mainthm2}, in conjunction with the general heat kernel estimates of \cite{Cest}, yield sub-diffusive transition density bounds for the limiting diffusion. Modulo the different exponents, these are of the same sub-Gaussian form as established for the Brownian continuum random tree in \cite{Cvol}, and for the two-dimensional UST in \cite{BCroyK}. Note in particular that our results imply that the spectral dimension of the continuous model, defined analogously to \eqref{specdim}, is equal to the value $d_s$ given at \eqref{dsdef}.

{\thm \label{mainthm4}
If $(\bP_{\delta_n})_{n\geq 0}$ is a convergent sequence with limit $\hat{\bP}$,
then the following statements hold.
\begin{enumerate}
\item[(a)] The annealed law of $(\phi_\mathcal{T}(X^\mathcal{T}_t))_{t\geq 0}$,
where $X^\mathcal{T}$ is Brownian motion on
$(\mathcal{T},d_\mathcal{T},\mu_\mathcal{T})$ started from $\rho_\mathcal{T}$, i.e.~
\[{\mathbb{P}}^\T\left(\cdot\right):=\int P_{\rho_\mathcal{T}}^\mathcal{T}\circ\phi_\mathcal{T}^{-1}(\cdot) d\hat{\bP},\]
is a well-defined probability measure on $C(\R_+,\R^3)$.
\item[(b)] Let $(X^\U_t)_{t\geq 0}$ be the simple random walk on $\U$ started from $\rU$, then the annealed laws of the rescaled processes
\[\left(\delta_n X^\U_{t\delta_n^{-(3+\beta)}}\right)_{t\geq 0}\]
converge to the annealed law of  $(\phi_\mathcal{T}(X^\mathcal{T}_t))_{t\geq 0}$.
\item[(c)] $\hat{\bP}$-a.s., the process $X^\mathcal{T}$ is recurrent and admits a jointly continuous transition density $(p^\mathcal{T}_t(x, y))_{x,y\in\mathcal{T},t>0}$. Moreover, it $\hat{\bP}$-a.s.\ holds that, for any $R >0$, there exist random constants $c_1(\mathcal{T}),c_2(\mathcal{T}),c_3(\mathcal{T}),c_4(\mathcal{T})$ and $t_0(\mathcal{T} ) \in(0,\infty)$ and deterministic constants $\theta_1,\theta_2,\theta_3,\theta_4\in(0,\infty)$ (not depending on $R$) such that
\[p_t^\mathcal{T}(x,y)\leq c_1t^{-d_s/2}\ell(t^{-1})^{\theta_1} \exp\left\{-c_2\left(\frac{d_\mathcal{T}(x,y)^{d_w}}{t}\right)^{\frac{1}{d_w-1}}\ell(d_\mathcal{T}(x,y)/t)^{-\theta_2}\right\},\]
\[p_t^\mathcal{T}(x,y)\geq c_3t^{-d_s/2}\ell(t^{-1})^{-\theta_3} \exp\left\{-c_4\left(\frac{d_\mathcal{T}(x,y)^{d_w}}{t}\right)^{\frac{1}{d_w-1}}\ell(d_\mathcal{T}(x,y)/t)^{\theta_4}\right\},\]
for all $x,y\in B_\mathcal{T}(\rho_\mathcal{T},R)$, $t\in(0,t_0)$, where $\ell(x):=1\vee \log x$.
\item[(d)]
\begin{enumerate}
\item[(i)]
$\hat{\bP}$-a.s., there exists a random $t_0(\mathcal{T}) \in(0,\infty)$ and deterministic $c_1,c_2,\theta_1,\theta_2\in(0,\infty)$ such that
\[c_1t^{-d_s/2}(\log\log t^{-1})^{-\theta_1}
\leq p_t^\mathcal{T}(\rho_\mathcal{T},\rho_\mathcal{T})\leq c_2t^{-d_s/2}(\log\log t^{-1})^{\theta_2},\]
for all $t\in(0,t_0)$.
\item[(i)] There exist constants $c_1,c_2\in(0,\infty)$ such that
\[c_1t^{-d_s/2}
\leq \hat{\mathbf{E}} p_t^\mathcal{T}(\rho_\mathcal{T},\rho_\mathcal{T})\leq c_2t^{-d_s/2},\]
for all $t\in(0,1)$.
\end{enumerate}
\end{enumerate}}

\subsection*{Organization of the paper}

The remainder of the article is organised as follows. In Section \ref{topsec}, we introduce the topologies that provide the framework for Theorem \ref{mainthm1}, and set out three conditions that imply tightness in this topology. Then, in Section \ref{sec:lerw}, we collect together the properties of loop-erased random walks that will be useful for this article. After these preparations, the three tightness conditions are checked in Section \ref{sec:assump}, and the volume estimates contained within this are strengthened in Sections \ref{sec: exp lower volume} and \ref{sec:expupper} in a way that yields more detailed properties concerning the limit space and simple random walk. In Section \ref{finitesec}, we demonstrate our finite-dimensional convergence result for subtrees of $\U$ that span a finite number of points. The various pieces for proving Theorem \ref{mainthm1} are subsequently put together in Section \ref{proofsec}, and the properties of the limiting space are explored in Section \ref{limitsec}, with Theorem \ref{mainthm2} being proved in this part of the article. Finally, Section \ref{srwsec} covers the results relating to the simple random walk and its diffusion scaling limit.

\section{Topological framework}\label{topsec}

In this section, we introduce the Gromov-Hausdorff-type topology on measured, rooted spatial trees with respect to which Theorem \ref{mainthm1} is stated.
This topology is metrizable, and for completeness sake we include a possible metric (see Proposition \ref{deltametric}).
Moreover, we provide a sufficient criterion (Assumptions 1,2, and 3 below) for tightness of a family of measures on measured, rooted spatial trees in the relevant topology (see Lemma \ref{relcompact}).
This will be applied in order to prove tightness under scaling of the three-dimensional UST.
In the first part of the section, we follow closely the presentation of \cite{BCroyK}.

Define $\mathbb{T}$ to be the collection of quintuplets of the form
\[\uT=(\T,d_\T,\mu_{\T},\phi_\T,\rho_\T),\]
where: $(\T,d_\T)$ is a complete and locally compact real tree (for the definition of a real tree, see \cite[Definition 1.1]{rrt}, for example); $\mu_\T$ is a locally finite Borel measure on $(\T,d_\T)$; $\phi_\T$ is a continuous map from $(\T,d_\T)$ into a separable metric space $(M,d_M)$; and $\rho_\T$ is a distinguished vertex in $\T$. (In this article, the image space $(M,d_M)$ we consider is $\R^3$ equipped with the Euclidean distance.) We call such a quintuplet a \textbf{measured, rooted, spatial tree}. We will say that two elements of $\mathbb{T}$, $\uT$ and $\uT'$ say, are equivalent if there exists an isometry $\pi:(\T,d_\T)\rightarrow(\T',d_\T')$ for which $\mu_\T\circ\pi^{-1}=\mu_\T'$, $\phi_\T=\phi_\T'\circ\pi$ and also $\pi(\rho_\T)=\rho_\T'$.

We now introduce a variation on the Gromov-Hausdorff-Prohorov topology on $\mathbb{T}$ that also takes into account the mapping $\phi_\T$.
In order to introduce this topology, we start by recalling from \cite{BCroyK} the metric $\Delta_c$ on $\mathbb{T}_c$, which is the subset of elements of $\mathbb{T}$ such that $(\T,d_\T)$ is compact. In particular, for two elements of $\mathbb{T}_c$, we set $\Delta_c\left(\uT,\uT'\right)$ to be equal to
\begin{equation}\label{deltacdef}
\inf_{\substack{Z,\psi,\psi',\mathcal{C}:\\(\rho_\T,\rho_\T')\in\mathcal{C}}}
\left\{d_P^Z\left(\mu_{\T}\circ\psi^{-1},\mu_{\T}'\circ\psi'^{-1}\right)+
\sup_{(x,x')\in\mathcal{C}}\left(d_Z\left(\psi(x),\psi'(x')\right)+d_M\left(\phi_\T(x),\phi_\T'(x')\right)\right)\right\},
\end{equation}
where the infimum is taken over all metric spaces $Z=(Z,d_Z)$, isometric embeddings $\psi:(\T,d_\T)\rightarrow Z$, $\psi':(\T',d_\T')\rightarrow Z$, and correspondences $\mathcal{C}$ between $\T$ and $\T'$, and we define $d_P^Z$ to be the Prohorov distance between finite Borel measures on $Z$. Note that, by a correspondence $\mathcal{C}$ between $\T$ and $\T'$, we mean a subset of $\T\times \T'$ such that for every $x\in \T$ there exists at least one $x'\in \T'$ such that $(x,x')\in\mathcal{C}$ and conversely for every $x'\in \T'$ there exists at least one $x\in \T$ such that $(x,x')\in \mathcal{C}$.
(Except for the term involving $\phi$ and $\phi'$, this is the usual metric for the Gromov-Hausdorff-Prohorov topology.)

Given the definition of $\Delta_c$ at \eqref{deltacdef}, we then define a pseudo-metric $\Delta$ on $\mathbb{T}$ by setting
\begin{equation}\label{deltadef}
{\Delta\left(\uT,\uT'\right)}:=\int_0^\infty e^{-r}\left(1\wedge
\Delta_c\left(\uT^{(r)},\uT'^{(r)}\right)\right) dr,
\end{equation}
where $\uT^{(r)}$ is obtained by taking the closed ball in $(\T,d_\T)$ of radius $r$ centred at $\rho_\T$, restricting $d_\T$, $\mu_\T$ and $\phi_\T$ to $\T^{(r)}$, and taking $\rho_\T^{(r)}$ to be equal to $\rho_\T$. We have the following result, and it is the corresponding topology that provides the framework for Theorem \ref{mainthm1}.

\begin{prop}[{\cite[Proposition 3.4]{BCroyK}}]
  \label{deltametric}
  The function $\Delta$ defines a metric on the equivalence classes of $\mathbb{T}$. Moreover, the resulting metric space is separable.
\end{prop}

\begin{rem}
 For those less familiar with Gromov-Hausdorff-type topologies and their application to the study of random graphs, we highlight that the embeddings of metric spaces into $(Z,d_Z)$ allow the comparison of properties with respect to their intrinsic metrics (which will be the rescaled graph distance on $\mathcal{U}$ in the discrete example considered here), whereas the map into $(M,d_M)$ allows the comparison of extrinsic properties (in our setting, features of $\mathcal{U}$ in Euclidean space). See \cite[Chapter 7]{BBI} for an introduction to the original Gromov-Hausdorff distance between compact metric spaces.
\end{rem}

We next present a criterion for tightness of a sequence of random measured, rooted spatial trees.
This is a probabilistic version of \cite[Lemma 3.5]{BCroyK} (which adds the spatial embedding to the result of \cite[Theorem 2.11]{ADH})
Recall the definition of \textbf{stochastic equicontinuity}:
Suppose for some index set $\mathcal{A}$ there are random metric spaces $(X_i,d_i)$ and random functions $\phi_i:X_i\to M$ for a metric space $(M,d_M)$.
The functions are stochastically equicontinuous if their moduli of continuity converge to 0 uniformly in probability,
i.e. for every $\eps>0$,
\[\lim_{\eta\to0} \sup_{i\in\mathcal{A}} \bP\left(\sup_{\substack{{x,y\in X_i}:\\d_i(x,y)\leq \eta}} d_M\left(\phi_i(x),\phi_i(y)\right)>\eps\right)=0.\]

\begin{lemma}\label{relcompact}
  Suppose $(M,d_M)$ is proper (i.e.~every closed ball in $M$ is compact), and $\rho_M$ is a fixed point in $M$.
  Let $\uT_\delta = (\T_\delta,d_{\T_\delta},\mu_{\T_\delta},\phi_{\T_\delta},\rho_{\T_\delta})$, ${\delta\in\mathcal{A}}$ (where $\mathcal{A}$ is some index set), be a collection of random measured, rooted spatial trees. Moreover, assume that for every $R>0$, the following quantities are tight:
  \begin{enumerate}[nosep,label=(\roman*)]
  \item For every $\eps>0$, the number $N\left(\uT_\delta,R,\eps\right)$ of
    balls of radius $\eps$ required to cover the ball $\T^{(R)}_\delta$,
  \item The measure of the ball: $\mu_{\T_\delta}\left(\T_\delta^{(R)}\right)$;
  \item The distances $d_M\left(\rho_M, \phi_{\T_\delta}\left(\rho_{\T_\delta}\right)\right)$.
  \end{enumerate}
  And additionally the restrictions of $\phi_{\T_\delta}$ to $\T_\delta^{(R)}$ are stochastically equicontinuous.
  Then the laws of $(\uT_\delta)_{\delta\in\mathcal{A}}$, form a tight sequence of probability measures on the space of measured, rooted spatial trees.
\end{lemma}

For convenience in applying Lemma \ref{relcompact} to the three-dimensional UST, we next summarise the conditions that we will check for this example. Since these are of a different form to those given above, we complete the section by verifying their sufficiency in Lemma \ref{tightness}. We recall that the notation $\BU(x,r)$ is used for balls in $(\sU, d_{\sU})$.

{\assu\label{a1} For every $R\in (0,\infty)$, it holds that
\[\lim_{\lambda\rightarrow \infty}\limsup_{\delta\rightarrow 0}\bP\left(\delta^3\mU\left(\BU(0,\delta^{-\beta}R)\right)>\lambda\right)=0.\]}

{\assu\label{a2} For every $\varepsilon, R\in (0,\infty)$, it holds that
\[\lim_{\eta\rightarrow 0}\limsup_{\delta\rightarrow 0}
\bP\left(\inf_{x\in \BU(0,\delta^{-\beta}R)}\delta^3\mU\left(\BU(x,\delta^{-\beta}\varepsilon)\right)<\eta\right)=0.\]}

{\assu\label{a3} For every $\varepsilon, R\in (0,\infty)$, it holds that
\[\lim_{\eta\rightarrow 0}\limsup_{\delta\rightarrow 0}
\bP\left(\inf_{\substack{x,y\in \BU(0,\delta^{-\beta}R):\\\delta d_E(x,y)>\varepsilon}}\delta^{\beta}\dU(x,y)<\eta\right)=0.\]}

{\lem
\label{tightness}
If Assumptions \ref{a1}, \ref{a2} and \ref{a3} hold, then so does the tightness claim of Theorem \ref{mainthm1}.}
\begin{proof} We first check that if Assumptions \ref{a1} and \ref{a2} hold, then, for every $\varepsilon, R\in (0,\infty)$,
\begin{equation}\label{a20}
\lim_{\lambda\rightarrow \infty}\limsup_{\delta\rightarrow 0}\bP\left(N_{\U}\left(\delta^{-\beta}R,\delta^{-\beta}\varepsilon\right)>\lambda\right)=0,
\end{equation}
where $N_{\U}(\delta^{-\beta}R,\delta^{-\beta}\varepsilon)$ is the minimal number of intrinsic balls of radius $\delta^{-\beta}\varepsilon$ needed to cover $\BU(0,\delta^{-\beta}R)$. Towards proving this, suppose that
\begin{equation}\label{e1}
\delta^3\mU\left(\BU(0,\delta^{-\beta}(R+\varepsilon/2))\right)\leq \lambda\eta,
\end{equation}
and also
\begin{equation}\label{e2}
\inf_{x\in \BU(0,\delta^{-\beta}R)}\delta^3\mU\left(\BU(x,\delta^{-\beta}\varepsilon/2)\right)\geq\eta.
\end{equation}
Set $x_1=0$, and choose
\[x_{i+1}\in \BU(0,\delta^{-\beta}R)\backslash\cup_{j=1}^{i}\BU(x_j,\delta^{-\beta}\varepsilon),\]
stopping when this is no longer possible, to obtain a finite sequence $(x_i)_{i=1}^M$. By construction, $\cup_{i=1}^{M}\BU(x_i,\delta^{-\beta}\varepsilon)$ contains $\BU(0,\delta^{-\beta}R)$, and so $M\geq N_{\U}(\delta^{-\beta}R,\delta^{-\beta}\varepsilon)$. Moreover, since $\dU(x_i,x_j)\geq \delta^{-\beta}\varepsilon$ for $i\neq j$, it is the case that the balls $(\BU(x_i,\delta^{-\beta}\varepsilon/2))_{i=1}^M$ are disjoint. Putting these observations together with (\ref{e1}) and (\ref{e2}), we find that
\begin{eqnarray*}
N_{\U}(\delta^{-\beta}R,\delta^{-\beta}\varepsilon)&\leq& M\\
&\leq& \eta^{-1}\sum_{i=1}^{M}\delta^3\mU\left(\BU(x_i,\delta^{-\beta}\varepsilon/2)\right)\\
&=&\eta^{-1}\delta^3\mU\left(\cup_{i=1}^{M}\BU(x_i,\delta^{-\beta}\varepsilon/2)\right)\\
&\leq &\eta^{-1}\delta^3\mU\left(\BU(0,\delta^{-\beta}(R+\varepsilon/2))\right)\\
&\leq &\lambda.
\end{eqnarray*}
From this, we conclude that
\begin{eqnarray*}
\lefteqn{\bP\left(N_{\U}\left(\delta^{-\beta}R,\delta^{-\beta}\varepsilon\right)>\lambda\right)}\\
&\leq&
\bP\left(\delta^3\mU\left(\BU(0,\delta^{-\beta}(R+\varepsilon/2))\right)>\lambda\eta\right)\\
&&+
\bP\left(\inf_{x\in \BU(0,\delta^{-\beta}R)}\delta^3\mU\left(\BU(x,\delta^{-\beta}\varepsilon/2)\right)<\eta\right),
\end{eqnarray*}
and so \eqref{a20} follows by letting $\delta\rightarrow 0$, $\lambda\rightarrow\infty$ and then $\eta\rightarrow 0$.

Second, we show that if Assumption \ref{a3} holds, then, for every $\varepsilon, R\in (0,\infty)$,
\begin{equation}\label{a30}
\lim_{\eta\rightarrow 0}\limsup_{\delta\rightarrow 0}
\bP\left(\sup_{\substack{x,y\in \BU(0,\delta^{-\beta}R):\\ \dU(x,y)< \delta^{-\beta}\eta}}\delta d_E(x,y)>\varepsilon\right)=0.
\end{equation}
Indeed, this follows from the elementary observation that
\[\bP\left(\sup_{\substack{x,y\in \BU(0,\delta^{-\beta}R):\\ \dU(x,y)< \delta^{-\beta}\eta}}\delta d_E(x,y)>\varepsilon\right)\leq \bP\left(\inf_{\substack{x,y\in \BU(0,\delta^{-\beta}R):\\\delta d_E(x,y) >\varepsilon}}\delta^{\beta}\dU(x,y)<\eta\right).\]

Given \eqref{a20}, Assumption \ref{a1}, the fact that $\delta\phi_\mathcal{U}(\rho_\mathcal{U})=0$, and \eqref{a30}, the result is a straightforward application of Lemma \ref{relcompact}.
\end{proof}

\subsection{Path ensembles}

Finally, we also define the path ensemble topology used in Theorem~\ref{mainthm1}. This topology was introduced by Schramm \cite{Schramm} in the context of scaling of two-dimensional uniform spanning trees, and a related topology (based on  quad-crossings) have been used in the context of scaling limits of critical percolation. Recall that $\gamma_\T(x,y)$ is the unique path from $x$ to $y$ in a topological tree $\T$.

We denote by $\cH(X)$ the Hausdorff space of compact subsets of a metric space $X$, endowed with the Hausdorff topology.
This is generated by the \textbf{Hausdorff distance}, given by
\[  \gls{hausdorffD} (A,B) = \inf \left\{ r \geq 0\::\: A \subset B_r,\: B \subseteq A_r \right\},\]
where $B_r = \{x\in X\::\: d(x,B)\leq r\}$ is the $r$-expansion of $B$.

We shall consider the sphere $\S^3$ as the one-point compactification of $\R^3$, on which we also consider the one-point compactification of a uniform spanning tree of $\mathbb{Z}^3$. For concreteness, fix some homeomorphism from $\R^3$ to $\S^3$ and endow it with the Euclidean metric on the sphere. Given a compact topological tree $\T\subset \S^3$, we consider the set $\Gamma_\T \subset \S^3 \times \S^3 \times \cH(\S^3)$
\[  \Gamma_\T = \{ (x,y,\gamma_\T(x,y)) \,:\, x,y\in\T\}.\]
Thus $\Gamma_\T$ consists of a pair of points and the path between them. We call $\Gamma_\T$ the \textbf{path ensemble} of the tree $\T$. Clearly $\Gamma_\T$ is a compact subset of $\S^3\times\S^3\times\cH(\S^3)$. Since each tree corresponds to a compact subset of $\S^3\times\S^3\times\cH(\S^3)$, the Hausdorff topology on this product space induces a topology on trees. Theorem~\ref{mainthm1} states that the laws of the uniform spanning on $\delta \Z^3$ are tight and have a subsequential weak limit with respect to this topology (in addition to the Gromov-Hausdorff-type topology described above).

\section{Loop-erased random walks}\label{sec:lerw}

As noted in the introduction, the fundamental connection between loop-erased random walks (LERWs) and uniform spanning tree (USTs) will be crucial to this study.
In this section, we recall the definition of the LERW, and collect together a number of properties of the three-dimensional LERW that hold with high probability.
These properties will be useful in our study of the three-dimensional UST.
We start by introducing some general notation and terminology.

\subsection{Notation for Euclidean subsets}  \label{subsec:notationset}

The \textbf{discrete $\ell^2$ Euclidean ball} will be denoted by
\[
  \gls{discreteBall} := \left\{ y \in \Z^{3}\::\:|x - y| < r \right\},
\]
where we write $|x-y|=d_E(x,y)$ for the Euclidean distance between $x$ and $y$.
(We will use the notation $|x-y|$ and $d_E(x,y)$, interchangeably.)
A \textbf{$\delta$-scaled discrete $\ell^2$ ball}, for $\delta > 0$, will be denoted by
\[
  \gls{scaledBall}  := \left\{ y \in \delta \Z^{3}\::\:|x - y| < r \right\},
\]
and the Euclidean $\ell^2$ ball is
\[
  \gls{euclideanBall} := \left\{ y \in \mathbb{R}^3 \::\:|x - y| < r \right\}.
\]
We will also use the abbreviation $B(r)=B(0,r)$, similarly for $B_{\delta}$ and $B_E$. We also write $ B_n (0, r) = B_{2^{-n}} (r) $.
The \textbf{discrete cube} (or $\ell^\infty$ ball of radius $r$) with side-length $2r$ centred at $x$ is defined to be the set
\[
  \gls{discreteCube} :=\left\{y \in \Z^{3} \::\: \|x-y\|_\infty < r  \right\}.
\]
Similarly to the definitions above, but with $\ell^{\infty}$ balls, $\gls{scaledCube}$ denotes the $\delta$-scaled discrete cube  and $\gls{euclideanCube}$ the Euclidean cube.
We further write $D(R)=D(0,R)$ and $D_n (r) = D_{2^{-n}}(0, r)$.
The Euclidean distance between a point $x$ and a set $A$ is given by
\[
  \dist (x , A) := \inf \left\{ |x - y|\::\:y \in A \right\}.
\]
For a subset $A$ of $\mathbb{Z}^{3}$, the \textbf{inner boundary} $\partial_{i} A$ is defined by
\[
\gls{innerB} := \left\{ x \in A \::\:\exists y \in \mathbb{Z}^{3} \setminus A \text{ such that } |x-y| =1 \right\}.
\]

\subsection{Notation for paths and curves} \label{subsec:notationpaths}

A path in $\Z^3$ is a finite or infinite sequence of vertices $[  v_0, v_1, \ldots ]$ such that $v_{i-1}$ and $v_{i}$ are nearest neighbours, i.e.\ $|v_{i-1}-v_i|=1$, for all $i \in \{1, 2, \dots \}$. The \textbf{length} of a finite path $\gamma = [v_0, v_1, ... , v_m ]$ will be denoted $\gls{length}$ and is defined to be the number of steps taken by the path, that is $ \len(\gamma) = m$.

A \textbf{(parameterized) curve} is a continuous function $\gamma : [0,T] \rightarrow \mathbb{R}^3$. For a curve $\gamma : [0,T] \to \mathbb{R}^3 $, we say that $T$ is its \textbf{duration}, and will sometimes use the notation $\gls{duration} := T$. When the specific parameterization of a curve $\gamma$ is not important, then we might consider only its trace, which is the closed subset of $\mathbb{R}^3$ given by $\tr \gamma=\{\gamma(t):\:t\in[0,T]\}$.
To simplify notation, we sometimes write $\gamma$ for instead of $\tr \gamma$ where the meaning should be clear.
A curve is \textbf{simple} if $\gamma$ is an injective function.
All curves in this article are assumed to be simple, often implicitly.

The space of parameterized curves of finite duration, $\gls{spaceCf}$, will be endowed with a metric $\gls{psi}$, as defined by
\begin{equation} \label{eq:distpsi}
 \psi (\gamma_1, \gamma_2) = \left| T_1 - T_2 \right|  +  \max_{0 \leq s \leq 1} \left|\gamma_1 (sT_1) - \gamma_2 ( sT_2 ) \right|,
\end{equation}
where $\gamma_i:[0,T_i]\rightarrow\mathbb{R}^3$, $i=1,2$ are elements of $\mathcal{C}_f$.

We say that a continuous function $\gamma^{\infty} : [0, \infty) \rightarrow \mathbb{R}^3$ is a \textbf{transient (parameterized) curve} if $\lim_{t \to \infty} \vert \gamma^{\infty} (t) \vert  = \infty $. We let $\mathcal{C}$ be the set of transient curves, and endow $\gls{spaceC}$ with the metric $\gls{chi}$ given by
\begin{equation} \label{eq:distchi}
\chi(\gamma^{\infty}_1,\gamma^{\infty}_2) =
  \sum_{k=1}^{\infty} 2^{-k} \left( 1\wedge \max_{t \leq k} |\gamma^{\infty}_1(t)- \gamma^{\infty}_2 (t)| \right).
\end{equation}

The concatenation of two curves $ \gamma_1 : [0, T_1] \to \mathbb{R}^3$ and $ \gamma_2 : [0, T_2] \to \mathbb{R}^3$ with $ \gamma_1 (T_1) = \gamma_2 (0)$ is the curve $\gamma_1 \oplus \gamma_2$ of length $T_1+T_2$ given by
\[  \gamma_1 \oplus \gamma_2 (t) :=
  \begin{cases}
    \gamma_1 (t)    & \text{ if } 0 \leq t \leq T_1 \\
    \gamma_2 (t - T_1) & \text{ if } T_1 < t \leq T_1 + T_2 .
  \end{cases}\]
The time-reversal of $\gamma : [0, T] \to \mathbb{R}^3$ is the curve $\vec{\gamma} : [0, T] \to \mathbb{R}^3$ defined by
\[
  \vec{\gamma} (t) := \gamma  (T - t), \qquad  t \in [0,T].
\]

We define several kinds of restrictions for a curve $\gamma : [0,T] \to \mathbb{R}^3$.
Analogous restrictions are defined for transient curves.
The restriction of $\gamma$ to an interval $ [a,b] \subseteq [0,T]$ is the curve $\gamma \vert_{[a,b]} : [0,b-a] \to \mathbb{R}^3 $ defined by setting
\[  \gamma \vert_{[a,b]} (t) = \gamma (t + a), \qquad  0 \leq t \leq b-a.\]
Similarly, if $\gamma$ is a simple parametrized curve, and $x, y \in \tr \gamma $ and $x$ appears before $y$ in $\gamma$,
then we define the restriction of $\gamma$ between $x$ and $y$ to be the curve $\gamma(x,y)$, where
\[  \gamma (x,y) (t) = \gamma(t + t_x), \qquad 0 \leq t \leq t_y - t_x,\]
with $t_x\leq t_y$ satisfying  $\gamma(t_x) = x$ and $\gamma (t_y) = y$.
(Note that the simplicity of $\gamma$ ensures that $t_x$ and $t_y$ are well-defined.)
Finally, the restriction of $\gamma$ to the Euclidean ball of radius $R$, with $R > 0$, is the curve $\gamma \vert^R :=  \gamma \vert_{[0, \xi_R \wedge T]} $, where  $\xi_R = \inf \{ t \in [0, T]:\:| \gamma(t)|\geq R \}$  is the time $\gamma$ exits the ball of radius $R$.

\begin{prop} \label{prop:continuity-psi}
Let $ ( \gamma_n )_{n \in \mathbb{N}} \subset \mathcal{C}_f$ be a sequence of curves. Assume that  $\gamma_n \to \gamma \in  \mathcal{C}_f$.
Then, the convergence is preserved in $\mathcal{C}_f$ under the following operations.
\begin{enumerate}[label=(\alph*)]
  \item \label{item:time-reversal}
  Time reversal: for the sequence of curves under time-reversal
    \[
      \vec{\gamma}_n \rightarrow \vec{\gamma}  \qquad
      \text{ as } n \to \infty.
    \]
  \item \label{item:restriction}
  Restriction: for $0 \leq a < b < T(\gamma)$, the restrictions
    \[ 
      \gamma_n\vert_{[a, b]} \rightarrow \gamma \vert_{[a, b]} \qquad
      \text{ as } n \to \infty ,
    \]
  where the sequence above is defined for $n$ large enough.
  \item \label{item:concatenation}
  Concatenation:
    if $ \tilde{\gamma}_n \to \tilde{\gamma} $ in $\mathcal{C}_f$, then
    \[
      \gamma_n \oplus \tilde{\gamma}_n \to \gamma \oplus \tilde{\gamma} \qquad
      \text{ as } n \to \infty.
    \]
\end{enumerate}
\end{prop}

\begin{proof} In this proof, we write $T_n = T (\gamma_n)$  and $T = T (\gamma)$.
  The convergence after a time-reversal is immediate from the definition and we get \ref{item:time-reversal}.
  For \ref{item:restriction}, we consider the case $a = 0$.
  Let $ r_n, r \in [0, 1]$ be such that $b = r_n T_n$ and $b = r T$. Then
  \begin{align*}
    \psi ( \gamma_n\vert_{[0, b]}, \gamma \vert_{[0, b]} )
    &=
    \max_{0 \leq s \leq 1} \vert \gamma_n ( s b ) - \gamma (sb) \vert
    =
    \max_{0 \leq s \leq 1} \vert \gamma_n ( s r_n T_n ) - \gamma (s r T) \vert \\
    &\leq
    \max_{0 \leq s \leq 1} \vert \gamma_n ( s r_n T_n ) - \gamma (s r_n T) \vert
    +
    \max_{0 \leq s \leq 1} \vert \gamma ( s r_n T ) - \gamma (s r T) \vert.
  \end{align*}
  The convergence of $\gamma_n \to \gamma$ implies that the first term above goes to $0$ as $n \to \infty$.
  Note that $\vert r_n - r \vert = b \vert T_n^{-1} - T^{-1} \vert \to 0 $, and hence the convergence of the last term above follows from uniform continuity of $\gamma$.
  The convergence of $\gamma_n$ under time-reversal gives the general when $a > 0$.

  Next we prove \ref{item:concatenation}.
  We write $\tilde{T}_n = T (\tilde{\gamma}_n)$, $\tilde{T} = T (\tilde{\gamma})$ and $\delta_n = \vert T_n + \tilde{T}_n - (\tilde{T} + T) \vert $.
  Note that $\delta_n \to 0$ as $n \to \infty$.
  For $ 0 \leq s \leq 1$, when we compare the times that we compare for $\psi$,
  $ \vert s(T_n +\tilde T_n ) -  s(T  +\tilde T ) \vert \leq \delta $.
  Then $\psi ( \gamma_n \oplus \tilde{\gamma}_n, \gamma \oplus \tilde{\gamma} )$
  is bounded above by
  \begin{align*}
  &\delta_n +
  \max_{ \substack{ \vert r - s \vert \leq \delta_n \\ r \leq T_n \vee \tilde{T}_n, \, s \leq T \vee \tilde{T}  }}
  \vert \gamma_n \oplus \tilde{\gamma}_n (r) - \gamma \oplus \tilde{\gamma} (s) \vert \\
  &
  \leq
  \delta_n
  + \max_{\substack{ \vert r - s \vert \leq \delta_n \\ r \leq T_n, \, s \leq T }}
    \vert \gamma_n (r) - \gamma (s) \vert
  +
    \max_{\substack{ \vert r - s \vert \leq \delta_n \\ r \leq \tilde{T}_n, \, s \leq \tilde{T} }}
    \vert \tilde{\gamma}_n (r) - \tilde{\gamma} (s) \vert
  +
  \delta_n M_n,
  \end{align*}
  where $M_n$ is an upper bound for the norms of $\gamma_n,  \tilde{\gamma}_n, \gamma$, and $ \tilde{\gamma}$.
  The last term in the inequality above comes from comparisons between $\gamma_n$ and $\tilde{\gamma}$ (or between $\tilde{\gamma}_n$ and $\gamma$) close to the concatenation point.
  The convergence of $\gamma_n \to \gamma$ and $\tilde{\gamma}_n \to \tilde{\gamma}$, and the uniform continuity of each curve give the desired result.
\end{proof}

\begin{prop} \label{prop:continuity-chi}
Let $ ( \gamma_n^{\infty})_{n \in \mathbb{N}} \subset \mathcal{C} $ be a sequence of parameterized curves with limit  $\gamma_n^{\infty} \to \gamma^{\infty}$ in $(\mathcal{C}, \chi)$.
The convergence is preserved  under the operations below.
\begin{enumerate}[label=(\alph*)]
  \item \label{item:restriction-infty}
  Restriction: for any $b > 0$
    \[
    \gamma_n^{\infty} \vert_{[0, b]} \to \gamma^{\infty} \vert_{[0, b]}  \qquad \text{ as } n \to \infty,
    \]
  in the space $\mathcal{C}_f$.
  \item \label{item:concatenation-infty}
  Concatenation: if $ (\gamma_n )_{n \in \mathbb{N}} \subset \mathcal{C}_f$ converges to a finite parameterized curve $\gamma$ as $n \to \infty$, then
    \[  \gamma_n \oplus \gamma^{\infty}_n \to \gamma \oplus \gamma^{\infty} \qquad
    \text{ as } n \to \infty,
    \]
  in $\mathcal{C}$.
  \item \label{item:evaluation-infty}
  Evaluation: if $t_n \to t$ then
  \[
    \gamma^{\infty}_n(t_n) \to \gamma^{\infty} (t) \qquad \text{ as } n \to \infty.
  \]
\end{enumerate}

\end{prop}

\begin{proof}
  The convergence in \ref{item:restriction-infty} follows from the definition of the metric $\chi$.
  Similarly, \ref{item:concatenation-infty} is a consequence of Proposition \ref{prop:continuity-psi} \ref{item:concatenation} and the definition of $\chi$.
  Finally, \ref{item:evaluation-infty} follows from the uniform continuity of $\gamma_n \vert_{[0,k]}$.
\end{proof}

If $\gamma$ is a parameterized (simple) curve and $x, y \in \tr \gamma$, we define the \textbf{Schramm metric} (cf.\ \eqref{dtschramm}) by setting
\begin{equation} \label{eq:Schramm}
  \gls{schrammMetric} (x,y) := \diam \tr \gamma(x,y).
\end{equation}
The \textbf{intrinsic distance} between $x$ and $y$ is given by
\begin{equation} \label{eq:Intrinsic}
  \gls{intrinsicMetric} (x,y) := T(\gamma(x,y))=t_y-t_x,
\end{equation}
where $\gamma(t_x) = x$ and $\gamma (t_y) = y$, i.e.\ this is the time duration of the curve segment between $x$ and $y$.
Formally, both \eqref{eq:Schramm} and \eqref{eq:Intrinsic} are only defined when $x$ comes before $y$ in $\gamma$, but the definition is extended symmetrically in the obvious way.

\subsection{Definition and parameterization of loop-erased random walks}

We will now define the loop-erased random walk.
Let $S = [v_0,\dots,v_m]$ be a path in some graph (which we take to be $\Z^3$ or $\delta\Z^3$).
By erasing the cycles (or loops) in $S$ in chronological order, we obtain a simple path from $v_0$ to $v_m$.
This operation is called \textbf{loop-erasure}, and is defined as follows. Set $T(0) = 0$ and $\widetilde{v_0} = v_0$.
Inductively, we set $T(j)$ according to the last visit time to each vertex:
\begin{equation} \label{eq:deflerw}
  T(j) = 1 + \sup \left\{ n \colon v_n = \widetilde{v}_j \right\}, \qquad \tilde{v}_j = v_{T(j)}.
\end{equation}
We continue until $\tilde{v}_l = v_m$, at which time $T(j)=m+1$ and there is no additional vertex $\tilde v_j$.
The \textbf{loop-erased random walk (LERW)} is the simple path $\LE(S) = [\tilde{v}_0, \dots , \tilde{v}_l]$.

The exact same definition also applies to an infinite, transient path $S$.
Since the path $S$ is transient, the times $T(j)$ in $\eqref{eq:deflerw}$ are finite, almost surely, for every $j \in \mathbb{N}$.
In this case $\LE(S)$ is an infinite simple path.

The loop-erased random walk is just what the name implies: the loop erasure of a random walk.
In $\Z^3$ (or $\delta\Z^3$) we can take $S_\infty$ to be an infinite random walk.
$S_\infty$ is almost surely transient, so the path $L(S_{\infty})$, called the \textbf{infinite loop-erased random walk (ILERW)}, is a.s. well defined.
We will also need loop-erased random walks in a domain $D\subset \R^3$.
We will write $\hat{D} = \Z^3 \cap D$ for the subset of vertices of $\Z^3$ inside $D$. Moreover, the (inner vertex) boundary of $\hat{D}$ is the set $\partial\hat{D}$ defined as the collection of vertices $v \in D$ for which $v$  is connected to $v_1 \in \Z^3 \setminus \hat{D} $. In this case, for a given starting vertex $v_0$, we may take $S$ to be a simple random walk up to the stopping time $m$ when it first hits $\partial\hat D$.
(We will apply this to bounded domains, so that $m$ is almost surely finite, though the definition is valid even if $m=\infty$.)
Examples of domains of a loop-erased random walk include the family of $L^2$-balls $\{ B(R) \}_{R > 0}$ and of $L^{\infty}$-balls $\{  D(R)\}_{r > 0}$.

A discrete simple path $\gamma=(v_i)$ may naturally be considered as a curve by setting $\gamma(i) = v_i$, for $i\in\N$, and linearly interpolating between $\gamma(i)$ and $\gamma(i+1)$.
With this parameterization, the length of $\gamma$ (as a path) is equal to its duration as a curve: $\len(\gamma) = T(\gamma)$.
If $\gamma$ is a loop-erased random walk on $\delta\Z^3$, its length as $\delta\to 0$, and the curve needs to be reparameterized.
To obtain a macroscopic curve in the scaling limit, we reparameterize loop-erased random walks by $\beta$-\textbf{parameterization}:
\[\gls{betaParameterization} (t) := \gamma (\delta^{-\beta}t),\qquad\forall t \in[0,\delta^{-\beta} \len(\gamma)],\]
where $\beta$ is the LERW growth exponent.
Similarly, for an infinite loop-erased random walk $\gamma_{\infty} = [v_0, v_1,\dots]$,
we consider its associated curve $\gamma_{\infty}$ by linearly interpolating between integer times, and its $\beta$-parameterization is given by
\[   \bar{\gamma}_{\infty} (t) = \gamma_{\infty} (\delta^{-\beta} t), \qquad\forall t \geq 0.\]

In this article, we will sometimes consider the ILERW restricted to a finite domain. Specifically, if $\gamma_{\infty}$ is an ILERW starting at the origin, we denote its restriction to a ball of radius $r > 0$ by $\gamma_{\infty} \vert^r = \LE(S^{\infty}) \vert_{[0,\xi_r(LE(S^{\infty}))]}$, where $\xi_r(\LE(S^{\infty}))$ is the first time $\LE(S^{\infty})$ exits $B(r)$. Note that this is a different object to a LERW started at the origin and stopped at the first hitting time of $\partial B(r)$.
However, the two are closely related, see \cite[Corollary 4.5]{Mas}.

\subsection{Path properties of the infinite loop-erased random walk}
\label{subsec:pathprop}

In this section, we summarize some path properties of the ILERW that hold with high probability.
Typically the events will involve some property that holds on the appropriate scale in a neighbourhood of radius $R \delta^{-1}$ about the starting point of the ILERW, for $\delta$ the scaling parameter, and for some fixed $R \geq 1$. Since the results hold uniformly in the scaling parameter $\delta \in (0,1]$, they will also be useful in the scaling limit. As for notation, for $x\in\Z^3$, we let $\gls{ILERW}$ be an ILERW on $\Z^3$ starting at $x$. If $x = 0$, then we simply write $\gamma_{\infty}$. We highlight that in this section the space is not rescaled.

\subsubsection{Quasi-loops}

A path $\gamma$ is said to have an \textbf{$(r,R)$--quasi-loop} if it contains two vertices $v_1, v_2 \in \gamma$ such that $|v_1 - v_2 |< r$, but $\gamma (v_1, v_2) \not\subseteq B(v_1, R)$.
(Up to changing the parameters slightly, this is almost the same as $d_\gamma^S(x,y)\ge R$.)
We denote the set of $(r,R)$--quasi-loops of $\gamma$ by $\QL (r, R; \gamma)$.
Estimates on probabilities of quasi-loops in LERWs were central to Kozma's work \cite{Kozma}.
The following bound on the probability of quasi-loops for the ILERW was established in \cite{SS} for loop-erased random walks.
The extension to the infinite case follows from \cite[Theorem 6.1]{SS} in combination with \cite[Corollary 4.5]{Mas}.

\begin{prop}[{cf.\ \cite[Theorem 6.1]{SS}}] \label{result:quasiloops}
  For every $R \geq 1$, there exist constants $C,M < \infty$, and $\tilde{\eta} > 0$ such that for any $\delta,\eps \in (0,1)$,
\[ \bP \left( \QL (\eps^M \delta^{-1} , \sqrt{\eps} \delta^{-1} ; \gamma_{\infty} \vert^{R\delta^{-1}}) \neq \emptyset \right) \leq C \eps^{\tilde{\eta}}.\]
\end{prop}

\subsubsection{Intrinsic length and diameter}

Let $\xi_{n} $ be the first time that the loop-erased walk $\gamma_{\infty}$ exits the ball $B(n)$ (i.e.\ the number of steps after the loop erasure).
The next result is a quantitative tightness result for $n^{-\beta}\xi_n$.
It is a combination of the exponential tail bounds of \cite{S}, together with the estimates on the expected value of $\xi_n$ from \cite{Escape}. We note that the result in \cite{S} is for the LERW, but the proof is the same for the ILERW.

\begin{prop}[{\cite[Theorems 1.4 and 8.12]{S} and \cite[Corollary 1.3]{Escape}}] \label{result:upperLowTail}
There exist constants $C, c_1, c_2\in(0,\infty)$ such that: for all $\lambda \geq 1$ and $n \geq 1$,
\begin{align*}
\mathbf{P} \left( \xi_{n} \leq \lambda n^{\beta} \right)& \geq 1 - 2e^{-c_1 \lambda},\\
\mathbf{P} \left( \xi_{n} \geq \lambda^{-1} n^{\beta}   \right)& \geq 1 -  C e^{-c_2 \lambda^{1/2}}.
\end{align*}
\end{prop}

While any possible pattern appears in $\gamma_\infty$,
the scaling relation (given by $\beta$) between the intrinsic distance and the Euclidean distance holds uniformly along the path of $\gamma_{\infty}$.
We quantify this relation in terms of equicontinuity.

Let $R \geq 1$, $\delta \in (0,1)$ and $\lambda \geq 1$. We say that $\gamma_{\infty}$ is $\lambda$-\textbf{equicontinuous} in the ball $B (R \delta^{-1})$ (with exponents $ 0 < b_1, b_2 < \infty $) if the following event holds:
\[ E^*_\delta(\lambda,R) =
  \left\{
  \begin{array}{c}
  \forall x,y \in \gamma_{\infty} \vert^{R \delta^{-1}} , \\
    \text{ if }  d_{\gamma_{\infty}} (x,y) \leq \lambda^{-b_1} \delta^{- \beta}, \text{ then }     \vert x - y \vert < \lambda^{-b_2} \delta^{-1}
  \end{array}
  \right\}.
\]
The bound for the ILERW was proved in \cite{LS}.

\begin{prop}[{cf.\ \cite[Proposition 7.1]{LS}}] \label{result:boundE}
There exist constants $ 0 <  b_1, b_2 < \infty$ such that the following is true. Given $R\geq 1$, there exists a constant $C$ such that: for all $\delta \in (0,1) $and $\lambda \geq 1$,
\begin{equation*}
  \mathbf{P} (E^*_\delta(\lambda,R)) \geq 1 - C \lambda^{-b_2}.
\end{equation*}
\end{prop}

A partial converse bounds the intrinsic distance in terms of the Schramm distance, where we recall that the Schramm distance was defined at \eqref{eq:Schramm}. For $\delta,r \in (0,1]$, $\lambda \geq 1$, set
\[  S^*_\delta (\lambda, r):=
  \left\{
  \begin{array}{c}
  \forall x,y \in \gamma_{\infty} \vert^{\lambda r \delta^{-1}}, \\ \text{ if }   d^S_{\gamma_{\infty}} (x,y) < r \delta^{- 1}, \text{ then } d_{\gamma_{\infty}} (x,y) < \lambda r^{\beta} \delta^{- \beta}
  \end{array}
  \right\}.\]
The following result follows from \cite[(7.51)]{LS}.

\begin{prop} \label{result:boundS}
  There exist constants $0 < c,C < \infty$ such that: for any $\delta,r \in (0,1]$ and $\lambda \geq 1$,
  \[\bP(S^*_\delta (\lambda, r)) \geq 1 - C \lambda^3 e^{-c \lambda}.\]
\end{prop}

\begin{proof}
  For $u\in\Z^3$, let $B_u$ be the box of side length $3r\delta^{-1}$ centred at $u$, and let $X_u = |\gamma_\infty\vert^{\lambda r \delta^{-1} } \cap B_u|$ be the number of points in $B_u$ hit by $\gamma_\infty \vert^{\lambda r \delta^{-1} }$.
  We recall \cite[equation (7.51)]{LS}, which states that for some absolute $c,C$ and any $u$,
  \[\bP \left(X_u \ge \lambda r^{\beta} \delta^{-\beta} \right) \leq C e^{-c \lambda}.\]

  Cover the ball $B(0,\lambda r\delta^{-1})$ by boxes of side length $r\delta^{-1}$ centred at some $\{u_1,\dots,u_N\}$ with $N\asymp\lambda^3$.
  If some pair $x,y$ violates the event $S_\delta^*$, and $x$ is in the box of
  side length $r\delta^{-1}$ around $u_i$, then the segment $\gamma_\infty(x,y)$ is in the thrice larger box around the same $u_i$, and so $X_{u_i} \ge \lambda r^\beta\delta^{-\beta}$.
  A union bound gives the conclusion.
\end{proof}

\subsubsection{Capacity and hittability}

As noted in the introduction, one of the key differences from the two-dimensional case is that in three dimensions it is much easier for a random walk to avoid a LERW. The electrical capacity of a connected path of diameter $r$ in $\Z^3$ can be as large as $Cr$, but can also be as low as $O(r/\log r)$. However, the latter occurs only when the path is close to a smooth curve. The fractal nature of the scaling limit of LERWs suggests that a segment of LERW has capacity comparable to its diameter, and consequently, is likely to be hit by a second random walk starting nearby.
In this subsection, we give bounds on the hitting probability of $\gamma^x_{\infty}$ by a random walk started from a point $y$. The hitting bounds are uniform over the starting points $y \in B:= B(x, R \delta^{-1})$ with $\dist (y, \gamma^x_{\infty}) < r \delta^{-1}$.

We begin with a quantitative definition of hittability on $\mathbb{Z}^3$.
Let $R \geq 1$, $\delta \in (0,1)$, and let $D \subset \mathbb{Z}^3$ be a subset with $x \in D$ and diameter at least $\delta^{-1}R$.
For $y \in B:= B(x, R \delta^{-1})$, let $S$ be simple random walk on $\mathbb{Z}^3$ starting from $y$. If $D$ is a random set, then we further assume that $S$ is independent of $D$. $P^y_S$ denotes the corresponding probability measure for the random walk.
We say that $D$ is $\eta$-\textbf{hittable} around $x$ in $B$, and with parameters  $R \geq 1$ and $r \in (0,1)$, if the following holds:
\begin{equation} \label{def:etaHittable}
  \left\{
\begin{array}{c}
  \forall y \in B (x, R \delta^{-1})\text{ with } \dist (y, D) \leq r  \delta^{-1}, \\
  P^y_{S} \left(  S \left[ 0, \xi_S ( B(y, r^{1/2} \delta^{-1}) ) \right] \cap D= \emptyset  \right) \leq     r^{\eta}
\end{array}
  \right\},
\end{equation}
where $\xi_S (B (y, r^{1/2} \delta^{-1}))$ is the first time that $S$ exits from $B(y, r^{1/2} \delta^{-1})$.
(Recall $\dist(\cdot, \cdot)$ stands for the Euclidean distance between a point and a set.)

If $\gamma^x$ is an ILERW started at $x$, let  $ A_\delta(x , R, r;  \eta) $  denote the event where $\gamma^x$ is $\eta$-hittable as in \eqref{def:etaHittable}.
A local version of this event, restricted to starting points near $x$, is given by
\[
G_\delta(x,r; \eta) =
\left\{
\begin{array}{c}
\forall y \in B (x, r \delta^{-1}), \\ P^y_{S} \left(  S \left[ 0, \xi_S ( B(y, r^{1/2} \delta^{-1}) ) \right] \cap \gamma_{\infty}^x = \emptyset  \right) \leq        r^{\eta}
\end{array}
\right\}.
\]
The next result, which was established in \cite{SS}, indicates that $\gamma^x_{\infty}$ is $\eta$-hittable with high probability.

\begin{prop}[{cf.\ \cite[Lemma 3.2 and Lemma 3.3]{SS}}] \label{result:hit}
There exists a constant $\hat{\eta} \in (0,1)$ such that the following is true. Given $R\geq 1$, there exists a constant $C$ such that: for all $ \delta,r \in (0,1) $,
\[  \mathbf{P} \left(A_\delta(x, R, r; \hat{\eta}) \right) \geq 1 - Cr.\]
In particular, $\mathbf{P} (G_\delta(x,r ; \hat{\eta})) \geq 1 - Cr$.
\end{prop}

We write $\mathbf{P}^{x,y}_{S}$ for the joint probability law of $\gamma_{\infty}^x$ and an independent simple random walk $S$ starting at $y$.
Working on the joint probability space, together with a change of variable, Proposition~\ref{result:hit} implies the following result. This result is well-known and simply states that a simple random walk hits a ILERW almost surely.

\begin{prop}[{cf.\ \cite[Theorem 1.1, Corollary 5.3]{LPS}}] \label{prop:hitas}
  For $x , y \in \mathbb{Z}^3$ we have that, for all $R > 0$,
  \[
    \inf_{\delta \in (0,1]} \mathbf{P}^{x,y}_S \left(  S [ 0, \xi_S ( B ( y, R \delta^{-1}) ) ] \cap \gamma_{\infty}^x  = \emptyset \right) = 0.
  \]
\end{prop}

\subsubsection{Hittability of sub-paths}

The main result of this subsection, Proposition \ref{prop:iterate}, is crucial for obtaining exponential tail bounds on the volume of balls in the UST in Section \ref{sec: exp lower volume}. It establishes that the path $\gamma_\infty=\text{LE} \left( S[0, \infty ) \right)$, i.e.\ the infinite LERW, has hittable sections across a range of distances from its starting point.

For $1 \le \lambda < R $, consider a sequence of boxes $D_{i} = D \left( \frac{i R}{\lambda} \right)$, $i = 1, 2, \dots , \lambda $, where $D (r)$ was defined in Subsection \ref{subsec:notationset}. Let $t_{i}$ be the first time that $\gamma_{\infty}$ exits $D_{i}$. We denote $x_{i} = \gamma_{\infty} (t_{i} )$, and write
\[\sigma_{i} = \inf \left\{ n\ge t_{i} \ | \ \gamma_{\infty} (n) \notin B \left(x_{i}, \frac{R}{2 \lambda } \right) \right\}.\]
For each $i = 1,2 , \dots , \lambda$, we define the event $A_{i}$ by
\begin{equation}\label{aidef}
A_{i} = \left\{ P^{z} \left( R^{z} [0, \xi_{i} ] \cap \gamma_{\infty} [t_{i}, \sigma_{i} ]\cap D_{\frac{R}{2 \lambda}} (x_{i} ) \neq \emptyset   \right) \ge c_{0} \text{ for all } z \in B \left( x_{i}, \frac{R}{16 \lambda } \right)  \right\},
\end{equation}
where: $R^{z}$ is a simple random walk started at $z$, independent of $\gamma_{\infty}$, with law denoted $P^{z}$; $\xi_{i}$ is the first time that $R^{z}$ exits $B( x_{i}, \frac{R}{2 \lambda })$; and $D_{\frac{R}{2 \lambda}} (x_{i} )$ is the box centered on the infinite half line started at $x_i$ that does not intersect $D_i$ and is orthogonal to the face of $D_i$ containing $x_i$, with centre at distance $R/4\lambda$ from $x$ and radius  $\frac{R}{2,000\lambda}$, see Figure \ref{figf}.

\begin{figure}[!ht]
\begin{center}
  \includegraphics[width=0.75\textwidth]{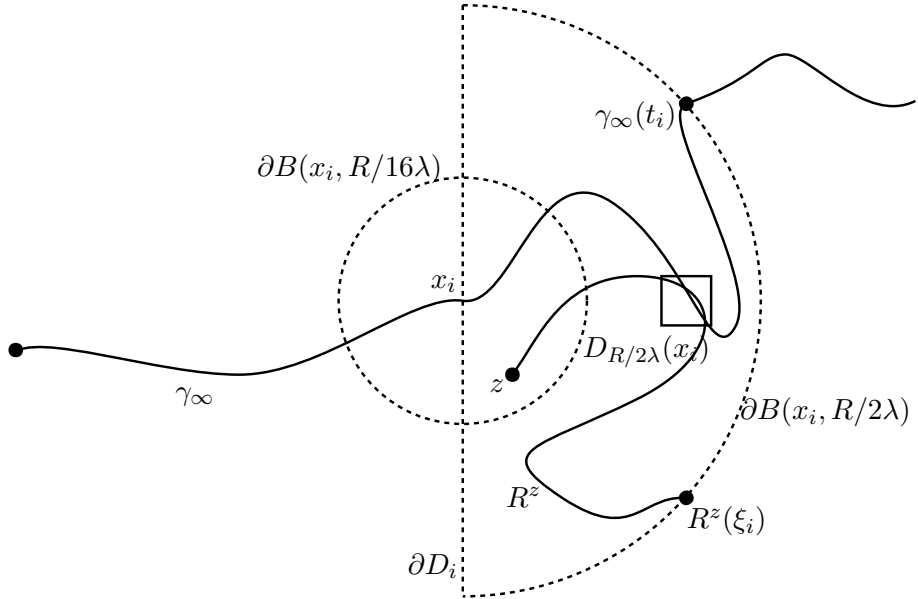}
\rput(	-181	pt,	118.4	pt	){	$x_{i}$	}
\rput(	-274	pt,	76	pt	){	$\gamma _{\infty }$	}
\rput(	-216	pt,	162	pt	){	$\partial B( x_{i} ,R/16\lambda )$	}
\rput(	-38	pt,70	pt	){	$\partial B( x_{i} ,R/2\lambda )$	}
\rput(	-75	pt,	29	pt	){	$R^{z}( \xi _{i})$	}
\rput(	-161	pt,	80	pt	){	$z$	}
\rput(	-151	pt,	38	pt	){	$R^{z}$	}
\rput(	-109	pt,	182	pt	){	$\gamma _{\infty }( t_{i})$	}
\rput(	-185	pt,	12	pt	){	$\partial D_{i}$	}
\rput(	-105	pt,	93	pt	){	$D_{R/2\lambda }( x_{i})$	}
\end{center}
\caption{On the event $A_i$, as defined at \eqref{aidef}, the above configuration occurs with probability greater than $c_0$ for any $z\in B(x_i,R/16\lambda)$.}\label{figf}
\end{figure}

Now, for fixed $a \in (0,1 )$, we consider a sequence of subsets of the index set $\{ 1,2, \dots ,  \lambda \}$ as follows. Let $ q =\lfloor \lambda^{1-a} / 3\rfloor$. For each $j = 0, 1, \dots , q $, define the subset $I_{j}$ of the set $\{ 1,2, \dots ,  \lambda \}$ by setting
\begin{equation}\label{index}
I_{j} := \left\{ 2 j \lambda^{a} + 1, 2 j \lambda^{a} + 2 , \dots , (2 j + 1 ) \lambda^{a} \right\},
\end{equation}
and the event $F_{j}$ by
\begin{equation}\label{index-2}
F_{j} = F_{j}^{a} = \bigcup_{i \in I_{j}} A_{i},
\end{equation}
i.e.\ $F_{j}$ is the event that there exists at least one index $i \in I_{j}$ such that $ \gamma_{\infty} [t_{i}, \sigma_{i} ] $ is a hittable set in the sense that $A_{i}$ holds. The next proposition shows that with high probability the event $F_{j}$ holds for all $j =1,2, \dots , q$. We will prove it in the following subsection.

\begin{prop}\label{prop:iterate}
Define the events $F_{j}$ as in \eqref{index-2}. There exists a universal constant $c_{1} > 0$ such that
\begin{equation}\label{good}
\mathbf{P} \left( \bigcap_{j = 1}^{q} F_{j} \right) \ge 1 - \lambda^{1-a}e^{- c_{1} \lambda^{a}}.
\end{equation}
\end{prop}

\begin{rem}
(i) The reason that we decompose the ILERW $\gamma_{\infty}$ using the sequence of random times $t_{i}$ as in the above definition is that we need to control the future path $\gamma_{\infty} [t_{i}, \sigma_{i} ]$ uniformly on the given past path $\gamma_{\infty} [0, t_{i}]$ via \cite [Proposition 6.1]{S}.\\
(ii) We expect that each $\gamma_{\infty} [t_{i} , \sigma_{i} ]$ is a hittable set not only with positive probability as in Proposition \ref{second-moment} below, but also with high probability in the sense of \cite[Theorem 3.1]{SS}. However, since Proposition \ref{iterate} is enough for us, we choose not to pursue this point further here.
\end{rem}

\subsection{Loop-erased random walks on polyhedrons} \label{subsec:LERWPoly}

We suppose that a loop-erased random walk on a domain $ \hat{D} \subset \mathbb{Z}^3  $ starts at an interior vertex of $\hat{D}$ and ends with its first hitting time to the boundary $\partial \hat{D}$.
As we have discussed above, the geometry of the domain $\hat{D}$ affects the path properties of loop-erased random walks on it.
In this subsection we will see that the results in \cite{Kozma, SS, LS} hold for a collection of scaled polyhedrons, which we define below.
Similarly to Subsection~\ref{subsec:pathprop}, and under the assumption that the polyhedrons are scaled with a large parameter,
the proofs in the aforementioned papers carry over without major modifications to our setting.
For clarity, we  comment on the differences between the work in \cite{SS, LS} and this subsection.

A \textbf{dyadic polyhedron} on $\mathbb{R}^3$ is a connected set $\mathcal{P}$ of the form
\begin{equation*}
  \gls{dyadic} = \bigcup_{j = 1}^{m} C_j,
\end{equation*}
where each $C_j \subset \mathbb{R}^3$ is a closed cube of the form $[a_1, b_1]\times [a_2, b_2] \times [a_3, b_3]$ with $a_i, b_i \in \mathbb{Z}^3 + \frac{1}{2}$ (cf.\ \eqref{eq:dyadicP} below, where we scale the lattice instead of the polyhedron). Note that
$ \mathbb{Z}^3 + \frac{1}{2} = \{ z \in \mathbb{R}^3 \: : \: z + 1/2 \in \mathbb{Z}  \}  $.
We say that a polyhedron $\mathcal{P}$ is bounded by $R$  if $\mathcal{P} \subset B(R) $.
Let us assume that $ 0 \in \mathcal{P}$ and write
\begin{equation}  \label{eq:expansion}
 2^n \mathcal{P} :=  \{ z \in \mathbb{Z}^3 \: : \:  2^{-n} z \in \mathcal{P}   \},
\end{equation}
for the $2^n$-expansion of the polyhedron $\mathcal{P}$.
In this subsection we restrict our scaling to powers of $2$, and note that $2^{n}\mathcal{P}$  is a dyadic polyhedron as well.
If $\mathcal{P}$ is bounded by $R$, then   $B(0,1) \subset 2^n \mathcal{P} \subset B(0,2^n R)$, for all $n \geq 1$.

Let $S$ be a simple random walk starting at $0$ and let $\xi_{\partial \mathcal{P}}$ be the first exit time of the random walk from the polyhedron (with respect to the lattice $\mathbb{Z}$).
In this section we study the path properties of the loop-erased random walk
\begin{equation} \label{eq:defGammaPoly1}
  \gamma_n^{\mathcal{P}} = \LE (S [  0, \xi_{\partial 2^{n}\mathcal{P}} ] ).
\end{equation}
Note that the index $n$ indicates a $2^n$-expansion of $\mathcal{P}$ (cf.~\eqref{def:gammaEscaledPoly} below).


We say that $\gamma_n^{\mathcal{P}}$ is $\eta$-\textbf{hittable} if the following event holds:
\[
A_n^{\mathcal{P}} (r;  \eta) :=
\left\{
\begin{array}{c}
  \forall y \in 2^n \mathcal{P}  \text{ with } \dist (y, \gamma_n^{\mathcal{P}} ) \leq r 2^n, \\
  P^y_{S} \left(  S \left[ 0, \xi_S ( B(y, r^{1/2} 2^n) ) \right] \cap\gamma_n^{\mathcal{P}} = \emptyset  \right) \leq     r^{\eta}
\end{array}
\right\},
\]
where $\xi_S (B (y, r^{1/2} 2^n))$ is the first time that $S$ exits from $B(y, r^{1/2} 2^n)$.


\begin{prop}[{cf.\ Proposition~\ref{result:hit}}] \label{result:hitPoly}
Fix $R\geq 1$, let $\mathcal{P}$ be a dyadic polyhedron containing $0$ and bounded by $R$, and let $\gamma_n^{\mathcal{P}}$ be the loop-erased random walk in \eqref{eq:defGammaPoly1}.
There exists a constant $\hat{\eta} \in (0,1)$ such that there exists a constant $C$ (depending on $R$) and $N \geq 1$ for which the following is true: for all $ r \in (0,1) $ and $n \geq N$,
\[  \mathbf{P} \left(A_n^{\mathcal{P}} (r; \hat{\eta}) \right) \geq 1 - Cr.\]
\end{prop}

Proposition~\ref{result:hitPoly} follows from \cite[Lemma 3.2]{SS} and \cite[Lemma 3.3]{SS}, using the argument for the proof of \cite[Theorem 3.1]{SS}.
The argument for Proposition~\ref{result:hitPoly} considers two cases, depending on the starting point of the simple random walk $S(0) = y $.
For some $\varepsilon > 0$, either $ y \in B(0, \varepsilon n) $ or $y \in \mathcal{P} \setminus B (0, \varepsilon n) $.
For the first case we apply \cite[Lemma 3.2]{SS}, and here we use that $\gamma_n^{\mathcal{P}}$ is a ``large'' path when $n$ is large enough.
If $y \in \mathcal{P} \setminus B (0, \varepsilon n)$, we then consider a covering of $\mathcal{P}$ with a collection of balls $\{ B(v_i, \varepsilon^2 n ) \}_{1 \leq i \leq L}$, with $v_1, \ldots v_L \in \mathcal{P} \setminus B (0, \varepsilon n)$ and $L \leq 10 R^3 \lfloor \varepsilon \rfloor^{-6}$. We then use \cite[Lemma 3.3]{SS} on each one of these balls and a union bound gives the desired result.

Recall the definition of $(r,R)$--quasi-loop in Subsection~\ref{subsec:pathprop} and that $\QL (r, R; \gamma)$ denotes the set of $(r,R)$--quasi-loops of $\gamma$. Proposition~\ref{result:quasiloops} indicates the ILERW does not have quasi-loops with high probability. A similar statement holds for a polyhedral domain. The proof makes use of Proposition~\ref{result:hitPoly} and we use modifications over the stopping times and the covering of the domain (as in Proposition~\ref{result:hitPoly}).
 Indeed, the proof of \cite[Theorem 6.1]{SS} is divided in three cases.
 If the LERW has a quasi-loop at a vertex $v$, then either $v$ is close to the starting point of the LERW, or $v$ is close to the boundary, or $v$ is in an intermediate region. The probability of the first two cases is bounded by escape probabilities for random walks. We can use the same bounds in \cite[Theorem 6.1]{SS} as long as the scale $ n $ is large enough (as we assume in Proposition~\ref{result:quasiloopsPoly}). The bound for the third case follows from a union bound over a covering of the domains. We can use this argument because $\mathcal{P}$ has a regular boundary.

\begin{prop}[{cf.\ Proposition~\ref{result:quasiloops}}] \label{result:quasiloopsPoly}
Fix $R\geq 1$ and let $\mathcal{P}$ be a dyadic polyhedron containing $0$ and bounded by $R$,
and let $\gamma_n^{\mathcal{P}}$ be the loop-erased random walk in \eqref{eq:defGammaPoly1}.
There exist constants $C,M < \infty$, $N \geq 1$ and $\tilde{\eta} > 0$ such that for any $\eps \in (0,1)$ and $n \geq N$,
\[ \bP \left( \QL (\eps^M 2^n , \sqrt{\eps} 2^n ;\gamma_n^{\mathcal{P}}) \neq \emptyset \right) \leq C \eps^{\tilde{\eta}}.\]
\end{prop}


Since Propositions~\ref{result:hitPoly} and~\ref{result:quasiloopsPoly} hold for scaled dyadic polyhedrons, we can follow the argument in \cite{LS} leading to the proof of the scaling limit of the LERW.  From this argument we obtain control of the paths and the scaling limit for the LERW $\gamma_n^{\mathcal{P}}$ with $\beta$-parameterization.
We finish this sections stating these three results.


For a LERW $\gamma_n^{\mathcal{P}}$,  $n \geq 1$ and $\lambda \geq 1$, the path $\gamma_n^{\mathcal{P}}$ is $\lambda$-equicontinuous (with exponents   $ 0 <  b_1, b_2 < \infty$) if
 \[ E^{\mathcal{P}}_{n} (\lambda,R) :=
   \left\{
    \forall x,y \in\gamma_n^{\mathcal{P}} ,
     \text{ if }  d_{\gamma }   (x,y) \leq \lambda^{-b_1} 2^{\beta}, \text{ then }     \vert x - y \vert < \lambda^{-b_2} 2^{n}
     \right\}.\]
The partial converse is the event:
\[  S^{\mathcal{P}}_n (\lambda, r):=\left\{\forall x,y \in \gamma_n^{\mathcal{P}}, \text{ if }   d^S_{\gamma} (x,y) < r 2^{n}, \text{ then } d_{\gamma} (x,y) < \lambda r^{\beta} 2^{\beta}\right\}.\]

\begin{prop}[{cf.\ Proposition~\ref{result:boundE}}] \label{result:boundEPoly}
There exist constants $ 0 <  b_1, b_2 < \infty$ such that the following is true. Given $R\geq 1$, there exist constants $0 < C < \infty$ and $N \geq 1$ such that: for all $\lambda \geq 1 $ and $n \geq N$,
\begin{equation*}
  \mathbf{P} (E^{\mathcal{P}}_n (\lambda,R)) \geq 1 - C \lambda^{-b_2}.
\end{equation*}
\end{prop}



\begin{prop}[{cf.\ Proposition~\ref{result:boundS}}] \label{result:boundSPoly}
  There exist constants $0 < c,C < \infty$ and $N \geq 1$ such that: for any $ r \in (0,1]$, $\lambda \geq 1$  and $n \geq N$,
  \[ \bP(S^{\mathcal{P}}_n (\lambda, r)) \geq 1 - C \lambda^3 e^{-c \lambda} .\]
\end{prop}

\begin{prop}[{cf.\ \cite[Theorem 1.4]{LS}}] \label{result:scalingLimitPoly}
  Let $\mathcal{P}$ be a dyadic polyhedron containing $0$ and bounded by $R$
  and let $\gamma_n^{\mathcal{P}}$ be the loop-erased random walk in \eqref{eq:defGammaPoly1}.
  The $\beta$-parameterization of this loop-erased random walk is the curve given by
  \[
    \bar{\gamma}_n^{\mathcal{P}} (t) =  \gamma_n^{\mathcal{P}} (2^{\beta n  } t),
    \qquad t \in [0, 2^{-\beta n}  \len ( \gamma_n^{\mathcal{P}} )]
  \]
and let $\bar{\gamma}_n^{\mathcal{P}}$ be the $\beta$-parameterization of the loop-erased random walk in \eqref{eq:defGammaPoly1}.
  Then the law of  $ \bar{\gamma}_n^{\mathcal{P}} $ converges as $n \to \infty$ with respect to the metric space $(\mathcal{C}_f, \psi)$.
\end{prop}

\subsection{Proof of Proposition \ref{prop:iterate} }

In this subsection we show that sub-paths of the ILERW are hittable in the sense required for the event \eqref{aidef} to hold, see Proposition \ref{second-moment} below. The latter result leads to the proof of Proposition \ref{prop:iterate}. With this objective in mind, we first study a conditioned LERW. We begin with a list of notation.
\begin{itemize}
\item Recall that $D(R)$ is the cube of radius $R$ centered at $0$, as defined in Subsection \ref{subsec:notationset}.
\item Take positive numbers $m, n$. Let $x \in \partial D (m)$ be a point lying in a ``face" of $D (m)$ (we denote the face containing $x$ by $F$). Write $\ell$ for the infinite half line started at $x$ which lies in $D(m)^{c}$ and is orthogonal to $F$. We let $y$ be the unique point which lies in $\ell$ and satisfies $|x-y| = n/2$. We set $D_{n} (x) := D (y, n/1000 ) $ for the box centered at $y$ with side length $n/500$. (Cf.\ the definition of $D_{\frac{R}{2 \lambda}} (x_{i} )$ above.)
\item Suppose that $m, n, x, D_{n} (x) $ are as above. Take $K \subseteq D(m) \cup \partial D (m)$. Let $X$ be a random walk started at $x$ and conditioned that $X [1, \infty ) \cap K = \emptyset $. We set $\eta = \text{LE} \left( X [0, \infty ) \right)$ for the loop-erasure of $X$, and $\sigma $ for the first time that $\eta$ exits $B (x, n)$. Finally, we denote the number of points lying in $\eta [0, \sigma ] \cap D_{n} (x)$ by $J^{K}_{m,n,x}$. This is an analogue of \cite[Definition 8.7]{S}.
\item Suppose that $X$ is the conditioned random walk as above. We write $G^{X} (\cdot , \cdot )$ for Green's function of $X$.
\end{itemize}
This setup is illustrated in Figure \ref{fige} (cf.\ \cite[Figure 3]{S}).

\begin{figure}[!b]
\begin{center}
  \includegraphics[width=0.75\textwidth]{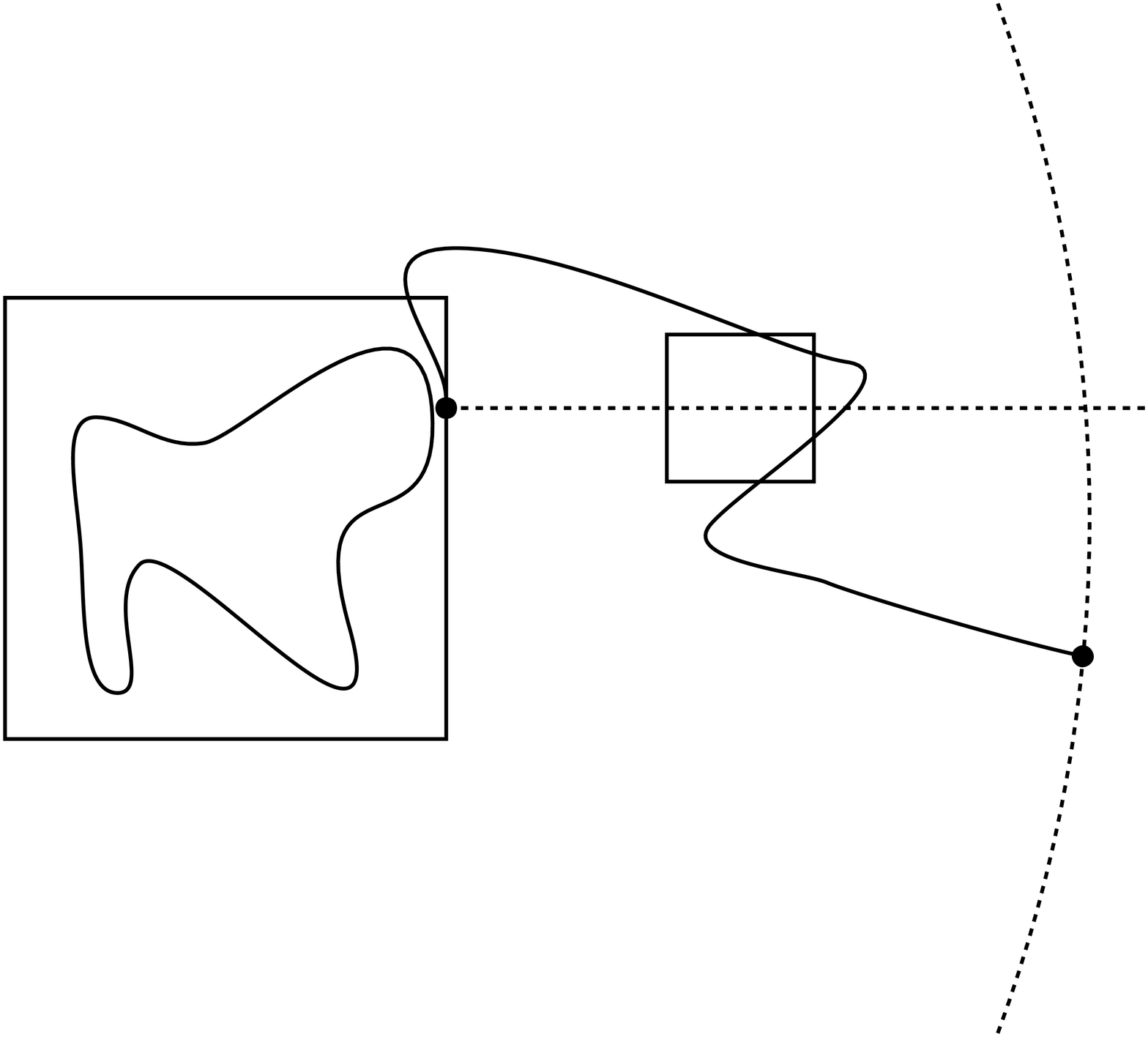}
\rput(	-325.2	pt,	208	pt	){	$D( m)$	}
\rput(	-282	pt,	160	pt	){	$K$	}
\rput(	-204	pt,	180	pt	){	$x$	}
\rput(	-126.6	pt,	198	pt	){	$D_{n}( x)$	}
\rput(	-30	pt,	194	pt	){	$\ell $	}
\rput(	-17	pt,	21.4	pt	){	$\partial B( x,n)$	}
\rput(	-95.4	pt,	125	pt	){	$\eta $	}
\rput(	-8	pt,	110.8	pt	){	$\eta ( \sigma )$	}
\end{center}
\caption{Notation used in the proof of Proposition \ref{prop:iterate}.}\label{fige}
\end{figure}

We will give one- and two-point function estimates for $\eta$ in the following proposition.

\begin{prop}
Suppose that $m, n, x, K, X, \eta, \sigma, J^{K}_{m,n,x}$ are as above. There exists a universal constant $c$ such that for all $z, w \in D_{n} (x) $ with $z \neq w$,
\begin{align}
\mathbf{P} \left( z \in \eta [0, \sigma ] \right) &\ge c n^{-3 + \beta}, \label{eq1} \\
\mathbf{P} \left( z, w \in \eta [0, \sigma ] \right) &\le \frac{1}{c} n^{-3 + \beta} |z-w|^{-3 + \beta}. \label{eq2}
\end{align}
 \end{prop}

\begin{proof}
The inequality \eqref{eq1} follows from \cite[(8.29)]{S} and \cite[Corollary 1.3]{Escape}. So, it remains to prove \eqref{eq2}. We first recall \cite[Proposition 8.1]{S}, the setting of which is as follows. Take $z_{1}, z_{2} \in D_{n} (x) $ with $z_{1} \neq z_{2}$. We set $z_{0} = x$, and write $l = |z_{1}- z_{2}|$. Note that $1 \le l \le n/100$. For $i =0,1,2$, we let $X^{i}$ be independent versions of $X$ with $X^{i} (0) = z_{i}$. We write $\sigma^{i}_{w}$ for the first time that $X^{i}$ hits $w$. For $i=0,1$, let $Z^{i}$ be $X^{i}$ conditioned on the event $\{ \sigma^{i}_{z_{i+1}} < \infty \}$, and also let $Z^{2} = X^{2}$. Also for $i=0,1$, write $u (i) $ for the last time that $Z^{i}$ passes through $z_{i+1}$, and set $u(2) = \infty$. Define the event $F^{\eta}_{z_{1}, z_{2}}$ by
\[F^{\eta}_{z_{1}, z_{2}} = \left\{ \text{There exist } 0 < t_{1} < t_{2} < \infty \text{ such that } \eta (t_{1}) = z_{1} \text{ and } \eta (t_{2} ) = z_{2} \right\},\]
and non-intersection events $F_{1}$ and $F_{2}$ by
\begin{align*}
&F_{1} = \left\{ \text{LE} \left( Z^{0} [0, u (0) ] \right) \cap \left( Z^{1} [1, u(1) ] \cup Z^{2} [1, \infty ) \right) = \emptyset \right\}, \\
&F_{2} = \left\{ \text{LE} \left( Z^{1} [0, u (1) ] \right) \cap Z^{2} [1, \infty ) = \emptyset \right\}.
\end{align*}
Then \cite[Proposition 8.1]{S} shows that
\[\mathbf{P} \left( F^{\eta}_{z_{1}, z_{2}} \right) = G^{X} (z_{0}, z_{1} ) G^{X} (z_{1}, z_{2} ) \mathbf{P} \left( F_{1} \cap F_{2} \right).\]

Now, in the proof of \cite[Lemma 8.9]{S}, it is shown that
\[ G^{X} (z_{0}, z_{1} )  \le \frac{C}{n}, \ \  G^{X} (z_{1}, z_{2} )  \le \frac{C}{l},\]
and so it suffices to estimate $\mathbf{P}( F_{1} \cap F_{2})$. To do this, we consider four balls
\[ B_{1} = B (z_{1},  l/8 ), \ B_{2} = B (z_{2}, l/8 ), \ B_{1}' = B (z_{1}, 2 l ), \ B_{1}'' = B (z_{1},  n/16 ).\]
Note that $B_{1} \cup B_{2} \subset B_{1}' \subset B_{1}''$ and $B_{1} \cap B_{2} = \emptyset$. For $i=0,1$, let $Y^{i} = \left( Z^{i} [0, u (i) ] \right)^{R} $ be the time reversal of $Z^{i} [0, u (i) ] $ where for a path $\lambda = [ \lambda (0 ), \lambda (1), \dots , \lambda ( u) ]$, we write $( \lambda )^{R} =  [ \lambda (u ), \lambda (u-1), \dots , \lambda ( 0) ]$ for its time reversal. By the time reversibility of LERW (see \cite[Lemma 7.2.1]{Lawb} for the time reversibility), we see that $\mathbf{P} (F_{1} \cap F_{2} )  = \mathbf{P} (F_{1}' \cap F_{2}' )$, where the events $F_{1}'$ and $F_{2}'$ are defined by
\begin{align*}
 &F_{1}' = \left\{ \text{LE} \left( Y^{0} [0, \sigma (0) ] \right) \cap \left( Y^{1} [0,  \sigma (1) -1 ] \cup Z^{2} [1, \infty ) \right) = \emptyset \right\}, \\
&F_{2}' = \left\{ \text{LE} \left( Y^{1} [0, \sigma (1) ] \right) \cap Z^{2} [1, \infty ) = \emptyset \right\}.
\end{align*}
 Here $\sigma (i)$ is the first time that $Y^{i}$ hits $z_{i}$. We define several random times as follows:
 \begin{itemize}
 \item $s_{0}$ is the first time that $\text{LE} \left( Y^{0} [0, \sigma (0) ] \right)$ exits $B_{1}$;
 \item $s_{2}$ is the first time that $\text{LE} \left( Y^{0} [0, \sigma (0) ] \right)$ exits $B_{1}''$;
 \item $s_{1}$ is the last time up to $s_{2}$ that $\text{LE} \left( Y^{0} [0, \sigma (0) ] \right)$ exits $B_{1}'$;
 \item $t_{0}$ is the first time that  $\text{LE} \left( Y^{1} [0, \sigma (1) ] \right)$ exits $B_{2}$;
 \item $t_{1}$ is the last time up to $\sigma (1)$ that  $ Y^{1} [0, \sigma (1) ] $ hits $\partial B_{1}$;
 \item $u_{0}$ is the first time that $Z^{2}$ exits $B_{2}$;
 \item $u_{1}$ is the first time that $Z^{2}$ exits $B_{1}''$.
 \end{itemize}
See Figure \ref{figd} for an illustration showing these random times. If we write $\gamma_{i} = \text{LE} \left( Y^{i} [0, \sigma (i) ] \right)$ for $i =0,1$, we see that $\mathbf{P} (F_{1}' \cap F_{2}' ) \le \mathbf{P} (H_{1} \cap H_{2} \cap H_{3} )$, where the events $H_{1}, H_{2}, H_{3}$ are defined by
\begin{align*}
&H_{1} = \left\{ \gamma_{0} [0, s_{0} ] \cap Y^{1} [ t_{1} , \sigma (1) -1 ] = \emptyset \right\}, \\
&H_{2} = \left\{ \gamma_{1} [0, t_{0} ] \cap Z^{2} [ 1 ,  u_{0} ] = \emptyset \right\}, \\
&H_{3} = \left\{ \gamma_{0} [s_{1}, s_{2} ] \cap Z^{2} [ u_{0} , u_{1} ] = \emptyset \right\}.
\end{align*}

\begin{figure}[!t]
\begin{center}
  \includegraphics[width=0.75\textwidth]{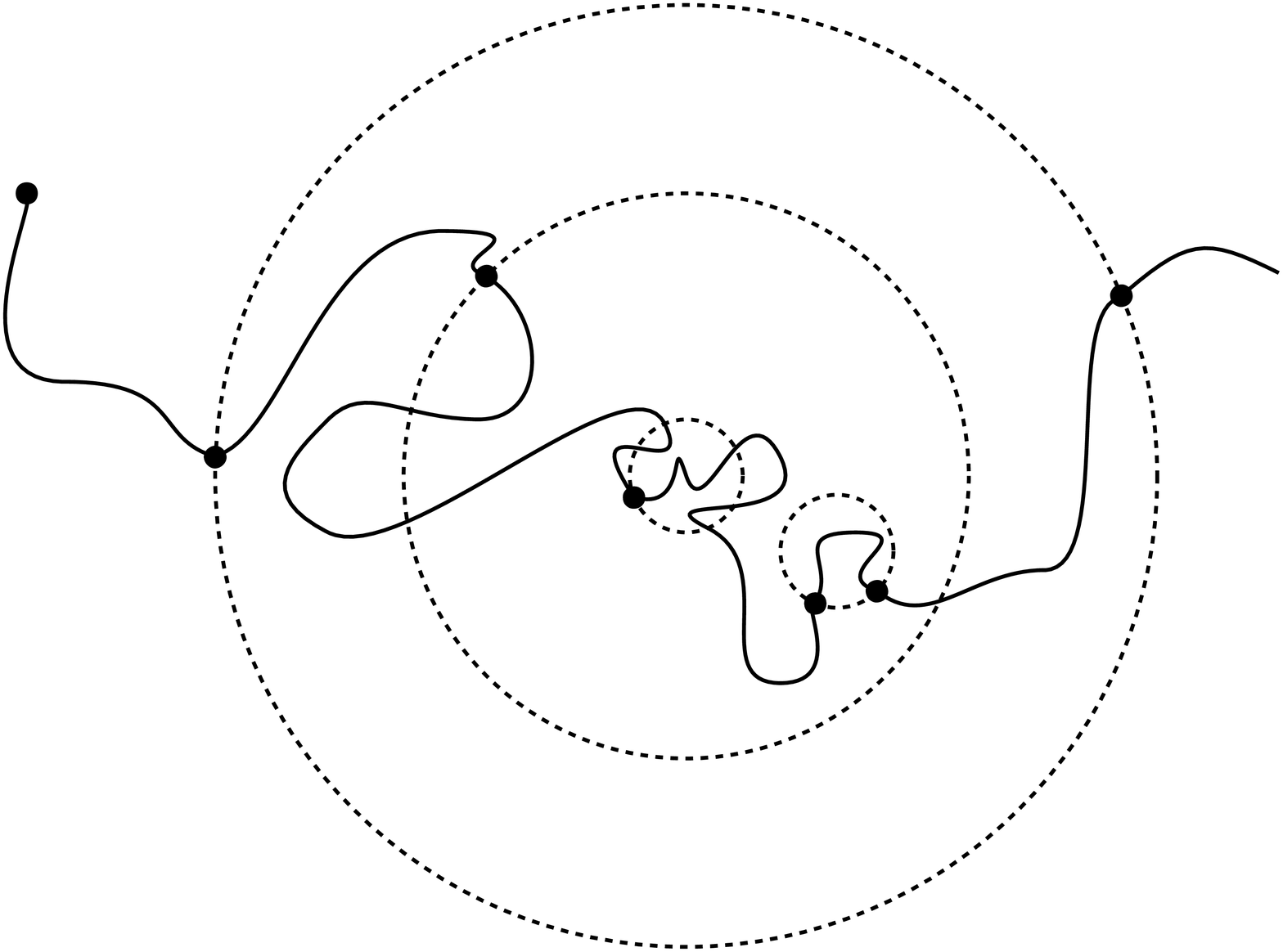}
\rput(-345pt,205pt){$z_{0}$}
\rput(-352pt,167pt){$\gamma _{0}$}
\rput(-304pt,122pt){$\gamma _{0}( s_{2})$}
\rput(-233pt,178pt){$\gamma _{0}( s_{1})$}
\rput(-191pt,110pt){$\gamma _{0}( s_{0})$}
\rput(-157pt,147pt	){	$\partial B_{1}$	}
\rput(-157pt,208pt	){	$\partial B_{1}'$	}
\rput(-157pt,258pt	){	$\partial B_{1}''$	}
\rput(-110pt,125pt	){	$\partial B_{2}$	}
\rput(	-136pt,	141pt	){	$\gamma _{1}$	}
\rput(	-66pt,	177pt	){	$Z^{2}( u_{1})$	}
\rput(	-97pt,	83pt	){	$Z^{2}( u_{0})$	}
\rput(	-143pt,	83pt	){	$\gamma _{1}( t_{0})$	}
\end{center}
\caption{The random times $s_0$, $s_1$, $s_2$, $t_0$, $u_0$, $u_1$.}\label{figd}
\end{figure}

Since $\text{dist} \left( D (m) , B_{1}'' \right) \ge n/4$, it follows from the discrete Harnack principle (see \cite[Theorem 1.7.6]{Lawb}, for example) that the distribution of $Z^{2} [0, u_{1} ]$ is comparable to that of $R_{2} [0, u_{1}']$, assuming $R_{2} (0) = z_{2}$ where $R_{2}$ is a simple random walk, and $u_{1}'$ is the first time that $R_2$ exits $B_{1}''$. More precisely, there exist universal constants $c, C \in (0, \infty)$ such that for any path $\lambda$
\[c \mathbf{P} \left( R_{2} [0, u_{1}' ] = \lambda \right) \le \mathbf{P}  \left( Z^{2} [0, u_{1} ] = \lambda \right) \le C  \mathbf{P} \left( R_{2} [0, u_{1}' ] = \lambda \right).\]
Also, since $\gamma_{0} [s_{1}, s_{2} ] \subseteq (B_{1}')^{c}  $, using the Harnack principle again, we see that
\begin{equation} \label{indep1}
\mathbf{P} (H_{1} \cap H_{2} \cap H_{3} ) \asymp E_{Y^{0}, Y^{1}} \left\{ {\bf 1}_{H_{1}} P_{R_{2}}^{z_{2}} \left(  \gamma_{1} [0, t_{0} ] \cap R_{2} [ 1 ,  u_{0}' ] = \emptyset \right)  P_{R_{2}}^{z_{1}} \left(  \gamma_{0} [s_{1}, s_{2} ] \cap R_{2} [ 0 ,  u_{1}' ] = \emptyset \right)  \right\},
\end{equation}
where $u_{0}'$ is the first time that $R_{2}$ exits $B_{2}$ and $E_{Y^{0}, Y^{1}}$ stands for the expectation with respect to the probability law of $(Y^{0}, Y^{1})$.

Another application of the Harnack principle tells that $ \gamma_{1} [0, t_{0} ]$ and $Y^{1} [ t_{1} , \sigma (1) -1 ] $ are ``independent up to constant'' (see \cite[Lemma 4.3]{LS}). Namely, there exist universal constants $c, C \in (0, \infty)$ such that for any paths $\lambda_{1}, \lambda_{2}$
\begin{align*}
\lefteqn{c \mathbf{P} \left(  \gamma_{1} [0, t_{0} ] = \lambda_{1} \right) \mathbf{P} \left(  Y^{1} [ t_{1} , \sigma (1) -1 ]  =\lambda_{2} \right)}\\
& \le  \mathbf{P} \left(  \gamma_{1} [0, t_{0} ] = \lambda_{1}, \ Y^{1} [ t_{1} , \sigma (1) -1 ]  = \lambda_{2} \right) \\
& \le C  \mathbf{P} \left(  \gamma_{1} [0, t_{0} ] = \lambda_{1} \right) \mathbf{P} \left(  Y^{1} [ t_{1} , \sigma (1) -1 ]  = \lambda_{2}\right).
\end{align*}
This implies that given $Y^{0}$, the two functions ${\bf 1}_{H_{1}}$ and $ P_{R_{2}}^{z_{2}} \left(  \gamma_{1} [0, t_{0} ] \cap R_{2} [ 1 ,  u_{0}' ] = \emptyset \right) $ are independent up to constant. Also, it is proved in \cite[Propositions 4.2 and 4.4]{Mas} that the distribution of $ \gamma_{1} [0, t_{0} ]$ is comparable with that of the ILERW started at $z_{2}$ until it exits $B_{2}$. Using the discrete Harnack principle again, we see that the distribution of the time reversal of $Y^{1} [ t_{1} , \sigma (1) -1 ] $ coincides with that of the SRW started at $z_{1}$ until it exits $B_{1}$. Therefore, if we write $R_{1}$ and $R_{3}$ for independent SRWs, the right hand side of \eqref{indep1} is comparable to
\begin{align}
\lefteqn{E_{Y^{0}} \left\{ P_{R_{2}}^{z_{1}} \left(  \gamma_{0} [s_{1}, s_{2} ] \cap R_{2} [ 0 ,  u_{1}' ] = \emptyset \right) P^{z_{1}}_{R_{1}} \left(  \gamma_{0} [0, s_{0} ] \cap  R_{1} [1, \sigma_{1}' ] = \emptyset \right) \right\} } \notag \\
&\hspace{100pt}\times P_{R_{2}, R_{3}}^{z_{2}, z_{2}} \left(  R_{2} [ 1 ,  u_{0}' ] \cap \text{LE} \left( R_{3} [0, \infty ) \right) [0, t_{3}' ] = \emptyset \right), \label{indep2}
\end{align}
where $\sigma_{1}'$ is the first time that $R_{1}$ exits $B_{1}$, and $t_{3}'$ is the first time that $\text{LE} \left( R_{3} [0, \infty ) \right)$ exits $B_{2}$. Moreover, it follows from \cite[Proposition 6.7]{S} and \cite[Corollary 1.3]{Escape} that
\[P_{R_{2}, R_{3}}^{z_{2}, z_{2}} \left(  R_{2} [ 1 ,  u_{0}' ] \cap \text{LE} \left( R_{3} [0, \infty ) \right) [0, t_{3}' ] = \emptyset \right) \asymp l^{-2 + \beta}.\]

Finally, let $R_{0}$ be a SRW started at $z_{1}$ and $\gamma_{0}' = \text{LE} \left( R_{0} [0, \infty ) \right)$ be the ILERW. Similarly to above, define:
\begin{itemize}
\item $s_{0}'$ to be the first time that $\gamma_{0}' $ exits $B_{1}$;
\item $s_{2}'$ to be the first time that $\gamma_{0}' $ exits $B_{1}''$;
\item $s_{1}'$ to be the last time up to $s_{2}'$ that $\gamma_{0}' $ exits $B_{1}'$.
\end{itemize}
We then have from \cite[Propositions 4.2 and 4.4]{Mas} that the distribution of $\gamma_{0} [0, s_{2}]$ is comparable with that of $\gamma_{0}' [0, s_{2}']$. Moreover, \cite[Proposition 4.6]{Mas} ensures that $\gamma_{0}' [0, s_{0}' ]$ and $\gamma_{0}' [s_{1}', s_{2}' ] $ are independent up to a constant. Therefore the expectation with respect to $Y^{0}$ in \eqref{indep2} is comparable to
\begin{equation}\label{indep4}
P^{z_{1}, z_{1}}_{R_{0}, R_{1}} \left(  \gamma_{0}' [0, s_{0}' ] \cap  R_{1} [1, \sigma_{1}' ] = \emptyset \right) P_{R_{0}, R_{2}}^{z_{1}, z_{1}} \left(  \gamma_{0}' [s_{1}', s_{2}' ] \cap R_{2} [ 0 ,  u_{1}' ] = \emptyset \right) \asymp \text{Es} (l) \text{Es} (l, n ),
\end{equation}
where we use the notation $\text{Es}$ defined in \cite{S}. Finally, by \cite[Corollary 1.3]{Escape}, it holds that the right hand side of \eqref{indep4} is comparable to $n^{-2 + \beta}$. This gives \eqref{eq2} and finishes the proof.
\end{proof}

\begin{dfn}\label{def1}
Suppose that $m, n, x, K, X, \eta, \sigma, J^{K}_{m,n,x}$ are as above. For $z \in B (x, n/8 )$, let $R^{z}$ be a SRW on $\mathbb{Z}^{3}$ started at $z$, independent of $X$. Write $\xi$ for the first time that $R^{z}$ exits $B (x, n)$, and let
\[N_{z} = \left|  R^{z} [0, \xi ] \cap \eta[0, \sigma ] \cap D_{n} (x) \right|\]
be the number of points in $D_{n} (x)$ hit by both $R^{z} [0, \xi ] $ and $\eta[0, \sigma ]$. Furthermore, define the (random) function $g(z)$ by setting
\[g (z) := P^{z} \left( N_{z} > 0 \right) = P^{z} \left( R^{z} [0, \xi ] \cap \left( \eta[0, \sigma ] \cap D_{n} (x) \right) \neq \emptyset \right),\]
where $P^{z}$ stands for the probability law of $R^{z}$. Note that $g (z)$ is a measurable function of $\eta [0, \sigma ]$,  and that,  given $ \eta [0, \sigma ]$, $g (\cdot )$ is a discrete harmonic function in $D_{n} (x)^{c}$.
\end{dfn}

The next proposition says that with positive probability (for $\eta$), $g (z)$ is bounded below by some universal positive constant for all $z \in B (x, n/8 )$.

\begin{prop}\label{second-moment}
Suppose that the function $g(z)$ is defined as in Definition \ref{def1}. There exists a universal constant $c_0> 0$ such that
\begin{equation}\label{2nd}
\mathbf{P} \left( g (z) \ge c_{0} \text{ for all } z \in B (x, n/8 ) \right) \ge c_{0}.
\end{equation}
\end{prop}
\begin{proof}
We claim that it suffices to show that
\begin{equation}\label{iikae}
\mathbf{P} \left( g (x) \ge c_{0} \right) \ge c_{0}
\end{equation}
for some $c_{0} > 0$. The reason for this is as follows. Suppose that \eqref{iikae} is true and the event $\{ g( x) \ge c_{0} \}$ occurs. Since $\text{dist} \left( B (x, n/8), D_{n} (x) \right) \ge n/4 $, using the Harnack principle, there exists a universal constant $c_{1} > 0$ such that $g (z) \ge c_{1} g (x) \ge c_{1} c_{0}$ for all $z \in B (x, n/8 )$. Thus we have
$\mathbf{P}( g (z)  ) \ge c_{0}c_{1} \text{ for all } z \in B (x, n/8 )) \ge c_{0}$, which gives \eqref{2nd}.

We will prove \eqref{iikae}. Recall the definition of $N_{z}$ from Definition \ref{def1}. By \eqref{eq1}, we see that
\[\mathbf{E} (N_{x} ) = \sum_{w \in D_{n} (x) } \mathbf{P} \left( w \in \eta[0, \sigma ] \right) P^{x} \left( w \in R^{x} [0, \xi ]  \right) \ge c n^{-1 + \beta }\]
for some $c > 0$. On the other hand, by \eqref{eq2}, we have
\begin{align*}
\mathbf{E} (N_{x}^{2} ) &= \sum_{w_{1}, w_{2} \in D_{n} (x) } \mathbf{P} \left( w_{1}, w_{2} \in \eta[0, \sigma ] \right) P^{x} \left( w_{1}, w_{2} \in R^{x} [0, \xi ]  \right) \\
&\le C n^{-4 + \beta} \sum_{w_{1}, w_{2} \in D_{n} (x) }  |w_{1} - w_{2} |^{-4 + \beta } \\
&\le C n^{-4 + \beta} n^{2 + \beta}\\
& = C n^{-2 + 2 \beta}.
\end{align*}
This gives $\mathbf{E} (N_{x}^{2} ) \le C \{ \mathbf{E} (N_{x} ) \}^{2}$. Therefore, the second moment method tells us that $\mathbf{E} ( g (x) ) \ge c_{2}$ for some universal constant $c_{2} > 0$. This implies $\mathbf{P} \left( g (x) \ge c_{2}/ 2 \right) \ge c_{2}/3$, which gives \eqref{iikae}.
\end{proof}

\begin{proof}[Proof of Proposition \ref{prop:iterate}]
We will prove that for each $j =1,2, \dots , q$
\begin{equation}\label{enough}
\mathbf{P} (F_{j}^{c} ) \le (1 -c_{0} )^{\lambda^{a}},
\end{equation}
where $c_{0}$ is the constant of Proposition \ref{second-moment}. Since $q \le \lambda^{1-a}$, the inequality \eqref{enough} gives the desired inequality \eqref{good}. Take $j \in\{1,2, \dots , q\}$. Suppose that $F_{j}^{c}$ occurs. This implies that for every $i \in I_{j} $, the event $A_{i}$ does not occur. Setting $l = 2 j \lambda^{a} $, we need to estimate
\[\mathbf{P} \left( \bigcap_{i= l +1}^{l + \lambda^{a} } A_{i}^{c} \right) = \mathbf{P} \left( A_{l + \lambda^{a}}^{c} \:\vline\: \bigcap_{i= l +1}^{l + \lambda^{a} -1 } A_{i}^{c} \right) \mathbf{P} \left( \bigcap_{i= l +1}^{l + \lambda^{a} -1 } A_{i}^{c} \right).\]
Note that the event $\bigcap_{i= l +1}^{l + \lambda^{a} -1 } A_{i}^{c} $ is measurable with respect to $\gamma [0, t_{l + \lambda^{a} }] $ while the event $ A_{l + \lambda^{a}}^{c}$ is measurable with respect to $\gamma [ t_{l + \lambda^{a} }, \sigma_{l + \lambda^{a} } ]$. Therefore, using the domain Markov property of $\gamma$ (see \cite[Proposition 7.3.1]{Lawb}), Proposition \ref{second-moment} tells us that
\[\mathbf{P} \left( A_{l + \lambda^{a}}^{c} \:\vline\: \bigcap_{i= l +1}^{l + \lambda^{a} -1 } A_{i}^{c} \right) \leq 1-c_{0},\]
where we apply Proposition \ref{second-moment} with $m = \frac{(l + \lambda^{a}) R}{\lambda}$, $n = \frac{R}{2 \lambda} $, $x = \gamma \left( t_{l + \lambda^{a}} \right)$ and $K =  \gamma [0, t_{l + \lambda^{a} }] $. Thus we have that
\[\mathbf{P} \left( \bigcap_{i= l +1}^{l + \lambda^{a} } A_{i}^{c} \right) \le (1-c_{0} ) \mathbf{P} \left( \bigcap_{i= l +1}^{l + \lambda^{a} -1 } A_{i}^{c} \right).\]
Repeating this procedure $\lambda^{a}$ times, we obtain \eqref{enough}, and thereby finish the proof.
\end{proof}

\section{Checking the assumptions sufficient for tightness}\label{sec:assump}

The aim of this section is to check Assumptions \ref{a1}, \ref{a2} and \ref{a3}, as set out in Section \ref{topsec}. In what follows, we let $\gamma_{{\cal U}} (x,y)$ be the unique injective path in ${\cal U}$ between $x$ and $y$. In particular, $\gamma_{{\cal U}} (x,y) (k)$ is the location at $k$th step of the path. Note that $\gamma_{{\cal U}} (x,y) (0) = x$ and $\gamma_{{\cal U}} (x,y) \left( d_{{\cal U} } (x, y) \right) = y$.

\subsection{Assumption 1}  That the first assumption holds is readily obtained from the upper inclusion for $B_{{\cal U}} \left( 0,  \delta^{-\beta} \right)$ in the following proposition. The lower inclusion, which requires only a small additional argument, is necessary for Proposition \ref{last-step} below.

\begin{prop}\label{1st-assump}
There exist constants $c, C \in (0, \infty)$ such that: for all $\lambda \ge 1$ and $\delta \in (0, 1)$,
\[\mathbf{P} \left( B \left( 0, \lambda^{-1} \delta^{-1} \right) \subseteq  B_{{\cal U}} \left( 0,  \delta^{-\beta} \right) \subseteq B \left( 0, \lambda \delta^{-1} \right) \right) \ge 1 - C\lambda^{-c}.\]
In particular, Assumption \ref{a1} in Section \ref{topsec} holds.
\end{prop}

\begin{proof}

Note that reparameterizing the upper inclusion estimate for $B_{{\cal U}} \left( 0,  \delta^{-\beta} \right)$ above gives that there exist constants $c_{1}, c_{2}, c_{3} > 0$ such that: for all $\lambda \ge 1$, $\delta \in (0, 1)$ and $R \ge 1$,
\[\mathbf{P} \left( \delta^{3} \mu_{{\cal U}} \left( B_{{\cal U}} \left( 0, R \delta^{-\beta} \right) \right) \le \lambda  \right) \ge 1- c_{1} R^{c_{2}} \lambda^{-c_{3}}.\]
Thus Assumption \ref{a1} follows.

We may assume that $\delta > 0$ is sufficiently small. The reason for this is as follows. Suppose that there exist universal constants $c, C > 0$ such that the result holds for $\delta \in (0, M^{-1}]$ where $M > 1$ is some absolute constant. For $\delta \ge M^{-1}$, it trivially follows that
\begin{equation*}
B(0,1)  \subseteq B_{{\cal U}} \left( 0,  \delta^{-\beta} \right)  \subseteq B\left(0,\delta^{-\beta}\right)\subseteq B\left(0,\delta^{-1}M^{\beta-1}\right),
\end{equation*}
where for the final inclusion we have used that $\beta>1$. Thus, since $\beta<2$, if $\delta \ge M^{-1}$ and $\lambda \ge  M$, the probabilities in the result are equal to 1. So the inequalities hold in this case.  However, when $\delta \ge M^{-1}$ and $1 \le \lambda < M$, by replacing the constant $C$ so that $C \ge M^{ c } $ if necessary, it follows that
\begin{equation*}
 1- C  \lambda^{-c} \le 0
 \end{equation*}
 in this case. Therefore, by changing the constant $C$ appropriately, the inequalities also hold for $\delta \ge M^{-1}$. Similarly, by taking $C \ge 2^{c}$, we may assume that $\lambda \ge 2$.

As mentioned above, we may assume that $\delta>0$ is sufficiently small so that
\begin{equation}\label{delta-small}
\frac{\delta^{-1}}{- 2 \log_{2} \delta + 2} \ge 10,
\end{equation}
and also that $\lambda \ge 2$. Let $\sigma$ and $\tau$ be the first time that the infinite LERW $\gamma_{\infty} = \text{LE} ( S[0, \infty  ) )$ exits $B ( \lambda^{2} \delta^{-1} ) $ and $B ( \lambda^{-4} \delta^{-1} ) $ respectively, where $S = ( S (n))_{n \ge 0}$ is a SRW on $\mathbb{Z}^{3}$ started at the origin. Define the event $F$ by setting
\[F = \left\{ \gamma_{\infty} [\sigma, \infty ) \cap B \left( 0, \lambda \delta^{- 1} \right) = \emptyset, \ \sigma \le \lambda^{4} \delta^{-\beta}, \ \tau \ge \lambda^{-9} \delta^{- \beta} \right\}.\]
Suppose that $\gamma_{\infty}$ returns to the ball $B( 0, \lambda \delta^{- 1}) $ after time $\sigma$. Then so does the SRW that defines $\gamma_\infty$ after the first time that it exits $B ( 0, \lambda^{2} \delta^{-1} )$. The probability of such a return by the SRW is, by \cite[Proposition 1.5.10]{Lawb}, smaller than $C \lambda^{-1}$ for some universal constant $C < \infty$. On the other hand, combining \cite[Theorem 1.4]{S} with \cite[Corollary 1.3]{Escape}, it follows that the probability that $\sigma$ is greater than $\lambda^{4} \delta^{-\beta}$ is bounded above by $C \exp \left\{ - c \lambda^{1/3} \right\}$ for some universal constants $c, C \in (0, \infty)$. Finally, the exponential tail lower bound on $\tau$ derived in \cite[Theorem 8.12]{S} with \cite[Corollary 1.3]{Escape} ensures that the probability that $\tau$ is smaller than $\lambda^{-9} \delta^{- \beta}$ is bounded above by $C \exp \{ - c \sqrt{\lambda} \}$ for some universal constants $c, C \in (0, \infty)$. Thus we have
\begin{equation}\label{4-1}
\mathbf{P} \left( F \right) \ge 1- C \lambda^{-1}.
\end{equation}
Note that on the event $F$, the number of steps (in $\gamma_{\infty}$) between the origin and $x \in \gamma_{\infty} \cap B \left( 0, \lambda \delta^{- 1} \right)$ is smaller than $\lambda^{4} \delta^{-\beta}$.

Next we introduce an event $G (\zeta )$, which ensures that $\gamma_{\infty}$ is a ``hittable'' set in the sense that if we consider another independent SRW $R$ whose starting point is close to $\gamma_{\infty} $, then, with high probability for $\gamma_{\infty} $, it is likely that $R$ intersects $\gamma_{\infty} $ quickly. Such hittability of LERW paths was studied in \cite[Theorem 3.1]{SS}. With this in mind, for $\zeta > 0$, we set
\[G (\zeta) = \left\{\forall x \in B \left( 0, 2  \delta^{-1} \right),  P_{R}^{x} \left( R \left[ 0, T_{R} \left( 0,  \lambda \delta^{-1} \right) \right] \cap \gamma_{\infty} = \emptyset \right) \le \lambda^{- \zeta}\right\},\]
where $R$ is a SRW, independent of $\gamma_{\infty}$, $P^{y}_{R}$ stands for its law assuming that $R(0) = y$, and $T_{R} (x, r)$ is the first time that $R$ exits $B (x, r)$. (For convenience, we omit the dependence of $G (\zeta)$ on $\delta$.) Note that the event $G (\zeta)$ roughly says that when $R (0)$ is close to  $\gamma_{\infty} $, it is likely for $R$ to intersect with $\gamma_{\infty} $ before it travels very far. From \cite[Lemma 3.2]{SS}, we have that there exist universal constants $C < \infty$ and $\zeta_{1} \in (0, 1)$ such that: for all $\lambda \ge 2$ and $\delta > 0$
\[\mathbf{P} \left( G (\zeta_{1} ) \right) \ge 1 - C \lambda^{-1}.\]

For each $k \ge 1$, let $\varepsilon_{k} = \lambda^{-\frac{\zeta_{1}}{6}} 2^{-k-10}$, $\eta_{k} = (2 k)^{-1}$ and
\[A_{k} = B ( (1+ \eta_{k}) \delta^{-1} ) \setminus B ((1-\eta_{k} ) \delta^{-1} ).\]
Write $k_{0}$ for the smallest integer satisfying $\delta^{-1} \varepsilon_{k_{0}} < 1$. We remark that the condition at \eqref{delta-small} ensures that $(1- \eta_{k_{0}} ) \delta^{-1} \le \delta^{-1} - 10 \le \delta^{-1} + 10 \le (1+ \eta_{k_{0}} ) \delta^{-1}$. Thus both the inner boundary $\partial_{i} B (0, \delta^{-1} )$ and the outer boundary $\partial B (0, \delta^{-1} )$ are contained in $A_{k_{0}}$. Here, for a subset $A$ of $\mathbb{Z}^{3}$, the inner boundary $\partial_{i} A$ is defined by
\begin{equation*}
\partial_{i} A := \left\{ x \in A \::\:\exists y \in \mathbb{Z}^{3} \setminus A \text{ such that } |x-y| =1 \right\}.
\end{equation*}
Moreover, let $D_{k}$ be a ``$\delta^{-1} \varepsilon_{k}$-net'' of $A_k$, i.e.\ $D_{k}$ is a set of lattice points in $A_k$ such that $A_{k} \subseteq \bigcup_{z \in D_{k} } B \left( z, \delta^{-1} \varepsilon_{k} \right)$. We may suppose that the number of points in $D_{k}$ is bounded above by $C \varepsilon_{k}^{-3}$. Since $\delta^{-1} \varepsilon_{k_{0}} < 1$ and $\partial_{i} B (0, \delta^{-1} ) \subseteq A_{k_{0}}$, it follows that $\partial_{i} B (0, \delta^{-1} ) \subseteq D_{k_{0}}$.

Now, to construct ${\cal U}$, we perform Wilson's algorithm (see \cite{Wilson}) as follows:
\begin{itemize}
\item Consider the infinite LERW $\gamma_{\infty} = \text{LE} ( S[0, \infty  ) )$, where $S = ( S (n))_{n \ge 0}$ is a SRW on $\mathbb{Z}^{3}$ started at the origin. We think of ${\cal U}_{0} = \gamma_{\infty } $ as the ``root'' in this algorithm.
\item Consider a SRW started at a point in $D_{1}$, and run until it hits ${\cal U}_{0}$; we add its loop-erasure to ${\cal U}_{0}$, and denote the union of them by ${\cal U}_{1}^{1}$.  We next consider a SRW from another point in $D_{1}$ until it hits ${\cal U}_{1}^{1}$; let ${\cal U}_{1}^{2}$ be the union of ${\cal U}_{1}^{1}$ and the loop-erasure of the second SRW. We continue this procedure until all points in $D_{1}$ are in the tree. Write ${\cal U}_{1}$ for the output random tree.
\item We now repeat the above procedure for $D_{2}$. Namely, we think of ${\cal U}_{1}$ as a root and add a loop-erasure of SRWs from each point in $D_{2}$. Let ${\cal U}_{2}$ be the output tree. We continue inductively to define ${\cal U}_{3}, {\cal U}_{4}, \dots, {\cal U}_{k_{0}}$.
\item Finally, we perform Wilson's algorithm for all points in $\mathbb{Z}^{3} \setminus {\cal U}_{k_{0}}$ to obtain ${\cal U}$.
\end{itemize}
Note that, by construction, ${\cal U}_{k} \subseteq {\cal U}_{k+1}$, and also $\partial_{i} B (0, \delta^{-1} )  \subseteq {\cal U}_{k_{0}}$.

Next, for each $x \in \mathbb{Z}^{3}$ and $0 \le j \le k_{0}$, let $t_{x, j} = \inf \{ k \ge 0 :\:\gamma_{{\cal U}} (x, 0)  (k) \in {\cal U}_{j} \}$
be the first time that $\gamma_{{\cal U}} (x, 0) $ hits ${\cal U}_{j}$ (recall that $\gamma_{{\cal U}}$ is defined in the beginning of this section). We write $\gamma_{{\cal U}} (x, {\cal U}_{j} ) = \gamma_{{\cal U}} (x, 0) [0, t_{x,j}]$ for the path in ${\cal U}$ connecting $x$ and ${\cal U}_{j}$. We remark that $t_{x,j} = 0$ and $\gamma_{{\cal U}} (x, {\cal U}_{j} ) = \{ x \}$ when $x \in {\cal U}_{j}$.

We condition the root $\gamma_{\infty}$ on the event $F \cap G (\zeta_{1} )$ and think of it as a deterministic set. Since the number of points in  $D_{1}$ is bounded above by $C \lambda^{\zeta_{1} / 2 }$, it follows that with high (conditional) probability, every branch $\gamma_{{\cal U}} (x, {\cal U}_{0} ) $ with $x \in D_{1}$ is contained in $ B (0, \lambda \delta^{-1} )$.
On the other hand, suppose that $\gamma_{{\cal U}} (x, {\cal U}_{0} )$ hits $B \left( 0, \lambda^{-4} \delta^{-1} \right)$ for some $x \in D_{1}$. This implies that in Wilson's algorithm, as described above, the SRW $R$ started at $x$ enters into $B \left( 0, \lambda^{-4} \delta^{-1} \right)$ before it hits $\gamma_{\infty}$. Since $\delta^{-1} / 2 \le |x| \le 2 \delta^{-1}$, it follows from \cite[Proposition 1.5.10]{Lawb} that
\begin{equation*}
P_{R}^{x} \left( R[0, \infty ) \cap B \left( 0, \lambda^{-4} \delta^{-1} \right) \neq \emptyset \right) \le C \lambda^{-4}
\end{equation*}
for some universal constant $C < \infty$, where $R$ is a SRW started at $x$, with $P_{R}^{x}$ denoting the law of the latter process. Taking the sum over $x \in D_{1}$ (recall that the number of points in $D_{1}$ is comparable to $\lambda^{\frac{\zeta_{1}}{2}}$ and that $0 < \zeta_{1} < 1$), we find that with high (conditional) probability, every branch $\gamma_{{\cal U}} (x, {\cal U}_{0} ) $ with $x \in D_{1}$ does not hit $B \left( 0, \lambda^{-4} \delta^{-1} \right)$. Consequently, if we define the event $H$ by
\begin{align*}
&H = \Big\{ \gamma_{{\cal U}} (x, {\cal U}_{0} ) \subseteq B (0, \lambda \delta^{-1} ),  \  \gamma_{{\cal U}} (x, {\cal U}_{0} ) \cap B \left( 0, \lambda^{-4} \delta^{-1} \right)  = \emptyset \\
& \  \  \    \    \    \   \    \  \text{ and } d_{{\cal U}} (x, {\cal U}_{0} ) \le \lambda^{2} \delta^{-\beta} \text{ for all } x \in D_{1}  \Big\},
\end{align*}
where $d_{{\cal U}} (x, {\cal U}_{0} )$ stands for the number of steps of the branch $\gamma_{{\cal U}} (x, {\cal U}_{0} ) $, then the condition of the event $G (\zeta_{1} )$, \cite[Theorem 1.4]{S} and \cite[Corollary 1.3]{Escape} ensure that the conditional probability of the event $H $ satisfies
\[P  (H ) \ge 1 - C \lambda^{-\zeta_{1} / 2}.\]

To complete the proof of the upper inclusion estimate on $B_{{\cal U}} \left( 0,  \delta^{-\beta} \right)$, we will consider several ``good'' events that ensure hittability of $\gamma_{\infty} \cup \gamma_{{\cal U}} (x, {\cal U}_{0} )$ with $x \in D_{k}$ ($k=1,2, \dots , k_{0}$),
similarly to the event $G (\zeta )$ defined above. Namely, for $k \ge 1$ and $\zeta > 0$, we define the event $I (k, \zeta)$ by setting
\begin{align}
&I (k, \zeta ) = \notag \\
&\left\{\forall x \in D_{k},\:y \in B \left( x, \varepsilon_{k} \delta^{-1} \right):
P_{R}^{y} \left( R \left[ 0, T_{R} \left( x,  \varepsilon_{k}^{1/2}  \delta^{-1} \right) \right] \cap \left( \gamma_{\infty} \cup \gamma_{{\cal U}} (x, {\cal U}_{0} ) \right) = \emptyset \right) \le \varepsilon_{k}^{ \zeta} \right\}. \label{event-i}
\end{align}
See Figure \ref{fig1} for this event.
From \cite[Lemma 3.2]{SS}, we see that the probability of the event $I (k, \zeta)$ is greater than $1-C \varepsilon_{k}^{2}$ if we take $\zeta$ sufficiently small. The reason for this is as follows. Suppose that the event $I (k, \zeta )$ does not occur, which means that there exist $x \in D_{k}$ and $y  \in B ( x, \varepsilon_{k} \delta^{-1} )$  such that the probability considered in \eqref{event-i} is greater than $\varepsilon_{k}^{\zeta}$. The existence of those two points $x \in D_{k}$ and $ y  \in B ( x, \varepsilon_{k} \delta^{-1} )$ implies the occurrence of the event $I (x, k, \zeta)$, as defined by
\[I (x, k, \zeta ) = \left\{\exists y \in B \left( x, \varepsilon_{k} \delta^{-1} \right) \text{ such that }P_{R}^{y} \left( R \left[ 0, T_{R} \left( x, \varepsilon_{k}^{1/2} \delta^{-1} \right) \right] \cap \gamma_{\infty}^{x} = \emptyset \right) > \varepsilon_{k}^{ \zeta}\right\},\]
where we write $\gamma_{\infty}^{x}$ for the unique infinite path started at $x$ in ${\cal U}$ (notice that $\gamma_{\infty}^{0} = \gamma_{\infty}$). Namely, we have
\begin{equation*}
I (k, \zeta)^{c} \subseteq \bigcup_{x \in D_{k}} I (x, k, \zeta ).
\end{equation*}
We mention that the distribution of $\gamma_{\infty}^{x}$ coincides with that of the infinite LERW started at $x$. With this in mind, applying \cite[Lemma 3.2]{SS} with $s = \varepsilon_{k} \delta^{-1} $, $t= \varepsilon_{k}^{1/2}   \delta^{-1}$ and $K = 10$, it follows that there exist universal constants $\zeta_{2} > 0$ and $C < \infty$ such that for all $k \ge 1$, $\lambda \ge 2$, $\delta \in (0, 1)$ and $x \in D_{k}$,
\[\mathbf{P} \left( I (x, k, \zeta_{2} ) \right) \le C \varepsilon_{k}^{5}.\]
Since the number of points in $D_{k}$ is bounded above by $ C \varepsilon_{k}^{-3}$, we see that
\[\mathbf{P} \left( I (k, \zeta_{2} ) \right) \ge 1 - C \varepsilon_{k}^{2},\]
as desired.

\begin{figure}[b!]
\begin{center}
  \includegraphics[width=0.75\textwidth]{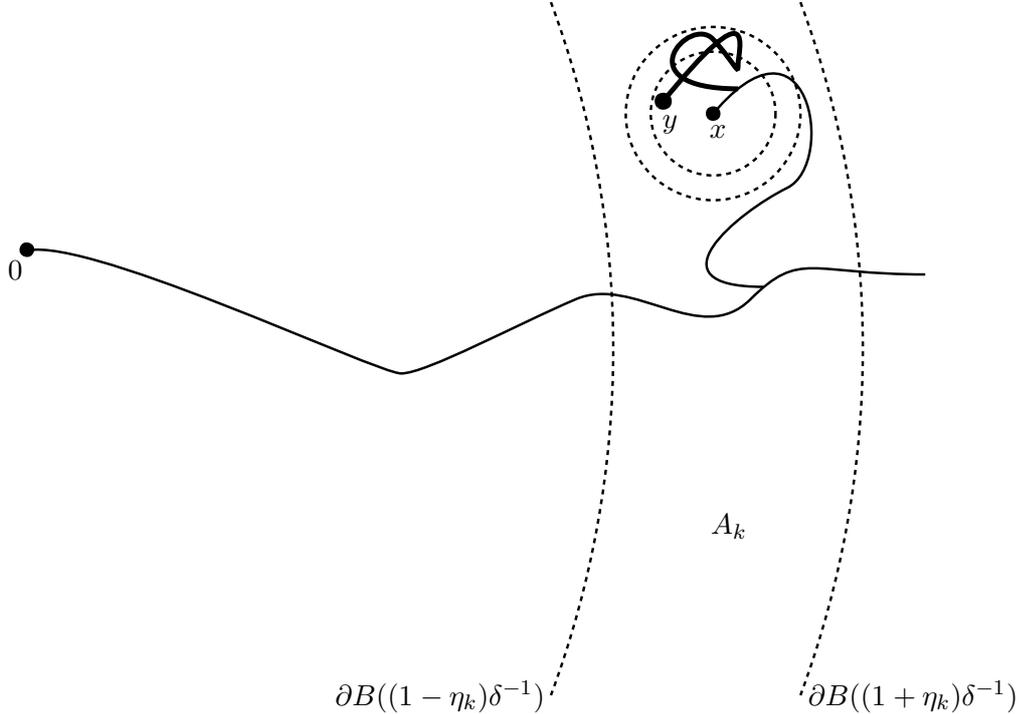}
  \rput(-345pt,161pt){$0$}
  \rput(-100pt,216pt){$y$}
  \rput(-82pt,213pt){$x$}
  \rput(-78pt,64pt){$A_k$}
  \rput(-9pt,0pt){$\partial B((1+ \eta_{k}) \delta^{-1})$}
  \rput(-186pt,0pt){$\partial B((1-\eta_k)\delta^{-1})$}
\end{center}
\caption{On the event $I(k,\zeta)$, as defined at \eqref{event-i}, the above configuration occurs with probability greater than $1-\varepsilon_k^\zeta$ for any $x\in D_k$, $y\in B(x,{\varepsilon_k}\delta^{-1})$. The circles shown are the boundaries of $B(x,{\varepsilon_k}\delta^{-1})$ and $B(x,\sqrt{\varepsilon_k}\delta^{-1})$. The non-bold paths represent $\gamma_\infty\cup\gamma_\mathcal{U}(x,\mathcal{U}_0)$, and the bold path $R[0,T_R(x,\sqrt{\varepsilon_k}\delta^{-1})]$.}\label{fig1}
\end{figure}

Set $A_{1}' := F \cap G (\zeta_{1} ) \cap H \cap  I (1, \zeta_{2} )$. Note that the event $A_{1}'$ is measurable with respect to ${\cal U}_{1}$ (recall that ${\cal U}_{1} $ is the union of $\gamma_{\infty}$ and all branches $ \gamma_{{\cal U}} (x, {\cal U}_{0} ) $ with $x \in D_{1}$). We have already proved that $\mathbf{P} (A_{1}' ) \ge 1 - C \lambda^{-1}$. Moreover, we note that on the event $A_{1}'$, it follows that
\begin{itemize}
\item  $d_{{\cal U}} (0, y) \ge \lambda^{-9} \delta^{-\beta}$ for all $y \in \gamma_{\infty} \cap B \left( 0, \delta^{-1} /3 \right)^{c}$,

\item  $d_{{\cal U}} (0, y) \le 2 \lambda^{4} \delta^{-\beta}$ for any $y \in \left( \gamma_{\infty} \cap B (0, \lambda \delta^{-1} ) \right) \cup (\cup_{x \in D_{1} } \gamma_{{\cal U}} (x, {\cal U}_{0} ))$,

\item the branch $ \gamma_{{\cal U}} (x, {\cal U}_{0} ) $ does not intersect with $B \left( 0, \lambda^{-4} \delta^{-1} \right)$ for any $x \in D_{1}$,

\item $d_{{\cal U}} (0, y) \ge \lambda^{-9} \delta^{-\beta}$ for all $y \in \gamma_{{\cal U}} (x, {\cal U}_{0} ) $ with $x \in D_{1}$.
\end{itemize}
Conditioning ${\cal U}_{1} $ on the event $A_{1}'$, we perform Wilson's algorithm for points in $D_{2}$. It is convenient to think of ${\cal U}_{1} $ as deterministic sets in this algorithm. Adopting this perspective, we take $y \in D_{2}$ and consider the SRW $R$ started at $y$ until it hits ${\cal U}_{1} $. Suppose that $R$ hits $\partial B (y, \varepsilon_{1}^{1/3} \delta^{-1}) $ before it hits ${\cal U}_{1}$. Since the number of ``$\varepsilon_{1}^{1/2} \delta^{-1}$-displacements'' of $R$ until it hits $\partial B (y, \varepsilon_{1}^{1/3} \delta^{-1})$ is bigger than $10^{-1} \varepsilon_{1}^{-1/6}$, the hittability condition $I (1, \zeta_{2} )$ ensures that
\begin{equation}\label{iterate}
P^{y}_{R} \left( R \text{ hits } \partial B (y, \varepsilon_{1}^{1/3} \delta^{-1}) \text{ before it hits } {\cal U}_{1} \right) \le \varepsilon_{1}^{c \zeta_{2} \varepsilon_{1}^{-1/6}},
\end{equation}
for some universal constant $c > 0$. On the other hand, it follows from \cite[Theorem 1.4]{S} and \cite[Corollary 1.3]{Escape} that
\begin{equation*}
P^{y}_{R} \left( \gamma_{{\cal U}} (y, {\cal U}_{1} ) \cap  \partial B (y, \varepsilon_{1}^{1/3} \delta^{-1}) = \emptyset \text{ and }  d_{{\cal U}} (y, {\cal U}_{1} ) > \varepsilon_{1}^{1/4} \delta^{-\beta} \right) \le C \exp \{ -c \varepsilon_{1}^{-\frac{1}{12}} \}.
\end{equation*}
With this in mind, we define the event $B_{2}$ by
\[B_{2} = \left\{ \forall y \in D_{2}, \gamma_{{\cal U}} (y, {\cal U}_{1} ) \cap  \partial B (y, \varepsilon_{1}^{1/3} \delta^{-1}) = \emptyset \text{ and }  d_{{\cal U}} (y, {\cal U}_{1} ) \le \varepsilon_{1}^{1/4} \delta^{-\beta} \right\}.\]
Taking the sum over $y \in D_{2}$, the conditional probability (recall that we condition ${\cal U}_{1}$ on the event $A_{1}'$) of the event $B_{2}$ satisfies
\[\mathbf{P} (B_{2} ) \ge 1 - C \varepsilon_{1}^{-3} \exp \{ -c \varepsilon_{1}^{-\frac{1}{12}} \}.\]
Thus, letting $A_{2}' := A_{1}' \cap B_{2} \cap I (2, \zeta_{2} )$, it follows that
\[\mathbf{P} (A_{2}' ) \ge 1 -C \lambda^{-\zeta_{1}/6},\]
where we also use that $\varepsilon_{1}$ is comparable to $ \varepsilon_{2}$, and that the number of points in $D_{2}$ is comparable to $\varepsilon_{2}^{-3}$.

Conditioning ${\cal U}_{2}$ on the event $A_{2}'$, we can do the same thing as above for a SRW started at $z \in D_{3}$. Hence if we define the event $B_{3}$ by setting
\[B_{3} = \left\{ \forall y \in D_{3}, \gamma_{{\cal U}} (y, {\cal U}_{2} ) \cap  \partial B (y, \varepsilon_{2}^{1/3} \delta^{-1}) = \emptyset \text{ and }  d_{{\cal U}} (y, {\cal U}_{2} ) \le \varepsilon_{2}^{1/4} \delta^{-\beta} \right\},\]
then the conditional probability of the event $B_{3}$ satisfies
\[\mathbf{P} (B_{3} ) \ge 1 - C \varepsilon_{2}^{-3} \exp \{ -c \varepsilon_{2}^{-\frac{1}{12}} \}.\]
So, letting $A_{3}' := A_{2}' \cap B_{3} \cap I (3, \zeta_{2} )$, it follows that
\[\mathbf{P} (A_{3}' ) \ge 1 -C \lambda^{-\zeta_{1}/6},\]
and we continue this until we reach the index $k_{0}$. In particular, if we define $B_{k}$ and $A_{k}'$ for each $k = 2, 3, \dots , k_{0}$ by
\[B_{k} = \left\{ \forall y \in D_{k}, \gamma_{{\cal U}} (y, {\cal U}_{k-1} ) \cap  \partial B (y, \varepsilon_{k-1}^{1/3} \delta^{-1}) = \emptyset \text{ and }  d_{{\cal U}} (y, {\cal U}_{k-1} ) \le \varepsilon_{k-1}^{1/4} \delta^{-\beta} \right\},\]
and $A_{k}' := A_{k-1}' \cap B_{k} \cap I (k, \zeta_{2} )$, we can conclude that
\[\mathbf{P} \left(A_{k_{0}}' \right) = \mathbf{P} \left(A_{1}'  \right) \prod_{k=2}^{k_{0}} \mathbf{P} \left(A_{k}'  | A_{k-1}' \right)  \ge \left(1 - C \lambda^{-\zeta_{1}/6}\right)  \prod_{k=1}^{\infty} \left( 1 - C \varepsilon_{k}^{2} \right) \ge 1 - C \lambda^{-\zeta_{1}/6}.\]
However, on the event $A_{k_{0}}'$, it is easy to see that:
\begin{itemize}
\item $d_{{\cal U}} (0, y) \ge \lambda^{-9} \delta^{-\beta} $ for all $y \in \gamma_{\infty} \cap B \left( 0, \delta^{-1}/3 \right)^{c}$,
\item $d_{{\cal U}} (0, y) \ge \lambda^{-9} \delta^{-\beta} $ for all $y \in \gamma_{{\cal U}} (x, {\cal U}_{0} ) $ and all $x \in D_{k}$ with $k = 1, 2, \dots , k_{0}$,
\item $d_{{\cal U}} (0, y) \le C \lambda^{4} \delta^{-\beta}$ for all $y \in  ( \gamma_{\infty} \cap B (0, \lambda \delta^{-1} ) ) \cup  {\cal U}_{k_{0}}$
\end{itemize}
This implies that $d_{{\cal U}} (0, y) \ge \lambda^{-9} \delta^{-\beta} $ for all $y \in B (0, \delta^{-1} )^{c}$ on the event $A_{k_{0}}'$ since $\partial_{i} B (0, \delta^{-1} ) \subseteq {\cal U}_{k_{0}}$. Therefore, it follows that
\[\mathbf{P} \left( B_{{\cal U}} \left( 0, \lambda^{-9} \delta^{-\beta} \right) \subseteq B (0, \delta^{-1} ) \right) \ge 1 - C \lambda^{-c}.\]
Reparameterizing this, we have
\begin{equation}\label{fukutsu-2}
\mathbf{P} \left( B_{{\cal U}} \left( 0, R \delta^{-\beta} \right) \subseteq B \left( 0, \lambda \delta^{-1} \right) \right) \ge 1 - c_{1} R^{c_{2}} \lambda^{-c_{3}},
\end{equation}
for some universal constants $c_{1}, c_{2}, c_{3} >0$. This gives the upper inclusion of the proposition.

For the lower inclusion estimate, since ${\cal U}_{k_{0}}$ contains $\partial_{i} B (0,  \delta^{-1} )$, using \cite[Theorem 1.4]{S} and \cite[Corollary 1.3]{Escape} again, we see that
\begin{equation}\label{tsukau}
\mathbf{P} \left( d_{{\cal U}} (0, y) \le C \lambda^{4} \delta^{-\beta} \text{ for all } y \in B (0,  \delta^{-1}) \right) \ge 1 - C \lambda^{-c},
\end{equation}
for some universal constants $c,C > 0$. This implies that
\[\mathbf{P} \left(  B (0,  \delta^{-1})   \subseteq  B_{{\cal U}} \left( 0,  C \lambda^{4} \delta^{-\beta} \right) \right) \ge 1- C \lambda^{-c}.\]
Reparameterizing this, it follows that for all $\lambda \ge 1$, $\delta \in (0, 1)$ and $\theta \in (0,1]$,
\begin{equation}\label{fukutsu}
\mathbf{P} \left(    B \left( 0, \lambda^{-1} \delta^{-1} \right)     \subseteq  B_{{\cal U}} \left( 0, \theta \delta^{-\beta} \right) \right) \ge 1- c_{1} \theta^{-c_{2}} \lambda^{-c_{3}},
\end{equation}
which finishes the proof.
\end{proof}

\subsection{Assumption \ref{a2}}

We will prove the following variation on Assumption \ref{a2}. Given Proposition \ref{1st-assump}, it is easy to check that this implies Assumption \ref{a2}. The restriction of balls to the relevant Euclidean ball will be useful in the proof of the scaling limit part of Theorem \ref{mainthm1}.

{\assu\label{a2'} For every $\varepsilon, R\in (0,\infty)$, it holds that
\[\lim_{\eta\rightarrow 0}\limsup_{\delta\rightarrow 0}
\bP\left(\inf_{x\in B(\delta^{-1}R)}\delta^3\mU\left(\BU(x,\delta^{-\beta}\varepsilon)\cap B(\delta^{-1}R)\right)<\eta\right)=0.\]}

We recall that Proposition \ref{1st-assump} gives a lower bound on the volume of $B_{\cal U} (0, \theta \delta^{-\beta} ) $ for each fixed $\theta \in (0,1]$ (see \eqref{fukutsu} for this). The next proposition shows that Assumption \ref{a2'} holds. We mention that a similar reparameterizing technique used in \eqref{fukutsu-2} and \eqref{fukutsu} cannot be used here, since only one of the parameters $R$ and $\theta$ can be eliminated in this way.

\begin{prop}\label{last-step}
There exist constants $c_{1}, c_{2}, c_{3}$ such that: for all $\lambda \ge 1$, $\delta \in (0, 1)$, $\theta \in (0, 1]$ and $R \ge 1$,
\begin{equation}\label{1255}
\bP\left(\inf_{x\in B(R \delta^{-1} )}\delta^3\mU\left(\BU(x,\theta\delta^{-\beta})\cap B( R \delta^{-1} )\right)<\lambda^{-1}\right) \le c_{1} ( R + \theta^{- 1} )^{c_{2}} \lambda^{-c_{3}}.
\end{equation}
In particular, Assumption \ref{a2'} holds.
\end{prop}

\begin{proof}
Fix parameters $\theta \in (0, 1]$ and $R \ge 1$. By the reason explained at the beginning of the proof of Proposition \ref{1st-assump}, we may assume that
\begin{equation}\label{1110}
\lambda \ge \max \{ 2, R^{c_{\ast}} \}
\end{equation}
 is sufficiently large and that $\delta > 0$ is sufficiently small, where $c_{\ast} > 0$ is some absolute constant which will be given at \eqref{c-ast} below.

Let $\gamma_{\infty}^{v}$ be the infinite simple path in ${\cal U}$ started at $v$. When $v=0$, we write $\gamma_{\infty}^{0} = \gamma_{\infty}$. Fix a point $v$. In order to prove \eqref{1255}, we will first give a lower bound on the volume of $B_{{\cal U}} \left( x, \theta \delta^{-\beta} \right) \cap B(0,  R \delta^{-1}  )$ uniformly in $x \in \gamma_{\infty}^{v} \cap B (0,  R \delta^{-1}  )$. Namely, we will show that there exist universal constants $b_{0}, c_{1}, c_{2}, c_{3} \in (0, \infty) $ such that for all $v \in B ( 0, \lambda^{b_{0}} \delta^{-1} )$,
\begin{equation}
\mathbf{P} \left( \inf_{x \in \gamma_{\infty}^{v} \cap B (0, R \delta^{-1} )} \delta^{3} \mu_{{\cal U}} \left( B_{{\cal U}} \left( x, \theta \delta^{-\beta} \right)  \cap B (0, R \delta^{-1}  ) \right) < \lambda^{-1}  \right)
\le c_{1} (R + \theta^{-1} )^{c_{2}} \lambda^{-c_{3}}. \label{warm-2nd}
\end{equation}

Recall that we proved at \eqref{tsukau} that there exist universal constants $a_{1}, a_{2} \in (0, \infty )$ such that
\begin{equation}\label{use-1}
\mathbf{P} (A^{1} ) := \mathbf{P} \left( B (0, \lambda^{-5} \delta^{-1}) \subseteq  B_{{\cal U}} \left( 0, a_{2} \lambda^{-1/4} \delta^{-\beta} \right)   \right) \ge 1 - a_{2} \lambda^{-a_{1}}.
\end{equation}
Similarly to previously, we need to deal with the hittability of $\gamma_{\infty}$. To this end, define the event $A (\zeta)$ by
\[A ( \zeta ) = \left\{\sup_{x \in B (0, \lambda \delta^{-1} ):\:\text{dist} \left( x, \gamma_{\infty}  \right) \le \lambda^{-\frac{a_{1}}{10}}  \delta^{-1}} P_{R}^{x} \left( R \left[ 0, T_{R} \left( x, \lambda^{-\frac{a_{1}}{20}}  \delta^{-1} \right) \right] \cap \gamma_{\infty} = \emptyset \right) \le \lambda^{-\zeta a_{1}}\right\}.\]
From \cite[Lemmas 3.2 and 3.3]{SS}, it follows that there exist universal constants $\zeta_{4} \in (0, 1)$ and $C < \infty$ such that
\begin{equation}\label{zeta4}
\mathbf{P} \left( A (\zeta_{4} ) \right) \ge 1 - C \lambda^{-a_{1}}.
\end{equation}
Now we let $b_{0} = a_{1} \zeta_{4} /5000$ and take $v \in B \left( 0, \lambda^{b_{0}}  \delta^{-1} \right)$, henceforth in this proof only, we write $\gamma^{v}_{\infty} = \gamma_{\infty}$ to simplify notation. We also write $\rho_{0}$ for the first time that $\gamma_{\infty}$ exits $B (0, R \delta^{-1} )$ (we set $\rho_{0} = 0$ if $v \notin B (0, R \delta^{-1} )$), set $L = \lambda^{\frac{\zeta_{4} a_{1}}{100}} \delta^{-1}$, and define $\rho$ to be the first time that $\gamma_{\infty}$ exits $B (0, L)$.
Now we define the constant $c_{\ast} $ in \eqref{1110}. We let
\begin{equation}\label{c-ast}
c_{\ast} = \frac{200}{\zeta_{4} a_{1}}.
\end{equation}
The condition \eqref{1110} ensures that $L \ge R^{2} \delta^{-1}$. Then a similar argument used to deduce \eqref{4-1} gives that
\[\mathbf{P} (A^{2} ) := \mathbf{P} \left( \gamma_{\infty} [\rho, \infty ) \cap B (0, R \delta^{-1} ) = \emptyset, \ \rho \le \lambda^{\frac{a_{1}}{50}} \delta^{-\beta} \right) \ge 1 - C R  \lambda^{-\frac{\zeta_{4}  a_{1}}{100}} .\]
So, it suffices to deal with the event that there exists an $x \in \gamma_{\infty} [0, \rho ] \cap B (0,  R \delta^{-1} )$ for which the volume of $B_{{\cal U}} \left( x, \theta \delta^{-\beta} \right)  \cap B (0, R \delta^{-1} )$ is less than $\lambda^{-1} \delta^{-3}$.

Given these preparations, and moreover writing $r = \lambda^{-\frac{\zeta_{4} a_{1}}{100}} \delta^{-1}$, we decompose the path $\gamma_{\infty} [0, \rho]$ in the following way.
\begin{itemize}
\item Let $\tau_{0} = 0$. For $l \ge 1$, define $\tau_{l}$ by $\tau_{l} = \inf \{ j \ge \tau_{l-1} :\: | \gamma_{\infty} (j) - \gamma_{\infty}  (\tau_{l-1} ) | \ge r \}$.
\item Let $N$ be the unique integer such that $\tau_{N-1} < \rho_{0} \le \tau_{N}$.
\item Set $\tau_{1}' = \inf \{ j \ge \rho_{0} :\: |\gamma_{\infty} (j) - \gamma_{\infty} (\rho_{0} ) | \ge r \}$.
\item For $l \ge 1$, if $\rho_{l-1} < \rho$, then we define $\tau_{l}' = \inf \{ j \ge \rho_{l-1}:\:  |\gamma_{\infty} (j)  - \gamma_{\infty} (\rho_{l-1} ) | \ge r  \}$ and set $\rho_{l} = \inf \{ j \ge \tau_{l}' :\: \gamma_{\infty} (j) \in B (0, R \delta^{-1} ) \} \wedge \rho$. Otherwise, we let $\tau_{l}' = \infty$ and $\rho_{l} = \rho$.
\item Let $N'$ be the smallest integer $l$ such that $\rho_{l} = \rho$.
\item For $0 \le l \le N'-1$, we let $\tau_{l}'' = \max \{ j \le \rho_{l}  :\: |\gamma_{\infty} (j) -  \gamma_{\infty} ( \rho_{l} ) | \ge r \}$ if it is the case that $ \{ j \le \rho_{l}  :\: |\gamma_{\infty} (j) -  \gamma_{\infty} ( \rho_{l} ) | \ge r \} \neq \emptyset $. Otherwise, we set $\tau_{l}'' = \rho_{l}$.
\end{itemize}
Notice that we don't consider the sequence $\{ \tau_{l} \}$ if $v \notin B (0, R \delta^{-1} )$ since $\tau_{0} = \rho_{0} = 0$ in that case. (Namely, if $v \notin B (0, R \delta^{-1} )$, we only consider the sequence $\{ \tau_{l}' \}$.) We also note that for any $x \in \gamma_{\infty} [\rho_{0}, \rho] \cap B ( 0, R \delta^{-1} )$, there exists  $0 \le l  < N'$ such that $x \in \gamma_{\infty} [\rho_{l}, \tau_{l+1}' ] $.

Our first observation is that by considering the same decomposition for the corresponding SRW, it follows that the probability that $N + N' \ge \lambda^{\zeta_{4} a_{1}/10}$ is smaller than $C \exp \{ - c \lambda^{c} \}$. Furthermore, applying \cite[Theorem 1.4]{S} together with \cite[Corollary 1.3]{Escape}, with probability at least $1- C \exp \{ - c \lambda^{c} \}$, it holds that $\tau_{l} - \tau_{l-1} \le \lambda^{-\frac{\zeta_{4} a_{1}}{200}} \delta^{-\beta} $ for all $l =1,2, \dots, N$, and that $\tau_{l}' - \rho_{l-1} \le \lambda^{-\frac{\zeta_{4} a_{1}}{200}} \delta^{-\beta} $ for all $l =1,2, \dots, N'$. Consequently,
\[\mathbf{P} (A^{3} ) \ge 1- C \exp \{ - c \lambda^{c} \},\]
where the event $A^{3}$ is defined by setting
\[A^{3} =  \left\{
\begin{array}{c}
   N + N' \le \lambda^{\frac{\zeta_{4} a_{1}}{10}}, \ \tau_{l} - \tau_{l-1} \le \lambda^{-\frac{\zeta_{4} a_{1}}{200}} \delta^{-\beta}  \text{ for all } l =1,2, \dots, N\\
   \text{ and } \tau_{l+1}' - \tau_{l}'' \le \lambda^{-\frac{\zeta_{4} a_{1}}{200}}
 \delta^{-\beta} \text{ for all } l =0,1, \dots, N'-1
\end{array} \right\}.\]

Replacing the constant $\zeta_{4}$ by a smaller constant if necessary, \cite[Theorem 6.1]{SS} (see Proposition \ref{result:quasiloops}) guarantees that $\gamma_{\infty}$ has no ``quasi-loops''. Namely, it follows that
\[\mathbf{P} (A^{4} )  \ge 1 - C \lambda^{- c a_{1}},\]
where the event $A^{4}$ is defined by setting
\[A^{4}
=  \left\{
\begin{array}{c}
 B \left( \gamma_{\infty} (\tau_{l} ), \lambda^{-\frac{a_{1}}{30}} \delta^{-1} \right) \cap \left( \gamma_{\infty} [0, \tau_{l-1}] \cup  \gamma_{\infty} [\tau_{l+1}, \infty ) \right) = \emptyset \text{ for all } l= 1,2, \dots , N \text{ and}  \\
B \left( \gamma_{\infty} (\rho_{l} ), \lambda^{-\frac{a_{1}}{30}} \delta^{-1} \right) \cap \left( \gamma_{\infty} [0, \tau_{l}''] \cup  \gamma_{\infty} [\tau_{l+1}', \infty ) \right) = \emptyset \text{ for all } l= 0,2, \dots , N'-1
\end{array}\right\}.\]

We now consider a $\lambda^{-\frac{a_{1}}{10}} \delta^{-1}$-net of $B (L)$, which we denote by $D$. We may assume that for each $y \in D \cap B(R \delta^{-1} )$, it holds that $B (y, 2 \lambda^{-1} \delta^{-1} ) \subseteq B (R \delta^{-1} ) $. Notice that the number of the points in $D$ is bounded above by $C  \lambda^{a_{1}/3}$ (recall that we set $L = \lambda^{\frac{\zeta_{4} a_{1}}{100}} \delta^{-1}$). For each $1 \le l \le N=2$, we can find a point $x_{l} \in D \cap B (R \delta^{-1})$ satisfying $| x_{l} - \gamma_{\infty} (\tau_{l} ) | \le  \lambda^{-\frac{a_{1}}{10}} \delta^{-1}$. Also, for each $0 \le l \le N'-1$, there exists a point $x_{l}' \in D \cap B (R \delta^{-1} )$ satisfying $| x_{l}' - \gamma_{\infty} (\rho_{l} ) | \le  \lambda^{-\frac{a_{1}}{10}} \delta^{-1}$. (Here, note that we can find $x_{l}' $ in $B (R \delta^{-1} )$ since $\gamma_{\infty} ( \rho_{l } ) \in B (R \delta^{-1} )$.)

We perform Wilson's algorithm as follows.
\begin{itemize}
\item The root of the algorithm is $\gamma_{\infty}$.
\item Consider the SRW $R^{1}$ started from $x_{1}$, and run until it hits $\gamma_{\infty}$. We let ${\cal U}^{1}$ be the union of $\gamma_{\infty}$ and $\text{LE} (R^{1})$. Next, we consider the SRW $R^{2}$ started at $x_{2}$ until it hits ${\cal U}^{1}$; the union of it and ${\cal U}^{1}$ is denoted by ${\cal U}^{2}$. We define ${\cal U}^{l}$ for $l = 3, 4, \dots , N-2$ similarly.
\item Consider the SRW $Z^{0}$ starting from $x_{0}'$ until it hits ${\cal U}^{N-2}$. We let $\tilde{{\cal U}}^{0}$ be the union of ${\cal U}^{N-2}$ and $\text{LE} (Z^{0})$. Next, we consider the SRW $Z^{1}$ started at $x_{1}'$ until it hits $\tilde{{\cal U}}^{0}$; the union of $\text{LE} ( Z^{1} )$ and $\tilde{{\cal U}}^{0}$ is denoted by $\tilde{{\cal U}}^{1}$. We define $\tilde{{\cal U}}^{l}$ for $l = 2, 3, \dots , N'-1$ similarly.
\item Finally, run sequentially LERWs from every point in $\mathbb{Z}^{3} \setminus \tilde{{\cal U}}^{N'-1}$ to obtain ${\cal U}$.
\end{itemize}
Define $F^{0} :=  A^{2} \cap A^{3} \cap A^{4} \cap A (\zeta_{4} )$ as a ``good'' event for $\gamma_{\infty}$. Conditioning $\gamma_{\infty}$ on the event $F^{0}$, we consider all simple random walks $R^{1}, R^{2}, \dots , R^{N-2}, Z^{0}, Z^{1}, \dots, Z^{N'-1}$ starting from $x_{1}, x_{2}, \dots , x_{N-2}, x_{0}', x_{1}', \dots , x_{N'-1}'$ respectively. The event $A (\zeta_{4} )$ ensures that the probability that $R^{l}$ (respectively $Z^{l}$) exits $B ( x_{l},   \lambda^{-\frac{a_{1}}{20}}  \delta^{-1}) $ (resp. $B ( x_{l}',   \lambda^{-\frac{a_{1}}{20}}  \delta^{-1}$) before hitting $\gamma_{\infty}$ is smaller than $\lambda^{- \zeta_{4} a_{1}}$ for each $l$. Moreover, the event $A^{4}$ says that the endpoint of $R^{l}$ (resp. $Z^{l}$) lies in $\gamma_{\infty} [\tau_{l-1}, \tau_{l+1} ]$ (resp. $\gamma_{\infty} [ \tau_{l}'', \tau_{l+1}' ]$)  for each $l $. On the other hand, the number of SRW's $N + N' -2$ is less than $\lambda^{\frac{\zeta_{4} a_{1}}{10}}$ by the event $A^{3}$. Also, we can again appeal to \cite[Theorem 1.4]{S} and \cite[Corollary 1.3]{Escape} to see that with probability at least $1 - C \exp \{ -c \lambda^{c} \}$, the length of the branch $\text{LE} (R^{l})$ (resp. $\text{LE} (Z^{l})$) is less than $\lambda^{- \frac{a_{1}}{40}} \delta^{-\beta} $ for each $l = 1,2 , \dots , N-2$ (respectively $l=0,1, \dots , N'-1$). Thus, taking the sum over $l$, we see that
\[\mathbf{P} (F^{1} ) \ge 1 - C \lambda^{-\frac{\zeta_{4} a_{1}}{2}},\]
where the event $F^{1}$ is defined by setting
\begin{align*}
F^{1} =&  \left\{
\begin{array}{c}
   \text{LE} (R^{l} ) \subseteq B \left( x_{l},   \lambda^{-\frac{a_{1}}{20}}  \delta^{-1} \right), \text{ the endpoint of } R^{l} \text{ lies in } \gamma_{\infty} [\tau_{l-1}, \tau_{l+1} ], \notag \\
   \text{ and the length of } \text{LE} (R^{l} )  \text{ is smaller than } \lambda^{- \frac{a_{1}}{40}} \delta^{-\beta} \text{ for all } l = 1, 2, \dots , N-2
\end{array} \right\} \\
& \cap  \left\{
\begin{array}{c}
\text{LE} (Z^{l} ) \subseteq B \left( x_{l}',   \lambda^{-\frac{a_{1}}{20}}  \delta^{-1} \right), \text{ the endpoint of } Z^{l} \text{ lies in } \gamma_{\infty} [\tau_{l}'', \tau_{l+1}' ], \notag \\
 \text{ and the length of } \text{LE} (Z^{l} )  \text{ is smaller than } \lambda^{- \frac{a_{1}}{40}} \delta^{-\beta} \text{ for all } l = 0, 1, \dots , N'-1
\end{array}\right\}.
\end{align*}

Recall that for each $y \in D \cap B(R \delta^{-1} )$, it holds that $B (y, 2 \lambda^{-1} \delta^{-1} ) \subseteq B (R \delta^{-1} ) $. Since the number of the points in $D$ is bounded above by $C  \lambda^{a_{1}/3}$, the translation invariance of ${\cal U}$ and \eqref{use-1} tell that
\[\mathbf{P} (F^{2} )\ge 1 - a_{2}  \lambda^{- \frac{2 a_{1}}{3}},\]
where the event $F^{2}$ is defined by
\begin{align*}
&F^{2} =  \left\{ B (x_{l}, \lambda^{-5} \delta^{-1}) \subseteq  B_{{\cal U}} \left( x_{l}, a_{2} \lambda^{-1/4} \delta^{-\beta} \right) \text{ for all } l=1,2, \dots , N-2    \right\} \\
& \ \ \ \ \ \ \cap  \left\{ B (x_{l}', \lambda^{-5} \delta^{-1}) \subseteq  B_{{\cal U}} \left( x_{l}', a_{2} \lambda^{-1/4} \delta^{-\beta} \right) \text{ for all } l=0,1, \dots , N'-1    \right\}.
\end{align*}

We set $F^{3} : = F^{0} \cap F^{1} \cap F^{2}$. Suppose that the event $F^{3}$ occurs. Take a point $x \in \gamma_{\infty} [0, \rho_{0}] $. We can then find $l \in\{ 0, 1, \dots , N-2\} $ such that $x \in \gamma_{\infty} [\tau_{l}, \tau_{l+2} ]$. Let $y_{l}$ be the endpoint of $R^{l}$. Since $y_{l}$ lies in  $\gamma_{\infty} [\tau_{l-1}, \tau_{l+1} ]$, and the event $A^{3}$ holds, we see that $d_{\cal U} (x, y_{l} ) \le \tau_{l+2} - \tau_{l-1} \le 3 \lambda^{-\frac{\zeta_{4} a_{1}}{200}} \delta^{-\beta}$. However, the event $F^{1}$ says that $d_{\cal U} (y_{l}, x_{l} ) \le \lambda^{- \frac{a_{1}}{40}} \delta^{-\beta}$. Finally, the event $F^{2}$ ensures that for every point $z \in B (x_{l}, \lambda^{-5} \delta^{-1}) $, we have $d_{\cal U} (x_{l}, z ) \le a_{2} \lambda^{-1/4} \delta^{-\beta}$. So, the triangle inequality tells that $d_{\cal U} (x, z) \le 5 \lambda^{-\frac{\zeta_{4} a_{1}}{200}} \delta^{-\beta}$ for all $z \in B (x_{l}, \lambda^{-5} \delta^{-1}) \subseteq B (R \delta^{-1} ) $.

We next consider a point $x \in \gamma_{\infty} [\rho_{0}, \rho] \cap B (R \delta^{-1} )$. There then exists  $0 \le l  < N'$ such that $x \in \gamma_{\infty} [\rho_{l}, \tau_{l+1}' ] $. Let $y_{l}'$ be the endpoint of $Z^{l}$. Since $y_{l}'$ lies in  $\gamma_{\infty} [\tau_{l}'', \tau_{l+1}' ]$, and the event $A^{3}$ holds, we see that $d_{\cal U} (x, y_{l}' ) \le \tau_{l+1}' - \tau_{l}'' \le  \lambda^{-\frac{\zeta_{4} a_{1}}{200}} \delta^{-\beta}$. However, the event $F^{1}$ says that $d_{\cal U} (y_{l}', x_{l}' ) \le \lambda^{- \frac{a_{1}}{40}} \delta^{-\beta}$. Finally, the event $F^{2}$ ensures that for every point $z \in B (x_{l}', \lambda^{-5} \delta^{-1}) $, we have $d_{\cal U} (x_{l}', z ) \le a_{2} \lambda^{-1/4} \delta^{-\beta}$. So, the triangle inequality tells that $d_{\cal U} (x, z) \le 3 \lambda^{-\frac{\zeta_{4} a_{1}}{200}} \delta^{-\beta}$ for all $z \in B (x_{l}', \lambda^{-1} \delta^{-1}) \subseteq B (R \delta^{-1} ) $.

This implies that for all $x \in \gamma_{\infty} [0, \rho] \cap B (R \delta^{-1} )$,
\begin{equation}\label{tre}
\mu_{\cal U} \left\{ B_{\cal U} \left( x, 5 \lambda^{-\frac{\zeta_{4} a_{1}}{200}} \delta^{-\beta} \right) \cap B (R \delta^{-1} ) \right\} \ge c \lambda^{-15} \delta^{-3},
\end{equation}
which gives \eqref{warm-2nd} when $5 \lambda^{-\frac{\zeta_{4} a_{1}}{200}}  \le \theta$. For the case that $5 \lambda^{-\frac{\zeta_{4} a_{1}}{200}}  > \theta$, taking $c_{1}$ and $c_{2}$ in \eqref{warm-2nd} sufficiently large so that $c_{1} \ge 5^{\frac{200c_{3}}{\zeta_{4} a_{1}}}$ and $c_{2} \ge \frac{200c_{3}}{\zeta_{4} a_{1}}$ if necessary, the inequality  \eqref{warm-2nd} also holds in this case.

Now we are ready to prove the proposition. Recall that the constants $a_{1}$ and $\zeta_{4}$ appeared at \eqref{use-1} and \eqref{zeta4}, and that we defined $b_{0} := a_{1} \zeta_{4}/5000$ and $L :=\lambda^{\frac{\zeta_{4} a_{1}}{100}} \delta^{-1}$. For $v \in B ( \lambda^{b_{0}}  \delta^{-1})$,   $\rho $ was defined to be the first time that $\gamma_{\infty}^{v}$ exits $B ( L)$ ($\rho =0$ if $v \notin B (R \delta^{-1} )$). We have proved that for each $v \in B (   \lambda^{b_{0}}  \delta^{-1})$,
\begin{align}
&P \left( \mu_{\cal U} \left\{ B_{\cal U} \left( x,  \lambda^{-b_{1}} \delta^{-\beta} \right) \cap B (0, R \delta^{-1} ) \right\} \ge c \lambda^{-15} \delta^{-3} \text{ for all } x \in \gamma_{\infty}^{v} [0, \rho] \cap B (R \delta^{-1} ) \right) \notag \\
&\ge 1 - C \lambda^{-b_{1}}, \label{mendo}
\end{align}
for some $b_{1} > 0$, see \eqref{tre}. Let $b_{2} = \frac{\zeta_{4} a_{1}}{10^{8}} \wedge \frac{b_{1}}{10^{8}}$. We consider a $\lambda^{-b_{2}} \delta^{-1}$-net $D' = ( x^{l})_{l=1}^{M}$ of $B (0,  2 R \delta^{-1})$. Note the number of points in $D'$, which is denoted by $M$, can be assumed to be smaller than $C R^{3} \lambda^{3 b_{2}}$.

Now we perform Wilson's algorithm as follows:
\begin{itemize}
\item The root of the algorithm is $\gamma_{\infty} = \gamma_{\infty}^{0}$.
\item Consider the SRW $R^{1}$ started at $x^{1} \in D'$, and run until it hits $\gamma_{\infty}$. Let ${\cal U}_{1}$ be the union of $\text{LE} (R^{1})$ and $\gamma_{\infty}$. We then consider the SRW $R^{l}$ started from $x^{l} \in D'$, and run until it hits ${\cal U}_{l-1}$; add $\text{LE} (R^{l})$ to ${\cal U}_{l-1}$ -- this union is denoted by ${\cal U}_{l}$. Since $M \le  C R^{3} \lambda^{3 b_{2}}$, by applying \eqref{mendo} for each $x^{l}$, we have
    \[\mathbf{P} (V^{1} ) \ge 1- C R^{3} \lambda^{-b_{1}/2},\]
    where the event $V^{1}$ is defined by setting
    \begin{equation*}
    V^{1} := \left\{ \mu_{\cal U} \left\{ B_{\cal U} \left( x,  \lambda^{-b_{1}} \delta^{-\beta} \right) \cap B (R \delta^{-1} ) \right\} \ge c \lambda^{-15} \delta^{-3} \text{ for all } x \in {\cal U}_{M}  \cap B (R \delta^{-1} ) \right\}.
    \end{equation*}
\item Taking $a > 0$ such that $a \sum_{j=1}^{\infty} j^{-2} = 1/2$, we let $a_{k} = a \sum_{j=1}^{k} j^{-2}$, and consider a $2^{-k} \lambda^{-b_{2}} \delta^{-1}$-net $D^{k} = ( x_{i}^{k} )_i$ of $B (  (2- a_{k} ) R \delta^{-1} )$, where the number of points in $D^{k}$ is bounded above by $C R^{3} 2^{3k} \lambda^{3b_{2}}$. Let $k_{0}$ be the smallest integer $k$ such that $2^{-k} \lambda^{-b_{2}} \delta^{-1} \le 1$.
\item Perform Wilson's algorithm for all points in $D^{1}$ adding new branches to ${\cal U}_{M}$; the output tree is denoted by $\hat{\cal U}_{1}$. Then perform Wilson's algorithm for points $D^{k}$ ($k=2, 3, \dots , k_{0}$) inductively; the output trees are denoted by $\hat{\cal U}_{2}, \dots , \hat{\cal U}_{k_{0}}$. Note that $B( R \delta^{-1}) \subseteq \hat{\cal U}_{k_{0}}$.
\end{itemize}

Since every branch generated in the procedure above is a hittable set, we can prove that there exist universal $0< b_{3} < b_{2}$ and $C >0$ such that
\begin{equation}\label{aamo}
\mathbf{P} (V^{2})  \ge 1- C \lambda^{-b_{3}},
\end{equation}
where the event $V^{2}$ is defined by
\begin{equation*}
V^{2}:=\left\{ \forall x \in \hat{\cal U}_{k_{0}}, d_{{\cal U}} (x,  x(M) ) \le \lambda^{-b_{3}} \delta^{-\beta} \text{ and } x(M) \in B (R \delta^{-1} ) \right\}.
\end{equation*}
Here, for each $x$, we write $x(M) \in {\cal U}_{M}$ for the point such that $d_{{\cal U}} (x,  x(M) ) = d_{{\cal U}} (x,  {\cal U}_{M} )$. The inequality \eqref{aamo} can be proved in a similar way to the proof of Proposition \ref{1st-assump}, so the details are left to the reader.

Suppose that the event $V^{1} \cap V^{2}$ occurs. Since $B (  R \delta^{-1} ) \subseteq \hat{\cal U}_{k_{0}}$, this implies that for any $x \in B (R \delta^{-1} )$, we have
\[ \mu_{\cal U} \left\{ B_{\cal U} \left( x,  2 \lambda^{-b_{3}} \delta^{-\beta} \right) \cap B (0, R \delta^{-1} ) \right\} \ge c \lambda^{-3} \delta^{-3},\]
which finishes the proof when $\theta \ge 2 \lambda^{-b_{3}}$. Otherwise, taking $c_{1}$ and $c_{2}$ sufficiently large (similarly to the proof of \eqref{warm-2nd}, see \eqref{tre} and a few lines following it for this), the proposition follows.
\end{proof}

Combining Propositions \ref{1st-assump} and \ref{last-step}, we have the following.

\begin{cor}\label{ass-4}
Assumptions \ref{a2} holds.
\end{cor}

\subsection{Assumption \ref{a3}}

In this subsection, we will prove the following proposition.

\begin{prop}\label{5th}
Assumption \ref{a3} holds.
\end{prop}

\begin{proof}
In \cite{LS}, it is proved that there exist universal constants $b_{3}, b_{4} \in (0, \infty )$ such that for all $\delta \in  (0, 1)$ and $\lambda \ge 1$,
\begin{equation}\label{buhuo}
\mathbf{P} (J_{1} ) \ge 1 - b_{4} \lambda^{-b_{3}},
\end{equation}
where the event $J_{1}$ is defined by setting
\begin{equation*}
J_{1} = \left\{ \forall x, y \in \gamma_{\infty} \cap B \left( \lambda^{b_{3}} \delta^{-1} \right) \text{ with } d_{{\cal U}} (x,y) \le \lambda^{-b_{4}} \delta^{-\beta}, |x-y| \le \lambda^{-b_{3}} \delta^{-1} \right\},
\end{equation*}
see \cite[(7.19)]{LS} in particular. We also need the hittability of $\gamma_{\infty}$ as follows. For $\zeta > 0$, define the event $J (\zeta )$ by setting
\[J ( \zeta ) = \left\{P_{R}^{x} \left( R \left[ 0, T_{R} \left( x, \lambda^{b_{3}/2}  \delta^{-1} \right) \right] \cap \gamma_{\infty} = \emptyset \right) \le \lambda^{-\zeta b_{3}} \text{ for all } x \in B \left(\lambda^{b_{3}/4} \delta^{-1} \right)\right\}.\]
It follows from \cite[Lemma 3.2 and Lemma 3.3]{SS} that there exist universal constants $C< \infty$ and $\zeta_{5} \in (0,1)$ such that for all $\delta > 0$ and $\lambda \ge 1$,
\[\mathbf{P}\left(J (\zeta_{5}) \right) \ge 1 - C \lambda^{- b_{3}}.\]

With this in mind, we set $b_{5} = \frac{\zeta_{5} b_{3}}{1000}$ and $R_{1} = \lambda^{b_{5}} \delta^{-1}$. Let $D'' = ( z_{l} )_l$ be a $\lambda^{-b_{5} } \delta^{-1}$-net of $B (R_{1} )$. The number of points $M''$ of $D''$ can be assumed to be smaller than $C \lambda^{6 b_{5}}$. We perform Wilson's algorithm as follows. The root of the algorithm is $\gamma_{\infty}$ as usual. Then we consider the loop-erasure of the SRWs $R^{1}, R^{2}, \dots , R^{M''}$ started from $z_{1}, z_{2}, \dots , z_{M''}$ respectively; we denote the output tree by ${\cal U}_{M''}$. Finally, we consider LERW's starting from all points in $\mathbb{Z}^{3} \setminus {\cal U}_{M''}$.

Conditioning $\gamma_{\infty}$ on the event $J_{1} \cap J ( \zeta_{5} )$, for each $l = 1, 2, \dots , M''$, the probability that $R^{l}$ exits $B \left( z_{l}, \lambda^{b_{3}/2}  \delta^{-1}  \right)$ before hitting $\gamma_{\infty}$ is, on the event $J (\zeta_{5} )$, bounded above by $\lambda^{- \zeta_{5} b_{3}}$. Taking the sum over $l$, we see that if
\[J_{2}:=\left\{R^{l} \left[ 0, T_{R^{l}} \left( x, \lambda^{b_{3}/2}  \delta^{-1} \right) \right] \cap \gamma_{\infty} \neq \emptyset \text{ for all } l=1,2, \dots , M''\right\},\]
then
\[\mathbf{P} (J_{2} ) \ge 1 - C \lambda^{-b_{5}}.\]
On the other hand, if we define
\[J_{3}^{l} = \left\{ \forall x, y \in \gamma_{\infty}^{z_{l}} \cap B \left( z_{l}, \lambda^{b_{3}} \delta^{-1} \right) \text{ with } d_{{\cal U}} (x,y) \le \lambda^{-b_{4}} \delta^{-\beta}, |x-y| \le \lambda^{-b_{3}} \delta^{-1} \right\},\]
for each $l =1,2, \dots , M''$, (recall that $\gamma_{\infty}^{x}$ stands for the unique infinite path in ${\cal U}$ starting from $x$,) by the translation invariance of ${\cal U}$ and \eqref{buhuo}, it follows that $\mathbf{P} (J_{3}^{l} )  \ge 1 - b_{4} \lambda^{-b_{3}}$ for all $l$. Thus, letting
\[J_{3} = \bigcap_{l=1}^{M''} J_{3}^{l},\]
we have $\mathbf{P} (J_{3} ) \ge 1 - \lambda^{- b_{3}}$.

Now, suppose that the event $J: = J_{1} \cap J (\zeta_{5}) \cap J_{2} \cap J_{3}$ occurs. The triangle inequality tells that on the event $J$, for all $x, y \in {\cal U}_{M''} \cap B ( \lambda^{\frac{2 b_{3}}{3}} \delta^{-1} )$ with $d_{{\cal U}} (x,y) \le \lambda^{-b_{4}} \delta^{-\beta}$, we have $ |x- y| \le 3 \lambda^{-b_{3}} \delta^{-1}$. Thus
\[\mathbf{P} \left(\forall x, y \in {\cal U}_{M''} \cap B \left( \lambda^{\frac{2 b_{3}}{3}} \delta^{-1} \right) \text{ with } d_{{\cal U}} (x,y) \le \lambda^{-b_{4}} \delta^{-\beta}, |x- y| \le 3 \lambda^{-b_{3}} \delta^{-1} \right)  \ge 1 - C \lambda^{-b_{5}}.\]
By the translation invariance of ${\cal U}$ again, we can prove that each branch $\gamma^{z_{l}}_{\infty}$ is also a hittable set with high probability. Namely, if we let
\[J^{l} ( \zeta ) = \left\{
P_{R}^{x} \left( R \left[ 0, T_{R} \left( x, \lambda^{- b_{5}/2}  \delta^{-1} \right) \right] \cap \gamma_{\infty}^{z_{l}} = \emptyset \right) \le \lambda^{-\zeta b_{5}} \text{ for all } x \in B (z_{l}, \lambda^{-b_{5}} \delta^{-1} )
 \right\}\]
for each $l = 1,2, \dots , M''$, then by using \cite[Lemma 3.2]{SS}, we see that there exist universal constants $\zeta_{6} \in (0, 1)$ and $C < \infty$ such that for all $\delta \in (0, 1)$, $\lambda \ge 1$ and $l =1,2, \dots M''$,
\[\mathbf{P} \left( J^{l} (\zeta_{6} ) \right) \ge 1- C \lambda^{- 100 b_{5}}.\]
With this in mind, we let
\[J_{4} : = \bigcap_{l=1}^{M''} J^{l} (\zeta_{6} ),\]
so that $\mathbf{P} (J_{4} ) \ge 1 - C \lambda^{- b_{5}}$.

Conditioning ${\cal U}_{M''}$ on the event $J \cap J_{4}$, we perform Wilson's algorithm for all points in $B( R_{1}/2) \setminus {\cal U}_{M''}$, considering finer and finer nets there as in the proof of Proposition \ref{1st-assump}. The event $J_{4}$ ensures that every SRW starting from a point $w$ in $B( R_{1}/2) $ hits ${\cal U}_{M''}$ before it exits $B( w, \lambda^{-b_{5}/3} \delta^{-1} )$ with probability at least $1- C \lambda^{-\zeta_{6} b_{5} \lambda^{\frac{b_{5}}{6}}}$. Thus we can conclude that with probability at least $1 - C \lambda^{-b_{5}}$, we have $\text{diam} ( \gamma_{{\cal U} } (w ,  {\cal U}_{M''} ) ) \le \lambda^{- b_{5} / 3 } \delta^{-1}$ and $d_{{\cal U} } (w ,  {\cal U}_{M''} ) \le \lambda^{-b_{5}/4} \delta^{-\beta}$ for all $w \in B \left( 0, R_{1}/2 \right)$. Therefore, by the triangle inequality again, it follows that
\begin{equation}\label{owari}
\mathbf{P} \left(\forall x, y \in {\cal U} \cap B \left( R_{1}/2 \right) \text{ with } d_{{\cal U}} (x,y) \le \lambda^{-b_{4}} \delta^{-\beta}, |x- y| \le  \lambda^{-b_{3} /5} \delta^{-1} \right)  \ge 1 - C \lambda^{-b_{5}}.
\end{equation}
Finally, Proposition \ref{1st-assump} shows that with probability $1 - C \lambda^{- c b_{5}}$, $B_{\cal U} (0, L \delta^{-\beta} ) \subseteq B ( R_{1} / 2 )$ for each fixed $L$. Combining this with \eqref{owari} completes the proof.
\end{proof}

\section{Exponential lower tail bound on the volume} \label{sec: exp lower volume}

In Proposition \ref{last-step}, we established a polynomial (in $\lambda$) lower tail bound on the volume of a ball. In this section, we will improve this bound to an exponential one, see Theorem \ref{2nd-goal} below. We start by proving the following analogue of \cite[Theorem 3.4]{BM} in three dimensions. The proof strategy is modelled on that of the latter result, though there is a key difference in that the Beurling estimate used there (see \cite[Theorem 2.5.2]{Lawb}) is not applicable in three dimensions, and we replace it with the hittability estimate of Proposition \ref{prop:iterate}.

\begin{thm}\label{1st-goal}
There exist constants $c, C , b \in (0, \infty)$ such that: for all $R \ge 1$ and $\lambda \ge 1$,
\begin{equation}\label{gla}
\mathbf{P} \left( \mu_{{\cal U} } \left(  B_{{\cal U}} (0, R ) \right) \le \lambda^{-1} R^{\frac{3}{\beta}} \right) \le C \exp \left\{ - c \lambda^{b} \right\}.
\end{equation}
\end{thm}

\begin{proof} We begin by describing the setting of the proof. We assume that $\lambda \ge 1$ is sufficiently large, and let $a = \frac{99}{100}$. Let $q = [\lambda^{(1-a)}/3 ]$ be the number of subsets $I_{0}, I_{1}, \dots , I_{q}$ of the index set $\{ 1, 2, \dots , \lambda \}$, as defined in \eqref{index}. Note that for all $0 \le j_{1} < j_{2} \le q$ and all $i_{1} \in I_{j_{1}}, i_{2} \in I_{j_{2}}$ we have
\begin{equation}\label{dist}
\text{dist} \left( \partial D_{i_{1}}, \partial D_{i_{2}} \right) \ge \lambda^{a-1} R,
\end{equation}
where $D_{i} = D (0, \frac{i R }{\lambda} )$ was defined in Subsection \ref{subsec:notationset}.
For each $j = 0, 1, \dots , q$, recall that the event $F_{j}$ stands for the event that there exists a ``good'' index $i \in I_{j}$ in the sense that $\gamma [t_{i}, \sigma_{i} ]$ is a hittable set.  By Proposition \ref{prop:iterate}, with probability at least $1 - \lambda^{1-a} \exp \left\{ - c_{1} \lambda^{a} \right\}$, the ILERW $\gamma$ has a good index in $I_{j}$ for every $j=0,1, \dots , q$. Let
\begin{equation}\label{evf}
F =  \bigcap_{j = 1}^{q} F_{j},
\end{equation}
and suppose that the event $F$ occurs. It then holds that, for each $j = 0, 1, \dots , q$, we can find a good index $i_{j} \in I_{j}$ such that the event $A_{i_{j}}$ occurs. We will moreover fix deterministic nets $W^{p} = ( w^{p}_{k})_{ k}$, $p =1, 2, 3$, of $B ( 2 R)$ satisfying
\begin{equation*}
B ( 2 R) \subseteq \bigcup_{k} B \left( w^{p}_{k},  \frac{R}{10^{2} \lambda^{2 p} } \right) \ \text{ and } \ |w^{p}_{k} - w^{p}_{k'} | \ge \frac{R}{10^{4} \lambda^{2 p} } \text{ for all } k \neq k'.
\end{equation*}
Note that we may assume that $|W^{p}| \asymp \lambda^{6 p}$.

From now on, we assume that the event $F$ occurs whenever we consider $\gamma$. We also highlight the correspondence between our setting and that of \cite[Theorem 3.4]{BM}. In the proof of \cite[Theorem 3.4]{BM}, $k$ points $z_{1}, z_{2}, \dots , z_{k}$ were  chosen on the ILERW. Here the points $x_{i_{0}} = \gamma (t_{i_{0}} ),  x_{i_{1}} = \gamma (t_{i_{1}} ), \dots , x_{i_{q}} = \gamma (t_{i_{q}} )$ correspond to those points, where $t_{i}$ stands for the first time that $\gamma$ exits $D_{i}$.  Setting $n = \frac{R}{2 \lambda }$, we write $B_{j} = B ( x_{i_{j}}, n )$ for $j =0, 1, \dots , q$. Note that for each $j_{1} \neq j_{2}$
\begin{equation}\label{dist-2}
\text{dist} \left( B_{j_{1}}, B_{j_{2}}\right) \ge \frac{ \lambda^{a-1} R}{2}
\end{equation}
by \eqref{dist}.

As in \cite[(3.18) and (3.19)]{BM}, we define the events $F^{1}, F^{2}$ by setting
\begin{equation}\label{f1def}
F^{1} = \left\{ \gamma [T_{2 R}, \infty ) \text{ hits more than } q/2 \text{ of } B_{0}, B_{1}, \dots B_{q} \right\},
\end{equation}
\[F^{2} = \left\{ T_{2 R} \ge \lambda^{a'} R^{\beta} \right\},\]
where $T_{r}$ is the first time that $\gamma$ (recall that $\gamma (0) = 0$) exits $B (r)$, and $a' = \frac{1}{1000}$, see Figure \ref{fig1b}. Here we also need to introduce the event $F^{3}$, as given by
\[F^{3} = \left\{ \exists w_{k}^{1} \in W^{1} \text{ such that } N^{1}_{k} \ge \lambda^{5} \right\},\]
where $N^{1}_{k}$ is defined by
\[N^{1}_{k} = \left| \left\{ w^{2}_{l} \in W^{2} \::\: B \left( w^{2}_{l},  \frac{R}{10^{2} \lambda^{4} } \right)  \subseteq B \left( w^{1}_{k},  \frac{R}{10^{2} \lambda^{2 } } \right) \text{ and } B \left( w^{2}_{l},  \frac{R}{10^{2} \lambda^{4} } \right) \cap \gamma [0, \infty ) \neq \emptyset \right\}\right|,\]
i.e., $N^{1}_{k}$ stands for the number of balls of the net $W^{2}$ contained in $B ( w^{1}_{k},  \frac{R}{10^{2} \lambda^{2 } } )$ and hit by ILERW $\gamma$.

\begin{figure}[!b]
\begin{center}
  \includegraphics[width=0.75\textwidth]{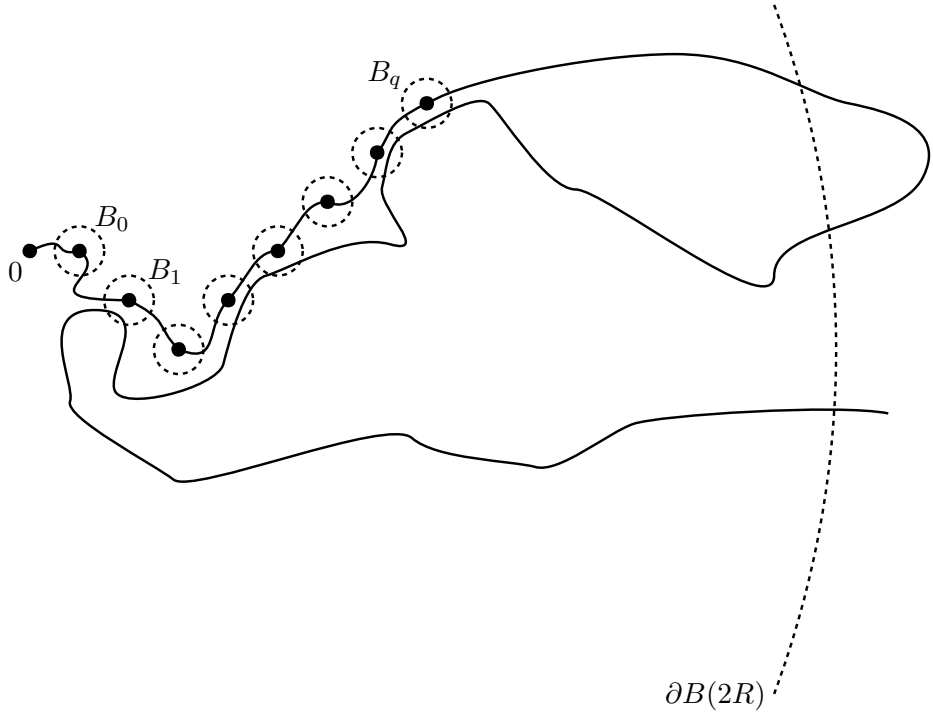}
  \rput(-346pt,160pt){$0$}
  \rput(-310pt,180pt){$B_0$}
  \rput(-290pt,160pt){$B_1$}
  \rput(-208pt,234pt){$B_q$}
  \rput(-84pt,0pt){$\partial B(2R)$}
\end{center}
\caption{A typical realisation of $\gamma$ on the event $F^1$, as defined at \eqref{f1def}.}\label{fig1b}
\end{figure}

We will first show that $\mathbf{P} (F^{1} )$ is exponentially small in $\lambda$. Let $\Gamma_{r} $ be the set of paths $\zeta$ satisfying $\mathbf{P} ( \gamma [0, T_{r} ] = \zeta ) > 0$. Namely, $\Gamma_{r} $ stands for  the set of all possible candidates for $\gamma [0, T_{r} ]$. Take $\zeta \in \Gamma_{2 R}$, and let $z = \zeta \left( \len (\zeta ) \right) $ be the endpoint of $\zeta$. Write $Y$ for the random walk started at $z$ and conditioned on the event that $Y [1, \infty ) \cap \zeta = \emptyset $. The domain Markov property (see \cite[Proposition 7.3.1]{Lawb}) yields that the distribution of $\gamma [T_{2 R}, \infty ) $ conditioned on the event $\{ \gamma [0, T_{2 R} ] = \zeta \}$ coincides with that of $\text{LE} ( Y [0, \infty )) $. Therefore, we have
\begin{equation*}
\mathbf{P} (F^{1} ) \le \sum_{\zeta \in \Gamma_{2 R} } \mathbf{P} (H_{\zeta} ) \mathbf{P} \left( \gamma [0, T_{2R} ] = \zeta \right),
\end{equation*}
where the event $H_{\zeta}$ is defined by
\begin{equation*}
H_{\zeta} = \left\{  Y \text{ hits more than } q/2  \text{ of } B_{0}, B_{1} \dots , B_{q} \right\},
\end{equation*}
where $B_{i}$ was defined in one line before \eqref{dist-2}.
Recall that $R^{z}$ stands for the SRW started at $z$. We remark that $\text{dist} (z, B_{j} ) \ge R/4 $ for all $j \in\{0, 1, \dots , q\}$. Let $\tau$ be the first time that $R^{z}$ exits $B (z, R/ 8)$, and observe that
\begin{equation}\label{ob}
\mathbf{P} \left( R^{z} [1, \infty ) \cap \zeta = \emptyset \right) \asymp  \mathbf{P} \left( R^{z} [1, \tau ] \cap \zeta = \emptyset \right).
\end{equation}
Indeed, it is clear that the left-hand side is bounded above by the right-hand side. To see the opposite inequality, we note that \cite[Proposition 6.1]{S} (see also \cite[Claim 3.4]{SS}) yields that
\begin{align*}
\lefteqn{\mathbf{P} \left( R^{z} [1, \infty ) \cap \zeta = \emptyset \right)  }\\
&\ge \mathbf{P} \left( R^{z} [1, \tau ] \cap \zeta = \emptyset, \ \text{dist} \left( B (2R), R^{z} ( \tau ) \right) \ge \frac{R}{16}, \  R^{z} [ \tau , \infty ) \cap B (2 R) = \emptyset  \right)  \\
&\ge c \mathbf{P} \left( R^{z} [1, \tau ] \cap \zeta = \emptyset, \ \text{dist} \left( B (2R), R^{z} ( \tau ) \right) \ge \frac{R}{16} \right) \ge c' \mathbf{P} \left( R^{z} [1, \tau ] \cap \zeta = \emptyset \right),
\end{align*}
which gives \eqref{ob}. Consequently, we obtain that
\begin{align*}
\mathbf{P} (H_{\zeta} )& \le  \frac{ C \mathbf{P} \left( R^{z} [1, \tau ] \cap \zeta = \emptyset, \ R^{z} \text{ hits more than } q/2  \text{ of } B_{0}, B_{1} \dots , B_{q} \right) }{ \mathbf{P} \left( R^{z} [1, \tau ] \cap \zeta = \emptyset \right) } \\
&\le C \max_{z' \in B (z, R/ 8)} \mathbf{P} \left( R^{z'} \text{ hits more than } q/2  \text{ of } B_{0}, B_{1} \dots , B_{q} \right).
\end{align*}
Take $z' \in B (z, R/ 8)$, and note that $\text{dist} (z', B_{j} ) \ge R/ 8$ for all $j\in\{0, 1, \dots , q\}$. We define a sequence of stopping times $u_{1}, u_{2}, \dots $ as follows. Let
\[u_{1} = \inf \left\{ t \ge 0 \::\: R^{z'} (t) \in \bigcup_{j =0}^{q} B_{j} \right\},\]
and $j^{1}$ be the unique index such that $R^{z'} (u_{1} ) \in B_{j^{1}}$. For $l \ge 2$, we define $u_{l}$ by setting
\[u_{l} = \inf \left\{ t \ge u_{l-1} \::\: R^{z'} (t) \in \left( \bigcup_{j =0}^{q} B_{j} \right) \setminus B_{j^{l-1}} \right\},\]
and write $j^{l}$ for the unique index such that $R^{z'} (u_{l} ) \in B_{j^{l}}$. Since the distance between two different balls is bigger than $\lambda^{a-1} R/ 2$ by \eqref{dist-2}, and each ball has radius $n = R/ 2 \lambda $, it follows from \cite[Proposition 1.5.10]{Lawb} that
\[\mathbf{P} \left( u_{l} < \infty \:\vline\: u_{l-1} < \infty \right) \le C \lambda^{-a} \lambda^{1-a} = C \lambda^{- \frac{49}{50}},\]
for all $l$. Thus, taking $\lambda$ sufficiently large so that $C \lambda^{- \frac{49}{50}} <1/2$, it holds that
\[\mathbf{P} (H_{\zeta} ) \le C (1/2)^{q/2} \le C \exp \left\{ - c \lambda^{\frac{1}{100} } \right\},\]
which gives
\begin{equation}\label{f1small}
\mathbf{P} (F^{1} ) \le C \exp \left\{ - c \lambda^{\frac{1}{100} } \right\}.
\end{equation}
As for the event $F^{2}$, we have from Proposition \ref{result:upperLowTail} that
 \begin{equation}\label{f2small}
\mathbf{P} (F^{2} ) \le C \exp \left\{ - c \lambda^{a'} \right\}.
\end{equation}
Finally, we will deal with the event $F^{3}$. Define
\begin{equation*}
M^{1}_{k} = \left| \left\{ w^{2}_{l} \in W^{2} \::\: B \left( w^{2}_{l},  \frac{R}{10^{2} \lambda^{4} } \right)  \subseteq B \left( w^{1}_{k},  \frac{R}{10^{2} \lambda^{2 } } \right) \text{ and } B \left( w^{2}_{l},  \frac{R}{10^{2} \lambda^{4} } \right) \cap S [0, \infty ) \neq \emptyset \right\}\right|,
\end{equation*}
i.e.\ $M^{1}_{k}$ stands for the number of balls of the net $W^{2}$ contained in $B \left( w^{1}_{k},  \frac{R}{10^{2} \lambda^{2 } } \right)$ and hit by SRW $S[0, \infty )$. It is clear that $N_{k}^{1} \le M_{k}^{1}$. Thus, on the event $F^{3}$, there exists $w_{k}^{1} \in W^{1}$  such that $M^{1}_{k} \ge \lambda^{5}$. However, for each $k$, it is easy to see that $\mathbf{P} ( M^{1}_{k} \ge \lambda^{5} ) \le C e^{- c \lambda }$. Therefore, since $| W^{1}| \asymp \lambda^{6}$, we see that
 \begin{equation}\label{f3small}
\mathbf{P} (F^{3} ) \le C \exp \left\{ - c \lambda^{1/2} \right\}.
\end{equation}

We are now ready to follow the proof of \cite[Theorem 3.4]{BM}. If the event $F^{c} \cup F^{1} \cup F^{2} \cup F^{3}$ (recall that the event $F$ is defined at \eqref{evf}) occurs, we terminate the algorithm with a `Type 1' failure.  Otherwise, for each $j = 0, 1, \dots , q$, we can find $z_{j} \in W^{3} \cap B ( x_{i_{j}}, n/8 )$ such that $ B ( z_{j} , \lambda^{-4} ) \cap \gamma [0, \infty ) = \emptyset $. Using this point $z_{j} $, we write
\begin{equation*}
B_{j}' = B \left( z_{j} , \lambda^{-4} R \right), \qquad B_{j}'' = B \left( z_{j} , \lambda^{-6} R \right).
\end{equation*}
Let ${\cal U}_{0} = \gamma [0, \infty )$. Suppose that the event $F \cap \bigcap_{k=1}^{3} (F^{k})^{c}$ occurs. We consider the SRW $R^{z_{0}}$ until it hits ${\cal U}_{0}$. Let $\gamma_{0} = \text{LE} (R^{z_{0}})$ be its loop-erasure which is the branch on ${\cal U}$ between $z_{0}$ and ${\cal U}_{0}$. Define the event $G^{0}_{1}$ by $G^{0}_{1} = \{ R^{z_{0}} \not\subseteq B_{0}\}$. Since $\gamma$ satisfies the event $F$, we see that
\begin{equation}\label{g01}
\mathbf{P} \left( (G^{0}_{1} )^{c} \right) \ge c_{0}.
\end{equation}
Suppose that the event $G^{0}_{1} $ occurs. We mark the ball $B_{j}$ as `bad' if $R^{z_{0}} \cap B_{j} \neq \emptyset$. Otherwise, we define the event $G^{0}_{2}:=\{ \len (\gamma_{0} ) \ge \lambda^{-1/2} R^{\beta} \} \cap  \{ R^{z_{0}} \subseteq B_{0} \}$. If the event $G^{0}_{2} $ occurs, we also mark $B_{0}$ as `bad' (we only mark $B_{0}$ in this case). By Proposition \ref{result:upperLowTail}, it holds that
\begin{equation}\label{g02}
\mathbf{P} \left( G^{0}_{2} \right) \le C \exp \left\{ - c \lambda^{1/2} \right\}.
\end{equation}
If the event $(G^{0}_{1})^{c} \cap (G^{0}_{2})^{c} $ occurs, we use Wilson's algorithm to fill in the reminder of $B_{0}''$. Define the event $G^{0}_{3}$ by setting
\[G^{0}_{3} \hspace{-1.35pt} =\hspace{-1.35pt} \left\{ \exists v \in B_{0}'' \text{ such that } \gamma_{{\cal U}} \left( v, \gamma_{0} \cup {\cal U}_{0} \right) \not\subseteq B_{0}' \text{ or }   \len \left( \gamma_{{\cal U}} \left( v, \gamma_{0} \cup {\cal U}_{0} \right) \right) \ge \lambda^{-2} R^{\beta } \right\}  \cap (G^{0}_{1})^{c} \cap (G^{0}_{2})^{c},\]
where we recall that $\gamma_{{\cal U} } (v , A)$ stands for the branch on ${\cal U}$ between $v$ and $A$. Modifying the proof of Proposition \ref{1st-assump}, we see that
\begin{equation}\label{g03}
\mathbf{P} \left( G^{0}_{3} \right) \le C \lambda^{-c}
\end{equation}
for some universal constants $c, C \in (0, \infty ) $. Suppose that the event $G^{0}_{3}$ occurs. We again mark the ball $B_{j}$ as `bad' if $S^{v}$ hits $B_{j}$ for some $v \in B_{0}''$ in the algorithm above. If the event $(G^{0}_{1})^{c} \cap (G^{0}_{2})^{c}  \cap (G^{0}_{3})^{c}$ occurs, we label this first `ball step' as successful and we terminate the whole algorithm. In this case, for all $v \in B_{0}''$
\begin{equation*}
d_{{\cal U}} (0, v) \le \left( \lambda^{a'} + \lambda^{-1/2} + \lambda^{-2} \right)  R^{\beta} \le C \lambda^{a'} R^{\beta},
\end{equation*}
and so
\begin{equation}\label{vol}
\mu_{{\cal U} } \left( B_{{\cal U} } \left( 0,  C \lambda^{a'} R^{\beta} \right) \right) \ge c \lambda^{-18} R^{3}.
\end{equation}
If the event $G^{0}_{1} \cup G^{0}_{2} \cup G^{0}_{3}$ occurs, we denote the number of bad balls by $N^{B}_{0}$.  Using a similar idea used to establish \eqref{f1small}, we see that
\begin{equation}\label{ballhit}
\mathbf{P} \left( N^{B}_{0} \ge \sqrt{q}/4 \right) \le C \exp \left\{ - c \lambda^{1/200} \right\}.
\end{equation}
If  $N^{B}_{0} \ge \sqrt{q}/4 $, we terminate the whole algorithm as `Type 2' failure. If $ N^{B}_{0} < \sqrt{q}/4 $, we can choose $B_{j}$ which is not bad and perform the second `ball step', replacing $B_{0}$ with $B_{j}$ in the above. We terminate this ball step algorithm whenever we get a successful ball step or we have Type 2 failure. We write $F^{4}$ for the event that some ball step ends with a Type 2 failure. Since we perform at most $q^{1/2}$ ball steps, it follows from \eqref{ballhit} that
\begin{equation}\label{f4small}
\mathbf{P} (F^{4} ) \le C \exp \left\{ - c \lambda^{1/400} \right\}.
\end{equation}
Finally, we let $F^{5}$ be the event that we can perform the $j$th ball step for all $j =1, 2 \dots , q^{1/2}$ without Type 2 failure and success. By combining \eqref{g01}, \eqref{g02} and \eqref{g03}, taking $\lambda$ sufficiently large, we see that each ball step has a probability at least $c_{0}/2$ of success. Therefore, we have
\begin{equation}\label{f5small}
\mathbf{P} (F^{5} ) \le C \exp \left\{ - c \lambda^{1/200} \right\}.
\end{equation}
Once we terminate the ball step algorithm with a success, we end up with a good volume estimate as in \eqref{vol}. Combining \eqref{f1small}, \eqref{f2small}, \eqref{f3small}, \eqref{f4small}, \eqref{f5small} with Proposition \ref{prop:iterate}, we conclude that
\[\mathbf{P} \left( \mu_{{\cal U} } \left( B_{{\cal U} } \left( 0,  C \lambda^{a'} R^{\beta} \right) \right) \ge c \lambda^{-18} R^{3} \right) \ge 1 - C \exp \left\{ - c \lambda^{a'} \right\}.\]
Reparameterizing this gives the desired result.
\end{proof}

We are now ready to derive the main result of this section, which gives exponential control of the volume of balls, uniformly over spatial regions.

\begin{thm}\label{2nd-goal} There exist constants $c, C , b \in (0, \infty)$ such that: for all $R \ge 1$ and $\lambda \ge 1$,
\begin{equation}\label{gl}
\mathbf{P} \left( \inf_{x \in B(R^{1/\beta} ) } \mu_{{\cal U} } \left( B_{{\cal U}} \left( x, \lambda^{-b'} R \right) \right)  \le \lambda^{-1} R^{\frac{3}{\beta}} \right) \le C \exp \left\{ - c \lambda^{b} \right\}.
\end{equation}
\end{thm}

\begin{proof} We will follow the strategy used in the proof of Proposition \ref{last-step}. Theorem \ref{1st-goal} tells us that if $A_{1} := \{ |B_{{\cal U}} ( 0, \lambda^{-1} R )|  \le \lambda^{-4} R^{\frac{3}{\beta}} \}$, then
\begin{equation}\label{30-1}
\mathbf{P} (A_{1} ) \le C \exp \left\{ - c \lambda^{b} \right\},
\end{equation}
We may assume that $b \in (0,1)$. We also let $b_{1} = b/1000$. Applying Proposition \ref{result:hit} with $s= \exp \{ - \lambda^{b_{1}} \} R^{1/\beta}$, $r = \exp \{ -  \lambda^{b_{1}} /2 \} R^{1/\beta}$ and $K=100$, we find that there exists universal constants $\eta \in (0, 1)$ and $C > 0$ such that
\[\mathbf{P} (A_{2} ) \le C \exp \left\{ -  \lambda^{b_{1}} \right\},\]
where $A_{2}$ is defined to be the event
\[\left\{\exists v \in B(5 R^{1/\beta}) \text{ such that } \text{dist} (v, \gamma ) \le e^{- \lambda^{b_{1}}}  R^{1/\beta}\text{ and } P^{v} \left( R^{v} [0, t_{v} ] \cap \gamma = \emptyset \right) \ge e^{- \eta \lambda^{ b_{1}}}\right\}.\]
Here, $\gamma$ represents the ILERW started at the origin, $R^{v}$ stands for a SRW started at $v$, the probability law of which is denoted by $P^{v}$, and $t_{v}$ stands for the first time that $R^{v}$ exits $B( v, \exp \{ - \lambda^{b_{1}} / 2 \} R^{1/\beta})$. We next use Proposition \ref{result:quasiloops} to conclude that there exists universal constants $C, M < \infty$ such that
\[\mathbf{P} (A_{3} ) \le C \exp \left\{ -  \lambda^{b_{1}} / M \right\},\]
where the event $A_{3}$ is defined by
\begin{eqnarray*}
A_{3}= \left\{
\begin{array}{cc}
\exists v \in B(5 R^{1/\beta}) \text{ and } i< j \text{ such that } \gamma (i), \gamma (j) \in B \left( v, 10 \exp \{ - \lambda^{b_{1}} /2 \} R^{1/\beta} \right) \\
 \text{ and } \gamma [i, j] \not\subseteq B \left( v, 10^{-1} \exp \{ - \lambda^{b_{1}} / M \} R^{1/\beta} \right)
\end{array}
\right\},
\end{eqnarray*}
Namely, the event $A_{3}$ says that $\gamma$ has a quasi-loop in $B (5 R^{1/\beta})$. We next let
\begin{equation}\label{del}
\delta = 10^{-3} \min \{ \eta, 1/M \},
\end{equation}
and define a sequence of random times $s_{1}, s_{2}, \dots $ by setting $s_{0} = 0$,
\begin{align*}
s_{1} &= \inf \left\{ t \ge 0 \::\: \gamma (t) \notin B \left( \exp \{ - \delta \lambda^{b_{1}} \} R^{1/\beta} \right) \right\},\\
s_{i} &= \inf \left\{ t \ge s_{i-1} \::\: \gamma (t) \notin B \left( \gamma (s_{i-1} ), \exp \{ - \delta \lambda^{b_{1}} \} R^{1/\beta} \right) \right\}, \qquad \forall i \ge 2.
\end{align*}
Let $x_{i} = \gamma (s_{i} )$, write
\[I = \left\{ i \ge 1 \::\: \left(\gamma [s_{i-1}, s_{ i} ] \cup \gamma [s_{i}, s_{ i+ 1} ]  \cup \gamma [s_{i+1}, s_{ i+ 2} ] \right)\cap B (4 R^{1/\beta}) \neq \emptyset \right\},\]
and set $N = |I|$. By considering the number of balls of radius $\exp \{ - \delta \lambda^{b_{1}} \} R^{1/\beta} $ crossed by a SRW before ultimately leaving $B (4 R^{1/\beta})$, we see that
\[\mathbf{P} (A_{4} ) \le C \exp \left\{ - c  e^{\delta \lambda^{b_{1}} }     \right\},\]
where $ A_{4}: = \{ N \ge \exp \{ 3 \delta \lambda^{b_{1}} \}\}$. A similar argument to that used in \cite[(7.51)]{LS} yields that
\[\mathbf{P} (A_{5} ) \le C \exp \left\{ - c  e^{\frac{\delta \lambda^{b_{1}}}{2} }     \right\},\]
where $A_{5}$ is the event that there exists an $i \in I$ such that $s_{i} - s_{i-1} \ge \exp \{ - \delta \lambda^{b_{1}} /2 \} R$. Thus, defining the event $A$ by setting
 \begin{equation*}
 A = \bigcap_{i=1}^{5} A_{i}^{c},
 \end{equation*}
combining the above estimates gives us that
\begin{equation}\label{30-5}
\mathbf{P} (A) \ge 1 - C \exp \left\{ -  \lambda^{b_{1}} / M \right\}.
\end{equation}

We now fix a net $W = ( w_{j})_{j}$ of $B (5 R^{1/\beta})$ such that
\[ B (5 R^{1/\beta}) \subseteq \bigcup_{j } B \left( w_{j},  \exp \{ - \lambda^{b_{1}} \} R^{1/\beta} \right)\]
and $| W| \asymp \exp \{ 3 \lambda^{b_{1}} \}$. For $i \in I$, let $w_{i} \in W$ be a point for which $| x_{i} - w_{i} | \le \exp \{ - \lambda^{b_{1}} \} R^{1/\beta} $. We now use Wilson's algorithm for all points $w_{i}$. On $A_{2}^{c}$, it holds that, for each $i \in I$,
\[P^{w_{i}} \left( R^{w_{i}} [0, t_{w_{i}} ] \cap \gamma = \emptyset \right) \le \exp \{ - \eta \lambda^{ b_{1}} \}.\]
Therefore we have
\begin{equation}\label{30-6}
\mathbf{P} (B_{1} ) \le C \exp \{ - \eta \lambda^{ b_{1}} \} \exp \{ 3 \delta \lambda^{b_{1}} \}  \le C \exp \{ - \eta \lambda^{ b_{1}} / 2 \},
\end{equation}
where $B_{1}$ is the event that there exists $i \in I$ such that $R^{w_{i}} [0, t_{w_{i}} ] \cap \gamma = \emptyset$. Suppose that the event $B_{1}^{c}$ occurs. For $i \in I$, write $u_{i}$ for the first time that $R^{w_{i}}$ hits $\gamma$, and let $z_{i} = R^{w_{i}} (u_{i} )$. On $A_{3}^{c}$, we have that $z_{i} \in \gamma [s_{i-1}, s_{i} ] \cup  \gamma [s_{i}, s_{i+1} ]$, because otherwise $\gamma$ has a quasi-loop. We define the events $B_{2}$ and $B_{3}$ by setting
\[B_{2} = \left\{\exists i \in I \text{ such that } \len \left( \text{LE} \left( R^{w_{i}} [0, u_{i} ] \right) \right) \ge \exp \{ - \lambda^{b_{1}}/4 \} R \right\},\]
\[B_{3} = \left\{ \exists i \in I \text{ such that }  \left| B_{{\cal U}} \left( w_{i}, \lambda^{-1} R \right)\right|  \le \lambda^{-4} R^{\frac{3}{\beta}} \right\}.\]
Combining the translation invariance of ${\cal U}$ with Proposition \ref{result:upperLowTail} ensures that
\begin{equation}\label{30-7}
\mathbf{P} (B_{2} ) \le C \exp \left\{ -c e^{\lambda^{b_{1}}/ 4 } \right\}.
\end{equation}
Moreover, by \eqref{30-1} and the translation invariance of the UST again, we have
\[\mathbf{P} (B_{3} ) \le C e^{- c \lambda^{b} } \times  e^{3 \lambda^{b_{1}}} \le C e^{ - c \lambda^{b} / 2 },\]
where we use the fact that $| W| \asymp e^{3 \lambda^{b_{1}}}$ and $b_{1} = b/1000$. Defining
\[B = \bigcap_{j= 1}^{3} B_{j}^{c},\]
we have proved that
\[\mathbf{P} (A \cap B) \ge 1 - C \exp \left\{ - \delta \lambda^{b_{1}}  \right\},\]
where we recall that $\delta > 0$ was defined as in \eqref{del}.

Next, suppose that the event $A \cap B$ occurs. Take $x \in \gamma \cap B (4 R^{1/\beta})$. We can then find some $i \in I$ such that $x \in \gamma [s_{i}, s_{i+1} ]$. On $A_{5}^{c}$, we have
\[d_{{\cal U}} (x_{i-1}, x_{i} ) \le \exp \{ - \delta \lambda^{b_{1}} /2 \} R \text{ and } d_{{\cal U}} (x_{i}, x_{i+1} ) \le \exp \{ - \delta \lambda^{b_{1}} /2 \} R.\]
Furthermore, on $B_{1}^{c}\cap A_{3}^{c}\cap B_{2}^{c}$, it holds that
\[z_{i} \in \gamma [s_{i-1}, s_{i} ] \cup \gamma [s_{i}, s_{i+ 1} ] \text{ and } d_{{\cal U}} (w_{i}, z_{i} ) \le \exp \{ - \lambda^{b_{1}}/4 \} R.\]
This implies that $d_{{\cal U}} (w_{i}, x) \le \exp \{ - \delta \lambda^{b_{1}} /4 \} R$. If $B_{3}^{c}$ also holds, it follows that we have
\[\mu_{{\cal U}} \left( B_{{\cal U}} \left( x, 2 \lambda^{-1} R \right) \right) \ge \lambda^{-4} R^{\frac{3}{\beta}}.\]
Consequently, we have proved that there exist universal constants $C, \delta, b_{1} \in (0, \infty )$ such that for all $R$ and $\lambda$
\begin{equation}\label{hobo}
\mathbf{P} \left(  \mu_{{\cal U}} \left( B_{{\cal U}} \left( x,  \lambda^{-1} R \right) \right) \ge \lambda^{-5} R^{\frac{3}{\beta}} \text{ for all } x \in \gamma \cap B (4 R^{1/\beta}) \right) \ge 1 - C \exp \left\{ - \delta \lambda^{b_{1}}  \right\}.
\end{equation}

Finally, once we get \eqref{hobo}, the proof of \eqref{gl} can be completed by following the strategy used to prove \eqref{1255} given \eqref{warm-2nd}. Indeed, thanks to \eqref{hobo}, we can use a net whose mesh size is exponentially small in $\lambda$, which guarantees the exponential bound as in \eqref{gla}. The simple modification is left to the reader.
\end{proof}

\section{Exponential upper tail bound on the volume}\label{sec:expupper}

Complementing the main result of the previous section, we next establish an exponential tail upper bound on the volume, see Theorem \ref{darui-1}, which improves the polynomial tail upper bound on the volume proved in Proposition \ref{1st-assump}. We begin with the following proposition.

\begin{prop}\label{2-4-1}
There exist constants $c, C, a \in (0, \infty )$ such that: for all $R \ge 1$ and $\lambda \ge 1$,
\[\mathbf{P} \left( B_{{\cal U}} \left( 0, \lambda^{-1} R^{\beta}  \right) \not\subseteq B ( R) \right) \le C \exp \{ - c \lambda^{a}\}.\]
In particular, it holds that
\[\mathbf{P}\left(\mu_{{\cal U}}\left(B_{{\cal U}}\left(0,\lambda^{-1}R^{\beta}\right)\right)\ge R^{3}\right)\le C\exp\{-c\lambda^{a}\}.\]
\end{prop}

\begin{proof}
The second inequality immediately follows from the first one. Thus it remains to prove the first inequality. We follow the strategy used in the proof of \cite[Theorem 3.1]{BM}. We may assume that $\lambda$ is sufficiently large. It follows from Proposition \ref{result:upperLowTail} that there exist constants $C, c$ and $a_{0} > 0$ such that
\[\mathbf{P} \left( T_{R/ 8}  < \lambda^{-1} R^{\beta} \right) \le C \exp \{ - c \lambda^{a_{0} } \},\]
where again $T_{r}$ stands for the first time that the ILERW $\gamma$ exits $B (r)$. Setting $a_{1} = a_{0} / 10$, we define a sequence of nets $D_{k}$ as follows. For $k \ge 1$, set $\delta_{k} = 2^{-k} \exp \{ - \lambda^{a_{1}} \}$, $\eta_{k} = (2k)^{-1}$, and $k_{0}$ be the smallest integer such that $\delta_{k_{0}} R < 1$. Defining
\[A_{k} := B (R) \setminus B \left( (1- \eta_{k} ) R \right),\]
let $D_{k}$ be a set of points in $A_{k}$ satisfying $|D_{k}| \asymp \delta_{k}^{-3}$ and also that
\[A_{k} \subseteq \bigcup_{w \in D_{k} } B \left( w, \delta_{k} R \right).\]
We then perform Wilson's algorithm as follows.
\begin{itemize}
\item Let ${\cal U}_{0} = \gamma$ be the ILERW, which is the root of the algorithm.
\item Take $w \in D_{1}$, and consider the SRW $R^{w}$ started at $w$, and run until it hits ${\cal U}_{0}$. We add $\text{LE} (R^{w} ) $ to ${\cal U}_{0}$. We choose another point $w' \in D_{1}$ and add the loop-erasure of $R^{w'}$, a SRW started at $w'$ and run until it hits the part of the tree already constructed. We perform the same procedure for every point in $D_{1}$. Let ${\cal U}_{1}$ be the output tree.
\item We perform the same algorithm as above for all points in $D_{2}$. Let ${\cal U}_{2}$ be the output tree. Similarly, we define ${\cal U}_{k}$.
\item We perform Wilson's algorithm for all points in $ {\cal U}_{k_{0}}^{c}$.
\end{itemize}
Since $\delta_{k_{0}} R < 1$, we note that $\partial_{i} B (R)  \subseteq A_{k_{0}} \subseteq {\cal U}_{k_{0}}$.

Now, take $w \in D_{1}$, and let $N_{w}$ be the first time that $\gamma_{{\cal U}} (0, w)$  exits $B (R/ 8 )$. Using \cite[Proposition 4.4]{Mas}, we see that
\[\mathbf{P} \left( N_{w}  < \lambda^{-1} R^{\beta} \right) \le C \mathbf{P} \left( T_{R/ 8}  < \lambda^{-1} R^{\beta} \right) \le C \exp \{ - c \lambda^{a_{0} } \}\]
for each $w \in D_{1}$. Thus if we define the event $F_{1}$ by setting
\[F_{1} = \left\{ T_{R/ 8}  < \lambda^{-1} R^{\beta} \right\}\cup\bigcup_{w \in D_{1}} \left\{ N_{w}  < \lambda^{-1} R^{\beta} \right\},\]
then it follows that
\[\mathbf{P} (F_{1} ) \le C \delta_{1}^{-3} \exp \{ - c \lambda^{a_{0} } \} \le C \exp \{ - c' \lambda^{a_{0} } \},\]
where we have used the fact that  $|D_{1}| \asymp \delta_{1}^{-3} \asymp \exp \{ 3 \lambda^{a_{1}} \}$ and that $a_{1} = a_{0} /10$.

Next, for $b > 0$, we define $G_{1}^{w} (b)$ to be the event
\[\left\{\exists  v \in B (2 R) \text{ with } \text{dist} \left( v, \gamma_{{\cal U}} (w, \infty ) \right) \le \delta_{1} R \text{ such that } P^{v} \left( R^{v} [0, \xi ] \cap \gamma_{{\cal U}} (w, \infty ) = \emptyset \right) \ge \delta_{1}^{b}\right\},\]
where $\xi$ is the first time that $R^{v}$ exits $B \left( v, \sqrt{\delta_{1}} R \right)$. Applying Proposition \ref{result:hit} to the case that $K = 100$, it holds that there exists $b_{0} > 0$ such that
\begin{equation}\label{darui}
\mathbf{P} (G_{1}^{w} ) := \mathbf{P} \left(  G_{1}^{w} (b_{0}) \right) \le C \delta_{1}^{50}.
\end{equation}
So, if we define the event $G_{1}:=\cup_{w \in D_{1}} G_{1}^{w}$, then
\[\mathbf{P} (G_{1} ) \le C \delta_{1}^{47}.\]

Suppose that the event $F_{1}^{c} \cap G_{1}^{c}$ occurs, and perform Wilson's algorithm (see \cite{Wilson}) from all points in $D_{2}$. For $w \in D_{2}$, define $H_{2}^{w}:=\{ \gamma_{{\cal U}} (w , 0 ) \text{ enters } B (R/ 2 ) \text{ before it hits } {\cal U}_{1}\}$, and let $H_{2} = \cup_{w \in D_{2}}  H^{w}_{2}$. The event $H^{w}_{2}$ implies that $R^{w}$ enters $B (R/ 2)$ without hitting ${\cal U}_{1}$. Since the event $G_{1}^{c}$ occurs, we see that
\[\mathbf{P} (H^{w}_{2} ) \le \left( \delta_{1}^{b_{0}} \right)^{c \delta_{1}^{-1/2} },\]
and thus we have
\[\mathbf{P} (H_{2} ) \le C \delta_{1}^{10}.\]

For $w \in D_{2}$, we then define $G_{2}^{w} = G_{2}^{w} (b_{0})$ to be the event
\[\left\{\exists  v \in B (2 R) \text{ with } \text{dist} \left( v, \gamma_{{\cal U}} (w, \infty ) \right) \le \delta_{2} R \text{ such that }P^{v} \left( R^{v} [0, \xi ] \cap \gamma_{{\cal U}} (w, \infty ) = \emptyset \right) \ge \delta_{2}^{b_{0}}\right\},\]
where $\xi$ is the first time that $R^{v}$ exits $B( v, \sqrt{\delta_{2}} R)$, and $b_{0}$ is the constant defined as above (see \eqref{darui} for $b_{0}$). Using \cite[Lemma 3.2 and Lemma 3.3]{SS} once again (with $r = \sqrt{\delta_{2}} R$ and $s = \delta_{2} R$), we have
\[\mathbf{P} (G^{w}_{2} ) \le C \delta_{2}^{50}.\]
Importantly, we can take $b_{0}$ depending only on $K = 100$. Define the event $G_{2}$ by setting $G_{2} := \cup_{w \in D_{2} } G^{w}_{2}$, and then
\[\mathbf{P} (G_{2} ) \le C \delta_{2}^{47}.\]
Defining $H_{k}$ and $G_{k}$, $k \ge 3$ similarly, it follows that
\[\mathbf{P} (H_{k} \cup G_{k} ) \le C \delta_{k}^{47}.\]

Finally, we define
\[J = F_{1}^{c} \cap G_{1}^{c} \cap \bigcap_{k=2}^{k_{0}} (H_{k}^{c} \cap G_{k}^{c} ).\]
On the event $J$, we have the following.
\begin{itemize}
\item For all $k = 1, 2, \dots k_{0}$ and every $w \in D_{k}$, the first time that $\gamma_{{\cal U}} (0, w)$ exits $B (R/ 8) $ is greater than $\lambda^{-1} R^{\beta} $.
\item The set $D_{k_{0}}$ disconnects $0$ and $B (R)^{c}$.
\end{itemize}
Thus, on the event $J$, it holds that $ B_{{\cal U}} \left( 0, \lambda^{-1} R^{\beta} \right) \subseteq B (R)$. Since
\[\mathbf{P} (J^{c} ) \le C \exp \{ - c' \lambda^{a_{0} } \} + C \sum_{k= 1}^{k_{0}} \delta_{k}^{47} \le C \exp \{ - \lambda^{a_{1}} \},\]
we have thus completed the proof.
\end{proof}

We are now ready to establish the main result of the section.

\begin{thm}\label{darui-1}
There exist constants $c', C', a' \in (0, \infty )$ such that: for all $R \ge 1$ and $\lambda \ge 1$,
\begin{equation}\label{darui-2}
\mathbf{P} \left( \max_{z \in B (R^{1/\beta}) }  \mu_{{\cal U}} \left( B_{{\cal U}} \left( z, \lambda^{a'} R \right) \right) \ge \lambda R^{3/\beta} \right) \le C' \exp \{ - c' \lambda^{a'} \}.
\end{equation}
\end{thm}

\begin{proof}
Since the proof is very similar to that of Theorem \ref{2nd-goal}, we will only explain how to modify it here. Also, we will use the same notation used in the proof of Theorem \ref{2nd-goal}. Proposition \ref{2-4-1} tells that there exist constants $c, C, b \in (0, \infty )$ such that
\begin{equation}\label{tsurai}
\mathbf{P} (A_{1}' ) \le C \exp \{ - c \lambda^{b} \},
\end{equation}
where $A_{1}':=\{ \mu_{{\cal U}} ( B_{{\cal U}} ( 0, \lambda R )) \le \lambda^{10} R^{\frac{3}{\beta}} \}$. In this proof, we choose the constant $b$ in this way, and let $b_{1} = b / 1000$. Using this constant $b_{1}$, we define the events $A_{2}, \dots, A_{5}$ as in the proof of Theorem \ref{2nd-goal}. Let $A = (A_{1}' )^{c} \cap (\cap_{i= 2 }^{5} A_{i}^{c})$ so that
\[\mathbf{P} (A) \ge 1 - C \exp \left\{ -  \lambda^{b_{1}} / M \right\},\]
see \eqref{30-5}. We also recall the events $B_{1}$ and $B_{2}$ defined in the proof of Theorem \ref{2nd-goal}, for which
\[\mathbf{P} (B_{1}  \cup B_{2}) \le C \exp \{ - \eta \lambda^{ b_{1}} / 2 \},\]
see \eqref{30-6} and \eqref{30-7}. Moreover, let
\begin{equation*}
B_{3}' = \left\{ \exists i \in I \text{ such that }  \mu_{{\cal U}} \left( B_{{\cal U}} \left( w_{i}, \lambda R \right) \right) \ge \lambda^{10} R^{\frac{3}{\beta}} \right\}.
\end{equation*}
Combining \eqref{tsurai} with the translation invariance of the UST, we have
 \[\mathbf{P} (B_{3}' ) \le C e^{- c \lambda^{b} } \times  e^{3 \lambda^{b_{1}}} \le C e^{ - c \lambda^{b} / 2 },\]
where we have also used the fact that $|W| \asymp e^{3 \lambda^{b_{1}}}$ and $b_{1} = b/1000$. Setting $B = B_{1}^{c} \cap B_{2}^{c} \cap (B_{3}' )^{c} $, we then have that
\[\mathbf{P} (B)  \ge 1-  C \exp \{ - \eta \lambda^{ b_{1}} / 4 \}.\]

Now, suppose that the event $A \cap B$ occurs, and let $x \in \gamma \cap B (4 R^{1/\beta})$. We can then find some $i \in I$ such that $x \in \gamma [s_{i}, s_{i+1} ]$. Since $A_{5}^{c}$ holds, we have
\[ d_{{\cal U}} (x_{i-1}, x_{i} ) \le \exp \{ - \delta \lambda^{b_{1}} /2 \} R \text{ and } d_{{\cal U}} (x_{i}, x_{i+1} ) \le \exp \{ - \delta \lambda^{b_{1}} /2 \} R.\]
Furthermore, since $B_{1}^{c}\cap A_{3}^{c}\cap B_{2}^{c}$ holds, we have that
\[z_{i} \in \gamma [s_{i-1}, s_{i} ] \cup \gamma [s_{i}, s_{i+ 1} ] \text{ and } d_{{\cal U}} (w_{i}, z_{i} ) \le \exp \{ - \lambda^{b_{1}}/4 \} R.\]
This implies that $d_{{\cal U}} (w_{i}, x) \le \exp \{ - \delta \lambda^{b_{1}} /4 \} R$. Given $(B_{3}')^{c}$ also holds, we therefore have
\[\mu_{\cal U} \left( B_{{\cal U}} \left( x,  \lambda R /2  \right) \right)  \le \lambda^{10} R^{\frac{3}{\beta}}.\]
Consequently, we have proved that there exist universal constants $C, \delta, b_{1} \in (0, \infty )$ such that: for all $R$ and $\lambda$,
\begin{equation}\label{hobo-1}
\mathbf{P} \left(  \mu_{\cal U} \left(  B_{{\cal U}} \left( x,  \lambda R/2 \right) \right)  \le \lambda^{10} R^{\frac{3}{\beta}} \text{ for all } x \in \gamma \cap B (4 R^{1/\beta}) \right) \ge 1 - C \exp \left\{ - \delta \lambda^{b_{1}}  \right\}.
\end{equation}
Similarly to the comment at the end of the proof of Theorem \ref{2nd-goal}, given \eqref{hobo-1}, the proof of \eqref{darui-2} follows by applying the same strategy as that used to prove \eqref{1255} given \eqref{warm-2nd}. Indeed, given \eqref{hobo-1}, we can use a net whose mesh size is exponentially small in $\lambda$, which guarantees the exponential bound as in \eqref{darui-2}. The simple modification is again left to the reader.
\end{proof}

\section{Convergence of finite-dimensional distributions}\label{finitesec}

As noted in the introduction, the existence of a scaling limit for the three-dimensional LERW was first demonstrated in \cite{Kozma}. The work in \cite{Kozma} established the result in the Hausdorff topology, and this was recently extended in \cite{LS} to the uniform topology for parameterized curves. Whilst the latter seems a particularly appropriate topology for understanding the scaling limit of the LERW, the results in \cite{Kozma, LS} are restrictive when it comes to the domain upon which the LERW is defined. More specifically, we say that a LERW is defined in a domain $D$ if it starts in an interior point of $D$ and ends when it reaches the boundary of $D$. The assumptions in \cite{Kozma} cover the case of LERWs defined in domains with a polyhedral boundary, while \cite{LS} requires the domain to be a ball or the full space.

In this section, we extend the existence of the scaling limit to LERWs defined in the domain $\mathbb{R}^3 \backslash \cup_{j=1}^K \tr \mathcal{K}_j$, where each $\mathcal{K}_j$ is itself a path of the scaling limit of a LERW. Once we gain this level of generality, we use Wilson's algorithm to obtain the convergence in distribution of rescaled subtrees of the UST (see Figure \ref{3dsubtree} for an example realisation of the subtree spanning a finite collection of points). This will be crucial for establishing the convergence part of Theorem \ref{mainthm1}. We begin by introducing some notation for subtrees.

\begin{figure}[t]
\begin{center}
 \includegraphics[width=0.4\textwidth]{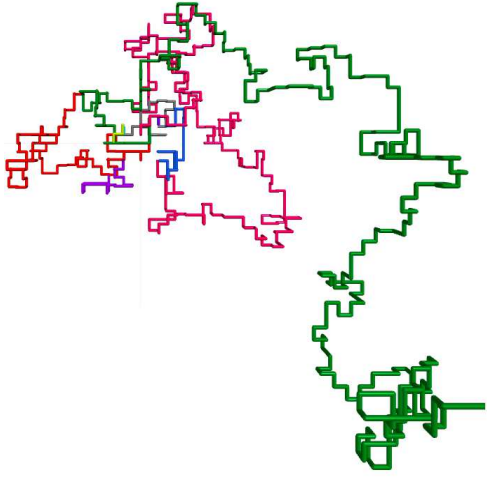}
\end{center}
\caption{A realisation of a subtree of the UST of $\delta\mathbb{Z}^3$ spanned by $0$ and the corners of the cube $[-1,1]^3$. The tree includes part of its path towards infinity (in green). Colours indicate different LERWs used in Wilson's algorithm.}\label{3dsubtree}
\end{figure}

\subsection{Parameterized trees}

A parameterized tree is an encoding for an infinite tree embedded in the closure of $\mathbb{R}^3$.
This encoding is specialized for infinite trees with a finite number of spanning points and one end.
More precisely, a \textbf{parameterized tree} $\gls{parameterizedTree}$ with $K$ spanning points is defined as $\mathscr{T} = (X, \Gamma) $ where:
\begin{enumerate}
  \item $X = \{ x(1), \ldots , x(K) \} \subset \mathbb{R}^3  $ are the spanning (or distinguished) points; and
  \item $ \gamma^{x(i)} $ is a transient parameterized (simple) curve starting at $x(i)$, and
  \begin{equation*}
    \Gamma = \{ \gamma^{x(i)} \colon 0 \leq i \leq K \}.
  \end{equation*}
  We require that for any pair $i,j$ there exist \textbf{merging times} $s^{i,j}, s^{j,i} \geq 0$ satisfying
  \begin{enumerate}
  \item $ \gamma^{x(i)} \vert_{[s^{i,j}, \infty)} = \gamma^{x(j)} \vert_{[s^{j,i}, \infty)} $; and  \label{condition:conected}
  \item $  \tr \gamma^{x(i)} \vert_{[0, s^{i,j})} \cap \tr \gamma^{x(j)} \vert_{[0, s^{j,i})}  = \emptyset $. \label{condition:simple}
  \end{enumerate}
\end{enumerate}

Let $\gls{parameterizedForest}$ be the space of parameterized trees with $K$ distinguished points. We endow $\mathscr{F}^K$ with the distance
 \[ d_{\mathscr{F}^K} \left( \mathscr{T} , \tilde{\mathscr{T}} \right) :=
  \max_{1 \leq i \leq K} \left\lbrace \chi \left( \gamma^{x(i)}, \tilde{\gamma}^{\tilde{x}(i)} \right) \right\rbrace + \max_{ 1 \leq i,j \leq K} \{ \vert s^{i,j} - \tilde{s}^{i,j} \vert \},\]
for $\mathscr{T} = (X, \Gamma) , \tilde{\mathscr{T}} = (\tilde{X}, \tilde{\Gamma}) \in \mathscr{F}^K$. Recall that $\chi$ is the distance for transient parameterized curves defined in \eqref{eq:distchi}.

We write
\begin{equation*}
  \tr \mathscr{T} = \bigcup_{\gamma \in \Gamma} \tr \gamma
\end{equation*}
for the trace of a parameterized tree.

\begin{prop} \label{prop:treeistree}
  Let $\mathscr{T}$ be a parameterized tree. Then $\tr \mathscr{T}$ is a topological tree with one end.
  Additionally,   for any $z,w \in \tr \mathscr{T}$ there exists a unique curve from $z$ to infinity on $\mathscr{T}$, denoted by $\gamma^z$ and a unique curve from $z$ to $w$ in $\mathscr{T}$ denoted by $\gamma^{z,w}$.
\end{prop}

\begin{proof}
  The set $\tr \mathscr{T}$ is path-connected as a consequence of
  condition \eqref{condition:conected} in the definition of a parameterized tree.
  It is also one-ended, since $\cap_{i = 1}^K \gamma^{x(i)}$ is a single parameterized curve towards infinity.

  The main task  in this proof is to show that there cannot be cycles embedded in $\tr \mathscr{T}$.
  We proceed by contradiction.
  Let $S^1$ be the circle and assume that $ \varphi : S^{1} \to \tr \mathscr{T} $ is an injective embedding.
  Since every curve in $\Gamma$ is simple and $\varphi$ is injective, then $\varphi(S^1)$ intersects at least two different curves, say $\gamma^{x(i)}$ and $\gamma^{x(j)}$.
  From the definition of merging times, we see that $T^2 = \tr \gamma^{x(i)} \cup \tr \gamma^{x(j)} $ is homeomorphic to $ \left( [0 , \infty) \times \{ 0 \} \right) \cup \left( \{ 1 \} \times [0, 1] \right) $, but the latter space cannot contain a embedding of $S^1$.
  It follows that $\varphi (S^1)$ intersects at least a third curve $\gamma^{x(\ell)}$.
  We assume that $\varphi (S^1)$ is contained in $ T^3 =  \tr (\gamma^{x(i)}) \cup \tr (\gamma^{x(j)}) \cup \tr (\gamma^{x(\ell)}) $.
  Under the last assumption, it is necessary that $\gamma^{x(\ell)}$ intersects $\gamma^{x(i)}$ and $\gamma^{x(j)}$ before these last two curves merge (otherwise the case is similar to $T^2$).
  Denote the intersection times by $t^{\ell,i}$ and $ t^{i, \ell} $, so $\gamma^{x(\ell)} (t^{\ell, i}) = \gamma^{x(i)} (t^{i, \ell})$. We use the same notation for $\gamma^{x(j)}$.
  Then, we have that $ t^{i, \ell} < s^{i, j} $ and $ t^{j,\ell} < s^{j, i} $. Without loss of generality, $ t^{\ell, i} < t^{\ell, j}$.
  However, it is easy to verify that $ t^{\ell, i} $ is not the merging time $s^{\ell, i}$, since $\gamma^{x(\ell)} \vert _{ [t^{\ell, i}, \infty) }$ does not merge with $\gamma^{x(i)}$ at that point. Therefore $s^{\ell, i}$ does not exist, and this conclusion contradicts the definition of $\Gamma$.
  It follows that $\varphi (S^1)$ is not contained in $T^3$, but it intersects more curves, e.g. all of them in $\tr \mathscr{T} =  \cup_{i = 1}^K \tr ((\gamma^{x(i)})  )$. However, the argument that we used for $T^3$ also applies to $\tr \mathscr{T}$. We conclude that the embedding $\varphi$ does not exist.

  Finally, observe that $\tr \mathscr{T}$ is one ended and all curves in $\Gamma$ are parameterized towards infinity. It is then straightforward to define $\gamma^z$ and $\gamma^{z,w}$.
\end{proof}

A corollary of Proposition \ref{prop:treeistree} is that the intrinsic distance in $\tr \mathscr{T}$ is well-defined.
It is given by
\begin{equation*}
  d_{\mathscr{T}} (z,w) := T (\gamma^{z,w} ), \qquad z , w \in \tr \mathscr{T}
\end{equation*}
where $T(\cdot)$ is the duration of a curve.

We will consider restrictions of parameterized trees to balls centred at the origin. For a parameterized tree $\mathscr{T} = (X, \Gamma)$, let $R \geq 1$ be large enough so that $X \subseteq B_E(R)$. We restrict each curve in $\Gamma$ to $\gamma^{x(i)} \vert^{R}$ (where the restriction to the ball of radius $R > 0$ is in the sense described in Subsection \ref{subsec:notationpaths}), and define the restriction of a parameterized tree to $B_E(R)$ as the subset of $\mathbb{R}^3$
\[\mathscr{T}\vert^{R} := \bigcup_{\gamma \in \Gamma} \tr  \gamma^{x(i)} \vert^{R} .\]
Note that$\mathscr{T} \vert^{R}$ may not be connected for some values of $R > 0$. But for $R$ large enough, $\mathscr{T} \vert^{R}$ is a topological tree (Figure~\ref{fig:trees-differ} gives an example of both cases).

\begin{figure}
\begin{center}
  \includegraphics{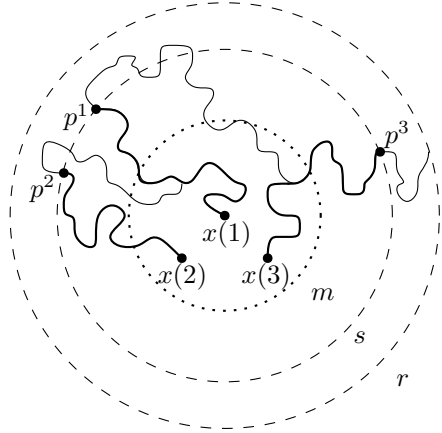}
\end{center}
\caption{
$\mathscr{S}$ is a parameterized tree with spanning points $ x(1), x(2) $ and $x(3)$. The restriction $\mathscr{S} \vert^s$ is the union of the paths between $x(i)$ and $p^i$, with $i  = 1, 2, 3$.
In this example, $\mathscr{S} \vert^r$ and $\mathscr{S} \vert^s$ are different inside  the radius $m$. A crucial difference between these two sets is that $\mathscr{S} \vert^r$ is connected, but $\mathscr{S}  \vert^s $ is disconnected. }\label{fig:trees-differ}
\end{figure}

\subsection{The scaling limit of subtrees of the UST}

We introduce the main results of this section.

Let $\mathcal{U}_n$ be the uniform spanning tree on $2^{-n} \mathbb{Z}^3$.
We are interested in subtrees of $\mathcal{U}_n$ spanned by $K$ distinguished points. Let $x(1), \ldots , x (K)$ be different points in $\mathbb{R}^3$ and let $X_n = \{  x_n (1), \ldots , x_n (K) \}$ be a subset of $2^{-n} \mathbb{Z}^3$ such that $x_n(i) \to x (i)$ as $n \to \infty$, for each $ i = 1, \ldots, n$.
Denote by $\gamma_n^{x_n(i)}$ the transient path in $\mathcal{U}_n$ starting at $x_n (i)$ and parameterized by path length.
We set $\bar{\gamma}_n^{x(i)}$ to be the $\beta$-parameterization of $\gamma_n^{x_n(i)}$ and $\Gamma_n = \{ \gamma_n^{x_n(i)} \colon  1 \leq i \leq K \}$.
Then $\mathscr{S}_n^K = (X , \Gamma_n ) $ is the parameterized tree corresponding to the subtree of $\mathcal{U}_n$ spanned by $x_n (1), \ldots , x_n (K)$ and the point at infinity.

\begin{thm} \label{thm:convergecetrees}
  The sequence of parameterized trees $(\mathscr{S}_n^K)_{n \in \mathbb{N}}$ converges weakly in the space $\mathscr{F}^K$ as $n \to \infty$.
  We denote a sample of the limit measure by $\hat{\mathscr{S}}^K$.
\end{thm}

The proof of Theorem \ref{thm:convergecetrees} relies on the convergence of the branches of the uniform spanning tree as they appear in Wilson's algorithm.
In the next section, we control the behaviour of a LERW before it hits an approximation of a parameterized tree. These loop-erased random walks correspond to the branches  of the UST.
Then Proposition \ref{prop:continuity} shows that convergence of such branches implies convergence of parameterized trees.
After these arguments, we are prepared for the proof of Theorem \ref{thm:convergecetrees}. We present it in Subsection \ref{subsec:essential-branches-UST}.

Conversely, Proposition \ref{prop:convdist} shows that convergence of parameterized trees implies the convergence of the intrinsic distance.
We thus get the following corollary of Theorem \ref{thm:convergecetrees}.

\begin{cor}\label{fddcor}
Let $(x_\delta (i))_{i=1}^K$ be a collection of points in $\delta\mathbb{Z}^3$ such that $x_\delta (i) \rightarrow x (i)$, for all  $i=1,\dots,K$, for some collection of distinct points $(x(i))_{i=1}^K$ in $\mathbb{R}^3$. Along the subsequence $\delta_n=2^{-n}$, it holds that
\[\left(\delta_n^\beta d_\sU\left(x_{\delta_n} (i),x_{\delta_n} (j)\right)\right)_{ i,j=1}^K\]
converges in distribution.
\end{cor}


\subsection{Parameterized trees and random walks}

We presented the notion of an $\eta$-hittable set in \eqref{def:etaHittable}.
Following the notation of Section~\ref{sec:lerw} and for $\delta \in (0,1)$, let $\delta^{-1} \tr \mathscr{T}_{\delta}$ be a $\delta^{-1}$-expansion  (this expansion is similar to the one in  \eqref{eq:expansion}).
We denote by $H (\mathscr{T}_{\delta}, R, \varepsilon; \eta )$ the event where the set  $\delta^{-1} \tr \mathscr{T}_{\delta}$  is $\eta$-hittable around $0$ in $B(0, R\delta^{-1} )$, and with parameters $R \geq 1$ and $\varepsilon \in (0,1)$.
In the context of this section, we scale the lattice by $\delta$ instead of scaling subsets of $\mathbb{Z}^3$ with an $\delta^{-1}$-expansion. Hence, we use the equivalent event:
\[
H (\mathscr{T}_{\delta}, R, \varepsilon; \eta ) =
\left\{
\begin{array}{c}
\forall x \in B_{\delta} ( 0, R )  \text{ with } \dist ( x, \tr \mathscr{T}_{\delta} ) \leq \eps^2, \\ P^x \left(  S^x \left[ 0, \xi_S ( B_{\delta}(x, \eps^{1/2} ) ) \right] \cap \tr \mathscr{T}_{\delta} = \emptyset  \right) \leq\eps^{\eta}
\end{array}
\right\}.
\]
In the event above, recall that $\xi_S ( B_{\delta}(x, \eps^{1/2})) $ stands for the first exit time from the $\delta$-scaled discrete ball $B_{\delta}(x, \eps^{1/2})$.

\begin{prop} \label{prop:treehittable}
There exist constants $\eta > 0$ and $C < \infty$ such that the following holds. If $\mathscr{S}^K_{\delta}$ is a parameterized subtree of the uniform spanning tree on $\delta\mathbb{Z}^3$ with $K$ spanning points then for all $\delta\in(0,1)$, $R \geq 1$ and $\eps > 0$,
\[\mathbf{P} \left( H ( \mathscr{S}^K_{\delta}\vert^{R}, \varepsilon ; \eta ) \right)\geq 1 - C K R^3  \eps.\]
\end{prop}

\begin{proof}
  Recall that any path towards infinity in the uniform spanning tree is equal, in distribution, to a ILERW.
  Then, the probability that  $x \in B_{\delta} (0, R)$ hits the tree $ \mathscr{S}^K_{\delta}\vert^{R} $ is at least the probability that $x$ hits a restricted ILERW, where such restriction is up to the first exit of the LERW from $B_{\delta} (0,R)$.
  Then Proposition \ref{prop:treeistree} is a consequence of Proposition \ref{result:hit}.
\end{proof}

\begin{rem}
The proof of Proposition \ref{prop:extension} can be generalized to any subset of $\mathbb{R}^3$ that is $\eta$-hittable with high probability. We restrict to the case of parameterized subtrees for clarity, and because it is the most relevant for our purposes. To further increase the clarity of the proof of Proposition \ref{prop:extension}, the reader can think of the subtree $\mathscr{S}^K_n$ as consisting of a single ILERW.
\end{rem}

\subsection{Essential branches of parameterized trees} 

Let $\mathscr{T} = (X, \Gamma)$ be a parameterized tree.
For a leaf $x(i) \in X$ with $i>1$, let
\begin{equation} \label{eq:branch-point}
  y(i) := \tr \gamma^{x(i)} \cap \bigcup_{j = 1}^{i-1} \tr \gamma^{x(j)}
\end{equation}
be the intersection point of $\gamma^{x(i)}$ with any of the curves with an smaller index. We define $y (1)$ to be the point at infinity and say $y (i)$ is a \textbf{branching point}.
When we compare \eqref{eq:branch-point} with conditions \eqref{condition:conected} and \eqref{condition:simple} in the definition of parameterized tree, we see that
\[
  y(i) = \gamma^{x(i)} (s^{i, m(j)}),
\]
where $ s^{i, m(j)} = \min_{ j < i } \{ s^{i,j} \}  $ is the first merging time.

The parameterized curves $\gamma^{x(i), y(i)}$ are called \textbf{essential branches} for $i = 1, \ldots, K$.
Note that   $ \gamma^{x(1), y(1)} $ is the transient curve $\gamma^{x(1)} \in \mathcal{C} $, while $  \gamma^{ x(i), y(i) } \in \mathcal{C}_f $ for $ i = 1, \ldots, K $.
We denote the set of essential branches by
$ \gls{essentialBranches} := \{ \gamma^{x(i),y(i)}  \}_{ 1 \leq i \leq K} $.

\begin{prop}  \label{prop:conv-essential-branch}
  Assume that $\mathscr{T}_n \to \mathscr{T}$ in the space of parameterized trees $\mathscr{F}^K$.
  Then
  \[
    \gamma_n^{x_n(1), y(1)} \to \gamma^{x_n(1), y(1)} \qquad \text{as } n \to \infty
  \]
  in the space $\mathcal{C}$. For $i = 2, \ldots, K$, the essential branches and the curves between branching points converge:
  \[
    \gamma_n^{x_n(i), y_n(i)} \to \gamma^{x(i), y(i)}, \qquad \gamma_n^{y_n(i), y_n(j)} \to \gamma^{y(i), y(j)} \qquad \text{as } n \to \infty
  \]
  in the space of finite parameterized curves $\mathcal{C}_f$.
\end{prop}

\begin{proof}
  The convergence of the first essential branch is immediate from the definition of the metric $ d_{\mathscr{F}^K} $, since $\gamma_n^{x_n(1), y(1)} = \gamma_n^{x_n(1)}$.

  To prove the convergence of the other essential branches, and the curves between branching points, we first need to show that spanning and branching points converge.

  Each $x_n(i) \in X$ is the initial point of a curve in $\Gamma$ and hence
  convergence in the space of parameterized trees implies that $x_n \to x$ as $n \to \infty$.
  Now we consider a branching point $y_n(i)$, with $i = 2, \ldots, K$.
  Recall that $y_n (i) = \gamma_n^{x(i)} (s_n^{i, m(j)})$, where $s_n^{i, m(j)} = \min_{ j < i} \{ s_n^{i, j} \}$.
  Since convergence of the parameterized trees $ \mathscr{T_n} $ imply convergence of the merging times $s^{i,j}_n \to s^{i,j}$ as $n \to \infty$, then, for the sequence of minima, $ s_n^{i,m(j)} \to s^{i,m(j)} $.
  With an application of Proposition \ref{prop:continuity-chi} \ref{item:evaluation-infty}, we get convergence of the branching points
  $ y_n (i) = \gamma_n^{x(i)} (s_n^{i, m(j)}) \to \gamma^{x(i)} (s^{i, m(j)}) = y(i) $.

  With convergence of both the spanning and branching points, Proposition~\ref{prop:continuity-psi} and Proposition~\ref{prop:continuity-chi} imply that the corresponding restrictions of $\gamma^{x(i)}$ converge.
\end{proof}

\begin{prop} \label{prop:convdist}
  Assume that $\mathscr{T}_n = (X_n, \Gamma_n)$ converges to $\mathscr{T}$ in the space of parameterized trees.
  If the corresponding collections of spanning points are $X_n = \{  x_n(1), \ldots , x_n (K)  \}$ and $X = \{ x(1), \ldots , x(K) \} $, then
  \begin{equation} \label{eq:convdist}
    \left( d_{\mathscr{T}_n} (x_n (i), x_n (j)) \right)_{1 \leq  i, j \leq K } \to \left( d_{\mathscr{T}} (x(i), x (j)) \right)_{1 \leq  i, j \leq K }
  \end{equation}
  as $n \to \infty$.
\end{prop}

\begin{proof}
  Proposition \ref{prop:treeistree} shows that restriction, concatenation and time-reversal of the curves in $\Gamma_n$ define $\gamma_n^{x_n(i), x_n(j)}$.  In fact,
  \begin{equation} \label{eq:curve-decomp}
    \gamma_n^{x_n(i), x_n(j)} = \gamma_n^{x_n(i), y_n(i)} \oplus \gamma_n^{y_n(\ell_1), y_n(\ell_2) } \oplus \ldots \oplus \gamma_n^{ y_n(\ell_{m-1}), y_n(\ell_m) } \oplus \gamma_n^{y_n(j) x_n(j)},
  \end{equation}
  where $\ell_1 = i$ and $\ell_m = j $.
  Then Proposition \ref{prop:conv-essential-branch} implies the convergence of each essential branch, and Proposition \ref{prop:continuity-psi} and Proposition \ref{prop:continuity-chi} imply the convergence of $ ( \gamma_n^{x_n(i), x_n(k)} )_{n \in \mathbb{N}}$.
  In particular, the duration of each curve in \eqref{eq:curve-decomp} converges and we get \eqref{eq:convdist}.
\end{proof}

Conversely, we can reconstruct a tree from a set of essential branches.

\begin{prop} \label{prop:essentialbranch-const}
  Let $X = \{ x(1), \ldots, x(K) \} \subset \mathbb{R}^3 $ and consider a collection of curves with the following conditions:
  \begin{enumerate}[label=(\alph*)]
    \item Let $\gamma^{x(1), \tilde{y}(1)}$ be a transient parameterized curve starting at $x(1)$; recall that $\tilde{y}(1)$ denotes the point at infinity.
    \item For $ i = 2, \ldots n$, $\gamma^{x(i), \tilde{y}(i)}$ is a parameterized curve starting at $x(i)$ and ending at $\tilde{y}(i)$, where the endpoint $\tilde{y}(i)$ is the first hitting point to $ \bigcup_{j = 1}^{i-1} \tr \gamma^{x(i), \tilde{y}(i)} $.
  \end{enumerate}
  Then $ \{ \gamma^{x(i), \tilde{y} (i)} \}_{1 \leq i \leq K} $ defines a set of transient curves $\Gamma = \{ \gamma^{x(i)} \}_{1 \leq i \leq K}$ and a  parameterized tree $ \mathscr{T} = (X, \Gamma) $.
\end{prop}

\begin{proof}
  First we to construct $\Gamma$ from the collection of curves $\{ \gamma^{ x(i), \tilde{y} (i) }  \}_{1 \leq i \leq K}$.
  Note that $ \gamma^{x(1), \tilde{y}(1)} $ is already a transient curve starting at $x (1)$.
  We construct the other elements in $\Gamma$ recursively.
  Assume that $\gamma^{x(1)}, \ldots \gamma^{x(i-1)}$ have been defined and satisfy conditions \eqref{condition:conected} and \eqref{condition:simple} in the definition of parameterized tree.
  Recall that the endpoint of $\gamma^{x(i), \tilde{y}(i)}$ is $\tilde{y} (i)$, and this point intersects some $\gamma^{x(j)}$ with $j < i$.
  Then
  \[
    \gamma^{x(i)} = \gamma^{x(i), \tilde{y}(i)} \oplus \gamma^{x(j)} \vert_{ [\tilde{y}(j) , \infty ) }.
  \]
  Since the endpoint of each  $ \gamma^{x(i), \tilde{y}(i)} $ is the  first hitting point to $ \bigcup_{j = 1}^{i-1} \tr \gamma^{x(i), \tilde{y}(i)} =  \bigcup_{j = 1}^{i-1} \tr \gamma^{x(i)}  $, we have that
  $( \tr  \gamma^{x(i)} \vert_{ [x(i), \tilde{y}^i)} \cap \tr  \gamma^{x(j)} ) = \emptyset  $ for $ j < i $.
  This construction ensures that $\gamma^{x(i)}$ satisfies conditions \eqref{condition:conected} and \eqref{condition:simple}, when we compare it against curves with smaller indexes.
  We continue with this construction for $i = 2, \ldots K$ to define $\Gamma$. Therefore $\mathscr{T} = (X, \Gamma) $ is a parameterized tree.
  Finally, note that  $\tilde{y} (i)$ satisfies \eqref{eq:branch-point} and hence $\tilde{y} (i) = y(i) $, for $i = 2, \ldots, K$.
\end{proof}

\begin{prop} \label{prop:continuity}
  Let $ ( \mathscr{T}_n )_{n \in \mathbb{N}} $ be a sequence of parameterized trees with essential branches $\Gamma^{e} (\mathscr{T}_n) = \{ \gamma_n^{x_n(i), y_n(i)} \}_{1 \leq i \leq K} $.
  Assume that
  \begin{equation} \label{eq:continuity1}
    \left( \gamma_n^{x_n(i), y_n(i)} \right)_{ 1 \leq i \leq K } \rightarrow \left( \gamma^{x(i), y(i)} \right)_{1 \leq i \leq K}
  \end{equation}
  in the product topology as $n \to \infty$ and $ \{ \gamma^{x(i), y(i)} \} $ satisfy the conditions in Proposition \ref{prop:essentialbranch-const}.
  Then $ ( \mathscr{T}_n )_{n \in \mathbb{N}} $ converges in the metric space $\mathscr{F}^K$ to a parameterized tree $\mathscr{T}$ for which $\Gamma^{e} (\mathscr{T}) = \{ \gamma^{x(i), y(i)} \}_{0 \leq i \leq K}$ is a set of essential branches.
\end{prop}

\begin{proof}
  Convergence of $ ( \gamma_n^{x_n(i), y_n(i)}  )_{ 0 \leq i \leq K } $ in the product topology implies that each element in $\Gamma_n^{e}$ converges.
  Proposition \ref{prop:essentialbranch-const} shows that every curve $\gamma_n^{x_n(i)}$ is the concatenation of sub-curves of $\Gamma^e$.
  Moreover, $\{ \gamma^{x(i), y(i)} \}_{ i = 1 \ldots K}$ satisfy the conditions in Proposition \ref{prop:essentialbranch-const} and hence they define a parameterized tree $\mathscr{T}$ with $\Gamma = ( \gamma^{x(i)} )$.
  Finally, \eqref{eq:continuity1} implies the convergence of the branching points $y_n (i)$, and from here we get convergence of the merging times.
  Then, Proposition \ref{prop:continuity-psi} and Proposition \ref{prop:continuity-chi} imply that $ \chi (\gamma_n^{x_n(i)}, \gamma^{x(i)}) \to 0 $ as $n \to \infty$.
  Therefore $ d_{\mathscr{F}^K} (\mathscr{T}_n, \mathscr{T}) \to 0 $ as $n \to \infty$.
\end{proof}

\subsection{Proof of Theorem \ref{thm:convergecetrees}} \label{subsec:essential-branches-UST}

The proof of Theorem \ref{thm:convergecetrees} is by mathematical induction.
The convergence in the scaling limit of the ILERW provides the base case.
We state the inductive step in Proposition~\ref{prop:extension}.

\begin{prop} \label{prop:extension}
  Let $\mathcal{U}_n$ be the uniform spanning tree on $2^{-n} \mathbb{Z}^3$.
  Let $(x_n (i) )_{i = 1, \ldots , K+1} $ be a set of vertices in $2^{-n} \mathbb{Z}^3$  and assume that $ x_n (i)$ converges to $x (i) \in \mathbb{R}^3$ as $n \to \infty$.
  Let $\bar{\gamma}_n^{x_n(i)}$ be the $\beta$-parameterization of the transient path in $\mathcal{U}_n$ starting at $x_n(i)$ and directed towards infinity.
  Assume that the law of $ ( \bar{\gamma}_n^{x_n(i)} )_{i = 1, \ldots, K} $ converges weakly, as measures on parameterized trees, to that of $\hat{\mathscr{S}}^K$. Then the law of $ (\bar{\gamma}_n^{x_n(i)})_{i= 1, \ldots, K+1} $ converges weakly, as measures on parameterized trees, to that of $ \hat{\mathscr{S}}^{K+1} $, with respect to the metric $\mathscr{F}^{K+1}$ for parameterized trees.
\end{prop}

We devote the rest of this section to the proof of Proposition~\ref{prop:extension}. It is based on  Proposition~\ref{prop:continuity}.
According to the latter proposition, it suffices to prove convergence of the essential branches with respect to the product topology.
We then shift our attention to the essential branches of an infinite subtree of the uniform spanning tree.
Wilson's algorithm provides a natural construction of them; and we present it below.
Subsection~\ref{subsec:prooflemmaext} develops the arguments for the proof of Proposition~\ref{prop:extension}.

Let $\mathcal{U}_n$ be the uniform spanning tree on $2^{-n} \mathbb{Z}^3$.
Let $ x_n (i) \in \mathbb{Z}^3 $ and $x(i) \in \mathbb{R}^3$ as in the statement of Proposition~\ref{prop:extension}, so $ x_{n} (i) \to x(i) $ as $n \to \infty$ for $ i = 1, \ldots, K+1$. Now we apply Wilson's algorithm on the scaled lattice $2^{-n} \mathbb{Z}^3$.
\begin{itemize}
  \item
  Let $\gamma_{n}^1$ be an ILERW starting at $x_{n} (1)$, and
\[\bar{\gamma}_{n}^{x(1), y(1)}(t) =  \gamma_{n}^{1} ( 2^{\beta n} t ), \qquad \forall t \geq 0.\]
  be its $\beta$-parameterization.
  Note that we omit the subscript $n$ on $x(1)$ and $y(1)$ to ease the notation.
  This transient curve is the first branch of the parameterized tree.
  \item
  Let $ \gamma_{n}^i$ be the loop-erased random walk started at $x_{n} (i)$, and stopped when it hits any of the previous loop-erased random walks $\gamma_{n}^1, \ldots , \gamma_{n}^{i-1}$.
  Let $y_{n} (i) \in 2^{-n} \mathbb{Z}^3 $ be the hitting point, and set
  \[\bar{\gamma}_{n}^{x(i), y(i)} = \gamma_{n}^i (2^{\beta n} t), \qquad \forall t\in[0,
 2^{-\beta n} \len(\gamma_{n}^i)].\]
  If $x_{n} (i) \in \cup_1^{i-1} \tr \gamma_{n}^j$, then $\bar{\gamma}_{n}^{x(i), y(i)}$ is a curve with length zero and  $\bar{\gamma}_{n}^{x(i), y(i)} (0) = x_{n} (i)$.
  The duration of the curve $\bar{\gamma}_{n}^{x(i), y(i)}$ is $2^{- \beta n}\len( \gamma_{n}^i)$, i.e. the length of the path $ \gamma_{n}^i $ with the appropriate scaling.
  We also omit the subscript $n$ on $x (i)$ and $y(i)$ when they appear as superscripts on the curve $\bar{\gamma}_{n}^{x(i), y(i)}$.
\end{itemize}

Set
$X_{n} = \{  x_{n} (1), \ldots , x_{n} (K)  \}$ and $\Gamma^e_{n} = \{ \bar{\gamma}_{n}^{x(i), y(i)}:\:i=1,\dots,K  \}$.
By Proposition \ref{prop:essentialbranch-const}, $X_{n}$ and  $\Gamma^e_{n}$ determine a parameterized tree $\mathscr{S}^K_{n}$, and Wilson's algorithm  shows that $\tr \mathscr{S}^K_{n}$ is equal in distribution to the subtree of $\mathcal{U}_{n}$ spanned by $X_{n}$ and the point at infinity.

As part of the proof of Theorem \ref{thm:convergecetrees}, we will show that the limit of parameterized trees $\hat{\mathscr{S}}^{K}$ has the following (formal) representation (see Lemma~\ref{lemma:Wilson-curves}). The next construction is Wilson's algorithm, but in this case, the branches have the distribution of the scaling limit of the ILERW.
\begin{itemize}
\item Let $ \hat{\gamma}^{x(1), y(1)} \in \mathcal{C} $ be the scaling limit of ILERW starting at $x(1)$, endowed with the natural parameterization, see \cite{LS}.
\item
Let $\hat{\gamma}^{x(i), y(i)} \in \mathcal{C}_f$ be the scaling limit of LERW started at $x(i)$, and stopped when it hits any of $\hat{\gamma}^{x(1), y(1)}, \ldots , \hat{\gamma}^{x (i-1), y(i-1)} $. (Our construction will give that this hitting time is finite, see Lemma~\ref{lemma:finallimit}.) Here we denote the hitting point by $y(i)$.
Similarly to the discrete case, if $x (i) \in \cup_{1}^{i-1} \tr \hat{\gamma}^{x(j), y(j)}$, then  $\hat{\gamma}^{x(i), y(i)}$  is a curve with length zero and   $\hat{\gamma}^{x(i), y(i)}(0) = x (i)$.
\end{itemize}
Set $X = \{ x(1), \ldots , x(K) \}$  and $\hat{\Gamma}^e  =  \{ \hat{\gamma}^{x(i), y(i)} \}_{1 \leq i \leq K} $.
Proposition~\ref{prop:essentialbranch-const} defines the parameterized tree $\hat{\mathscr{S}}^K = (X, \hat{\Gamma})$ with set of essential branches
$\Gamma^e (\mathscr{S}^K) =  \hat{\Gamma}^e$.

\subsection{Proof of Proposition \ref{prop:extension}}
\label{subsec:prooflemmaext}

The next proposition allows us to work with restrictions of parameterized trees.
We compare traces of parameterized trees  within a smaller subset with the symmetric difference $\triangle$.

\begin{prop} \label{prop:ust-equal-larger}
  Let $\mathscr{S}^K_{\delta}$ be a parameterized subtree of the uniform spanning tree on $\delta\mathbb{Z}^3$ with $K$ spanning points.
  Assume that $\vert x \vert < m$ for each $x \in K$.
  Then for $ r > s  \geq m^2 > 0$,
  \begin{equation} \label{eq:returntree}
    \mathbf{P} \left( (\tr \mathscr{S}^K_{\delta} \vert^{r} \triangle \tr \mathscr{S}^K_{\delta}\vert^{s})\cap B_E (m) \neq \emptyset  \right)
    \leq
    K \delta m^{-1}   [ 1 + O (m^{-1}) ].
  \end{equation}
\end{prop}

\begin{proof}{}
  The restrictions $\mathscr{S}^K_{\delta}  \vert^{r}$ and $\mathscr{S}^K_{\delta} \vert^s$ are different inside $B_E (m)$ when a path returns to $B_E (m)$ after its first exit from $B_E (r)$; we refer to Figure~\ref{fig:trees-differ} as an example of this situation.
  By virtue of Wilson's algorithm and a union bound, the probability in \eqref{eq:returntree} is bounded above by the probability of return to $B (m \delta^{-1})$ of $K$ simple random walks on $\mathbb{Z}^3$:
  \[
    K \sup_{x \in \partial B(s\delta^{-1})} P_S^x \left(  \tau_S (B(m\delta^{-1})) < \infty \right),
  \]
  where $P_S^x$ indicates the probability measure of a simple random walk on $\mathbb{Z}^3$ started at $x$, and $\tau_S (B(m\delta^{-1}))$ is the first time that the random walk $S$ hits the ball $B(m\delta^{-1})$.
  Therefore, the upper bound in \eqref{eq:returntree} follows from well-known estimates on the return probability for the simple random walk, see e.g. \cite[Proposition 6.4.2]{LawlerLimic}.
\end{proof}

Then we see an immediate consequence of the hypothesis in Proposition~\ref{prop:extension}.
The following lemma is straightforward since $x(K+1)$ is a fixed point.

\begin{lemma} \label{lemma:HausdorffDistance}
  Fix $r > 0$.
  Under the assumptions in Proposition~\ref{prop:extension},
  the law of the subsets $(\tr \mathscr{S}_n^{K}\vert^r)_{n}  $ converge weakly to the law of $\tr \hat{\mathscr{S}}^{K}\vert^r $, as $n \to \infty$, with respect to the Hausdorff distance.
  Moreover, with probability one, $x(K+1) \notin \tr   \hat{\mathscr{S}}^{K}$.
\end{lemma}

\begin{proof}
  The weak convergence of $ ( \bar{\gamma}_n^{x_n(i)} )_{i = 1, \ldots, K} $ as measures on parameterized trees implies that the laws of
   $ (  \tr \bar{\gamma}_n^{x_n(i)} )_{i = 1, \ldots, K} $ converge weakly to $ ( \tr \hat{\gamma}^{x_n(i)} )_{i = 1, \ldots , K}$ in the Hausdorff distance, as $n \to \infty$.
   Since each $\tr \mathscr{S}_n^{K}\vert^r$ is the union of $\tr \bar{\gamma}_n^{x_n(i)}$ over $i = 1, \ldots , K$,
    it then follows that $(\tr \mathscr{S}_n^{K}\vert^r)_{n}  $
   converge in law to  $\tr \hat{\mathscr{S}}^{K}\vert^r $  as $n \to \infty$, with respect to the Hausdorff distance.

   Estimates on the one-point function of the LERW imply that $ P \left( x_n (K+1) \notin \tr \mathscr{S}_n^{K}\vert^r \right) =  O( 2^{- \frac{4}{3}n})  $ as $n \to \infty$ \cite{LLERW,Escape}. In the limit, we get that, with probability one, $ x (K+1) \notin \tr \hat{\mathscr{S}}^{K}\vert^r $. Taking the union of these almost sure event for all $r \in \mathbb{N}$ gives the desired result.
\end{proof}

The proof of Proposition \ref{prop:extension} is divided into a sequence of lemmas, and these are grouped into five steps. The final and sixth step finishes the proof.

\subsubsection{Step 1: set-up.}

We begin with the set-up of the proof. First note that the assumptions of Proposition \ref{prop:extension} indicate that $(\bar{\gamma}_n^{x(i)})_{1 \leq i\leq K}$ converges in distribution. From now on, we work in the coupling given by the Skorohod representation theorem where $(\bar{\gamma}_n^{x(i)})_{1 \leq i\leq K}$ converges to a collection of continuous curves $(\hat{\gamma}^{x(i)})_{ 1 \leq i\leq K}$, almost surely.
In this case, note that the a.s. convergence of $ \mathscr{S}_n^{K,R} $ implies the a.s. convergence of $ \tr \mathscr{S}_n^{K,R} $ to $ \tr \hat{\mathscr{S}}^{K,R} $ in the Hausdorff topology (as in Lemma~\ref{lemma:HausdorffDistance}).

Let $S_n = (S_n(t))_{t \in \mathbb{N}}$ be an independent random walk on ${\delta}_n \mathbb{Z}^3$ starting at $x_n (K+1) $. Consider the hitting time of the parameterized tree $\mathscr{S}^K_n$, as given by
\[\xi^{\mathscr{S} }_n = \inf \left\{ t \geq 0 : S_n (t) \cap  \tr  \mathscr{S}^K_n  \neq \emptyset \right\}.\]
We let $ \gamma_n =  \LE  \left( S_n [0, \xi^{\mathscr{S}}] \right) $ be the corresponding LERW from $x_n (K+1)$ to  $\mathscr{S}^K_n$, and set
\[  \bar{\gamma}_n (t) = \gamma_n ( 2^{\beta n} t ), \qquad\forall t\in[0, 2^{-\beta n }\len (\gamma_n) ].\]
We want to show that $\bar{\gamma}_n$ converges to a scaling limit. Since the domain $\mathbb{Z}^3 \setminus \cup_{i = 1}^K \tr  \bar{\gamma}^{x(i)}  $ does not have a polyhedral boundary, we cannot use \cite[Theorem 1.3]{LS} directly. To get around this obstacle, we approximate with a simpler domain. Furthermore, to gain some control over the paths of the loop-erased random walks $(\bar{\gamma}_n^{x(i)})_{1\leq i\leq K}$, we also need to work within a bounded domain.

We write $D_n(R) = D_{2^{-n}} (R) $ to denote a scaled discrete box with side length $R \geq 1$ around the origin. Since the points $ x(1), \ldots, x (K+1)$ are fixed, we can take $R$ large enough so that $x(1), \ldots,  x (K+1) \in D_n(R)$.
For each curve $\bar{\gamma}_n^{x(i)} \in \Gamma$ and for the parameterized tree $\mathscr{S}^K_n$, we denote its restriction to the closed box $\bar{D_n}(R)$ with a superscript, e.g. $\gamma_n^{x(i),R}$ and $\mathscr{S}_n^{K,R}$, respectively. We also consider the ILERW in the domain $D_n(R) \setminus \tr \mathscr{S}_n^{K,R} $. The exit time from such domain for the random walk $S_n$ is
 \[\xi^{\mathscr{S}, R }_n = \inf \left\{ t \geq 0 : S_n (t) \cap  \left(  \partial D_n(R) \cup \tr  \mathscr{S}^{K,R}_n \right)  \neq \emptyset \right\}.\]
The curve $ \gamma_n^R  = \LE  \left( S_n [0, \xi^{\mathscr{S},R}] \right) $ is the LERW from $x_n (K+1)$ to either $\mathscr{S}^{K,R}_n$ or the boundary of $D_n(R)$; and we set
\begin{equation} \label{eq:gamma2}
  \bar{\gamma}_n^R (t) = \gamma_n^R ( 2^{\beta n} t ), \qquad\forall t\in[0,  2^{-\beta n }\len ( \gamma^R_n) ].
\end{equation}
Note that we omit $x_n (K+1)$ as a superscript of $\gamma_n^R$ to simplify the notation.
We emphasize that $\bar{\gamma}_n^R$ is not necessarily the same as $ \bar{\gamma}_n \vert^R $, where the latter is the restriction of the ILERW to the box $D_n(R)$,

For each integer $u$, the cubes $C^{(u)}$ at scale $u$ are closed cubes with vertices in $2^{-u} (\mathbb{Z}^3  + \frac{1}{2})$ and side length $2^{-u}$ (the translation by $1/2$ is relevant when $u = n$ since it allows us to guarantee that $ \bar{\gamma}_n^R =  \bar{\gamma}_n^{n,R} $).
For each $n \in \mathbb{N}$, let $d_n = d_H ( \mathscr{S}_n^{K,R},\hat{\mathscr{S}}^{K,R}  )$.
Then, for each $u \leq n$, we define the sets of cubes at scale $u$:
\begin{equation} \label{eq:dyadicCollection}
  \mathcal{C}_n^{u,R} = \left\{ C^{(u)} \: \colon \:  \inf \left\{   \vert x - y \vert \: \colon \: x \in C^{(u)}, y \in \tr \mathscr{S}_n^{K,R} \right\} \leq   \min\{ 2^{-2u}, 2d_n \} \right\},
\end{equation}
and for the limit tree we consider the collection
 $\mathcal{C}^{u,R}  = \{ C^{(u)} \: \colon \:   C^{(u)} \cap \tr \hat{\mathscr{S}}^{K,R}  \neq \emptyset  \}$.
We define $A_n^{u,R}$ and $A^{u,R}$ as the $\mathbf{u}$\textbf{-dyadic approximation} to $\tr \mathscr{S}_n^{K,R}$ and $\tr \hat{\mathscr{S}}_n^{K,R}$, respectively, defined as
\begin{equation} \label{eq:dyadicP}
  A_n^{u,R} := \bigcup_{\mathcal{C}_n^{u,R}}  C^{(u)} , \qquad  A^{u,R} := \bigcup_{\mathcal{C}^{u,R}}  C^{(u)}
\end{equation}

\begin{prop} \label{prop:aCauchy}
If $n \in \mathbb{N}$ is fixed, then $ \tr \mathscr{S}_n^{K,R} \subset  A_n^{u,R}$ and $d_H ( A_n^{u,R}, \tr \mathscr{S}_n^{K,R} )  \leq 2^{- (u-2)}$ for $n$ large enough.
If $u \in \mathbb{N}$ fixed then, with probability one, the sequence $  ( A_n^{u,R} )_{n \in \mathbb{N}} $ is eventually constant with limit $A^{u,R}$.

\end{prop}

\begin{proof}
    The definition of $\mathcal{C}_n^{u,R}$ on \eqref{eq:dyadicCollection} implies that $\mathscr{S}_n^{K,R} \subset A_n^{u,R}$.
    Then the bound on the Hausdorff distance follows immediately from the diameter of each cube $C^{(u)}$ in the collection $\mathcal{C}_n^{u,R}$.

    Now we fix $u \in \mathbb{N}$ and let $n \geq u$.
    The convergence of  $ \{  \mathscr{S}_n^{K,R} \}_{n \in \mathbb{N}} $ implies that $d_n \to 0$ (with $d_n$ as defined above), and hence, there exists $N_0 \geq 0$ such that
    the limit tree $ \hat{\mathscr{S}} \subset  A_n^{u,R} $ for all $n \geq N_0$. It follows that $A^{u,R} \subset A_n^{u,R}$ (equivalently $\mathcal{C}^{u,R} \subset \mathcal{C}_n^{u,R} $)for all $n \geq N_0$. Now we take $C_1^{(u)} \notin \mathcal{C}^{u,R} $. Since both sets are closed,
    then $\delta_1 = \inf \left\{   \vert x - y \vert \: \colon \: x \in C_1^{(u)}, y \in \tr \hat{\mathscr{S}}^{K,R} \right\} $   is positive.
    There exists $N_1 \geq 0$ such that $ d_n = d_H ( \mathscr{S}_n^{K,R}, \hat{\mathscr{S}}^{K,R} ) < \delta_1 / 4 $.
    It is straightforward to verify that $C_1^{u} \notin \mathcal{C}_n^{K,R}$ for all $n \geq  N_1$.
    We finish with the observation that, since $u$ is a fixed number, the total amount of cubes at scale $u$ within the box $D_n(R)$ is finite. So the previous argument shows the existence of some $N \geq 0$ after which  the collections $\mathcal{C}^{u,R} $  and $ \mathcal{C}_n^{u,R}$ are equal, and in particular, $A_n^{u,R} = A_m^{u,R} = A^{u,R}$ for all $n,  m \geq N$.
\end{proof}

If $x_n (K+1) \in \tr \mathscr{S}_n^{K,R}$ then we set $\bar{\gamma}^{u, R}_n$ as the curve with length zero and constant value $x_n(K+1)$ for all $u \leq n$.
On the event where $x_n (K+1) \notin \tr \mathscr{S}^{K,R}$, Proposition~\ref{prop:aCauchy} implies that  $x_n (K+1) \in D_n (R) \setminus A_n^{u,R}$ for $n$ and $u$ large enough.
Lemma~\ref{lemma:HausdorffDistance} implies that, with high probability on $n$, $x_n (K+1) \notin \tr \mathscr{S}_n^{K,R}$, so the latter case is the relevant one.
Let $S_n$ be the loop-erased random walk started from $x_n (K+1)$, and stopped when on its first entrance to  $A_n^{u,R}$ (w.r.t the lattice $2^{-n} \mathbb{Z}^3$). We denote the latter hitting time by $\xi^{\mathcal{P}, R}_u$, and the corresponding LERW by
\begin{equation} \label{def:gammaEscaledPoly} \gamma_{n}^{u,R}    = \LE (S_n [0, \xi^{\mathcal{P}, R}_u]).
\end{equation}
The superscript refers to the polyhedron $\mathcal{P} = D_n(R) \setminus A_n^{u,R}$. This curve has $\beta$-parameterization
\begin{equation} \label{eq:defGamma}
\bar{\gamma}_n^{u,R} (t) := \gamma_n^{u,R} (2^{\beta n} t), \qquad \forall t\in[0,2^{-\beta n}\len(\gamma_n^{u,R})].
\end{equation}
The weak convergence of $\eqref{eq:defGamma}$ is an immediate consequence of \cite[Theorem 1.4]{LS} and Proposition \ref{prop:aCauchy}. We state this observation below as Lemma \ref{lemma:convDyadicSet}.

\subsubsection{Step 2: comparing $\bar{\gamma}_n^{u,R}$ and $\bar{\gamma}_n^{v,R}$}

\begin{figure}
\begin{center}
  \includegraphics{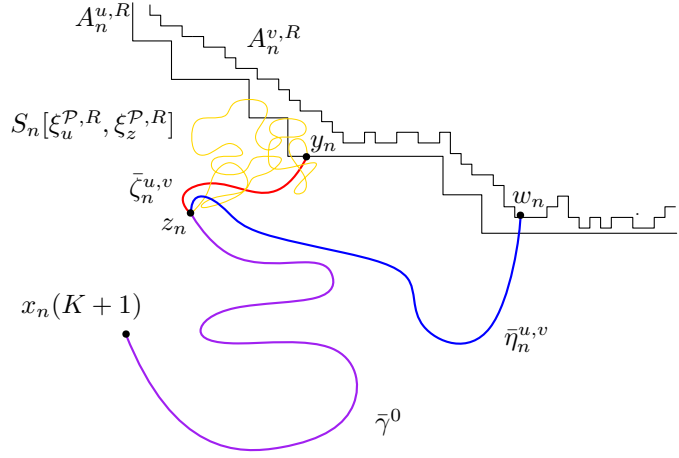}
\end{center}
\caption{The figure shows the decomposition of the curves $\bar{\gamma}_n^{u,R}$ and $\bar{\gamma}_n^{v,R}$ used in the proof of Proposition \ref{prop:extension}.
The curve $\bar{\gamma}_n^{u,R}$ is the concatenation of $\bar{\gamma}_0$ (in purple) and $\bar{\zeta}_n^{u,v}$ (in red). The curve $\bar{\gamma}_n^{v,R}$ is the concatenation of $\bar{\gamma}_0$ and $\bar{\eta}_n^{u,v}$ (in blue). The figure also shows a restriction of the random walk $S_n$ from $y_n$ to $z_n$ (in yellow). In this case, $S_n$ avoids hitting $A_n^{v,R}$ when it is close to $y_n$.}\label{fig:diagramAuv}
\end{figure}

Recall that $n$ is a fixed integer and $u < v$.
In this section we are interested in the event where $x_n (K+1) \notin \tr \mathscr{S}^{K,R}$, otherwise the sequence $\left( \bar{\gamma}^{u, R}_n \right)_{u \leq n}$ is constant.
Then we consider integers $u$ and $v$ with $u < v \leq n$ and $u$ and $v$ large enough so that $ x_n (K+1) \in D_n (R) \setminus A^{u,R} $ and $  x_n (K+1) \in D_n (R) \setminus A^{v,R} $.
Let $\bar{\gamma}_n^{u,R}$ and $\bar{\gamma}_n^{v,R}$ be the loop-erased random walks defined in \eqref{eq:defGamma}. Our aim is to bound the distance $ \psi ( \bar{\gamma}_n^{u,R}, \bar{\gamma}_n^{v,R} ) $ for large values of $n$, $u$ and $v$.
Let us consider the event where this distance is large. More precisely, for $\eps > 0$ and integers $ u < v \leq n$, we define
\[\mathcal{E}^{u,v}_n (\eps) := \left\{ \psi (\bar{\gamma}_n^{u,R}, \bar{\gamma}_n^{v,R}) \geq \eps   \right\},\]
where $\psi$ is the distance for finite curves defined in \eqref{eq:distpsi}.
Recall that we use the same random walk on $2^{-n} \mathbb{Z}^3$ to generate $\bar{\gamma}^{u,R}_n$ and $\bar{\gamma}^{v,R}_n$.  Typically, these two curves have a segment in common, $\bar{\gamma}^0$ (see Figure~\ref{fig:diagramAuv}). We claim that $\bar{\gamma}^{u,R}_n \setminus \bar{\gamma}^0$ and $\bar{\gamma}^{v,R}_n \setminus \bar{\gamma}^0$ have a small effect on $\psi (\bar{\gamma}^u_n, \bar{\gamma}^v_n)$.

Towards proving the preceding claim, we start by introducing some further notation for elements in the curves $\bar{\gamma}_n^{u,R}$ and $\bar{\gamma}_n^{v,R}$; Figure \ref{fig:diagramAuv} serves as a reference.
For clarity, and without loss of generality, we elaborate our arguments on the event where the random walk $S_n$ hits the boundary of $\mathcal{P}^{u,R}$ first, that is
\[  \mathcal{F}^u := \left\{ \xi^{\mathcal{P},R}_u \leq \xi^{\mathcal{P},R}_v \right\},\]
and hence it generates $\bar{\gamma}_n^{u,R}$ before $\bar{\gamma}_n^{v,R}$.
The symmetric event is $\mathcal{F}^v := \{ \xi^{\mathcal{P},R}_v \leq \xi^{\mathcal{P},R}_u \}$. We will consider the restriction of the random walk $S_n$:
\[S^{u,v}_n := S_n \vert_{[ \xi^{\mathcal{P},R}_u, \xi^{\mathcal{P},R}_v ]}.\]
Denote the endpoint of $\bar{\gamma}_n^{u,R}$ by $y_n := S_n (\xi^{\mathcal{P},R}_u)$.
To simplify notation, we denote the durations of $\bar{\gamma}_n^{u,R}$ and $\bar{\gamma}_n^{v,R}$  by
\[T^u = T (\bar{\gamma}_n^{u,R}), \qquad T^v = T (\bar{\gamma}_n^{v,R}),\]
respectively.
The last time that $S_n^{u,v}$ hits its past $\bar{\gamma}_n^{u,R}$ determines the endpoint of $\bar{\gamma}^0$.
Let
\[
  \xi^{\mathcal{P},R}_z := \sup \{ t \leq   \xi^{\mathcal{P},R}_v \colon S(t) \in \bar{\gamma}_n^{u,R} \}
\]
and set $z_n := S ( \xi^{\mathcal{P},R}_z ) $.
Let $T_{\text{z}}$ be such that $ \bar{\gamma}_n^{v,R} (T_{\text{z}}) = \bar{\gamma}_n^{v,R} (T_{\text{z}}) = z_n $.
We then define the common curve $\bar{\gamma}^0 := \bar{\gamma}_n^{u,R} [0, T_{\text{z}}] = \gamma_n^{v,R} [0, T_{\text{z}}]$. The difference between $\bar{\gamma}_n^{u,R}$ and $\bar{\gamma}_n^{v,R}$ is contained in the union of the curves
\[ \bar{\zeta}_n^{u,v} := \bar{\gamma}_n^{u,R} [ T_{\text{z}}, T^u ] , \qquad   \bar{\eta}_n^{u,v} = \bar{\gamma}_n^{v,R} [T_{\text{z}}, T^v].\]
Note that the range of $\bar{\eta}_n^{u,v}$ is a subset of $S^{u,v}_n$.

\begin{figure}
\begin{center}
  \includegraphics{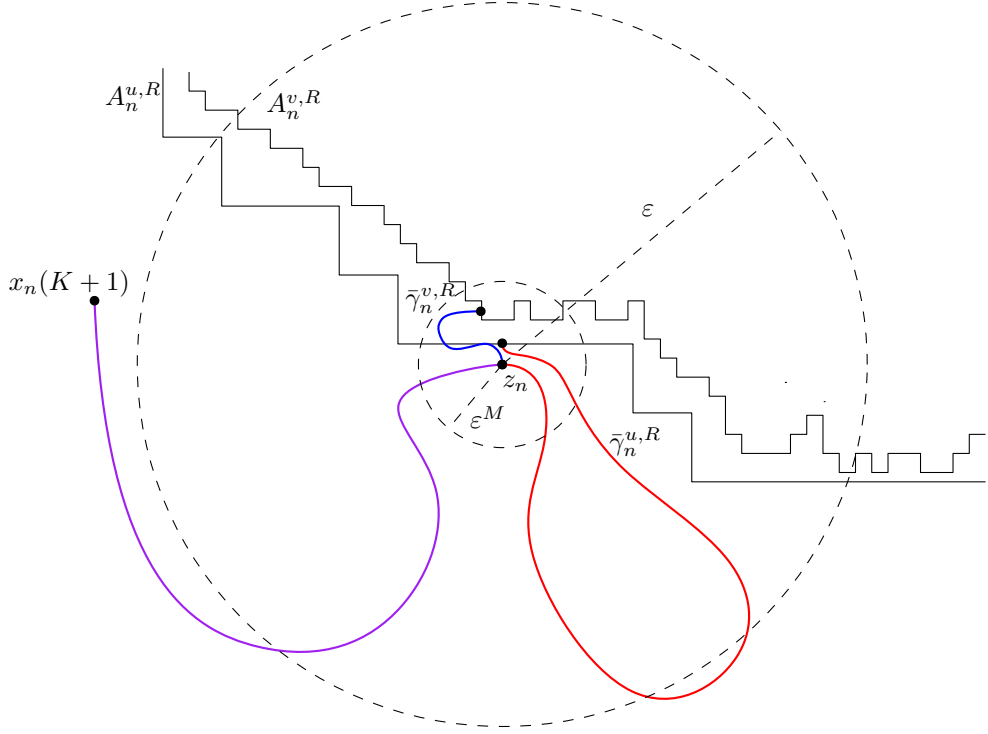}
\end{center}
\caption{A realization of the event $\mathcal{D}_n^{u,v} (\varepsilon)^c \cap \mathcal{Q} (\varepsilon^M, \varepsilon) $.
In the figure, $\bar{\gamma}_n^{u,R}$ is the concatenation of $\bar{\gamma}_0$ (in purple) and $\bar{\zeta}_n^{u,v}$ (in red). The curve $\bar{\gamma}_n^{v,R}$ is the concatenation of $\bar{\gamma}_0$ and $\bar{\eta}_n^{u,v}$ (in blue).}
\label{fig:quasiloops}
\end{figure}

We now compare the shapes of $\bar{\gamma}_n^{v,R}$ and $\bar{\gamma}_n^{u,R}$. In particular, we note that the respective traces of these curves can be significantly different if one of the next two bad events occur. The first event controls the diameter of $\bar{\eta}_n^{u,v}$, while the second event imposes a limit on the size of $\bar{\zeta}_n^{u,v}$.
\begin{itemize}
  \item Since $\bar{\eta}_n^{u,v}$ is a subset of $S_n^{u,v}$, $\bar{\eta}_n^{u,v}$ has a diameter larger than $\eps_0$ only if $S^{u,v}_n$, the segment of the random walk $S_n$ between the hitting times $\xi_u^{\mathcal{P},R}$ and $\xi_v^{\mathcal{P},R}$, has a similarly large diameter. We denote this event by
  \[\mathcal{D}_n^{u,v} (\eps_0) := \left\{ \diam (S^{u,v}_n) \geq \eps_0 \right\}.\]
  \item On the complementary event $ \mathcal{D}_n^{u,v} (\eps_0)^c  $, the curve $\bar{\zeta}_n^{u,v}$ has diameter larger than $\eps$ only if $\bar{\gamma}_n^{u,R}$ has an $(\eps_0, \eps)$--quasi-loop. Figure \ref{fig:quasiloops} shows an example of this situation.
  Let
\[\mathcal{Q} ( \eps_0, \eps; \gamma) := \{ \gamma \text{ has an } (\eps_0, \eps)\text{-quasi-loop} \},\]
and  $\mathcal{Q}_n ( \eps_0, \eps) = \mathcal{Q} ( \eps_0, \eps; \bar{\gamma}_n^u) \cup \mathcal{Q} ( \eps_0, \eps; \bar{\gamma}_n^v)$.
\end{itemize}
Combining the definitions above, we introduce a bad event for the shape by setting
\[\mathcal{B}^{u,v}_n (\eps) := \mathcal{D}_n^{u,v} (\eps^M) \cup \mathcal{Q}_n ( \eps^{M}, \eps),\]
noting that we have taken $\eps_0 = \eps^M $, with $M > 1$ being the exponent of Proposition \ref{result:quasiloops}.
We highlight that, on $(\mathcal{B}^{u,v}_n(\eps))^c$, it holds that $d_H (\bar{\gamma}^{u,R}_n, \bar{\gamma}^{v,R}_n) \leq \eps$.

The following result establishes that, on $(\mathcal{B}^{u,v}_n(\eps))^c$, $\bar{\gamma}_n^{u,R}$ and $\bar{\gamma}_n^{v,R}$ are also close as parameterized curves. The issue here is that even if the traces of two curves may be close in shape, they may take a large number of steps in a small diameter. We will compare the Schramm and intrinsic distances, as defined at \eqref{eq:Schramm} and \eqref{eq:Intrinsic}, on the event $( \mathcal{B}^{u,v}_n(\eps))^c$ where the shapes are close to each other.
The Schramm and intrinsic distances of $\bar{\gamma}^{u, R}_n$ are comparable on the events $S_{2^{-n}}^{\dag} (R, \eps)$ and $E_{2^{-n}}^{\dag} (R, \eps)$. These events are introduced in Subsection~\ref{subsec:LERWPoly}. To simplify notation, we write $S_{n}^{\dag} (R, \eps) = S_{2^{-n} }^{\dag} (R, \eps)$ and $E_{n}^{\dag} (R, \eps) = E_{2^{-n}}^{\dag} (R, \varepsilon)$.

\begin{lemma} \label{lemma:steps}
Fix $R \geq 1$ and let  $\eps \in (0,1)$. On the event $ \left( \mathcal{B}^{u,v}_n (\eps) \right)^c \cap S^{\dag}_{n}  (R \eps^{-1}, \eps) $, we have that
\[T (\bar{\eta}_n^{u,v}) \leq  R \eps^{\beta - 1},\qquad  T (\bar{\zeta}_n^{u,v}) \leq  R \eps^{\beta - 1}.\]
\end{lemma}

\begin{proof} In this proof, we write $G = \left( \mathcal{B}^{u,v}_n (\eps) \right)^c \cap S^{\dag}  (R \eps^{-1}, \eps)$. We begin with an upper bound for the duration of $\bar{\eta}_n^{u,v}$. On $G$, the random walk $S^{u,v}_n$ is localized in a neighbourhood around $y_n$. Indeed, on $\mathcal{D}_n^{u,v} (\eps^M)^c$ we have that $\diam (S^{u,v}_n) \leq \eps^{M}$. Since $\bar{\eta}_n^{u,v}$ is a subset of $S^{u,v}_n$, it follows that $\diam (\bar{\eta}_n^{u,v}) < \eps^{M}$, and, in particular, for the endpoints of $\bar{\eta}_n^{u,v}$, $z_n$ and $w_n$ say (as in Figure~\ref{fig:diagramAuv}), we have that $d_{\bar{\gamma}^v_n}^S (z_n,w_n) = d^S_{\bar{\eta}} (z_n,w_n) \leq \eps^M < \eps$. On $G \subseteq S^{\dag}_n (R \eps^{-1}, \eps) $, this implies that
\[T (\bar{\eta}_n^{u,v}) = d_{\bar{\gamma}_n^v} (z_n,w_n) \leq R  \eps^{\beta - 1}.\]
Next we bound the duration of $\bar{\zeta}_n^{u,v}$ on the event $G$. We have that the endpoints of $\bar{\zeta}_n^{u,v}$ are in $S_n^{u,v}$. Indeed, $y_n = S_n^{u,v} (0)$ and $z_n \in \bar{\zeta}_n^{u,v} \subseteq S^{u,v}_n $. Thus
\begin{equation} \label{eq:steps2}
\vert z_n - y_n \vert < \eps^{M}.
\end{equation}
On the event $G \subseteq \mathcal{Q}(\eps^{M}, \eps)^c$, the loop-erased random walk $\bar{\gamma}_n^{u,R}$ does not have $(\eps^M,\eps)$--quasi-loops, and so \eqref{eq:steps2} implies that $d_{\bar{\gamma}^u_n}^S (z_n,y_n) < \eps$. The argument used for $\bar{\eta}^{u,v}_n$ also gives $ T (\bar{\zeta}^{u,v}_n) = d_{\bar{\gamma}_n^u} (z_n,y_n) < R \eps^{\beta - 1} $.
\end{proof}

We finish this step by showing that $\mathcal{E}^{u,v}_n (\eps)$ can be contained in the events already described.

\begin{lemma} \label{lemma:decompE}
Let $ R \geq 1$ and $\eps > 0$.
On the event $ \mathcal{B}^{u,v}_n (\eps)^c \cap   S^{\dag}_n (R \eps^{-1}, \eps) \cap E^{\dag}_n (R, \eps) $, we have that
\[
  \psi (\bar{\gamma}_n^{u,R}, \bar{\gamma}_n^{v,R}) \leq C R^{b_3} \varepsilon^{\beta'},
\]
where $0 < b_3, C < \infty$ are universal constants, and $\beta' = b_3 (\beta - 1) $.
\end{lemma}

\begin{proof} It suffices to show that on the event $G_2 = \mathcal{B}_n^{u,v} (\eps)^c \cap S^{\dag} (R \eps^{-1}, \eps) \cap E^{\dag} (R, \eps)$,
\begin{equation} \label{eq:decompE1}
\psi (\bar{\gamma}_n^u, \bar{\gamma}_n^v)= \vert T^u - T^v \vert+ \max_{0 \leq s \leq 1}\vert \bar{\gamma}_n^u (s T^u ) - \bar{\gamma}_n^v (s T^v ) \vert \leq C R^{b_3} \eps^{\beta'} .
\end{equation}
Lemma \ref{lemma:steps} gives $\vert T^u - T^v \vert < 2 R \eps^{\beta - 1} $. Next we bound the second term in \eqref{eq:decompE1}.
Let $a = \bar{\gamma}_n^{u,R} (sT^u)$ and $b = \bar{\gamma}_n^{v, R} (sT^v)$ and assume that one of these points belongs to the common path, say $ b \in \bar{\gamma}^0$.
In this case, $ sT^v \leq T^u $ and we can re-write $ b = \bar{\gamma}_n^{u,R} ( ( s(T^v/T^{u}) ) T^{u}  ) $.
Then, with respect to the intrinsic metric of $\bar{\gamma}_n^{u,R}$, we compare points within distance
\[
  d_{\bar{\gamma}_n^{u}} (a,b) \leq \vert s T^u - \left( s T^v / T^{u}  \right) T^u \vert \leq 2 R \eps^{\beta - 1}.
\]
We introduce
\[N^u = \sup \left\{ \vert a - b \vert : a,b \in \tr \bar{\gamma}^{u,R}_n, d_{\bar{\gamma}_n^u} (x,y) \leq   2R \eps^{\beta - 1} \right\},\]
define $N^v$ similarly from $\bar{\gamma}^v_n$, and also introduce notation for the diameter of the segments $\bar{\eta}^{u,v}_n$ and $\bar{\zeta}^{u,v}_n$ by setting
\[N^{\eta} = \sup_{ 0 \leq t \leq \len (\bar{\eta}_n^{u,v}) } \vert \bar{\eta}^{u,v}_n (t) - z_n \vert, \qquad   N^{\zeta} = \sup_{ 0 \leq t \leq \len (\bar{\zeta}_n^{u,v})} \vert \bar{\zeta}_n^{u,v} (t) - z_n \vert.\]
It readily holds that we have the following bound:
\begin{equation} \label{eq:decompE2}
\max_{0 \leq s \leq 1}    \vert \bar{\gamma}_n^u (s T^u ) - \bar{\gamma}_n^v (s T^v ) \vert   \leq  N^u + N^v + N^{\eta} + N^{\xi}.
\end{equation}
Lemma \ref{lemma:steps} implies that $d_{\eta} (a, b) < R \eps^{\beta - 1} $ for all $a \in \tr \eta$.
Let $b_3 = \frac{b_2}{b_1}$, where $b_1$ and $b_2$ are the constants of Proposition \ref{result:boundE}.
On the event $ E^{\dag}_n (R, \varepsilon) $, we thus have that $N^{\eta} \leq R^{ b_3} \varepsilon^{ b_3 (\beta-1) }$, and similarly for $N^{\xi}$. On the event $ E^{\dag}_n (R, \eps) $, the loop-erased random walks $\bar{\gamma}_n^u$ and $\bar{\gamma}_n^v$ are uniformly equicontinuous, so that $N^u \leq C R^{ b_3 } \varepsilon^{ b_3 (\beta-1) }$, and the same bound holds for $N^v$. Adding the upper bounds for $N^u$, $N^v$ $N^{\eta}$ and $N^{\xi}$ in \eqref{eq:decompE2}, we get \eqref{eq:decompE1}.
\end{proof}

\subsubsection{Step 3: bounding $\mathbf{P} ( \mathcal{E}_n^{u,v} (\varepsilon) )$} 

In this step, we give an upper bound on the probability of the bad event $\mathcal{E}_n^{u,v} (\varepsilon)$. The key is that, given $u$, and $v$, this estimate is uniform over all $n$ for $u$ and $v$ large enough.

\begin{lemma} \label{lemma:cauchy}
  Fix $R \geq 1$.
  For each $ \varepsilon \in (0,1) $, there exists $U = U(\varepsilon)$ such that
  for all $ n \geq u, v \geq U (\varepsilon)$
  \[
    \mathbf{P} \left( \mathcal{E}_n^{u,v} (\varepsilon) \right) \leq C \varepsilon^{\theta},
  \]
  for constants $C = C(R) > 0$ and $\theta = \theta (R) > 0$, depending only on $R$.
\end{lemma}

\begin{proof}
Lemma~\ref{lemma:decompE} gives that
\begin{align}
\mathbf{P} \left( \mathcal{E}_n^{u,v} ( C R \eps^{\beta'} ) \right) &\leq
   \mathbf{P} \left( \mathcal{D}_n^{u,v} (\eps^{M}) \right)+    \mathbf{P} \left( \mathcal{Q} (\eps^M, \eps) \right)  \nonumber \\
   &\quad  +   \mathbf{P} \left( \left( S^{\dag}_n (R \eps^{-1}, \eps ) \right )^c  \right)    +   \mathbf{P} \left( \left( E^{\dag}_n (R, \eps ) \right)^c \right).\label{eq:unionBoundE}
\end{align}
Proposition~\ref{result:quasiloopsPoly} implies
\[
\mathbf{P} ( \mathcal{Q} (\eps^{M}, \eps))\leq \mathbf{P} ( \mathcal{Q} (\eps^{M}, \eps^2)) \leq C R^3 \varepsilon^{\hat{b}_2},
\]
and Propositions~\ref{result:boundEPoly} and \ref{result:boundSPoly} give upper bounds for the last two terms of \eqref{eq:unionBoundE}. Thus we are left to bound the probability of $\mathcal{D}_n^{u,v} (\eps^{M})$. For this, we need $U (\eps)$ large enough so that, by Proposition \ref{prop:aCauchy}, we have that
\begin{equation} \label{eq:closeA}
  d_H ( A^{u,R}_n, A^{v,R}_n )< \eps^{4M} \qquad \text{ for all } u, v \geq U(\eps).
\end{equation}
On $\mathcal{F}^u$, $\bar{\gamma}_n^{u,R}$ is the first walk LERW to stop, and we call its endpoint $y_n \in \partial \mathcal{P}^{u,R} $. From \eqref{eq:closeA}, we  have $\dist (y_n, \partial \mathcal{P}^{v,R} ) < \eps^{4M} $. But, along $S_n^{u,v}$, the random walk $S_n$ reaches distance $ \eps^{M} $ before hitting $\partial \mathcal{P}^{v,R}$. The same argument on the complement of $\mathcal{F}^u$, i.e. on $\mathcal{F}^v$. Hence Proposition \ref{prop:treehittable} implies that $\mathbf{P}(\mathcal{D}_n^{u,v} (\eps^{M}))\leq C K R^3 \eps^{2M} + \eps^{2M\hat{\eta}}$. In conjunction with the aforementioned bounds,
\begin{align*}
\sup_{\substack{n,u,v:\\n\geq u,v\geq U}}\mathbf{P} \left( \mathcal{E}_n^{u,v} ( R^{b_3} \eps^{\beta'}   ) \right)
    &\leq
      C K R^3 \eps^{2M} + \eps^{2M \hat{\eta} }
      +
      C \eps^{\tilde{\eta}} \nonumber \\
      &\quad + C \left( \frac{R}{\varepsilon} \right)^3 e^{-c \left( \frac{R}{\varepsilon} \right)^a}
      +
      C \eps^{b_2}.
\end{align*}
The dominant term above is $\varepsilon^{\theta (R)}$, and a reparameterization completes the proof.
\end{proof}

\subsubsection{Step 4: the scaling limit of a loop-erased random walk}

Recall that $  \bar{\gamma}_n^{R} $ is the LERW on $\bar{D}_n(R) \setminus \mathscr{S}_n^{K,R} $ defined in \eqref{eq:gamma2}.
In \eqref{eq:defGamma}, we defined $ \bar{\gamma}_n^{u,R} $, for $ u \leq n $, as the $\beta$-parameterization of the loop-erased random walk $ \LE \left(  S_n [0, \xi_m^{\mathcal{P},R}] \right) $, where $\xi_m^{\mathcal{P},R}$ is the first exit time from the dyadic polyhedron $\mathcal{P}^{m,R}$.
In this step, we establish that $\bar{\gamma}_n^{R} $ and  $ \bar{\gamma}_n^{u,R}$ converge to the same limit.
We take limits on each variable in the following order.
For $\bar{\gamma}_n^{u,R}$, we first take $n \to \infty$. The limit object is a curve on the bounded and polyhedral domain $D_E (R) \setminus A^{u,R} \subset \mathbb{R}^3$, where $A^{u,R}$ is the polyhedral domain defined in \eqref{eq:dyadicP}.
Then we take $u \to \infty$, and the limit is a curve on  the bounded set $D_E (R) \setminus \tr \mathscr{S}^{K, R} $.
In Step 5, we take $R \to \infty$, and we thus define $\hat{\gamma}$ as a limit curve on the full space $\mathbb{R}^3 \setminus \tr \mathscr{S}^K$.

\begin{lemma} \label{lemma:convDyadicSet}
Fix $R \geq 1$. For each $u \in \mathbb{N}$, the law of $ \bar{\gamma}_{n}^{u,R} $ converges with respect to the metric $ \psi $, as $n \to \infty$.
\end{lemma}

\begin{proof}
  Proposition \ref{prop:aCauchy} shows that $\bar{\gamma}_n^{u,R}$ is in the polyhedron $D_E (R) \setminus A^{u,R}$, for $n$ large enough.
  Then, the weak convergence of $\{ \bar{\gamma}_{n}^{u,R} \}_{n \in \mathbb{N}}$ is an immediate consequence of Proposition~\ref{result:scalingLimitPoly}.
\end{proof}

We denote by $\hat{\gamma}^{u,R}$ a curve with the limit law of Lemma~\ref{lemma:convDyadicSet}.

\begin{lemma}  \label{lemma:convBox}
Fix $R \geq  1$. Let $(\hat{\gamma}^{u,R})_{u \in  \mathbb{N} }$ be the sequence of limit elements from Lemma \ref{lemma:convDyadicSet}. It is then the case that $(\hat{\gamma}^{u,R})_{u \in  \mathbb{N} }$ converges in distribution in the metric $\psi$ as $u \to \infty$.
\end{lemma}

\begin{proof}
Denote the laws of $\bar{\gamma}^{u,R}_n$ and $ \hat{\gamma}^{u,R} $ by $\mathcal{L} (\bar{\gamma}_n^{u,R})$ and $\mathcal{L} (\hat{\gamma}^{u,R})$, respectively. Since $(\mathcal{C}_f, \psi)$ is a complete and separable metric space (see \cite[Section 2.4]{KS}), to prove weak convergence it suffices to show that $( \mathcal{L} ( \hat{\gamma}^{u,R}))_{u \in \mathbb{N}}$ is a Cauchy sequence in the Prohorov metric $d_\mathbf{P}$. Let $u,v \in \mathbb{N} $. By the triangle inequality, for $n \geq u,v$,
\begin{align}
d_\mathbf{P} (\mathcal{L}(\hat{\gamma}^{u,R}),\mathcal{L}(\hat{\gamma}^{v,R}) )
& \leq
d_\mathbf{P} ( \mathcal{L}( \hat{\gamma}^{u,R}), \mathcal{L} (\bar{\gamma}_n^{u,R})) +
d_\mathbf{P} (\mathcal{L}(\hat{\gamma}^{v,R}), \mathcal{L}(\bar{\gamma}_n^{v,R})) \nonumber \\
& \quad + \sup_{n \geq u,v} d_\mathbf{P} (   \mathcal{L}(\bar{\gamma}_n^{u,R}), \mathcal{L}(\bar{\gamma}_n^{v,R})). \label{eq:convergencegamma1}
\end{align}
Letting $n \to \infty$, the first two terms on the right hand side of \eqref{eq:convergencegamma1} converge to $0$ by Lemma \ref{lemma:convDyadicSet}. Moreover, Lemma \ref{lemma:cauchy} shows that the last term of \eqref{eq:convergencegamma1} converges to 0 as $u,v \to \infty $. Therefore $(\mathcal{L}( \hat{\gamma}^{u,R}))_{u \in \mathbb{N} }$ is a Cauchy sequence in the Prohorov metric. It follows that $(\hat{\gamma}^{u,R})_{u \in \mathbb{N}  }$ converges weakly.
\end{proof}

We denote by $\hat{\gamma}^R$ a curve with the limit law of Lemma~\ref{lemma:convBox}.
The random curve $\hat{\gamma}^R$ is the limit of dyadic approximations.
We see below that it is also the limit of the LERWs stopped when they hit $\mathscr{S}^{K,R}$.

\begin{lemma} \label{lemma:convLERW}
Fix $R \geq 1$. Then
$\bar{\gamma}_n^R \to \hat{\gamma}^R $ in distribution as $n \to \infty$, with respect to the metric $\psi$.
\end{lemma}

\begin{proof}
  Since $\bar{\gamma}_n^{u,R}\rightarrow \hat{\gamma}^{u,R}$ in distribution as $n\rightarrow\infty$, and $\bar{\gamma}^{u,R}\rightarrow\hat{\gamma}^{R}$ in distribution as  $u\rightarrow\infty$, to complete the proof it suffices to notice that, for $\varepsilon>0$,
  \[ \lim_{u\rightarrow\infty}\limsup_{n\rightarrow\infty}\mathbf{P}\left(\psi ( \bar{\gamma}_n^{u,R}, \bar{\gamma}_n^R )>\varepsilon\right)=0, \]
  see \cite[Theorem 3.2]{Billingsley}, for example.
  However, since $\bar{\gamma}_n^R =  \bar{\gamma}_n^{n,R} $, the above statement readily follows from Lemma \ref{lemma:cauchy}.
\end{proof}

\subsubsection{Step 5: taking $R \to \infty$}

Until this point, we have only considered loop-erased random walks inside a box $D_E (R)$. Indeed, $\bar{\gamma}_n^R$ is LERW in $D_n(R) \setminus \mathscr{S}^K_n$, and its scaling limit $\hat{\gamma}^R$ is within $D_E (R)$.
In this final step, we will take $R \to \infty$ to consider the tree $\hat{\mathscr{S}}^K$ and the random walk $S_n$ in the full space.

\begin{lemma}  \label{lemma:finallimit}
  Let $\left( \hat{\gamma}^R \right)_{R \geq 1}$ be the sequence of limit elements from Lemma \ref{lemma:convBox} and $\hat{\mathscr{S}}^K$ is the parameterized tree in Proposition \ref{prop:extension}. There exists a random element $\hat{\gamma} \in \mathcal{C}_f$ such that $\hat{\gamma}^R$ converges in distribution to $\hat{\gamma}$ in the metric $\psi$ as $ R \to \infty $. Moreover, the intersection of $\tr \hat{\gamma}$ and $\tr \hat{\mathscr{S}}^K $ is the endpoint of $\hat{\gamma}$.
\end{lemma}

\begin{proof}
  Denote the laws of $ \bar{\gamma}_n^R $ and $\hat{\gamma}^R$ by $\mathcal{L} ( \bar{\gamma}_n^R )$ and $ \mathcal{L} (\hat{\gamma}^R) $, respectively.
  This proof is similar to the one for Lemma \ref{lemma:convBox} as we will show that $ \left( \mathcal{L} (\hat{\gamma}^R) \right)_{R \geq 1} $ is a Cauchy sequence in the Prohorov metric $d_{\mathbf{P}}$.
  For two integers $ r  >   s > 0 $, the triangle inequality yields
  \begin{equation} \label{eq:gammardif1}
    d_{\mathbf{P}} \left( \mathcal{L} ( \hat{\gamma}^r, \hat{\gamma}^s ) \right)
    \leq
    d_{\mathbf{P}} \left( \mathcal{L} ( \hat{\gamma}^r, \bar{\gamma}^r_n ) \right)
    +
    d_{\mathbf{P}} \left( \mathcal{L} ( \hat{\gamma}^s, \bar{\gamma}^s_n ) \right)
    +
    \sup_n  d_{\mathbf{P}} \left( \mathcal{L} ( \bar{\gamma}^r_n, \bar{\gamma}^s_n ) \right).
  \end{equation}
  Letting $n \to \infty$, the first two terms on the right hand side of \eqref{eq:gammardif1} converge to 0 by Lemma \ref{lemma:convLERW}.
  Then we are left to bound $ \sup_n  d_{\mathbf{P}} \left( \mathcal{L} ( \bar{\gamma}^r_n, \bar{\gamma}^s_n ) \right) $.
  \begin{equation} \label{eq:gammardif2}
    d_{\mathbf{P}} \left( \mathcal{L} ( \hat{\gamma}^r, \hat{\gamma}^s ) \right)
    \leq
    \sup_n  d_{\mathbf{P}} \left( \mathcal{L} ( \bar{\gamma}^r_n, \bar{\gamma}^s_n ) \right).
  \end{equation}
  Recall that we sample $\bar{\gamma}^r_n$ and $\bar{\gamma}^r_n$ as loop-erasures of the simple random walk $S_n$.
  On the event that
  $ \mathscr{S}_n^{K,r} \triangle \mathscr{S}_n^{K, s}  \cap B(s^{1/2}) = \emptyset $,
   $\bar{\gamma}^r_n =\bar{\gamma}^r_n$ (as parameterized curves) whenever $S_n$ hits $\tr \mathscr{S}_n^{K,r} $ before reaching the boundary of $B_n(s^{1/2})$, and so
  \[
  \begin{aligned}
    \mathbf{P} \left( \bar{\gamma}^r_n \neq \bar{\gamma}^s_n \right)  \leq & \mathbf{P} \left( (\mathscr{S}^{K,r}_{n} \triangle  \mathscr{S}^{K,s}_{n})\cap B_n (s^{1/2}) \neq \emptyset  \right)  \\
    &
    + \mathbf{P} \left(  S_n[  0  , \xi_S (B_n(s^{1/2} )) ] \cap \tr \mathscr{S}_n^{K,r}  \neq \emptyset  \right).
  \end{aligned}
  \]
  Proposition~\ref{prop:ust-equal-larger} gives $\mathbf{P} \left( (\mathscr{S}^{K,r}_{n} \triangle  \mathscr{S}^{K,s}_{n})\cap B_n (s^{1/2}) \neq \emptyset  \right) \to 0$ as $s \to \infty$.
  Recall that $\mathscr{S}_n^K$ is a subtree of the uniform spanning tree, including a path to infinity. Then, Proposition \ref{prop:treehittable} implies that
  \[  \mathbf{P} \left(  S[  0  , \xi_S (B_n(s^{1/2} )) ] \cap \tr \mathscr{S}_n^{K,r}  \neq \emptyset  \right)  \to 0  \text{ as } r, s \to \infty .\]
  Therefore, \eqref{eq:gammardif2} converges to $0$ as $r ,  s  \to \infty$. It follows that $ \left( \mathcal{L} (\hat{\gamma}^R) \right)_{R \geq 1} $ is a Cauchy sequence in the Prohorov metric.
  Since $d_{\mathbf{P}}$ is a complete metric, we conclude that $ \left( \mathcal{L} (\hat{\gamma}^R) \right) $ converges weakly. Such limit is a random element $\hat{\gamma}$ taking values in $\mathcal{C}_f$, and in particular $\hat{\gamma}$ has finite duration.

  On the space of finite curves $ (\mathcal{C}_f, \psi)  $, the evaluation of the endpoint defines a continuous function $E : \mathcal{C}_f \to \mathbb{R}^3$.
  Therefore, as we take $n \to \infty$, the endpoint of $E (\bar{\gamma}_n^R) \in \tr \mathscr{S}_n^K$ converges to $ E (\hat{\gamma}^R) $ (see \cite[Theorem 5.1]{Billingsley}, for example).
  Proposition \ref{prop:hitas} implies that, with probability one,  $ E (\bar{\gamma}^{R}_n ) \in \tr \mathscr{S}^K_n $ for $R$ large enough.
  Additionally, note that $\mathscr{S}_n^K$ converges weakly to $\mathscr{S}^K$ as a parameterized tree, when $n \to \infty$.
  It follows that the law of $E (\hat{\gamma}^R)$ is supported on $\mathscr{S}^K$.
\end{proof}

\begin{lemma} \label{lemma:Wilson-curves}
 The collection of curves $\Gamma^e (  \hat{\mathscr{S}}^K  ) \cup \{ \hat{\gamma} \}$ define a parameterized tree $ \hat{\mathscr{S}}^{K+1}  $.
 This tree coincides with the description in Section~\ref{subsec:essential-branches-UST}.
\end{lemma}

\begin{proof}
  Lemma~\ref{lemma:finallimit} shows that $\Gamma^c (\mathscr{S}^K) \cup \{  \hat{\gamma} \}$ satisfies the conditions of Proposition~\ref{prop:essentialbranch-const}.
  It follows that $\Gamma^c (\mathscr{S}^K) \cup \{  \hat{\gamma} \}$ is the set of essential branches for a parameterized tree $\hat{ \mathscr{S}}^{K+1} $.

  Finally, note that Lemma~\ref{lemma:convLERW} shows that $\hat{ \gamma}$ is the limit of scaled loop-erased random walks, stopped when they hit the previous limit element $ \tr \mathscr{S}^K $, and such hitting time is finite. Therefore $ \hat{\mathscr{S}}^{K+1} $ is the tree of Section \ref{subsec:essential-branches-UST}.
\end{proof}

\subsubsection{Step 6: the scaling limit of parameterized trees }

\begin{proof}[Proof of Proposition \ref{prop:extension}]
First let us describe the probability measure induced by $(\mathscr{S}^K_n, \bar{\gamma}_n)$. Let $\mu_n$ be the probability measure on $\mathscr{F}^K$ induced by $\mathscr{S}^K_n$. For each $\mathscr{S}^K_n \in \mathscr{F}^K$, let $\nu_n^{\gamma_n}$ be the probability measure on $(\mathcal{C}_f, \psi)$ induced by the loop-erased random walk $\bar{\gamma}_n$; recall that $\bar{\gamma}_n$ is stopped when it exits $(\mathbb{R}^3\setminus \tr \mathscr{S}^K_n) \cap \mathbb{Z}^3$. The measure $\nu_n^{\gamma_n}$ defines the stochastic kernel
\[K_n (\mathscr{S}^K_n, A) = \nu^{\gamma_n}_n (A), \qquad \forall\mathscr{S}^K_n \in \mathscr{F}^K, A \in \mathcal{B} (\mathcal{C}_f),\]
where $\mathcal{B} (\mathcal{C}_f)$ is the Borel $\sigma$-algebra corresponding to $(\mathcal{C}_f, \psi)$. That is, the probability measure induced by $(\mathscr{S}^K_n, \bar{\gamma}_n)$, $\mu_n \otimes K_n$ say, is the unique probability measure such that
\[\mu_n \otimes K_n (A_1 \times A_2)=\int_{A_1} K_n (\mathscr{S}^K_n, A_2) \mu_n(d \mathscr{S}^K_n),\]
for Borel sets $A_1  \in \mathcal{B} (\mathscr{F}^K) $ and $A_2 \in \mathcal{B} (\mathcal{C}_f)$.

Now, recall we are supposing that we have a coupling so that $\mathscr{S}^K_n\rightarrow \hat{\mathscr{S}}^K$, almost-surely. In what follows, we write $\mathbf{P}^*$ for the corresponding probability measure. From Lemma \ref{lemma:finallimit}, we obtain that, $\mathbf{P}^*$-a.s., $\nu^{\gamma_n}_n \to \nu^{\hat{\gamma}}$ as $n \to \infty$, where $\nu^{\hat{\gamma}}$ is the law of $\hat{\gamma}$. Hence $\nu^{\hat{\gamma}}$ is $\mathbf{P}^*$-measurable, and, in particular, so is $\nu^{\hat{\gamma}}(A)$ for all $A\in \mathcal{B}(C_f)$. As a consequence, the integral
\[\mu\otimes K \left(A_1\times A_2\right):=\int\mathbf{1}_{A_1}(\hat{\mathscr{S}^K }) \nu^{\hat{ \gamma }}(A_2) d\mathbf{P}^*\]
is well-defined for every $A_1\in\mathcal{B}(\mathscr{F}^K)$, $A_2\in\mathcal{B}(\mathcal{C}_f)$. Moreover, $\mu\otimes K$ is readily extended to give a measure on the product space $\mathscr{F}^K\times\mathcal{C}_f$. Finally, let $A_1\in\mathcal{B}(\mathscr{F}^K)$, $A_2\in\mathcal{B}(\mathcal{C}_f)$ be continuity sets for $\mu\otimes K$, in the sense that $\mu\otimes K(\partial A_1\times \mathcal{C}_f)=0=\mu\otimes K(\mathscr{F}^K\times\partial A_2)$. We then have that, $\mathbf{P}^*$-a.s., $\mathbf{1}_{A_1}(\mathscr{S}^K_n) \nu_n^{\hat{\gamma}}(A_2)\to\mathbf{1}_{A_1}(\hat{\mathscr{S}^K}) \nu^{\hat{\gamma}}(A_2)$. An application of the dominated convergence theorem thus yields
\[\mu_n \otimes K_n(A_1\times A_2)\to\mu\otimes K \left(A_1\times A_2\right),\]
which is enough to establish that $\mu \otimes K$ is a measure on $( \mathcal{F}^{K}, \mathcal{C}_f )$ (see \cite[Theorem 2.8]{Billingsley}).
Lemma \ref{lemma:Wilson-curves} shows that $ \mu \otimes K $ defines a measure on the spaces of parameterized trees $\mathcal{F}^{K+1}$.
\end{proof}

\section{Proof of tightness and subsequential scaling limit}\label{proofsec}

Given the preparations in the previous sections, we are now in a position to establish the first main result of this article, namely Theorem \ref{mainthm1}.

\begin{proof}[Proof of Theorem \ref{mainthm1}] We start by establishing the parts of the result concerning the Gromov-Hausdorff-type topology. Applying Lemma \ref{tightness}, the tightness claim follows from Proposition \ref{1st-assump}, Corollary \ref{ass-4} and Proposition \ref{5th}. It remains to check the distributional convergence of $\underline{\sU}_{n}$ as $n\rightarrow\infty$, where we write $\underline{\sU}_n$ for the random measured, rooted spatial tree at \eqref{uddef}, indexed by $\delta_n=2^{-n}$. By the first part of the theorem and Prohorov's theorem (see \cite[Theorem 16.3]{Kall}, for example), we know that every subsequence $(\underline{\sU}_{n_i})_{i\geq 1}$ admits a convergent subsubsequence $(\underline{\sU}_{n_{i_j}})_{j\geq 1}$. Thus we only need to establish the uniqueness of the limit.

Now, suppose $(\underline{\sU}_{n_i})_{i\geq 1}$ is a convergent subsequence, and write $\uT=(\mathcal{T},d_\mathcal{T},\mu_{\mathcal{T}},\phi_\mathcal{T},\rho_\mathcal{T})$ for the limiting random element in $\mathbb{T}$. To show that the convergence specifies the law of $\uT$ uniquely, we will start by considering finite restrictions of $\underline{\sU}_{n_i}$, $i\geq 1$. In particular, for $R\in(0,\infty)$, set
\[\underline{\sU}_{n_i}^{(R)}=\left(B(\delta_{n_i}^{-1}R),\delta_{n_i}^{\beta}d_\mathcal{U}|_{B(\delta_{n_i}^{-1}R)\times B(\delta_{n_i}^{-1}R)},\delta_{n_i}^3\mu_{\mathcal{U}}\left(\cdot\cap B(\delta_{n_i}^{-1}R)\right),\delta_{n_i}\phi_\mathcal{U}|_{B(\delta_{n_i}^{-1}R)},\rho_\mathcal{U}\right),\]
i.e.\ the part of $\underline{\sU}_{n_i}$ contained inside $B(\delta_{n_i}^{-1}R)$. (We acknowledge this notation clashes with that used in Section \ref{topsec} for restrictions to balls with respect to the tree metric.) Note that, by \eqref{deltadef}, we have that
\begin{align*}
\lefteqn{\lim_{R\rightarrow\infty}\limsup_{i\rightarrow\infty}\mathbf{P}\left(\Delta\left(\underline{\sU}_{n_i}^{(R)},\underline{\sU}_{n_i}\right)>\varepsilon\right)}\\
&\leq \lim_{R\rightarrow\infty}\limsup_{i\rightarrow\infty}\left(\mathbf{1}_{\{e^{-\lambda^{-1}R^\beta}>\varepsilon\}}+\mathbf{P}\left(B_\sU(0,\lambda^{-1}\delta_{n_i}^{-\beta}R^\beta)\not\subseteq B(\delta_{n_i}^{-1}R)\right)\right)\\
&\leq Ce^{-c\lambda^a}
\end{align*}
for any $\varepsilon>0$ and $\lambda\geq 1$, where we have applied Proposition \ref{2-4-1} to deduce the final bound. In particular, since $\lambda$ can be taken arbitrarily large in the above estimate, we obtain that
\begin{equation}\label{conccc}
\lim_{R\rightarrow\infty}\limsup_{i\rightarrow\infty}\mathbf{P}\left(\Delta\left(\underline{\sU}_{n_i}^{(R)},\underline{\sU}_{n_i}\right)>\varepsilon\right)=0.
\end{equation}
As a consequence, to prove the uniqueness of the law of $\uT$, it will be enough to show that, for each $R\in (0,\infty)$, $(\underline{\sU}_{n_{i}}^{(R)})_{i\geq 1}$ converges in distribution to a uniquely specified limit. Indeed, if $\uT^{(R)}$ is the limit of $\underline{\sU}_{n_i}^{(R)}$, then, since $\underline{\sU}_{n_i}\buildrel{d}\over{\rightarrow}\uT$ (as $i\rightarrow\infty$) and \eqref{conccc} both hold, we have that $\uT^{(R)}\buildrel{d}\over{\rightarrow}\uT$ as $R\rightarrow\infty$.

Next, for given $n_i$ and $R$, consider the measure $\pi_{n_i}^{(R)}$ on $B(\delta_{n_i}^{-1}R)\times \mathbb{R}^3$ given by
\[\pi_{n_i}^{(R)}(dxdy)=\frac{\mu_{\sU}(dx)\delta_{\delta_{n_i}\phi_\sU(x)}(dy)}{\mu_\sU(B(\delta_{n_i}^{-1}R))},\]
where $\delta_z(\cdot)$ is the probability measure on $\mathbb{R}^3$ placing all its mass at $z$. We will check that the triple
\begin{equation}\label{triple}
\left(B(\delta_{n_i}^{-1}R),\delta_{n_i}^{\beta}d_\mathcal{U}|_{B(\delta_{n_i}^{-1}R)\times B(\delta_{n_i}^{-1}R)},\pi_{n_i}^{(R)}\right)
\end{equation}
converges in the marked Gromov-weak topology of \cite[Definition 2.4]{DGP}; a characterisation of this convergence that will be relevant to us is given in the following paragraph. Towards establishing tightness, we first note that the projections of $\pi_{n_i}^{(R)}$ onto the sets $B(\delta_{n_i}^{-1}R)$ and $\mathbb{R}^3$ are simply the uniform probability measures on $B(\delta_{n_i}^{-1}R)$ and $\delta_{n_i}B(\delta_{n_i}^{-1}R)$, respectively. Since the latter measure clearly converges to the uniform probability measure on $B_E(R)$, by \cite[Theorem 4]{DGP} (see also \cite[Theorem 3]{GPW}), the desired tightness is implied by the following two conditions.
\begin{enumerate}
  \item[(a)] The distributions of
  \[\delta_{n_i}^\beta d_\sU\left(\xi^{n_i,R}_1,\xi^{n_i,R}_2\right),\qquad i\geq 1,\]
  are tight, where $\xi^{n_i,R}_1$ and $\xi^{n_i,R}_2$ are independent uniform random variables on $B(\delta_{n_i}^{-1}R)$, independent of $\sU$.
  \item[(b)] For every $\varepsilon>0$, there exists an $\eta>0$ such that
  \[\mathbf{E}\left(\delta_{n_i}^3\mu_\sU\left(\left\{x\in B(\delta_{n_i}^{-1}R):\:\mu_{\sU}\left( B_\sU(x,\delta_{n_i}^{-\beta}\varepsilon )\cap B(\delta_{n_i}^{-1}R)\right)\leq \eta\right\}\right)\right)\leq\varepsilon.\]
\end{enumerate}
The fact that (b) holds readily follows from the mass lower bound of Corollary \ref{ass-4}. As for (a), this is a simple consequence of Corollary \ref{fddcor}. Moreover, if we write $(\xi^{n_i,R}_j)_{j\geq 1}$ for a sequence of independent uniform random variables on $B(\delta_{n_i}^{-1}R)$, independent of $\sU$, then Corollary \ref{fddcor} further implies that
\begin{equation}\label{uspec}
\left(\left(\delta_{n_i}^\beta d_\sU\left(\xi^{n_i,R}_j,\xi^{n_i,R}_k\right)\right)_{j,k\geq 1},\left(\xi_j^{n_i,R}\right)_{j\geq 1}\right)
\end{equation}
converges in distribution. This enables us to deduce, by applying \cite[Theorem 5, see also Remark 2.7]{DGP}, that the triple at \eqref{triple} in fact converges in distribution in the marked Gromov-weak  topology. We denote the limit by $(\mathcal{T}^{(R)},d_{\mathcal{T}^{(R)}},\pi_{\mathcal{T}^{(R)}})$, where $(\mathcal{T}^{(R)},d_{\mathcal{T}^{(R)}})$ is a complete, separable metric space, and $\pi_{\mathcal{T}^{(R)}}$ is a probability measure on $\mathcal{T}^{(R)}\times \mathbb{R}^3$ such that $\pi_{\mathcal{T}^{(R)}}(\cdot\times \mathbb{R}^3)$ has full support on $\mathcal{T}^{(R)}$. In addition, by combining \eqref{use-1} with Proposition \ref{5th}, we have the following adaptation of Assumption \ref{a3}: there exists a continuous, increasing function $h(\eta)$ with $h(0)=0$ such that
\[\lim_{\eta\rightarrow 0}\liminf_{\delta\rightarrow 0}
\bP\left(\sup_{\substack{x,y\in B(\delta^{-1}R):\\\delta^{\beta}\dU(x,y)<\eta}}\delta \left|\phi_\sU(x)-\phi_\sU(y)\right| \leq h(\eta)\right)= 1.\]
This allows us to apply \cite[Theorem 3.7]{KL} to deduce that
\[\pi_{\mathcal{T}^{(R)}}(dxdy)=\mu_{\mathcal{T}^{(R)}}(dx)\delta_{\phi_{\mathcal{T}^{(R)}}(x)}(dy),\]
where $\mu_{\mathcal{T}^{(R)}}$ is a probability measure on ${\mathcal{T}^{(R)}}$ of full support, and $\phi_{\mathcal{T}^{(R)}}:\mathcal{T}^{(R)}\rightarrow\mathbb{R}^3$ is a continuous function.

As a consequence of the convergence described in the previous paragraph and the separability of the marked Gromov-weak topology (see \cite[Theorem 2]{DGP}), we can assume that all the random objects are built on the same probability space with probability space with probability measure $\mathbf{P}^*$ such that, $\mathbf{P}^*$-a.s.,
\[\left(B(\delta_{n_i}^{-1}R),\delta_{n_i}^{\beta}d_\mathcal{U}|_{B(\delta_{n_i}^{-1}R)\times B(\delta_{n_i}^{-1}R)},\pi_{n_i}^{(R)}\right)\rightarrow\left(\mathcal{T}^{(R)},d_{\mathcal{T}^{(R)}},\pi_{\mathcal{T}^{(R)}}\right).\]
By \cite[Lemma 3.4]{DGP}, this implies that, $\mathbf{P}^*$-a.s., there exists a complete and separable metric space $(Z,d_Z)$ and isometric embeddings $\psi_{n_i}:(B(\delta_{n_i}^{-1}R),\delta_{n_i}^{\beta}d_\mathcal{U})\rightarrow (Z,d_Z)$, $\psi:(\mathcal{T}^{(R)},d_{\mathcal{T}^{(R)}})\rightarrow(Z,d_Z)$, such that
\begin{equation}\label{piconv}
\pi_{n_i}^{(R)}\circ (\tilde{\psi}_{n_i})^{-1}\rightarrow\pi_{\mathcal{T}^{(R)}}\circ \tilde{\psi}^{-1}
\end{equation}
weakly as probability measures on $Z\times\mathbb{R}^3$, where $\tilde{\psi}_{n_i}(x,y)=({\psi}_{n_i}(x),y)$ and $\tilde{\psi}(x,y)=({\psi}(x),y)$. From our initial assumption that $(\underline{\sU}_{n_i})_{i\geq 1}$ is distributionally convergent in $\mathbb{T}$, Corollary \ref{ass-4} and \eqref{use-1}, we further have the existence of a deterministic subsequence $(n_{i_j})_{j\geq 1}$ such that, $\mathbf{P}^*$-a.s., $\underline{\sU}_{n_{i_j}}\rightarrow{\uT}$ in $\mathbb{T}$,
\begin{equation}\label{mlb}
\inf_{j\geq1}\delta_{n_{i_j}}^3\inf_{x\in B(\delta_{n_{i_j}}^{-1}R)}\mu_\sU\left(B_\sU(x,\delta_{n_{i_j}}^{-\beta}\delta)\right)>0,\qquad \forall\delta>0,
\end{equation}
and also
\begin{equation}\label{contain}
\sup_{x\in B(\delta_{n_{i_j}}^{-1}R)}\delta_{n_{i_j}}^{\beta}d_\sU(0,x)\rightarrow\Lambda\in(0,\infty).
\end{equation}
Now, taking projections onto $Z$ and rescaling, we readily obtain from \eqref{piconv} that
\begin{equation}\label{asd1}
\delta_{n_i}^3\mu_{\sU}\left(({\psi}_{n_i})^{-1}(\cdot)\cap B(\delta_{n_i}R)\right)\rightarrow c \mu_{\mathcal{T}^{(R)}}\circ {\psi}^{-1}
\end{equation}
weakly as probability measures on $Z$, where the constant $c$ is the Lebesgue measure of $B_E(R)$. Moreover, appealing again to the mass lower bound of \eqref{mlb}, we also obtain the subsequential convergence of measure supports, i.e.
\[{\psi}_{n_{i_j}}\left(B(\delta_{n_{i_j}}R)\right)\rightarrow\psi\left(\mathcal{T}^{(R)}\right)\]
with respect to the Hausdorff topology on compact subsets of $Z$ (cf.\ the argument of \cite[Theorem 6.1]{ALW}, for example). That $\mathcal{T}^{(R)}$ is indeed compact is established as in \cite{ALW}, and that it is a real tree follows from \cite[Lemma 2.1]{EPW}. In particular, if we define a sequence of correspondences by setting
\[\mathcal{C}_{n_{i_j}}:=\left\{(x,x')\in B(\delta_{n_{i_j}}R)\times \mathcal{T}^{(R)}:\:d_Z\left(\psi_{n_{i_j}}(x),\psi(x')\right)\leq 2d_H^Z\left({\psi}_{n_{i_j}}(B(\delta_{n_{i_j}}R)),\psi(\mathcal{T}^{(R)})\right)\right\},\]
where $d_H^Z$ is the Hausdorff distance on $Z$, then we have that
\begin{equation}\label{asd2}
\sup_{(x,x')\in\mathcal{C}_{n_{i_j}}}d_Z\left(\psi_{n_{i_j}}(x),\psi(x')\right)\rightarrow0.
\end{equation}
Given that $\underline{\sU}_{n_{i_j}}\rightarrow{\uT}$ in $\mathbb{T}$ and \eqref{contain} holds, it is a straightforward application of \cite[Lemmas 3.5 and 5.1]{BCroyK} to also check that, $\mathbf{P}^*$-a.s.,
\[\lim_{\eta\rightarrow0}\limsup_{j\rightarrow\infty}\sup_{\substack{x,y\in B(\delta_{n_{i_j}}^{-1}R):\\\delta_{n_{i_j}}^{\beta}\dU(x,y)<\eta}}\delta_{n_{i_j}} \left|\phi_\sU(x)-\phi_\sU(y)\right|=0,\]
and, applying this equicontinuity in conjunction with \eqref{piconv}, this yields in turn that
\begin{equation}\label{asd3}
\sup_{(x,x')\in\mathcal{C}_{n_{i_j}}}\left|\phi_{\sU}(x)-\phi_{\mathcal{T}^{(R)}}(x')\right|\rightarrow0.
\end{equation}
Finally, although not included in the framework of \cite{DGP, GPW, KL}, it is not difficult to include the convergence of roots in the above arguments, i.e.\ we may further suppose that
\begin{equation}\label{asd4}
d_Z\left(\psi_{n_{i_j}}(\rho_\sU),\psi(\rho_{\mathcal{T}^{(R)}})\right)\rightarrow0
\end{equation}
for some $\rho_{\mathcal{T}^{(R)}}\in\mathcal{T}^{(R)}$ with $\phi_{\mathcal{T}^{(R)}}(\rho_{\mathcal{T}^{(R)}})=0$. Recalling the definition of $\Delta_c$ from \eqref{deltacdef}, combining \eqref{asd1}, \eqref{asd2}, \eqref{asd3} and \eqref{asd4} yields that $\Delta_c(\underline{\sU}^{(R)}_{n_{i_j}},\underline{\mathcal{T}}^{(R)})\rightarrow0$, $\mathbf{P}^*$-a.s., where $\uT^{(R)}:=(\mathcal{T}^{(R)},d_{\mathcal{T}^{(R)}},\mu_{\mathcal{T}^{(R)}},\phi_{\mathcal{T}^{(R)}},\rho_{\mathcal{T}^{(R)}})$. Since the distribution of  $\underline{\mathcal{T}}^{(R)}$ is uniquely specified by \eqref{uspec}, and the same limit can be deduced for some subsubsequence of any subsequence of $(n_i)_{i\geq 1}$, we obtain that $\underline{\sU}_{n_{i}}^{(R)}\rightarrow\underline{\mathcal{T}}^{(R)}$ in distribution in $\mathbb{T}$, and thus the part of the proof concerning the Gromov-Hausdorff-type topology is complete.

As for the path ensemble topology, we know from \cite[Lemma 3.9]{BCroyK} that convergence of compact measured, rooted spatial trees with respect to our Gromov-Hausdorff-type implies the corresponding path ensemble statement. To extend from this to the desired conclusion, we can proceed exactly as in the proof of \cite[Lemma 5.5]{BCroyK}, with the additional inputs required being provided by \eqref{use-1} and the coupling lemma that is stated below at Lemma \ref{couplem}.
\end{proof}

\section{Properties of the limiting space}\label{limitsec}

The aim of this section is to prove Theorem \ref{mainthm2}. To this end, we present several preparatory lemmas. In the first of these, we check that for large enough annuli there is only one disjoint crossing by a path in $\sU$. Precisely, for $r<R$, we introduce the event ${\cal C}^E_{\cal U} (r, R)$ by setting
\begin{equation*}
{\cal C}^E_{\cal U} (r, R) = \left\{ \exists x, y \in B (R)^{c} \text{ such that } \gamma_{\cal U} (x, y) \cap B (r) \neq \emptyset \right\},
\end{equation*}
and show that the probability of this occurring decays as the ratio $R/r$ increases.

\begin{lem}\label{mendo-i}
There exist universal constants $a, b, C \in (0, \infty ) $ such that for all $\delta \in (0, 1) $ and $\lambda \ge 1$,
\[\mathbf{P} \left( {\cal C}^E_{\cal U} \left( \lambda^{-a} \delta^{-1},  \delta^{-1} \right) \right) \le C \lambda^{-b}.\]
\end{lem}
\begin{proof} This is essentially established in the proof of Proposition \ref{1st-assump}. We will use the same notation as in that proof here. First, suppose that the event $A_{k_{0}}'$, as defined in the proof of Proposition \ref{1st-assump}, occurs. It then holds that: for every point $x \in \partial B (\delta^{-1} )$,
\begin{equation*}
\gamma_{\cal U} \left( x, \gamma_{\infty} \right) \cap B \left( \lambda^{-4} \delta^{-1} \right) = \emptyset,
\end{equation*}
where $\gamma_{\infty}$ is the unique infinite simple path in ${\cal U}$ started at the origin, and $\gamma_{\cal U} \left( x, \gamma_{\infty} \right)$ is shortest path in ${\cal U }$ from $x$ to a point of $\gamma_{\infty}$. Note that we have already proved that $\mathbf{P} (A_{k_{0}}' ) \ge 1 - C \lambda^{-1}$. Second, let $u$ be the first time that $\gamma_{\infty}$ exits $B(\lambda^{-4}\delta^{-1})$, and define
\[W = \left\{ \gamma_{\infty} [u, \infty ) \cap B (\lambda^{-5} \delta^{-1} ) = \emptyset \right\}.\]
By Proposition 1.5.10 of \cite{Lawb}, it holds that $\mathbf{P} (W) \ge 1 - C \lambda^{-1}$. Finally, suppose that the event $A_{k_{0}}' \cap W$ occurs. For $x,y\in B(\delta^{-1})^{c}$, let $x',y'\in\gamma_{\infty}$ be such that $\gamma_{\cal U}(x,\gamma_{\infty})=\gamma_{\cal U}(x,x')$ and $\gamma_{\cal U}(y,\gamma_{\infty})=\gamma_{\cal U}(y,y')$. We then have that $\gamma_{\cal U}(x,x')\cap B(\lambda^{-4}\delta^{-1})=\emptyset$ and $\gamma_{\cal U}(y,y')\cap B(\lambda^{-4}\delta^{-1})=\emptyset$. Also, it holds that $x',y'\in\gamma_{\infty}[u,\infty)$. In particular, it follows that $\gamma_{\cal U}(x,y)\cap B(\lambda^{-5}\delta^{-1})=\emptyset$ for all $x,y\in B(\delta^{-1})^{c}$. This completes the proof of the result with $a=5$ and $b=1$.
\end{proof}

We next establish a result which essentially gives the converse of Assumption \ref{a3}. In particular, we define the event ${\cal D} (a, b, c )$ by
\begin{equation*}
{\cal D} (a, b, c ) = \left\{\exists x, y \in B ( a) \text{ such that } d_{\cal U}^{\text{S}} (x, y) < b \text{ and } d_{\cal U} (x, y) > c \right\},
\end{equation*}
where we define the Schramm metric $d_\mathcal{U}^S$ on $\sU$ analogously to \eqref{dtschramm}, and check the following.

\begin{lem}\label{darudaru}
There exist universal $a_{1}, \dots , a_{4}, C \in (0, \infty)$ such that for all $\delta \in (0, 1)$ and $\lambda \ge 1,$
\[\mathbf{P} \left( {\cal D} \left( \lambda^{a_{1}} \delta^{-1}, \lambda^{-a_{2}} \delta^{-1}, \lambda^{-a_{3}} \delta^{- \beta } \right) \right) \le C \lambda^{- a_{4} }.\]
\end{lem}
\begin{proof} Consider the event $\hat{{\cal D}} (a, b, c )$ given by
\begin{equation*}
\hat{{\cal D}} (a, b, c ) = \left\{ \exists x, y \in B ( a) \cap \gamma_{\infty} \text{ such that } d_{\cal U}^{\text{S}} (x, y) < b \text{ and } d_{\cal U} (x, y) > c \right\}.
\end{equation*}
We first prove that there exist universal $a_{1}, \dots , a_{4}, C \in (0, \infty)$ such that for all $\delta \in (0, 1)$ and $\lambda \ge 1$,
\begin{equation}\label{darudaru-2}
\mathbf{P} \left( \hat{{\cal D}}  \left( \lambda^{a_{1}} \delta^{-1}, \lambda^{-a_{2}} \delta^{-1}, \lambda^{-a_{3}} \delta^{- \beta } \right) \right) \le C \lambda^{- a_{4} }.
\end{equation}
To do this, let $a_{1} = 10^{-4}$, $a_{2} = 1$ and $a_{3} = 1/2$. Moreover, let $D = (w_{k})_{k=1}^{M}$ be a $\lambda^{-a_{2}} \delta^{-1}$-net of $B ( \lambda^{a_{1}} \delta^{-1} ) $ such that $  B ( \lambda^{a_{1}} \delta^{-1} )  \subseteq \bigcup_{k=1}^{M} B (w_{k}, \lambda^{-a_{2}} \delta^{-1} )$ and $M \asymp \lambda^{3 (a_{1} + a_{2} )}$. Suppose that the event $\hat{{\cal D}}  \left( \lambda^{a_{1}} \delta^{-1}, \lambda^{-a_{2}} \delta^{-1}, \lambda^{-a_{3}} \delta^{- \beta } \right) $ occurs. Then there exists $w_{k} \in D$ such that $| \gamma_{\infty} \cap B (w_{k}, \lambda^{-a_{2}} \delta^{-1})| \ge c \lambda^{-a_{3}} \delta^{-\beta} $ for some universal $c > 0$. Now, it follows from \cite[(7.51)]{LS} that
\begin{equation*}
\mathbf{P} \left( \exists  w_{k} \in D \text{ such that } \left|  \gamma_{\infty} \cap B (w_{k}, \lambda^{-a_{2}} \delta^{-1} )\right| \ge c \lambda^{-a_{3}} \delta^{-\beta}  \right) \le C  e^{- c' \lambda^{1/2} },
\end{equation*}
for some universal $c', C \in (0, \infty )$. Thus, the inequality \eqref{darudaru-2} holds when we let $a_{4} = 100$.

We next consider a $\lambda^{-4} \delta^{-1}$-net $D' =(x_{i})_{i=1}^{N}$ of the ball $B (\lambda^{a_{1}} \delta^{-1})$ for which $B ( \lambda^{a_{1}} \delta^{-1} ) \subseteq \bigcup_{i=1}^{N} B (x_{i}, \lambda^{-4} \delta^{-1} )$ and $N \asymp \lambda^{3 (a_{1} + 4) }$. We perform Wilson's algorithm as follows:
\begin{itemize}
\item Consider a subtree spanned by $D' = (x_{i})_{i=1}^{N}$. The output random tree is denoted by ${\cal U}_{1}$.
\item Perform Wilson's algorithm for all remaining points $\mathbb{Z}^{3} \setminus D'$ to generate ${\cal U}$.
\end{itemize}
We define the event $L$ by
\begin{equation*}
L = \bigcap_{i=1}^{N}  \hat{{\cal D}}  \left( \lambda^{a_{1}} \delta^{-1}, \lambda^{-a_{2}} \delta^{-1}, \lambda^{-a_{3}} \delta^{- \beta }; i \right)^{c},
\end{equation*}
where the event $\hat{{\cal D}}  \left( a,b,c; i \right)$ is defined by
\begin{equation*}
\hat{{\cal D}} (a, b, c; i ) = \left\{ \exists x, y \in B ( a) \cap \gamma_{\infty}^{x_{i}} \text{ such that } d_{\cal U}^{\text{S}} (x, y) < b \text{ and } d_{\cal U} (x, y) > c \right\},
\end{equation*}
with $\gamma_{\infty}^{x}$ standing for the unique infinite simple path in ${\cal U}$ started at $x$. By \eqref{darudaru-2}, we have $\mathbf{P} (L) \ge 1 - C \lambda^{-80}$. Furthermore, if we define
\[J = \left\{ \forall x \in B (\lambda^{a_{1} } \delta^{-1} ), \text{diam} \left( \gamma_{\cal U} (x, {\cal U}_{1} ) \right) < \lambda^{-2} \delta^{-1} \text{ and } d_{\cal U} (x, {\cal U}_{1} ) < \lambda^{-2} \delta^{-\beta} \right\},\]
then applying the hittability of each branch of ${\cal U}$ as in the proof of Proposition \ref{1st-assump} guarantees that $\mathbf{P} (J ) \ge 1 - C \lambda^{-10}$. Finally, suppose that the event $L \cap J$ occurs. The event $L$ ensures that for all $x, y \in {\cal U}_{1} $ with $d_{\cal U}^{\text{S}} (x, y) < \lambda^{-a_{2}} \delta^{-1} $, we have $d_{\cal U} (x, y) < 2 \lambda^{- a_{3}} \delta^{-\beta} $. Also, the event $J$ guarantees that for all $x, y \in B (\lambda^{a_{1}} \delta^{-1} )$ with $d_{\cal U}^{\text{S}} (x, y) < \frac{1}{2} \lambda^{-a_{2}} \delta^{-1} $, we have $d_{\cal U} (x, y) < 3 \lambda^{- a_{3}} \delta^{-\beta} $. Thus the proof is complete, establishing the result with $a_{1} = 10^{-4}$, $a_{2} = 1$, $a_{3} = 1/2$ and $a_{4} = 10$.
\end{proof}

For the remainder of the section, including in the proof of Theorem \ref{mainthm2}, we fix a sequence $\delta_n\rightarrow 0$ such that $(\mathbf{P}_{\delta_n})_{n\geq 1}$ converges weakly (as measures on $(\mathbb{T},\Delta)$), and write $\sU_{\delta_n}=(\mathcal{U},\delta_n^\kappa d_\mathcal{U},\delta_n^2\mu_\mathcal{U},\delta_n\phi_\mathcal{U},0)$. Letting  $\hat{\mathbf{P}}$ be the relevant limiting law, we denote by $\uT=(\mathcal{T},d_\mathcal{T},\mu_\mathcal{T},\phi_\mathcal{T},\rho_\mathcal{T})$ a random element of $\mathbb{T}$ with law $\hat{\mathbf{P}}$. A key ingredient to the proof of Theorem \ref{mainthm2} is the following coupling between the discrete and continuous models, which is a ready consequence of this convergence assumption. Since the proof of the corresponding result in \cite{BCroyK} was not specific to the two-dimensional case, we omit the proof here.

{\lem[{cf.\ \cite[Lemma 5.1]{BCroyK}}]\label{couplem} There exist realisations of $(\sU_{\delta_n})_{n\geq 1}$ and $\uT$ built on the same probability space, with probability measure $\mathbf{P}^*$ say, such that: for some subsequence $(n_i)_{i\geq 1}$ and divergent sequence $(r_j)_{j\geq 1}$ it holds that, $\mathbf{P}^*$-a.s.,
\[D_{i,j}:=\Delta_c\left(\sU_{\delta_{n_i}}^{(r_j)},\uT^{(r_j)}\right)\rightarrow 0\]
as $i\rightarrow \infty$, for every $j\geq 1$.}

\begin{proof}[Proof of Theorem \ref{mainthm2}] We start by checking the measure bounds of parts (c) and (d), and we also remark that part (b) is an elementary consequence of (c) (see \cite[Proposition 1.5.15]{Edgar}, for example). The uniform bound of (c) will follow from the estimates: for $R>0$, there exist constants $c_i\in(0,\infty)$ such that, for every $r\in(0,1)$,
\begin{eqnarray}
\hat{\mathbf{P}}\left(\inf_{x\in B_\mathcal{T}(\rho_\mathcal{T},R)}\mu_\mathcal{T}\left(B_\mathcal{T}(x,r)\right)
\le c_1r^{d_f}(\log r^{-1})^{-c_2}\right)&\leq& c_3r^{c_4},\label{u1}\\
\hat{\mathbf{P}}\left(\sup_{x\in B_\mathcal{T}(\rho_\mathcal{T},R)}\mu_\mathcal{T}\left(B_\mathcal{T}(x,r)\right)
\geq c_5r^{d_f}(\log r^{-1})^{c_6}\right)&\leq& c_7r^{c_8}.\label{u2}
\end{eqnarray}
Indeed, given these, applying Borel-Cantelli along the subsequence $r_n=2^{-n}$, $n\in\mathbb{N}$, yields the result. By appealing to the coupling of Lemma \ref{couplem}, the above inequalities readily follow from the following discrete analogues:
\begin{eqnarray}
\limsup_{\delta\rightarrow\infty}{\mathbf{P}}\left(\delta^3\min_{x\in B_{\sU}(\rho_\sU,\delta^{-\beta}R)}\mu_\mathcal{T}\left(B_\sU(x,\delta^{-\beta}r)\right)
\le c_1r^{d_f}(\log r^{-1})^{-c_2}\right)&\leq& c_3r^{c_4},\label{dcdcdc1}\\
\limsup_{\delta\rightarrow\infty}{\mathbf{P}}\left(\delta^3\max_{x\in B_{\sU}(\rho_\sU,\delta^{-\beta}R)}\mu_\mathcal{T}\left(B_\sU(x,\delta^{-\beta}r)\right)
\ge c_5r^{d_f}(\log r^{-1})^{c_6}\right)&\leq& c_7r^{c_8}.\label{dcdcdc2}
\end{eqnarray}
To establish these, we start by noting that Proposition \ref{2-4-1} implies that the probability in \eqref{dcdcdc1} is bounded above by
\[Ce^{-cz^a}+{\mathbf{P}}\left(\delta^3\min_{x\in B(\delta^{-1}R^{1/\beta}z)}\mu_\mathcal{T}\left(B_\sU(x,\delta^{-\beta}r)\right)\le c_1r^{d_f}(\log r^{-1})^{-c_2}\right)\]
for any $z\geq 1$. Moreover, applying a simple union bound and Theorem \ref{2nd-goal} (with $R=\delta^{-\beta}r$, $\lambda= c_1^{-1}\log( r^{-1})^{c_2}$), we can bound this in turn by
\[Ce^{-cz^a}+\frac{C'R^{d_f}z^3}{r^{d_f}}e^{-c'c_1^{-a'}\log( r^{-1})^{a'c_2}}.\]
Choosing $z=(c^{-1}\log(r^{-1}))^{1/a}$, $c_1$ small enough so that $c'c_1^{-a}>d_f$, and $c_2=1/a'$, the above is bounded above by $C''r^{c''}$, as desired. The proof of \eqref{dcdcdc2} is similar, with Theorem  \ref{darui-1} replacing Theorem \ref{2nd-goal}. As for (d), this follows from a Borel-Cantelli argument and the following estimates: there exist constants $c_i\in(0,\infty)$ such that
\begin{equation}\label{dcdcdc3}
\hat{\mathbf{P}}\left(\mu_{\mathcal{T}}\left(B_\mathcal{T}(\rho_\mathcal{T},r)\right)\geq \lambda r^{d_f}
\right)\leq c_1e^{-c_2\lambda^{c_3}},
\end{equation}
\begin{equation}\label{dcdcdc4}
\hat{\mathbf{P}}\left(\mu_{\mathcal{T}}\left(B_\mathcal{T}(\rho_\mathcal{T},r)\right)\leq \lambda^{-1} r^{d_f}
\right)\leq c_4e^{-c_5\lambda^{c_6}},
\end{equation}
for all $r>0$, $\lambda\geq 1$. Similarly to the proof of the uniform estimates \eqref{u1} and \eqref{u2}, applying the coupling of Lemma \ref{couplem}, these readily follow from Theorem \ref{1st-goal} and Proposition \ref{2-4-1}.

For part (a), since $(\sU, d_{\sU})$ has infinite diameter, we immediately find that $(\mathcal{T},d_\mathcal{T})$ has at least one end at infinity. Thus we need to show that there can be no more than one end at infinity. Given Lemma \ref{couplem} and the inclusion results of \eqref{use-1} and Proposition \ref{2-4-1}, this can be proved exactly as in the two-dimensional case. In particular, as in \cite{BCroyK}, it follows from the following crossing estimate: for $r>0$,
\[\lim_{R\rightarrow \infty}\limsup_{\delta\rightarrow0}\mathbf{P}\left(C_\sU^E(\delta^{-1}r,\delta^{-1} R)\right)=0,\]
which is given by Lemma \ref{mendo-i}.

For part (e), we can proceed exactly as in the proof of \cite[Lemma 5.4]{BCroyK}. Given Lemma \ref{couplem}, the one additional ingredient we need to do this is the estimate corresponding to \cite[(5.12)]{BCroyK}:  for every $r,\eta>0$,
\[\lim_{\varepsilon\rightarrow 0}\limsup_{\delta\rightarrow 0}{\mathbf{P}}
\left(\inf_{\substack{x,y\in B_\mathcal{U}(0,\delta^{-\beta}r):\\d_\mathcal{U}(x,y)\geq \delta^{-\beta}\eta}}d_\mathcal{U}^S(x,y)<\delta^{-1}\varepsilon\right)=0,\]
and this was established in Lemma \ref{darudaru} (when viewed in conjunction with Proposition \ref{2-4-1}).

Given Lemma \ref{couplem} and \eqref{2-4-1}, the proof of part (f) is identical to that of \cite[Lemma 5.2]{BCroyK}.
\end{proof}

\section{Simple random walk and its diffusion limit}\label{srwsec}

In this section, we complete the article with the proofs of Theorem \ref{mainthm3}, Corollary \ref{corsrxexp} and Theorem \ref{mainthm4}.

\begin{proof}[Proof of Theorem \ref{mainthm3}] On the event
\begin{equation}\label{evenggg}
\left\{\inf_{x\in \BU(0,R)} \mU\left(\BU(x,R/8)\right)\geq \lambda^{-1}R^{d_f},\:\mU\left(\BU(0,2R)\right)\leq \lambda R^{d_f}\right\},
\end{equation}
one can find a cover $(\BU(x_i,R/4))_{i=1}^N$ of $\BU(0,R)$ of size $N\leq \lambda^2$ (cf.\ \cite[Lemma 9]{Cest}, for example). Following the argument of \cite[Lemma 2.4]{BCKum} (see alternatively \cite[Lemma 4.1]{Kumres}), it holds that on the event at \eqref{evenggg},
\[R_\sU\left(0,\BU(0,R)^c\right)\geq \frac{R}{\lambda^2}.\]
Hence the result is a consequence of Theorem \ref{2nd-goal} and Proposition \ref{2-4-1}.
\end{proof}

\begin{proof}[Proof of Corollary \ref{corsrxexp}] By Theorem \ref{mainthm3}, parts (1) and (4) of \cite[Assumption 1.2]{KM} hold. Moreover, since $R_\sU(0,\BU(0,R)^c)\leq R+1$, we also have that part (2) of \cite[Assumption 1.2]{KM} holds. Hence \eqref{star1}, \eqref{specdim}, \eqref{star3}, \eqref{star5} and \eqref{star7} follow from \cite[Proposition 1.4 and Theorem 1.5]{KM}. It remains to prove the claims involving the Euclidean distance. To this end, note that by \eqref{use-1} and Proposition \ref{2-4-1},
\[\mathbf{P}\left(\BU(0,\lambda^{-1}R^\beta)\subseteq B(R)\subseteq\BU(0,\lambda R^\beta) \right)\geq 1-c_1\lambda^{-c_2}.\]
Hence, by Borel-Cantelli, if $R_n:=2^n$ and $\lambda_n:=n^{2/c_2}$, then
\[\BU(0,\lambda_n^{-1}R_n^\beta)\subseteq B(R_n)\subseteq\BU(0,\lambda_n R_n^\beta)\]
for all large $n$, $\mathbf{P}$-a.s. Combining this with the results at \eqref{star1} and \eqref{star3}, we obtain \eqref{star2} and \eqref{star4}. As for \eqref{star6}, the lower bound follows from Jensen's inequality, Fatou's lemma and \eqref{star4}. Indeed,
\[\liminf_{R\rightarrow \infty}\frac{\log \mathbb{E}^{\sU}\left( \tau^E_{0,R}\right)}{\log R}\geq \liminf_{R\rightarrow \infty}\mathbb{E}^{\sU}\left( \frac{\log \tau^E_{0,R}}{\log R}\right)\geq\mathbb{E}^{\sU}\left( \liminf_{R\rightarrow \infty}\frac{\log \tau^E_{0,R}}{\log R}\right)=\beta d_w.\]
As for the upper bound, a standard estimate for exit times (see \cite[Corollary 2.66]{Barlowbook}, for example) gives that
\[E^{\sU}_0\tau_{0,R}^E\leq R^3 R_\sU(0,B(R)^c)\leq R^3\xi_R,\]
where $\xi_R$ is defined above Proposition \ref{result:upperLowTail}. The latter result thus yields
\[\mathbb{E}^\sU\left(\tau_{0,R}^E\right)\leq R^3 \mathbf{E}\left(\xi_R\right)\leq cR^{3+\beta}=cR^{\beta d_w},\]
which gives (a stronger statement than) the desired conclusion.
\end{proof}

\begin{proof}[Proof of Theorem \ref{mainthm4}] The result can be proved by a line-by-line modification of \cite[Theorems 1.4 and 7.2]{BCroyK}, and so we omit the details. However, as an aid to the reader, we summarise the key steps. As per the construction of \cite{Kden}, $\hat{\mathbf{P}}$-a.s., there is a `resistance form' $(\mathcal{E}_\mathcal{T},\mathcal{F}_\mathcal{T})$ on $(\mathcal{T},d_{\mathcal{T}})$, characterised by
\[d_\mathcal{T}(x,y)^{-1}=\inf\left\{\mathcal{E}_\mathcal{T}(f,f):\:f\in\mathcal{F}_\mathcal{T},\:f(x)=0,\:f(y)=1\right\},\qquad\forall x,y\in\mathcal{T},\:x\neq y.\]
Moreover, by taking
\[\mathcal{D}_\mathcal{T}:=\overline{\mathcal{F}_\mathcal{T}\cap C_0(\mathcal{T})},\]
where $C_0(\mathcal{T})$ are the compactly supported continuous functions on $(\mathcal{T},d_{\mathcal{T}})$, and the closure is taken with respect to $\mathcal{E}_\mathcal{T}(f,f)+\int_\mathcal{T}f^2d\mu_\mathcal{T}$, we obtain a regular Dirichlet form $(\mathcal{E}_\mathcal{T},\mathcal{D}_\mathcal{T})$ on $L^2(\mathcal{T},\mu_{\mathcal{T}})$ (see \cite[Remark 1.6]{AEW} or \cite[Theorem 9.4]{Kres}). Moreover, since $(\mathcal{T},d_\mathcal{T})$ is complete and has one end at infinity (by Theorem \ref{mainthm2}(a)), the naturally associated stochastic process $((X^\mathcal{T}_t)_{t\geq 0},(P^\mathcal{T}_x)_{x\in\mathcal{T}})$ is recurrent (see \cite[Theorem 4]{AEW}). And, from \cite[Theorem 10.4]{Kres}, we have that the process admits a jointly continuous transition density $(p^\mathcal{T}_t(x, y))_{x,y\in\mathcal{T},t>0}$.

Next, by appealing to the Skorohod representation theorem, it is possible to construct realisations of $(\U,\delta_n^\beta\dU,\delta_n^3\mU,\delta_n\pU,\rU)$, $n\geq 1$, and  $(\T,\dT,\mT,\pT,\rT)$ on the same probability space with probability measure $\mathbf{P}^*$ such that $(\U,\delta_n^\beta\dU,\delta_n^3\mU,\delta_n\pU,\rU)\rightarrow(\T,\dT,\mT,\pT,\rT)$, $\mathbf{P}^*$-a.s. Moreover, applying Theorem \ref{mainthm3} in a simple Borel-Cantelli argument allows one to deduce that, $\mathbf{P}^*$-a.s.,
\[\lim_{R\rightarrow\infty}\liminf_{n\rightarrow\infty}\delta_n^{\beta}{R}_\U\left(0,\BU(0,R\delta_n^{-\beta})^c\right)=\infty.\]
Hence we can apply \cite[Theorem 7.1]{Cres} to deduce that, $\mathbf{P}^*$-a.s.,
\begin{equation}\label{repok}
P^\U_0\left(\left(\delta_n X^\U_{t\delta_n^{-(3+\beta)}}\right)_{t\geq 0}\in\cdot\right)\rightarrow P^\mathcal{T}_{\rho_\mathcal{T}}\circ\pT^{-1}
\end{equation}
weakly as probability measures on $C(\mathbb{R}_+,\mathbb{R}^3)$. Since the left-hand side above is $\mathbf{P}^*$-measurable, so is the right-hand side. Moreover, for any measurable set $B\subseteq C(\mathbb{R}_+,\mathbb{R}^3)$, we have that
\[P^\mathcal{T}_{\rho_\mathcal{T}}\circ\pT^{-1}(B)=\mathbf{E}^*\left(P^\mathcal{T}_{\rho_\mathcal{T}}\circ\pT^{-1}(B)\:\vline\:\mathcal{T}\right),\]
where $\mathbf{E}^*$ is the expectation under $\mathbf{P}^*$, and so $P^\mathcal{T}_{\rho_\mathcal{T}}\circ\pT^{-1}$ is in fact $\hat{\mathbf{P}}$-measurable, as is required to prove part (a). For part (b), we apply \eqref{repok} and integrate out with respect to $\mathbf{P}^*$.

As for the heat kernel estimates, we note that the measure bounds of Theorem \ref{mainthm2}(c) are enough to apply the arguments of \cite{Cest} to deduce part (c) (for further details, see the proof of \cite[Theorem 1.4(c)]{BCroyK}). As for the on-diagonal estimates of part (d), similarly to the proof of \cite[Theorem 7.2]{BCroyK} (cf.\ \cite[Theorems 1.6 and 1.7]{Cvol}), these follow from the distributional estimates on the measures of balls at \eqref{dcdcdc3} and \eqref{dcdcdc4}, together with the following resistance estimate
\begin{equation}\label{limres}
\mathbf{P}\left(R_\mathcal{T}(\rho_\mathcal{T},B_\mathcal{T}(\rho_\mathcal{T},R)^c)\leq \lambda^{-1}R\right)\leq Ce^{-c\lambda^{a}},
\end{equation}
where $R_\mathcal{T}$ is the resistance associated with $(\mathcal{E}_\mathcal{T},\mathcal{F}_\mathcal{T})$. As in the proof of Theorem \ref{mainthm3}, to check \eqref{limres}, it is enough to combine \eqref{dcdcdc3} with the bound
\[\hat{\mathbf{P}}\left(\inf_{x\in B_\mathcal{T}(\rho_\mathcal{T},R)} \mu_\mathcal{T}\left(B_\mathcal{T}(x,R/8)\right)\leq \lambda^{-1}R^{d_f}\right)\leq Ce^{-c\lambda^{a}},\]
which is again a ready consequence of the discrete analogue (see Theorem \ref{2nd-goal} and Proposition \ref{2-4-1}).
\end{proof}

\section*{Acknowledgements}

DC would like to acknowledge the support of a JSPS Grant-in-Aid for Research Activity Start-up, 18H05832, a JSPS Grant-in-Aid for Scientific Research (C), 19K03540, and the Research Institute for Mathematical Sciences, an International Joint Usage/Research Center located in Kyoto University. SHT would like to acknowledge the support of a fellowship from the Mexican National Council for Science and Technology (CONACYT). DS is supported by a JSPS Grant-in-Aid for Early-Career Scientists, 18K13425 and JSPS KAKENHI Grant Number 17H02849 and 18H01123.
Furthermore, all the authors would like to thank Russell Lyons for suggesting numerous corrections and an anonymous referee for their insightful referee report. In particular, for the referee's detailed description of how Kozma's result from \cite{Kozma} could be extended (see Remark \ref{extension} for further comments).

\bibliography{3dUST}

\providecommand{\bysame}{\leavevmode\hbox to3em{\hrulefill}\thinspace}
\providecommand{\MR}{\relax\ifhmode\unskip\space\fi MR }
\providecommand{\MRhref}[2]{%
  \href{http://www.ams.org/mathscinet-getitem?mr=#1}{#2}
}
\providecommand{\href}[2]{#2}
\begin{thebibliography}{10}

\bibitem{ADH}
R.~Abraham, J.-F. Delmas, and P.~Hoscheit, \emph{A note on the
  {G}romov-{H}ausdorff-{P}rokhorov distance between (locally) compact metric
  measure spaces}, Electron. J. Probab. \textbf{18} (2013), no. 14, 21.

\bibitem{AEW}
S.~Athreya, M.~Eckhoff, and A.~Winter, \emph{Brownian motion on
  {$\mathbb{R}$}-trees}, Trans. Amer. Math. Soc. \textbf{365} (2013), no.~6,
  3115--3150.

\bibitem{ALW}
S.~Athreya, W.~L\"{o}hr, and A.~Winter, \emph{Invariance principle for variable
  speed random walks on trees}, Ann. Probab. \textbf{45} (2017), no.~2,
  625--667.

\bibitem{Barlowbook}
M.~T. Barlow, \emph{Random walks and heat kernels on graphs}, London
  Mathematical Society Lecture Note Series, vol. 438, Cambridge University
  Press, Cambridge, 2017.

\bibitem{BCKum}
M.~T. Barlow, T.~Coulhon, and T.~Kumagai, \emph{Characterization of
  sub-{G}aussian heat kernel estimates on strongly recurrent graphs}, Comm.
  Pure Appl. Math. \textbf{58} (2005), no.~12, 1642--1677.

\bibitem{BCroyK2}
M.~T. Barlow, D.~A. Croydon, and T.~Kumagai, \emph{Quenched and averaged tails
  of the heat kernel of the two-dimensional uniform spanning tree}, in
  preparation.

\bibitem{BCroyK}
\bysame, \emph{Subsequential scaling limits of simple random walk on the
  two-dimensional uniform spanning tree}, Ann. Probab. \textbf{45} (2017),
  no.~1, 4--55.

\bibitem{BJKS}
M.~T. Barlow, A.~A. J\'{a}rai, T.~Kumagai, and G.~Slade, \emph{Random walk on
  the incipient infinite cluster for oriented percolation in high dimensions},
  Comm. Math. Phys. \textbf{278} (2008), no.~2, 385--431.

\bibitem{BM}
M.~T. Barlow and R.~Masson, \emph{Spectral dimension and random walks on the
  two dimensional uniform spanning tree}, Comm. Math. Phys. \textbf{305}
  (2011), no.~1, 23--57.

\bibitem{Billingsley}
P.~Billingsley, \emph{Convergence of probability measures}, 2nd ed.,
  Wiley-Interscience, 1999.

\bibitem{BBI}
D.~Burago, Y.~Burago, and S.~Ivanov, \emph{A course in metric geometry},
  Graduate Studies in Mathematics, vol.~33, American Mathematical Society,
  Providence, RI, 2001.

\bibitem{Cest}
D.~A. Croydon, \emph{Heat kernel fluctuations for a resistance form with
  non-uniform volume growth}, Proc. Lond. Math. Soc. (3) \textbf{94} (2007),
  no.~3, 672--694.

\bibitem{CAIHP}
\bysame, \emph{Convergence of simple random walks on random discrete trees to
  {B}rownian motion on the continuum random tree}, Ann. Inst. Henri
  Poincar\'{e} Probab. Stat. \textbf{44} (2008), no.~6, 987--1019.

\bibitem{Cvol}
\bysame, \emph{Volume growth and heat kernel estimates for the continuum random
  tree}, Probab. Theory Related Fields \textbf{140} (2008), no.~1-2, 207--238.

\bibitem{Crange}
\bysame, \emph{Hausdorff measure of arcs and {B}rownian motion on {B}rownian
  spatial trees}, Ann. Probab. \textbf{37} (2009), no.~3, 946--978.

\bibitem{Csl}
\bysame, \emph{Scaling limits for simple random walks on random ordered graph
  trees}, Adv. in Appl. Probab. \textbf{42} (2010), no.~2, 528--558.

\bibitem{Cres}
\bysame, \emph{Scaling limits of stochastic processes associated with
  resistance forms}, Ann. Inst. Henri Poincar\'{e} Probab. Stat. \textbf{54}
  (2018), no.~4, 1939--1968.

\bibitem{DGP}
A.~Depperschmidt, A.~Greven, and P.~Pfaffelhuber, \emph{Marked metric measure
  spaces}, Electron. Commun. Probab. \textbf{16} (2011), 174--188.

\bibitem{DS}
P.~G. Doyle and J.~L. Snell, \emph{Random walks and electric networks}, Carus
  Mathematical Monographs, vol.~22, Mathematical Association of America,
  Washington, DC, 1984.

\bibitem{Edgar}
G.~A. Edgar, \emph{Integral, probability, and fractal measures},
  Springer-Verlag, New York, 1998.

\bibitem{EPW}
S.~N. Evans, J.~Pitman, and A.~Winter, \emph{Rayleigh processes, real trees,
  and root growth with re-grafting}, Probab. Theory Related Fields \textbf{134}
  (2006), no.~1, 81--126.

\bibitem{GPW}
A.~Greven, P.~Pfaffelhuber, and A.~Winter, \emph{Convergence in distribution of
  random metric measure spaces ({$\Lambda$}-coalescent measure trees)}, Probab.
  Theory Related Fields \textbf{145} (2009), no.~1-2, 285--322.

\bibitem{HS}
N.~Holden and X.~Sun, \emph{S{LE} as a mating of trees in {E}uclidean
  geometry}, Comm. Math. Phys. \textbf{364} (2018), no.~1, 171--201.

\bibitem{Kall}
O.~Kallenberg, \emph{Foundations of modern probability}, second ed.,
  Probability and its Applications (New York), Springer-Verlag, New York, 2002.

\bibitem{KS}
I.~Karatzas and S.~E. Shreve, \emph{Brownian motion and stochastic calculus},
  second ed., Graduate Texts in Mathematics, vol. 113, Springer-Verlag, New
  York, 1991.

\bibitem{Kenyon}
R.~Kenyon, \emph{The asymptotic determinant of the discrete {L}aplacian}, Acta
  Math. \textbf{185} (2000), no.~2, 239--286.

\bibitem{Kesten1987}
H.~Kesten, \emph{Hitting probabilities of random walks on {${\bf Z}^d$}},
  Stochastic Process. Appl. \textbf{25} (1987), no.~2, 165--184.

\bibitem{Kden}
J.~Kigami, \emph{Harmonic calculus on limits of networks and its application to
  dendrites}, J. Funct. Anal. \textbf{128} (1995), no.~1, 48--86.

\bibitem{Kres}
\bysame, \emph{Resistance forms, quasisymmetric maps and heat kernel
  estimates}, Mem. Amer. Math. Soc. \textbf{216} (2012), no.~1015, vi+132.

\bibitem{KL}
S.~Kliem and W.~L\"{o}hr, \emph{Existence of mark functions in marked metric
  measure spaces}, Electron. J. Probab. \textbf{20} (2015), no. 73, 24.

\bibitem{Kozma}
G.~Kozma, \emph{The scaling limit of loop-erased random walk in three
  dimensions}, Acta Math. \textbf{199} (2007), no.~1, 29--152.

\bibitem{Kumres}
T.~Kumagai, \emph{Heat kernel estimates and parabolic {H}arnack inequalities on
  graphs and resistance forms}, Publ. Res. Inst. Math. Sci. \textbf{40} (2004),
  no.~3, 793--818.

\bibitem{KM}
T.~Kumagai and J.~Misumi, \emph{Heat kernel estimates for strongly recurrent
  random walk on random media}, J. Theoret. Probab. \textbf{21} (2008), no.~4,
  910--935.

\bibitem{Lawb}
G.~F. Lawler, \emph{Intersections of random walks}, Probability and its
  Applications, Birkh\"{a}user Boston, Inc., Boston, MA, 1991.

\bibitem{LLERW}
\bysame, \emph{Loop-erased random walk}, Perplexing problems in probability,
  Progr. Probab., vol.~44, Birkh\"{a}user Boston, Boston, MA, 1999,
  pp.~197--217.

\bibitem{Lawler}
\bysame, \emph{Conformally invariant processes in the plane}, Mathematical
  Surveys and Monographs, vol.~14, American Mathematical Society, 2005.

\bibitem{LawlerLimic}
G.~F. Lawler and V.~Limic, \emph{{Random walk: a modern introduction}},
  Cambridge University Press, 2010.

\bibitem{LSW}
G.~F. Lawler, O.~Schramm, and W.~Werner, \emph{Conformal invariance of planar
  loop-erased random walks and uniform spanning trees}, Ann. Probab.
  \textbf{32} (2004), no.~1B, 939--995.

\bibitem{rrt}
J.-F. Le~Gall, \emph{Random real trees}, Ann. Fac. Sci. Toulouse Math. (6)
  \textbf{15} (2006), no.~1, 35--62.

\bibitem{LPW}
D.~A. Levin, Y.~Peres, and E.~L. Wilmer, \emph{Markov chains and mixing times},
  American Mathematical Society, Providence, RI, 2009, With a chapter by James
  G. Propp and David B. Wilson.

\bibitem{LS}
X.~Li and D.~Shiraishi, \emph{Convergence of three-dimensional loop-erased
  random walk in the natural parametrization}, preprint available at
  arxiv.org/1811.11685.

\bibitem{Escape}
X.~Li and D.~Shiraishi, \emph{One-point function estimates for loop-erased
  random walk in three dimensions}, Electron. J. Probab. \textbf{24} (2019),
  Paper No. 111, 46. \MR{4017129}

\bibitem{LP}
R.~Lyons and Y.~Peres, \emph{Probability on trees and networks}, Cambridge
  Series in Statistical and Probabilistic Mathematics, vol.~42, Cambridge
  University Press, New York, 2016.

\bibitem{LPS}
R.~Lyons, Y.~Peres, and O.~Schramm, \emph{Markov chain intersections and the
  loop-erased walk}, Ann. Inst. Henri Poincar\'{e} Probab. Stat. \textbf{39}
  (2003), 779--791.

\bibitem{Mas}
R.~Masson, \emph{The growth exponent for planar loop-erased random walk},
  Electron. J. Probab. \textbf{14} (2009), no. 36, 1012--1073.

\bibitem{Pemantle}
R.~Pemantle, \emph{Choosing a spanning tree for the integer lattice uniformly},
  Ann. Probab. \textbf{19} (1991), no.~4, 1559--1574.

\bibitem{SS}
A.~Sapozhnikov and D.~Shiraishi, \emph{On {B}rownian motion, simple paths, and
  loops}, Probab. Theory Related Fields \textbf{172} (2018), no.~3-4, 615--662.

\bibitem{Schramm}
O.~Schramm, \emph{Scaling limits of loop-erased random walks and uniform
  spanning trees}, Israel J. Math. \textbf{118} (2000), 221--288.

\bibitem{S}
D.~Shiraishi, \emph{Growth exponent for loop-erased random walk in three
  dimensions}, Ann. Probab. \textbf{46} (2018), no.~2, 687--774.

\bibitem{Smirnov}
S.~Smirnov, \emph{Critical percolation in the plane: conformal invariance,
  {C}ardy's formula, scaling limits}, C. R. Acad. Sci. Paris S\'{e}r. I Math.
  \textbf{333} (2001), no.~3, 239--244.

\bibitem{Wilson}
D.~B. Wilson, \emph{Generating random spanning trees more quickly than the
  cover time}, Proceedings of the {T}wenty-eighth {A}nnual {ACM} {S}ymposium on
  the {T}heory of {C}omputing ({P}hiladelphia, {PA}, 1996), ACM, New York,
  1996, pp.~296--303.

\bibitem{Wilest}
\bysame, \emph{Dimension of loop-erased random walk in three dimensions}, Phys.
  Rev. E \textbf{82} (2010), no.~6, 062102.

\end{thebibliography}
\bibliographystyle{amsplain}

\appendix

\printnoidxglossary[title=List of symbols, toctitle=List of symbols, sort=standard]

\end{document}